\documentclass{amsbook}

\newtheorem{thm}{Theorem}[section]
\newtheorem*{thm*}{Theorem}

\newtheorem{claim}[thm]{Claim}

\newtheorem{cor}[thm]{Corollary}

\newtheorem{lem}[thm]{Lemma}
\newtheorem*{lem*}{Lemma}
\newtheorem{mainthm}{Theorem}
\newtheorem*{mainthm*}{Theorem}
\newtheorem{maincor}[mainthm]{Corollary}
\newtheorem{prop}[thm]{Proposition}

\theoremstyle{definition}

\newtheorem*{case*}{Case}

\newtheorem{defn}[thm]{Definition}
\newtheorem*{defn*}{Definition}

\newtheorem*{exmp*}{Example}

\newtheorem{hyp}[thm]{Hypothesis}

\newtheorem{step}{Step}\renewcommand{\thestep}{}

\theoremstyle{remark}

\newtheorem{case}{Case}\renewcommand{\thecase}{}

\newtheorem{rmk}[thm]{Remark}
\newtheorem*{rmk*}{Remark}

\makeatletter
\def\alphenumi{
  \def\theenumi{\alph{enumi}}
  \def\p@enumi{\theenumi}
  \def\labelenumi{(\@alph\c@enumi)}}
\makeatother

\makeatletter
\def\thecase{\@arabic\c@case}
\makeatother

\makeatletter
\def\thestep{\@arabic\c@step}
\makeatother

\newcount\hh
\newcount\mm
\mm=\time
\hh=\time
\divide\hh by 60
\divide\mm by 60
\multiply\mm by 60
\mm=-\mm
\advance\mm by \time
\def\hhmm{\number\hh:\ifnum\mm<10{}0\fi\number\mm}

\let\oldmarginpar\marginpar
\renewcommand\marginpar[1]{\-\oldmarginpar[\raggedleft\footnotesize #1]%
{\raggedright\footnotesize #1}}

\newcommand\CC{\mathbb{C}}

\newcommand\EE{\mathbb{E}}

\newcommand\KK{\mathbb{K}}

\newcommand\NN{\mathbb{N}}

\newcommand\RR{\mathbb{R}}

\newcommand\VV{\mathbb{V}}

\newcommand\ZZ{\mathbb{Z}}

\newcommand\cX{{\mathcal{X}}}
\newcommand\cY{{\mathcal{Y}}}
\newcommand\cZ{{\mathcal{Z}}}

\newcommand\fg{{\mathfrak{g}}}

\newcommand\fu{{\mathfrak{u}}}

\newcommand\fX{{\mathfrak{X}}}

\newcommand\sA{{\mathscr{A}}}
\newcommand\sB{{\mathscr{B}}}
\newcommand\sC{{\mathscr{C}}}
\newcommand\sD{{\mathscr{D}}}
\newcommand\sE{{\mathscr{E}}}
\newcommand\sF{{\mathscr{F}}}
\newcommand\sG{{\mathscr{G}}}
\newcommand\sH{{\mathscr{H}}}

\newcommand\sL{{\mathscr{L}}}
\newcommand\sM{{\mathscr{M}}}

\newcommand\sO{{\mathscr{O}}}
\newcommand\sP{{\mathscr{P}}}

\newcommand\sR{{\mathscr{R}}}

\newcommand\sU{{\mathscr{U}}}
\newcommand\sV{{\mathscr{V}}}

\newcommand\sX{{\mathscr{X}}}
\newcommand\sY{{\mathscr{Y}}}
\newcommand\sZ{{\mathscr{Z}}}

\newcommand\bchi{{\boldsymbol{\chi}}}

\newcommand\bB{{\mathbf{B}}}

\newcommand\eps{\varepsilon}

\newcommand\su{{\mathfrak{s}\mathfrak{u}}}

\newcommand\SO{\operatorname{SO}}

\newcommand\U{\operatorname{U}}

\newcommand\less{\setminus}

\newcommand\ad{{\operatorname{ad}}}
\newcommand\Ad{{\operatorname{Ad}}}

\newcommand\Aut{\operatorname{Aut}}

\newcommand\Center{\operatorname{Center}}

\newcommand\dist{\operatorname{dist}}

\newcommand\End{\operatorname{End}}

\newcommand{\esssup}{\operatornamewithlimits{ess\ sup}}
\newcommand\Exp{\operatorname{Exp}}

\newcommand\Hom{\operatorname{Hom}}

\newcommand\Ind{\operatorname{Index}}

\newcommand\Ker{\operatorname{Ker}}

\newcommand\Real{\operatorname{Re}}
\newcommand\Ric{\operatorname{Ric}}

\newcommand\Ran{\operatorname{Ran}}

\newcommand\Stab{\operatorname{Stab}}

\newcommand\tr{\operatorname{tr}}

\newcommand\vol{\operatorname{vol}}
\newcommand\Vol{\operatorname{Vol}}

\newcommand\apriori{{\emph{a priori }}}
\newcommand\Apriori{{\emph{A priori }}}

\newcommand\id{{\mathrm{id}}}

\newcommand\loc{{\mathrm{loc}}}

\newcommand\mutatis{{\emph{mutatis mutandis }}}

\newcommand\spinc{\text{$\text{spin}^c$ }}

\newcommand\Spinc{\text{$\text{Spin}^c$}}

\numberwithin{equation}{section}
\numberwithin{section}{chapter}

\usepackage{amssymb}
\usepackage{graphicx}
\usepackage{hyperref}
\usepackage{mathrsfs}
\usepackage[usenames]{color}
\usepackage{paralist}
\usepackage{slashed}
\usepackage{url}
\usepackage{verbatim}
\usepackage{xr}

\hypersetup{pdftitle={{\L}ojasiewicz--Simon gradient inequalities for coupled Yang--Mills energy functions}}
\hypersetup{pdfauthor={Paul M. N. Feehan and Manousos Maridakis}}

\begin{document}

\frontmatter

\title[{\L}ojasiewicz--Simon gradient inequalities]{{\L}ojasiewicz--Simon gradient inequalities for coupled Yang--Mills energy functions}

\author[Paul M. N. Feehan]{Paul M. N. Feehan}
\author[Manousos Maridakis]{Manousos Maridakis}


\dedicatory{Paul Feehan dedicates this monograph to the memory of his parents, Martin and Odilla Feehan.\\[2pt]
Manousos Maridakis dedicates this monograph to his parents, Petros and Pavlina Maridakis.}




\maketitle

\setcounter{page}{7}

\tableofcontents

\chapter*{Preface}
\label{chap:Preface}
Our primary goal in this monograph is to prove \emph{{\L}ojasiewicz--Simon gradient inequalities} for coupled Yang--Mills energy functions using Sobolev spaces that impose \emph{minimal regularity requirements} on pairs of connections and sections. Our {\L}ojasiewicz--Simon gradient inequalities for coupled Yang--Mills energy functions generalize that of the pure Yang--Mills energy function due to the first author \cite[Theorems 23.1 and 23.17]{Feehan_yang_mills_gradient_flow_v4} for base manifolds of arbitrary dimension and due to R\r{a}de \cite[Proposition 7.2]{Rade_1992} for dimensions two and three.

Our {\L}ojasiewicz--Simon gradient inequalities for coupled Yang--Mills energy functions (for example, Theorems \ref{mainthm:Lojasiewicz-Simon_gradient_inequality_boson_Yang--Mills_energy_function} and \ref{mainthm:Lojasiewicz-Simon_gradient_inequality_fermion_Yang--Mills_energy_function} and the many other examples in Section \ref{sec:Lojasiewicz-Simon_gradient_inequality_coupled Yang--Mills_function}) are proved by applying the {\L}ojasiewicz--Simon gradient inequality for an abstract analytic function on a Banach space given by Theorem \ref{mainthm:Lojasiewicz-Simon_gradient_inequality} from our article \cite{Feehan_Maridakis_Lojasiewicz-Simon_harmonic_maps_v6}. Now Theorem \ref{mainthm:Lojasiewicz-Simon_gradient_inequality} requires that the Hessian operator for the analytic function be \emph{Fredholm with index zero}. While the Hessian operator for the coupled Yang--Mills energy function is a linear second-order partial differential operator, it only becomes \emph{elliptic} when combined with a Coulomb gauge condition \cite{DK, FU}. Coupled Yang--Mills energy functions are invariant under the action of gauge transformations (or bundle automorphisms) and so, in principle, one can always find a gauge transformation to produce the required Coulomb gauge condition with the aid of a \emph{slice theorem}. However, in order to prove the most useful version of the {\L}ojasiewicz--Simon gradient inequality, one must have a stronger version of the slice theorem for the action of the group of gauge transformations that goes beyond the usual statements found in standard references such as Donaldson and Kronheimer \cite{DK} or Freed and Uhlenbeck \cite{FU} for connections over four-dimensional manifolds and proved by applying the Implicit Function Theorem. Therefore, a secondary goal of this monograph is to prove a slice theorem using Sobolev norms with \emph{borderline} (or \emph{critical}) Sobolev exponents and which is valid in all dimensions.

Since its discovery by {\L}ojasiewicz in the context of analytic functions on Euclidean space \cite{Lojasiewicz_1965} and generalization by Simon to a class of analytic functions on certain H\"older spaces \cite[Theorem 3]{Simon_1983}, {\L}ojasiewicz--Simon gradient inequalities have played a major role in analyzing questions such as
\begin{inparaenum}[\itshape a\upshape)]
\item global existence, convergence, and analysis of singularities for solutions to nonlinear evolution equations that are realizable as gradient or gradient-like systems for an energy function,
\item uniqueness of tangent cones, and
\item energy gaps and discreteness of energies.
\end{inparaenum}

There are essentially four approaches to establishing a {\L}ojasiewicz--Simon gradient inequality for a particular energy function arising in geometric analysis or mathematical physics:
\begin{inparaenum}[(1)]
\item \label{item:LS_first_principles} establish the inequality from first principles without relying on the {\L}ojasiewicz inequality for analytic functions on Euclidean space,
\item \label{item:LS_Lyapunov-Schmidt_reduction_to_Lojasiewicz} employ Lyapunov--Schmidt reduction to deduce the gradient inequality for an analytic function on a Banach space from the {\L}ojasiewicz inequality for analytic functions on Euclidean space,
\item \label{item:Adapt_Simon} adapt the argument employed by Simon in the proof of his \cite[Theorem 3]{Simon_1983}, or
\item \label{item:Apply_abstract_LS} apply an abstract version of the {\L}ojasiewicz--Simon gradient inequality for an analytic or Morse--Bott function on a Banach space.
\end{inparaenum}
Approach \eqref{item:LS_Lyapunov-Schmidt_reduction_to_Lojasiewicz} is exactly that employed by Simon in \cite{Simon_1983} and by R\r{a}de for the Yang-Mills energy function \cite{Rade_1992}. Occasionally a development from first principles may be necessary, as discussed by Colding and Minicozzi in \cite{Colding_Minicozzi_2014sdg}, or may yield the best results, as discussed by the first author in \cite{Feehan_lojasiewicz_inequality_ground_state}. However, in most cases one can derive a {\L}ojasiewicz--Simon gradient inequality for a specific energy function from an abstract version for an analytic or Morse--Bott function on a Banach space. For this strategy to work well, one desires an abstract {\L}ojasiewicz--Simon gradient inequality with the weakest possible hypotheses and proofs of such gradient inequalities were provided by the authors in our article \cite[Theorems 1--4]{Feehan_Maridakis_Lojasiewicz-Simon_harmonic_maps_v6}.

Appendix \ref{chap:Convergence_gradient_flows_validity_Lojasiewicz-Simon_gradient_inequality} describes the application of {\L}ojasiewicz--Simon gradient inequalities to long-time existence and convergence for solutions to gradient systems for an analytic function while our monograph \cite{Feehan_yang_mills_gradient_flow_v4} provides a comprehensive development of these ideas. Given a constant $T > 0$, an open subset $\sU$ of a Banach space $\sX$, a smooth function $\sE:\sU\to\RR$, and a point $x_0 \in \sU$, a smooth map, $u:[0, T) \to \sX$, is called a \emph{gradient flow} for $\sE$ if it is a solution to the Cauchy problem for the \emph{gradient system},
\[
\dot u(t) = -\sE'(u(t)), \quad\text{for all } t\in [0,T), \quad u(0) = x_0,
\]
as an identity in $\sX^*$ (the continuous dual space of $\sX$), where we abbreviate $\dot u = dt/dt$ and $\sE'(x):\sX \to \sX^*$ is the differential of $\sE$ at a point $x\in\sU$. The best known examples of gradient flows occurring in geometric analysis include pure and coupled Yang-Mills flows, harmonic map flow, Ricci curvature flow, mean curvature flow, and Yamabe scalar curvature flow.

Simon's approach \cite{Simon_1983} to studying long-time existence and convergence of a solution to the preceding gradient system relies on his celebrated generalization to infinite dimensions of the {\L}ojasiewicz gradient inequality to a specific class of analytic energy functions on $C^{2,\alpha}$ H\"older spaces of sections of a vector bundle over a closed, smooth Riemannian manifold. Over the intervening years, his {\L}ojasiewicz--Simon gradient inequality has since been generalized by many authors --- see our article \cite{Feehan_Maridakis_Lojasiewicz-Simon_harmonic_maps_v6} and references cited therein and our Theorems \ref{mainthm:Lojasiewicz-Simon_gradient_inequality}, \ref{mainthm:Lojasiewicz-Simon_gradient_inequality_dualspace}, and \ref{mainthm:Lojasiewicz-Simon_gradient_inequality2} in this monograph. For example, if $\sX$ is continuously embedded in a Hilbert space $\sH$ and $\sE$ is analytic and $x_\infty \in \sX$ is a critical point such that the Hessian operator, $\sE''(x_\infty):\sX\to\sX^*$, is Fredholm with index zero, then there exist constants $Z \in (0,\infty)$, $\sigma \in (0,1]$, and $\theta \in [1/2, 1)$ such that gradient map obeys \cite[Theorem 1]{Feehan_Maridakis_Lojasiewicz-Simon_harmonic_maps_v6}
\[
\|\sE'(x)\|_{\sX^*} \geq Z|\sE(x) - \sE(x_\infty)|^\theta,
\quad\text{for all } x \in \sX \text{ such that } \|x-x_\infty\|_\sX < \sigma.
\]
The preceding inequality is precisely that asserted by Theorem \ref{mainthm:Lojasiewicz-Simon_gradient_inequality_dualspace} here. We shall apply our more general Theorems \ref{mainthm:Lojasiewicz-Simon_gradient_inequality} and \ref{mainthm:Lojasiewicz-Simon_gradient_inequality2} to derive {\L}ojasiewicz--Simon gradient inequalities for all of the energy functions considered in this monograph, namely the
\begin{inparaenum}[\itshape a\upshape)]
\item pure Yang--Mills energy function,
\item boson coupled Yang--Mills energy function,
\item fermion coupled Yang--Mills energy function,
\item Yang--Mills--Higgs energy function,
\item Seiberg--Witten energy function,
\item non-Abelian monopole energy function, and the
\item multiple spinor Seiberg--Witten energy function.
\end{inparaenum}

\chapter*{Acknowledgments}
\label{chap:Acknowledgments}
Paul Feehan is very grateful to the Max Planck Institute for Mathematics, Bonn, and the Institute for Advanced Study, Princeton, for their support during the preparation of this monograph. He would like to thank Peter Tak{\'a}{\v{c}} for many helpful conversations regarding the {\L}ojasiewicz--Simon gradient inequality, for explaining his proof of \cite[Proposition 6.1]{Feireisl_Takac_2001} and how it can be generalized as described in this monograph, and for his kindness when hosting his visit to the Universit{\"a}t R{\"o}stock. He would also like to thank Brendan Owens for several useful conversations and his generosity when hosting his visit to the University of Glasgow. He thanks Brendan Owens and Chris Woodward for helpful communications and comments regarding Morse--Bott theory, Alessandro Carlotto for useful comments regarding the integrability of critical points of the Yamabe function. We thank Thomas Parker and Penny Smith for helpful questions and comments.
\bigskip
\bigskip

\rightline{January 14, 2019\footnote{To appear in \emph{Memoirs of the American Mathematical Society.}}}
\bigskip

\leftline{Paul M. N. Feehan and Manousos Maridakis}
\leftline{Department of Mathematics}
\leftline{Rutgers, The State University of New Jersey}
\leftline{Piscataway, NJ 08854-8019}
\leftline{United States}





\mainmatter

\chapter{Introduction}
\label{chap:Introduction}
Our primary goal in this work is to prove {\L}ojasiewicz--Simon gradient inequalities for coupled Yang--Mills energy functions. A key feature of our results is that we use systems of Sobolev norms that are as \emph{weak as possible}. This property is very useful in applications to the analysis of gradient flows, the primary example of an application of {\L}ojasiewicz--Simon gradient inequalities in geometric analysis as illustrated by results of the first author in \cite{Feehan_yang_mills_gradient_flow_v4}. Our gradient inequalities use $W^{1,p}$ Sobolev norms for coupled Yang--Mills pairs over manifolds of arbitrary dimension $d \geq 2$, including the case $p=d/2$, where the Sobolev exponent is \emph{borderline} (or \emph{critical}) in sense that we explain later in this Introduction.

In the remainder of our Introduction, we outline the history of {\L}ojasiewicz--Simon gradient inequalities in Section \ref{sec:Lojasiewicz-Simon_history} and survey their applications in geometric analysis, mathematical physics, and applied mathematics. In Section \ref{sec:Lojasiewicz-Simon_gradient_inequality_abstract_function}, we review our abstract {\L}ojasiewicz--Simon gradient inequality for an analytic function on a Banach space. We state our results on {\L}ojasiewicz--Simon gradient inequalities for coupled Yang--Mills energy functions in Section \ref{sec:Lojasiewicz-Simon_gradient_inequality_coupled Yang--Mills_function}. Unlike the case of the harmonic map energy function considered in \cite{Feehan_Maridakis_Lojasiewicz-Simon_harmonic_maps_v6}, one must restrict the Hessian of a coupled Yang--Mills energy function to a suitable slice for the action of the group of gauge transformations in order to obtain an elliptic operator that has the Fredholm property required by our abstract {\L}ojasiewicz--Simon gradient inequality \cite[Theorem 2]{Feehan_Maridakis_Lojasiewicz-Simon_harmonic_maps_v6}. In order to obtain the strongest possible version of the resulting {\L}ojasiewicz--Simon gradient inequality for a coupled Yang--Mills energy function, we therefore need to prove existence of a global transformation to Coulomb gauge valid for borderline Sobolev exponents --- going beyond standard results described in \cite{DK, FU} or previous results due to the first author \cite{FeehanSlice} --- and we state the required theorem in Section \ref{sec:Automorphisms and transformation to Coulomb gauge}.

\section[A brief history of {\L}ojasiewicz--Simon gradient inequalities]{A brief history of {\L}ojasiewicz--Simon gradient inequalities and their applications to gradient flows and energy gaps in geometric analysis}
\label{sec:Lojasiewicz-Simon_history}
Since its discovery by {\L}ojasiewicz in the context of analytic functions on Euclidean space \cite[Proposition 1, p. 92]{Lojasiewicz_1965} and subsequent generalization by Simon to a class of analytic functions on certain H\"older spaces \cite[Theorem 3]{Simon_1983}, the \emph{{\L}ojasiewicz--Simon gradient inequality} has played a significant role in analyzing questions such as
\begin{inparaenum}[\itshape a\upshape)]
\item global existence, convergence, and analysis of singularities for solutions to nonlinear evolution equations that are realizable as gradient-like systems for an energy function,
\item uniqueness of tangent cones, and
\item energy gaps and discreteness of energies.
\end{inparaenum}
For applications of the {\L}ojasiewicz--Simon gradient inequality to gradient flows arising in geometric analysis, beginning with the harmonic map energy function, we refer to Irwin \cite{IrwinThesis}, Kwon \cite{KwonThesis}, Liu and Yang \cite{Liu_Yang_2010}, Simon \cite{Simon_1985}, and Topping \cite{ToppingThesis, Topping_1997}; for applications to gradient flow for the Chern--Simons function, see Morgan, Mrowka, and Ruberman \cite{MMR}; for applications to gradient flow for the Yamabe function, see Brendle \cite[Lemma 6.5 and Equation (100)]{Brendle_2005} and Carlotto, Chodosh, and Rubinstein \cite{Carlotto_Chodosh_Rubinstein_2015}; for applications to Yang--Mills gradient flow, we refer to our monograph \cite{Feehan_yang_mills_gradient_flow_v4}, R\r{a}de \cite{Rade_1992}, and Yang \cite{Yang_2003aim}; for applications to mean curvature flow, we refer to the survey by Colding and Minicozzi \cite{Colding_Minicozzi_2014sdg}; and for applications to Ricci curvature flow, see Ache \cite{Ache_2011arxiv}, Haslhofer \cite{Haslhofer_2012cvpde}, Haslhofer and M{\"u}ller \cite{Haslhofer_Muller_2014}, and Kr{\"o}ncke \cite{Kroncke_2015cvpde, Kroncke_2013arxiv}.

For applications of the {\L}ojasiewicz--Simon gradient inequality to proofs of global existence, convergence, convergence rate, and stability of nonlinear evolution equations arising in other areas of mathematical physics (including the Cahn--Hilliard, Ginzburg--Landau, Kirchoff--Carrier, porous medium, reaction-diffusion, and semilinear heat and wave equations), we refer to the monograph by Huang \cite{Huang_2006} for a comprehensive introduction and to the articles by Chill \cite{Chill_2003, Chill_2006}, Chill and Fiorenza \cite{Chill_Fiorenza_2006}, Chill, Haraux, and Jendoubi \cite{Chill_Haraux_Jendoubi_2009}, Chill and Jendoubi \cite{Chill_Jendoubi_2003, Chill_Jendoubi_2007}, Feireisl and Simondon \cite{Feireisl_Simondon_2000}, Feireisl and Tak{\'a}{\v{c}} \cite{Feireisl_Takac_2001}, Grasselli, Wu, and Zheng \cite{Grasselli_Wu_Zheng_2009}, Haraux \cite{Haraux_2012}, Haraux and Jendoubi \cite{Haraux_Jendoubi_1998, Haraux_Jendoubi_2007, Haraux_Jendoubi_2011}, Haraux, Jendoubi, and Kavian \cite{Haraux_Jendoubi_Kavian_2003}, Huang and Tak{\'a}{\v{c}} \cite{Huang_Takac_2001}, Jendoubi \cite{Jendoubi_1998jfa}, Rybka and Hoffmann \cite{Rybka_Hoffmann_1998, Rybka_Hoffmann_1999}, Simon \cite{Simon_1983}, and Tak{\'a}{\v{c}} \cite{Takac_2000}. For applications to fluid dynamics, see the articles by Feireisl, Lauren{\c{c}}ot, and Petzeltov{\'a} \cite{Feireisl_Laurencot_Petzeltova_2007}, Frigeri, Grasselli, and Krej{\v{c}}{\'{\i}} \cite{Frigeri_Grasselli_Krejcic_2013}, Grasselli and Wu \cite{Grasselli_Wu_2013}, and Wu and Xu \cite{Wu_Xu_2013}.

For applications of the {\L}ojasiewicz--Simon gradient inequality to proofs of energy gaps and discreteness of energies for Yang--Mills connections, we refer to our article \cite{Feehan_yangmillsenergygapflat}. A key feature of our versions of the {\L}ojasiewicz--Simon gradient inequality for the pure Yang--Mills energy function \cite[Theorems 23.1 and 23.17]{Feehan_yang_mills_gradient_flow_v4} is that they hold for $W^{1,2}$ Sobolev norms for base manifolds of dimensions two, three or four and $W^{2,p}$ Sobolev norms for base manifolds of arbitrary dimension. Those norms are considerably weaker than the $C^{2,\alpha}$ H{\"o}lder norms originally employed by Simon in \cite[Theorem 3]{Simon_1983} and this affords considerably greater flexibility in applications. For example, when $(X,g)$ is a closed, four-dimensional, Riemannian manifold, the $W^{1,2}$ Sobolev norm on (bundle-valued) one-forms is (in a suitable sense) \emph{quasi-conformally invariant} with respect to conformal changes in the Riemannian metric $g$.

\section[{\L}ojasiewicz--Simon gradient inequalities for analytic functions]{{\L}ojasiewicz--Simon gradient inequalities for analytic functions on Banach spaces}
\label{sec:Lojasiewicz-Simon_gradient_inequality_abstract_function}
There are essentially three approaches to establishing a {\L}ojasiewicz--Simon gradient inequality for a particular energy function arising in geometric analysis or mathematical physics:
\begin{inparaenum}[\itshape 1\upshape)]
\item establish the inequality from first principles,
\item adapt the argument employed by Simon in the proof of his
\cite[Theorem 3]{Simon_1983}, or
\item apply an abstract version of the {\L}ojasiewicz--Simon gradient inequality for an analytic or Morse--Bott function on a Banach space.
\end{inparaenum}
Most famously, the first approach is exactly that employed by Simon in \cite{Simon_1983}, although this is also the avenue followed by Kwon \cite{KwonThesis}, Liu and Yang \cite{Liu_Yang_2010} and Topping \cite{ToppingThesis, Topping_1997} for the harmonic map energy function and by R\r{a}de for the Yang--Mills energy function. Occasionally a development from first principles may be necessary, as discussed by Colding and Minicozzi in \cite{Colding_Minicozzi_2014sdg}. However, in almost all of the examples cited in Section \ref{sec:Lojasiewicz-Simon_history}, one can derive a {\L}ojasiewicz--Simon gradient inequality for a specific application from an abstract version for an analytic or Morse--Bott function on a Banach space. For this strategy to work well, one desires an abstract {\L}ojasiewicz--Simon gradient inequality with the weakest possible hypotheses and a proof of such a gradient inequality (quoted as Theorem \ref{mainthm:Lojasiewicz-Simon_gradient_inequality} here) was the one of the goals of our article \cite{Feehan_Maridakis_Lojasiewicz-Simon_harmonic_maps_v6}.

We now recall from \cite{Feehan_Maridakis_Lojasiewicz-Simon_harmonic_maps_v6} a generalization of Simon's infinite-dimensional version \cite[Theorem 3]{Simon_1983} of the {\L}ojasiewicz gradient inequality \cite{Lojasiewicz_1965}. As we explained in detail in \cite{Feehan_Maridakis_Lojasiewicz-Simon_harmonic_maps_v6}, Theorem \ref{mainthm:Lojasiewicz-Simon_gradient_inequality} generalizes Huang's \cite[Theorems 2.4.2 (i) and 2.4.5]{Huang_2006} and other previously published versions of the {\L}ojasiewicz--Simon gradient inequality for analytic functions on Banach spaces.

We begin with the concept of a gradient map \cite[Section 2.1B]{Huang_2006}, \cite[Section 2.5]{Berger_1977}.

\begin{defn}[Gradient map]
\label{defn:Huang_2-1-1}
(See \cite[Definition 2.1.1]{Huang_2006}.)
Let $\sU\subset \sX$ be an open subset of a Banach space, $\sX$, and let $\tilde\sX$ be a Banach space with continuous embedding, $\tilde\sX \subseteqq \sX^*$. A continuous map, $\sM:\sU\to \tilde\sX$, is called a \emph{gradient map} if there exists a $C^1$ function, $\sE:\sU\to\RR$, such that
\begin{equation}
\label{eq:Differential_and_gradient_maps}
\sE'(x)v = \langle v,\sM(x)\rangle_{\sX\times\sX^*}, \quad \forall\, x \in \sU, \quad v \in \sX,
\end{equation}
where $\langle \cdot , \cdot \rangle_{\sX\times\sX^*}$ is the canonical bilinear form on $\sX\times\sX^*$. The real-valued function, $\sE$, is called a \emph{potential} for the gradient map, $\sM$.
\end{defn}

When $\tilde\sX = \sX^*$ in Definition \ref{defn:Huang_2-1-1}, then the differential and gradient maps coincide.

Let $\sX$ be a Banach space and let $\sX^*$ denote its continuous dual space. We call a bilinear form\footnote{Unless stated otherwise, all Banach spaces are considered to be real in this monograph.}, $b:\sX\times\sX \to \RR$, \emph{definite} if $b(x,x) \neq 0$ for all $x \in \sX\less\{0\}$. We say that a continuous \emph{embedding} of a Banach space into its continuous dual space, $\jmath:\sX\to\sX^*$, is \emph{definite} if the pullback of the canonical pairing, $\sX\times\sX \ni (x,y) \mapsto \langle x,\jmath(y)\rangle_{\sX\times\sX^*} \to \RR$, is a definite bilinear form.

\begin{mainthm}[{\L}ojasiewicz--Simon gradient inequality for analytic functions on Banach spaces]
\label{mainthm:Lojasiewicz-Simon_gradient_inequality}
(See \cite[Theorem 2]{Feehan_Maridakis_Lojasiewicz-Simon_harmonic_maps_v6}.)
Let $\sX$ and $\tilde\sX$ be Banach spaces with continuous embeddings, $\sX \subset \tilde\sX \subset \sX^*$, and such that the embedding, $\sX \subset \sX^*$, is definite. Let $\sU \subset \sX$ be an open subset, $\sE:\sU\to\RR$ be a $C^2$ function with real analytic gradient map, $\sM:\sU\to\tilde\sX$, and $x_\infty\in\sU$ be a critical point of $\sE$, that is, $\sM(x_\infty) = 0$. If $\sM'(x_\infty):\sX\to \tilde\sX$ is a Fredholm operator with index zero, then there are constants, $Z \in (0,\infty)$, and $\sigma \in (0,1]$, and $\theta \in [1/2, 1)$, with the following significance. If $x \in \sU$ obeys
\begin{equation}
\label{eq:Lojasiewicz-Simon_gradient_inequality_neighborhood}
\|x-x_\infty\|_\sX < \sigma,
\end{equation}
then
\begin{equation}
\label{eq:Lojasiewicz-Simon_gradient_inequality_analytic_function}
\|\sM(x)\|_{\tilde\sX} \geq Z|\sE(x) - \sE(x_\infty)|^\theta.
\end{equation}
\end{mainthm}

\begin{rmk}[Comments on the embedding hypothesis in Theorem \ref{mainthm:Lojasiewicz-Simon_gradient_inequality}]
\label{rmk:Embedding_hypothesis_Huang_theorem_2-4-5}
The hypothesis in Theorem \ref{mainthm:Lojasiewicz-Simon_gradient_inequality} on the continuous embedding, $\sX \subset \sX^*$, is easily achieved given a continuous embedding of $\sX$ into a Hilbert space $\sH$.
\end{rmk}

\begin{rmk}[On the choice of Banach spaces in applications of Theorem \ref{mainthm:Lojasiewicz-Simon_gradient_inequality}]
\label{rmk:Choice_Banach_and_Hilbert_spaces_Lojasiewicz-Simon_gradient_inequality}
The hypotheses of Theorem \ref{mainthm:Lojasiewicz-Simon_gradient_inequality} are designed to give the most flexibility in applications of a {\L}ojasiewicz--Simon gradient inequality to analytic functions on Banach spaces. An example of a convenient choice of Banach spaces modeled as Sobolev spaces, when $\sM'(x_\infty)$ is realized as an elliptic partial differential operator of order $m$, would be
\[
\sX = W^{k,p}(X;V), \quad \tilde\sX = W^{k-m,p}(X;V), \quad\text{and}\quad \sX^* = W^{-k,p'}(X;V),
\]
where $k\in\ZZ$ is an integer, $p \in (1,\infty)$ is a constant with dual H\"older exponent $p'\in(1,\infty)$ defined by $1/p+1/p'=1$, while $X$ is a closed Riemannian manifold of dimension $d\geq 2$ and $V$ is a Riemannian vector bundle with a compatible connection, $\nabla:C^\infty(X;V) \to C^\infty(X;T^*X\otimes V)$, and $W^{k,p}(X;V)$ denotes a Sobolev space defined in the standard way \cite{Aubin_1998}. When the integer $k$ is chosen large enough, the verification of analyticity of the gradient map, $\sM:\sU\to\tilde\sX$, is straightforward. Normally, that is the case when $k\geq m+1$ and $(k-m)p>d$ or $k-m=d$ and $p=1$, since $W^{k-m,p}(X;\CC)$ is then a Banach algebra by \cite[Theorem 4.39]{AdamsFournier}. If the Banach spaces are instead modeled as H\"older spaces, as in Simon \cite{Simon_1983}, a convenient choice of Banach spaces would be
\[
\sX = C^{k,\alpha}(X;V) \quad \text{and} \quad \tilde\sX = C^{k-m,\alpha}(X;V),
\]
where $\alpha \in (0,1)$ and $k\geq m$, and these H\"older spaces are defined in the standard way \cite{Aubin_1998}. Following Remark~\ref{rmk:Embedding_hypothesis_Huang_theorem_2-4-5}, the definiteness of the embedding  $C^{k,\alpha}(X;V)=\sX \subset \sX^*$ in this case  is the achieved by observing that $C^{k,\alpha}(X;V) \subset L^2(X;V)$.
\end{rmk}

We refer the reader to \cite[Theorem 4]{Feehan_Maridakis_Lojasiewicz-Simon_harmonic_maps_v6} for a statement and proof of our abstract {\L}ojasiewicz--Simon gradient inequality for Morse--Bott functions on Banach spaces.

Theorem \ref{mainthm:Lojasiewicz-Simon_gradient_inequality} appears to us to be the most widely applicable abstract version of the {\L}ojasiewicz--Simon gradient inequality that we are aware of in the literature. However, for applications where $\sM'(x_\infty)$ is realized as an elliptic partial differential operator of \emph{even} order, $m=2n$, and the nonlinearity of the gradient map is sufficiently mild, it often suffices to choose $\sX$ to be the Banach space, $W^{n,2}(X;V)$, and choose $\tilde\sX = \sX^*$ to be the Banach space, $W^{-n,2}(X;V)$. The distinction between the differential, $\sE'(x) \in \sX^*$, and the gradient, $\sM(x) \in \tilde\sX$, then disappears. Similarly, the distinction between the Hessian, $\sE''(x_\infty) \in (\sX\times\sX)^*$, and the Hessian operator, $\sM'(x_\infty) \in \sL(\sX,\tilde\sX)$, disappears. Finally, if $\sE:\sX\supset\sU \to \RR$ is real analytic, then the simpler Theorem \ref{mainthm:Lojasiewicz-Simon_gradient_inequality_dualspace} is often adequate for applications.

\begin{mainthm}[{\L}ojasiewicz--Simon gradient inequality for analytic functions on Banach spaces]
\label{mainthm:Lojasiewicz-Simon_gradient_inequality_dualspace}
(See \cite[Theorem 1]{Feehan_Maridakis_Lojasiewicz-Simon_harmonic_maps_v6}.)
Let $\sX \subset \sX^*$ be a continuous, definite embedding of a Banach space into its dual space. Let $\sU \subset \sX$ be an open subset, $\sE:\sU\to\RR$ be an analytic function, and $x_\infty\in\sU$ be a critical point of $\sE$, that is, $\sE'(x_\infty) = 0$. Assume that $\sE''(x_\infty):\sX\to \sX^*$ is a Fredholm operator with index zero. Then there are constants $Z \in (0, \infty)$, and $\sigma \in (0,1]$, and $\theta \in [1/2,1)$, with the following significance. If $x \in \sU$ obeys
\begin{equation}
\label{eq:Lojasiewicz-Simon_gradient_inequality_neighborhood_dualspace}
\|x-x_\infty\|_\sX < \sigma,
\end{equation}
then
\begin{equation}
\label{eq:Lojasiewicz-Simon_gradient_inequality_analytic_function_dualspace}
\|\sE'(x)\|_{\sX^*} \geq Z|\sE(x) - \sE(x_\infty)|^\theta.
\end{equation}
\end{mainthm}

While Theorem \ref{mainthm:Lojasiewicz-Simon_gradient_inequality} has important applications to proofs of global existence, convergence, convergence rates, and stability of gradient flows defined by an energy function, $\sE:\sX\supset \sU \to \RR$, with gradient map, $\sM:\sX\supset \sU \to \tilde\sX$, (see \cite[Section 2.1]{Feehan_yang_mills_gradient_flow_v4} for an introduction and Simon \cite{Simon_1983} for his pioneering development), the gradient inequality \eqref{eq:Lojasiewicz-Simon_gradient_inequality_analytic_function} is most useful when it has the form,
\[
\|\sM(x)\|_{\sH} \geq Z|\sE(x) - \sE(x_\infty)|^\theta, \quad\forall\, x \in \sU \text{ with } \|x-x_\infty\|_\sX < \sigma,
\]
where $\sH$ is a Hilbert space and the Banach space, $\sX$, is a dense subspace of $\sH$ with continuous embedding, $\sX \subset \sH$, and so $\sH^* \subset \sX^*$ is also a continuous embedding.  We refer to Appendix \ref{chap:Convergence_gradient_flows_validity_Lojasiewicz-Simon_gradient_inequality} for applications of this version of the gradient inequality to prove convergence of gradient flows and to Feehan and Maridakis \cite[Section 1.2]{Feehan_Maridakis_Lojasiewicz-Simon_harmonic_maps_v6} for further discussion.

As we shall explain further in Section \ref{subsec:Lojasiewicz-Simon_gradient_inequality_boson_fermion_Yang--Mills_function_L2}, an $L^2$ gradient inequality for a coupled Yang--Mills energy function like \eqref{eq:Boson_Yang--Mills_energy_function} or \eqref{eq:Fermion_Yang--Mills_energy_function} does not follow from Theorem \ref{mainthm:Lojasiewicz-Simon_gradient_inequality} when $X$ has dimension $d \geq 5$. However, the desired $L^2$ gradient inequalities \emph{are} implied by the forthcoming Theorem \ref{mainthm:Lojasiewicz-Simon_gradient_inequality2}.

\begin{mainthm}[Generalized {\L}ojasiewicz--Simon gradient inequality for analytic functions on Banach spaces]
\label{mainthm:Lojasiewicz-Simon_gradient_inequality2}
(See \cite[Theorem 3]{Feehan_Maridakis_Lojasiewicz-Simon_harmonic_maps_v6}.)
Let $\sX$ and $\tilde\sX$ be Banach spaces with continuous embeddings, $\sX \subset \tilde\sX \subset \sX^*$, and such that the embedding, $\sX \subset \sX^*$, is definite.  Let $\sU \subset \sX$ be an open subset, $\sE:\sU\to\RR$ be an analytic function, and $x_\infty\in\sU$ be a critical point of $\sE$, that is, $\sE'(x_\infty) = 0$. Let
\[
\sX\subset \sG \subset \tilde\sG \quad \text{and} \quad \tilde\sX \subset \tilde\sG \subset \sX^*,
\]
be continuous embeddings of Banach spaces such that the compositions,
\[
\sX\subset \sG\subset \tilde\sG \quad \text{and}\quad \sX\subset \tilde\sX\subset \tilde\sG,
\]
induce the same embedding, $\sX \subset \tilde\sG$. Let $\sM:\sU\to\tilde\sX$ be a gradient map for $\sE$ in the sense of Definition \ref{defn:Huang_2-1-1}. Suppose that for each $x \in \sU$, the bounded, linear operator,
\[
\sM'(x): \sX \to \tilde \sX,
\]
has an extension
\[
\sM_1(x): \sG \to \tilde\sG
\]
such that the map
\[
\sU \ni x \mapsto \sM_1(x) \in \sL(\sG, \tilde\sG) \quad\hbox{is continuous}.
\]
If $\sM'(x_\infty):\sX\to \tilde\sX$ and $\sM_1(x_\infty):\sG\to \tilde\sG$ are Fredholm operators with index zero, then there are constants, $Z \in (0,\infty)$ and $\sigma \in (0,1]$ and $\theta \in [1/2, 1)$, with the following significance. If $x \in \sU$ obeys
\begin{equation}
\label{eq:Lojasiewicz-Simon_gradient_inequality_neighborhood_general2}
\|x-x_\infty\|_\sX < \sigma,
\end{equation}
then
\begin{equation}
\label{eq:Lojasiewicz-Simon_gradient_inequality_analytic_functional_general2}
\|\sM(x)\|_{\tilde\sG} \geq Z|\sE(x) - \sE(x_\infty)|^\theta.
\end{equation}
\end{mainthm}

\begin{rmk}[Generalized {\L}ojasiewicz--Simon gradient inequality for analytic functions on Banach spaces with gradient map valued in a Hilbert space]
\label{rmk:Lojasiewicz-Simon_gradient_inequality2_Hilbert}
Suppose now that $\tilde\sG = \sH$, a Hilbert space, so that the embedding $\sG\subset \sH$  in Theorem \ref{mainthm:Lojasiewicz-Simon_gradient_inequality2}, factors through $\sG\subset \sH\simeq \sH^* $ and therefore
\[
\sE'(x)v = \langle v, \sM(x) \rangle_{\sX\times\sX^*} = (v, \sM(x))_\sH, \quad\forall\, x \in \sU \text{ and } v \in \sX,
\]
using the continuous embeddings, $\tilde\sX \subset \sH \subset \sX^*$. As we noted in Remark \ref{rmk:Embedding_hypothesis_Huang_theorem_2-4-5}, the hypothesis in Theorem \ref{mainthm:Lojasiewicz-Simon_gradient_inequality2} that the embedding, $\sX \subset \sX^*$, is definite is implied by the assumption that $\sX \subset \sH$ is a continuous embedding into a Hilbert space. By Theorem \ref{mainthm:Lojasiewicz-Simon_gradient_inequality2}, if $x \in \sU$ obeys
\begin{equation}
\label{eq:Lojasiewicz-Simon_gradient_inequality_neighborhood_general_Hilbert_space}
\|x-x_\infty\|_\sX < \sigma,
\end{equation}
then
\begin{equation}
\label{eq:Lojasiewicz-Simon_gradient_inequality_analytic_function_Hilbert_space}
\|\sM(x)\|_{\sH} \geq Z|\sE(x) - \sE(x_\infty)|^\theta,
\end{equation}
as desired.
\end{rmk}

\section[{\L}ojasiewicz--Simon inequalities for coupled Yang--Mills energies]{{\L}ojasiewicz--Simon gradient inequalities for coupled Yang--Mills energy functions}
\label{sec:Lojasiewicz-Simon_gradient_inequality_coupled Yang--Mills_function}
In this subsection, we summarize consequences of Theorem \ref{mainthm:Lojasiewicz-Simon_gradient_inequality} for coupled Yang--Mills energy functions.

\subsection{{\L}ojasiewicz--Simon gradient inequalities for boson and fermion coupled Yang--Mills energy functions}
\label{subsec:Lojasiewicz-Simon_gradient_inequality_boson_fermion_Yang--Mills_function}
We begin with a definition (due to Parker \cite{ParkerGauge}) of two coupled Yang--Mills energy functions.

\begin{defn}[Boson and fermion coupled Yang--Mills energy functions]
\label{defn:Boson_and_fermion_coupled_Yang--Mills_energy_function}
\cite[Section 2]{ParkerGauge}
Let $(X,g)$ be a closed, smooth Riemannian manifold of dimension $d \geq 2$, and $G$ be a compact Lie group, $P$ be a smooth principal $G$-bundle over $X$, and $\EE$ be a complex finite-dimensional $G$-module equipped with a $G$-invariant Hermitian inner product, $\varrho: G \to \Aut_\CC(\EE)$ be a unitary representation \cite[Definitions 2.1.1 and 2.16]{BrockertomDieck}, and $E = P\times_\varrho\EE$ be a smooth Hermitian vector bundle over $X$, and $m$ and $s$ be smooth real-valued functions on $X$.

We define the \emph{boson coupled Yang--Mills energy function} by
\begin{equation}
\label{eq:Boson_Yang--Mills_energy_function}
\sE_g(A,\Phi)
:=
\frac{1}{2} \int_X \left(|F_A|^2 + |\nabla_A \Phi|^2 -  m|\Phi|^2 - s|\Phi|^4\right)
\,d\vol_g,
\end{equation}
for all smooth connections, $A$ on $P$, and smooth sections, $\Phi$ of $E$, where
\[
\nabla_A: C^\infty(X;E) \to C^\infty(T^*X\otimes E),
\]
is the covariant derivative induced on $E$ by the connection $A$ on $P$ and $F_A \in \Omega^2(X;\ad P)$ is the curvature of $A$ and $\ad P := P\times_{\ad}\fg$ denotes the real vector bundle associated to $P$ by the adjoint representation of $G$ on its Lie algebra,
$\Ad:G \ni u \to \Ad_u \in \Aut(\fg)$, with fiber metric defined through the Killing form on $\fg$.

Suppose that $X$ admits a \spinc structure comprising a Hermitian vector bundle $W$ over $X$ and a \emph{Clifford multiplication map}, $c:T^*X \to \End_\CC(W)$, thus
\begin{equation}
\label{eq:Clifford_multiplication}
c(\alpha)^2 = -g(\alpha,\alpha)\,\id_W, \quad\forall\, \alpha \in \Omega^1(X),
\end{equation}
and
\[
D_A := c\circ\nabla_A: C^\infty(X;W\otimes E) \to C^\infty(X;W\otimes E),
\]
is the corresponding \emph{Dirac operator} \cite[Appendix D]{LM}, \cite[Sections 1.1 and 1.2]{KMBook}, where $\nabla_A$ denotes the covariant derivative induced on $\otimes^n(T^*X)\otimes E$ (for $n \geq 0$) and $W\otimes E$ by the connection $A$ on $P$ and Levi-Civita connection for the metric $g$ on $TX$.

We define the \emph{fermion coupled Yang--Mills energy function} by
\begin{equation}
\label{eq:Fermion_Yang--Mills_energy_function}
\sF_g(A,\Psi)
:=
\frac{1}{2} \int_X \left(|F_A|^2 + \langle \Psi, D_A\Psi\rangle -  m|\Psi|^2\right)
\,d\vol_g,
\end{equation}
for all smooth connections, $A$ on $P$, and smooth sections, $\Psi$ of $W\otimes E$.
\end{defn}

We recall from \cite[Corollary D.4]{LM} that a closed orientable smooth manifold $X$ admits a \spinc structure if and only if the second Stiefel-Whitney class $w_2(X) \in H^2(X;\ZZ/2\ZZ)$ is the mod $2$ reduction of an integral class. One calls $W$ the \emph{fundamental spinor bundle} and it carries irreducible representations of $\Spinc(d)$; when $X$ is even-dimensional, there is a splitting $W = W^+\oplus W^-$ and Clifford multiplication restricts to give $\rho: T^*X \to \Hom_\CC(W^\pm, W^\mp)$ \cite[Definition D.9]{LM}.

Although initially defined for smooth connections and sections, the energy functions $\sE_g$ and $\sF_g$ in Definition \ref{defn:Boson_and_fermion_coupled_Yang--Mills_energy_function}, extend to the case of Sobolev connections and sections of class $W^{1,2}$.

A short calculation shows that the gradient of the boson coupled Yang--Mills energy function $\sE_g$ in \eqref{eq:Boson_Yang--Mills_energy_function} with respect to the $L^2$ metric on $C^\infty(X;\Lambda^1\otimes\ad P\oplus E)$,
\begin{equation}
\label{eq:Definition_gradient_boson_coupled_Yang--Mills_energy_function}
\left(\sM_g(A,\Phi), (a,\phi)\right)_{L^2(X,g)}
:=
\left.\frac{d}{dt}\sE_g(A+ta, \Phi+t\phi)\right|_{t=0}
=
\sE_g'(A,\Phi)(a,\phi),
\end{equation}
for all $(a,\phi) \in C^\infty(X;\Lambda^1\otimes\ad P\oplus E)$, is given by
\begin{multline}
\label{eq:Gradient_boson_coupled_Yang--Mills_energy_function}
\left(\sM_g(A,\Phi), (a,\phi)\right)_{L^2(X,g)}
\\
=
(d_A^*F_A, a)_{L^2(X)} + \Real (\nabla_A^*\nabla_A \Phi, \phi )_{L^2(X)} + \Real(\nabla_A\Phi, \rho(a)\Phi )_{L^2(X)}
\\
- \Real ( m\Phi, \phi )_{L^2(X)}
- 2\Real \int_X s|\Phi|^2 \langle \Phi,\phi\rangle\, d\vol_g,
\end{multline}
where $d_A^* = d_A^{*,g}: \Omega^l(X; \ad P) \to \Omega^{l-1}(X; \ad P)$ is the $L^2$ adjoint of the exterior covariant derivative $d_A:\Omega^l(X; \ad P) \to \Omega^{l+1}(X; \ad P)$, for integers $l\geq 0$. As customary, we let
\[
\Lambda^l = \Lambda^l(T^*X)
\]
denote the vector bundle over $X$ whose fiber $\Lambda^l(T_x^*X)$ over each point $x \in X$ is the $l$-th exterior power of the cotangent space, $T_x^*X$, with $\Lambda^0(T^*X) := X\times\RR$ and $\Lambda^1(T^*X) = T^*X$.

We call $(A,\Phi)$ a \emph{boson Yang--Mills pair} (with respect to the Riemannian metric $g$ on $X$) if it is a critical point for $\sE_g$, that is, $\sM_g(A,\Phi) = 0$.

Similarly, one finds that the gradient of the fermion coupled Yang--Mills energy function $\sF_g$ in \eqref{eq:Fermion_Yang--Mills_energy_function} with respect to the $L^2$ metric on $C^\infty(X;\Lambda^1\otimes\ad P\oplus W\otimes E)$,
\begin{multline}
\label{eq:Definition_gradient_fermion_coupled_Yang--Mills_energy_function}
\left(\sM_g(A,\Psi), (a,\psi)\right)_{L^2(X,g)}
:=
\left.\frac{d}{dt}\sF_g(A+ta, \Psi+t\psi)\right|_{t=0}
\\
= \sF_g'(A,\Psi)(a,\psi),
\end{multline}
for all $(a,\psi) \in C^\infty(X;\Lambda^1\otimes\ad P\oplus W\otimes E)$, is given by
\begin{multline}
\label{eq:Gradient_fermion_coupled_Yang--Mills_energy_function}
\left(\sM_g(A,\Psi), (a,\psi)\right)_{L^2(X,g)}
=
(d_A^*F_A, a)_{L^2(X)} + \Real(D_A \Psi - m\Psi, \psi )_{L^2(X)}
\\
+ \frac{1}{2}(\Psi, \rho(a)\Psi )_{L^2(X)},
\end{multline}
where the action of $a \in \Omega^1(X;\ad P) \equiv C^\infty(T^*X\otimes\ad P)$ on $\Psi \in C^\infty(X;W\otimes E)$ is defined by
\begin{multline*}
\rho(\alpha\otimes\xi)(\phi\otimes \eta)
:=
c(\alpha)\phi\otimes \varrho_*(\xi)\eta,
\\
\forall\, \alpha \in \Omega^1(X), \quad \xi \in C^\infty(X;\ad P),
\quad \phi \in C^\infty(X;W), \quad \eta \in C^\infty(X;E),
\end{multline*}
where $\varrho_*:\fg \to \End_\CC(\EE)$ is the representation of the Lie algebra induced by the representation $\varrho: G \to \End_\CC(\EE)$ of the Lie group.

We call $(A,\Psi)$ a \emph{fermion Yang--Mills pair} (with respect to the Riemannian metric $g$ on $X$) if it is a critical point for $\sF_g$, that is, $\sM_g(A,\Psi) = 0$.

Note that both the boson and fermion coupled Yang--Mills energy functions reduce to the pure \emph{Yang--Mills energy function} when $\Phi \equiv 0$ or $\Psi \equiv 0$, respectively,
\begin{equation}
\label{eq:Yang--Mills_energy_function}
\sE_g(A)  := \frac{1}{2}\int_X |F_A|^2\,d\vol_g,
\end{equation}
and $A$ is a \emph{Yang--Mills connection} (with respect to the Riemannian metric $g$ on $X$) if it is a critical point for $\sE_g$, that is,
\[
\sM_g(A) = d_A^{*,g}F_A = 0.
\]
Given a Hermitian or Riemannian vector bundle, $V$, over $X$ and covariant derivative, $\nabla_A$, which is compatible with the fiber metric on $V$, we denote the Banach space of sections of $V$ of Sobolev class $W^{k,p}$, for any $k\in \NN$ and $p \in [1,\infty]$, by $W_A^{k,p}(X; V)$, with norm,
\begin{equation}
\label{eq:Sobolev_norm_WAkp_sections_vector_bundle_over_manifold_finite_p}
\|v\|_{W_A^{k,p}(X)} := \left(\sum_{j=0}^k \int_X |\nabla_A^j v|^p\,d\vol_g \right)^{1/p},
\end{equation}
when $1\leq p<\infty$ and
\begin{equation}
\label{eq:Sobolev_norm_WAkp_sections_vector_bundle_over_manifold_infinite_p}
\|v\|_{W_A^{k,\infty}(X)} := \sum_{j=0}^k \esssup_X |\nabla_A^j v|,
\end{equation}
when $p=\infty$, where $v \in W_A^{k,p}(X; V)$. If $k=0$, then we denote $\|v\|_{W^{0,p}(X)} = \|v\|_{L^p(X)}$. For $p \in [1,\infty)$ and nonnegative integers $k$, we use \cite[Theorem 3.12]{AdamsFournier} (applied to $W_A^{k,p}(X;V)$ and noting that $X$ is a closed manifold) and Banach space duality to define
\[
W_A^{-k,p'}(X;V) := \left(W_A^{k,p}(X;V)\right)^*,
\]
where $p'\in (1,\infty]$ is the dual exponent defined by $1/p+1/p'=1$ and we use the fiber metric on $V$ to replace $V^*$ by $V$ on the left-hand side. Elements of the Banach space dual $(W_A^{k,p}(X;V))^*$ may be characterized via \cite[Section 3.10]{AdamsFournier} as distributions in the Schwartz space $\sD'(X;V)$ \cite[Section 1.57]{AdamsFournier}.

As our first application of Theorem \ref{mainthm:Lojasiewicz-Simon_gradient_inequality}, we have the following generalization of \cite[Theorem 23.17]{Feehan_yang_mills_gradient_flow_v4} from the case of the pure Yang--Mills energy function \eqref{eq:Yang--Mills_energy_function}, when $p=2$ and $X$ has dimension $d=2,3$, or $4$, and R\r{a}de's \cite[Proposition 7.2]{Rade_1992}, when $p=2$ and $X$ has dimension $d=2$ or $3$. Because gauge transformations of class $W^{2,2}$ are continuous when $d=2$ or $3$ and standard versions of the slice theorem \cite[Proposition 2.3.4]{DK}, \cite[Theorem 3.2]{FU}, \cite[Theorem 10.4]{Lawson} for the action of gauge transformations are applicable, the proof of the analogue of Theorem \ref{mainthm:Lojasiewicz-Simon_gradient_inequality_boson_Yang--Mills_energy_function} for the pure Yang--Mills energy function due to R\r{a}de is simpler for $d=2,3$ and $p = 2$.

\begin{mainthm}[{\L}ojasiewicz--Simon gradient inequality for the boson coupled Yang--Mills energy function]
\label{mainthm:Lojasiewicz-Simon_gradient_inequality_boson_Yang--Mills_energy_function}
Let $(X,g)$ be a closed, smooth Riemannian manifold of dimension $d\geq 2$, and $G$ be a compact Lie group, $P$ be a smooth principal $G$-bundle over $X$, and $E = P\times_\varrho\EE$ be a smooth Hermitian vector bundle over $X$ defined by a finite-dimensional unitary representation, $\varrho: G \to \Aut_\CC(\EE)$. Let $A_1$ be a $C^\infty$ reference connection on $P$, and $(A_\infty,\Phi_\infty)$ a boson coupled Yang--Mills pair on $(P,E)$ for $g$ of class $W^{1,q}$, with $q \in [2,\infty)$ obeying $q > d/2$. If $p \in [2,\infty)$ obeys $d/2 \leq p \leq q$, then the gradient map,
\[
\sM_g: (A_1,0)+W_{A_1}^{1,p}(X;\Lambda^1\otimes\ad P\oplus E)
\to W_{A_1}^{-1,p}(X;\Lambda^1\otimes\ad P\oplus E),
\]
is \emph{real analytic} and there are constants $Z \in (0, \infty)$, and $\sigma \in (0,1]$, and $\theta \in [1/2,1)$, depending on $A_1$, $(A_\infty,\Phi_\infty)$, $g$, $G$, $p$, and $q$ with the following significance. If $(A,\Phi)$ is a $W^{1,q}$ Sobolev pair on $(P,E)$ obeying the \emph{{\L}ojasiewicz--Simon neighborhood} condition,
\begin{equation}
\label{eq:Lojasiewicz-Simon_gradient_inequality_boson_Yang--Mills_pair_neighborhood}
\|(A,\Phi) - (A_\infty,\Phi_\infty)\|_{W^{1,p}_{A_1}(X)} < \sigma,
\end{equation}
then the boson coupled Yang--Mills energy function \eqref{eq:Boson_Yang--Mills_energy_function} obeys the \emph{{\L}ojasiewicz--Simon gradient inequality}
\begin{equation}
\label{eq:Lojasiewicz-Simon_gradient_inequality_boson_Yang--Mills_energy_function}
\|\sM_g(A,\Phi)\|_{W^{-1,p}_{A_1}(X)}
\geq
Z|\sE_g(A,\Phi) - \sE_g(A_\infty,\Phi_\infty)|^\theta.
\end{equation}
\end{mainthm}

The statement of Theorem \ref{mainthm:Lojasiewicz-Simon_gradient_inequality_boson_Yang--Mills_energy_function} simplifies with the addition of the rather mild assumption that $A_1 = A_\infty$ and that $(A_\infty,\Phi_\infty)$ is $C^\infty$ (which can be assumed, modulo a $W^{2,q}$ gauge transformation, provided by the regularity Theorem \ref{thm:Parker_1982_5-3}).

\begin{maincor}[{\L}ojasiewicz--Simon gradient inequality for the boson coupled Yang--Mills energy function]
\label{maincor:Lojasiewicz-Simon_gradient_inequality_boson_Yang--Mills_energy_function}
Let $(X,g)$ be a closed, smooth Riemannian manifold of dimension $d \geq 2$, and $G$ be a compact Lie group, $P$ be a smooth principal $G$-bundle over $X$, and $E = P\times_\varrho\EE$ be a smooth Hermitian vector bundle over $X$ defined by a finite-dimensional unitary representation, $\varrho: G \to \Aut_\CC(\EE)$. Let $(A_\infty,\Phi_\infty)$ be a smooth boson coupled Yang--Mills pair for $g$ on $(P,E)$. If $p \in [2,\infty)$ obeys $p \geq d/2$, then the gradient map,
\[
\sM_g: (A_\infty,0)+W_{A_\infty}^{1,p}(X;\Lambda^1\otimes\ad P\oplus E)
\to W_{A_\infty}^{-1,p}(X;\Lambda^1\otimes\ad P\oplus E),
\]
is \emph{real analytic} and, for $d/2 < q < \infty$ obeying $q\geq p$, there are constants $Z \in (0, \infty)$, and $\sigma \in (0,1]$, and $\theta \in [1/2,1)$, depending on $(A_\infty,\Phi_\infty)$, $g$, $G$, $p$, and $q$ with the following significance. If $(A,\Phi)$ is a $W^{1,q}$ Sobolev pair on $(P,E)$ that obeys the \emph{{\L}ojasiewicz--Simon neighborhood} condition,
\begin{equation}
\label{eq:Lojasiewicz-Simon_gradient_inequality_boson_Yang--Mills_pair_neighborhood_Ainfty}
\|(A,\Phi) - (A_\infty,\Phi_\infty)\|_{W^{1,p}_{A_\infty}(X)} < \sigma,
\end{equation}
then the boson coupled Yang--Mills energy function \eqref{eq:Boson_Yang--Mills_energy_function} obeys the \emph{{\L}ojasiewicz--Simon gradient inequality},
\begin{equation}
\label{eq:Lojasiewicz-Simon_gradient_inequality_boson_Yang--Mills_energy_function_Ainfty}
\|\sM_g(A,\Phi)\|_{W^{-1,p}_{A_\infty}(X)}
\geq
Z|\sE_g(A,\Phi) - \sE_g(A_\infty,\Phi_\infty)|^\theta.
\end{equation}
\end{maincor}

Similarly, for the fermion coupled Yang--Mills energy function, we have the

\begin{mainthm}[{\L}ojasiewicz--Simon gradient inequality for the fermion coupled Yang--Mills energy function]
\label{mainthm:Lojasiewicz-Simon_gradient_inequality_fermion_Yang--Mills_energy_function}
Assume the hypotheses of Theorem
\ref{mainthm:Lojasiewicz-Simon_gradient_inequality_boson_Yang--Mills_energy_function},
except that we require that $X$ admit a \spinc structure $(\rho,W)$,
replace the role of $\sE_g$ in
\eqref{eq:Boson_Yang--Mills_energy_function} by $\sF_g$ in
\eqref{eq:Fermion_Yang--Mills_energy_function}, and replace the role
of the pair $(A,\Phi)$ and critical point $(A_\infty,\Phi_\infty)$ of
$\sE_g$ by the pair $(A,\Psi)$ and critical point
$(A_\infty,\Psi_\infty)$ of $\sF_g$, where $\Psi$ and $\Psi_\infty$
are sections of $W\otimes E$. Then the conclusions of Theorem \ref{mainthm:Lojasiewicz-Simon_gradient_inequality_boson_Yang--Mills_energy_function} hold \emph{mutatis mutandis} for $\sF_g$.
\end{mainthm}

\begin{rmk}[{\L}ojasiewicz--Simon gradient inequality for coupled Yang--Mills energy functions on quotient spaces]
\label{rmk:Lojasiewicz-Simon_gradient_inequality_quotient_spaces}
We recall that the space of all smooth connections on $P$ is an affine space, $\sA(p) = A_1 + \Omega^1(X;\ad P)$. While the energy functions $\sE_g$ and $\sF_g$ in Definition \ref{defn:Boson_and_fermion_coupled_Yang--Mills_energy_function} were initially defined on affine spaces modeled on $C^\infty(X;\Lambda^1\otimes\ad P\oplus E)$ or $C^\infty(X;\Lambda^1\otimes\ad P\oplus W\otimes E)$, the functions are invariant under the action of the group of gauge transformations, $\Aut(P)$, and thus descend to the corresponding quotient spaces. The resulting configuration spaces may be given the structure of smooth Banach manifolds in a standard way \cite[Sections 4.2.1]{DK}, \cite[Chapter 3]{FU}, \cite[Section 1]{GroisserParkerGeometryDefinite} and, with minor modifications of standard proofs, the structure of real analytic Banach manifolds as discussed in Section \ref{subsec:Analyticity_of_quotient_spaces}.
\end{rmk}

\begin{rmk}[{\L}ojasiewicz--Simon gradient inequality for the Yang--Mills energy function over a Riemann surface]
\label{rmk:Lojasiewicz-Simon_gradient_inequality_Yang--Mills_Riemann_surface_Morse--Bott}
When $d=2$, it is known in many cases (see \cite{Feehan_lojasiewicz_inequality_ground_state}) that the pure Yang--Mills energy function obeys the Morse--Bott condition in the sense of \cite[Definition 1.9]{Feehan_Maridakis_Lojasiewicz-Simon_harmonic_maps_v6} and so by \cite[Theorem 4]{Feehan_Maridakis_Lojasiewicz-Simon_harmonic_maps_v6} (our abstract {\L}ojasiewicz--Simon gradient inequality for Morse--Bott functions on Banach spaces), one has the optimal {\L}ojasiewicz--Simon exponent, $\theta = 1/2$.
\end{rmk}

We have chosen to derive the {\L}ojasiewicz--Simon gradient inequalities (in Theorems \ref{mainthm:Lojasiewicz-Simon_gradient_inequality_boson_Yang--Mills_energy_function} and \ref{mainthm:Lojasiewicz-Simon_gradient_inequality_fermion_Yang--Mills_energy_function}) for two specific coupled Yang--Mills energy functions, motivated by physical considerations, namely the properties of \emph{regularity}, \emph{naturality}, and \emph{conformal invariance} (in dimension four) described by Parker in \cite[Section 2]{ParkerGauge}.

However, it is clear from the proofs of Theorems \ref{mainthm:Lojasiewicz-Simon_gradient_inequality_boson_Yang--Mills_energy_function} and \ref{mainthm:Lojasiewicz-Simon_gradient_inequality_fermion_Yang--Mills_energy_function} that one can expect the same conclusions for any energy function on pairs of connections and sections with the same nonlinearity structure. Indeed, proofs of such results can be obtained by simple modifications of our proof of the {\L}ojasiewicz--Simon gradient inequality for the boson coupled Yang--Mills energy function, just as we do in this monograph for the case of the fermion coupled Yang--Mills energy function.

\subsection{{\L}ojasiewicz--Simon $W^{-1,2}$ gradient inequalities for boson and fermion coupled Yang--Mills energy functions}
\label{subsec:Lojasiewicz-Simon_gradient_inequality_boson_fermion_Yang--Mills_function_L2}
For reasons noted prior to the statement of Theorem \ref{mainthm:Lojasiewicz-Simon_gradient_inequality2}, it is desirable to replace the {\L}ojasiewicz--Simon gradient inequality \eqref{eq:Lojasiewicz-Simon_gradient_inequality_boson_Yang--Mills_energy_function_Ainfty} in Theorem \ref{mainthm:Lojasiewicz-Simon_gradient_inequality_boson_Yang--Mills_energy_function} with one where the term $\|\sM_g(A,\Phi)\|_{W^{-1,p}_{A_\infty}(X)}$ is replaced by $\|\sM_g(A,\Phi)\|_{L^2(X)}$. Such a modification is trivial when we have a continuous embedding, $L^2(X) \subset W^{-1,p}(X)$, or equivalently when $W^{1,p'}(X) \subset L^2(X)$, where $p \in [1,\infty)$ and $p' = p/(p-1) \in (1,\infty]$ is the dual Sobolev exponent. By \cite[Theorem 4.12]{AdamsFournier}, with $d \geq 2$, we have a continuous Sobolev embedding $W^{1,s}(X) \subset L^{s^*}(X)$ for $1\leq s < d$. When $d = 2$, $3$, or $4$, then $p = 2$ is the smallest value allowed by Theorem \ref{mainthm:Lojasiewicz-Simon_gradient_inequality_boson_Yang--Mills_energy_function} and because $p' = 2$, the continuous Sobolev embedding $W^{1,2}(X) \subset L^2(X)$ and, in particular, its dual embedding $L^2(X) \subset W^{-1,2}(X)$ guarantee that the desired replacement described above is possible. For $d \geq 5$, the smallest and thus most favorable choice is $p = d/2$, in which case $p' = (d/2)/((d/2)-1) = d/(d-2) < d$. We need $(p')^* \geq 2$. Now,
\begin{multline*}
  (p')^* = dp'/(d-p') = (d^2/(d-2))/(d - d/(d-2))
  \\
  = (d/(d-2))/(1 - 1/(d-2)) = d/(d-3)
\end{multline*}
and thus $(p')^* \geq 2$ is equivalent to $d \geq 2(d-3)$, that is, $d \leq 6$. Consequently, Theorem \ref{mainthm:Lojasiewicz-Simon_gradient_inequality_boson_Yang--Mills_energy_function} only allows replacement of the term $\|\sM_g(A,\Phi)\|_{W^{-1,p}_{A_\infty}(X)}$ by $\|\sM_g(A,\Phi)\|_{L^2(X)}$ when $d \leq 6$.

However, the following {\L}ojasiewicz--Simon $W^{-1,2}$ gradient inequalities for boson and fermion coupled Yang--Mills energy functions are corollaries of the proofs of Theorems \ref{mainthm:Lojasiewicz-Simon_gradient_inequality_boson_Yang--Mills_energy_function} and \ref{mainthm:Lojasiewicz-Simon_gradient_inequality_fermion_Yang--Mills_energy_function}, respectively, that are valid for any $d \geq 2$. They are proved in Chapter \ref{chap:Lojasiewicz-Simon_gradient_inequality_Hilbert_space} and $L^2$ gradient inequalities follow immediately from these using the continuous Sobolev embedding $L^2(X) \subset W^{-1,2}(X)$. 

\begin{maincor}[{\L}ojasiewicz--Simon $W^{-1,2}$ gradient inequality for the boson coupled Yang--Mills energy function]
\label{maincor:Lojasiewicz-Simon_L2_gradient_inequality_boson_Yang-Mills_energy_function}
Assume the hypotheses of Theorem \ref{mainthm:Lojasiewicz-Simon_gradient_inequality_boson_Yang--Mills_energy_function}. If $(A,\Phi)$ is a $W^{1,q}$ Sobolev pair on $(P,E)$ obeying the \emph{{\L}ojasiewicz--Simon neighborhood} condition \eqref{eq:Lojasiewicz-Simon_gradient_inequality_boson_Yang--Mills_pair_neighborhood}, then the boson coupled Yang--Mills energy function \eqref{eq:Boson_Yang--Mills_energy_function} obeys the \emph{{\L}ojasiewicz--Simon gradient inequality},
\begin{equation}
\label{eq:Lojasiewicz-Simon_gradient_inequality_boson_Yang-Mills_energy_function_Wminus12}
\|\sM_g(A,\Phi)\|_{W^{-1,2}(X)}
\geq
Z|\sE(A,\Phi) - \sE(A_\infty, \Phi_\infty)|^\theta.
\end{equation}
\end{maincor}

\begin{maincor}[{\L}ojasiewicz--Simon $W^{-1,2}$ gradient inequality for the fermion coupled Yang--Mills energy function]
\label{maincor:Lojasiewicz-Simon_L2_gradient_inequality_fermion_Yang-Mills_energy_function}
Assume the hypotheses of Theorem \ref{mainthm:Lojasiewicz-Simon_gradient_inequality_fermion_Yang--Mills_energy_function}. Then the conclusions of Corollary \ref{maincor:Lojasiewicz-Simon_L2_gradient_inequality_boson_Yang-Mills_energy_function} hold \emph{mutatis mutandis} for the fermion coupled Yang--Mills energy function $\sF_g$ in \eqref{eq:Fermion_Yang--Mills_energy_function}.
\end{maincor}

\subsection{{\L}ojasiewicz--Simon gradient inequality for Yang--Mills--Higgs energy functions}
\label{subsec:Lojasiewicz-Simon_gradient_inequality_Yang--Mills--Higgs_function}
A well-known example in complex differential geometry of a coupled Yang--Mills energy function is the \emph{Yang--Mills--Higgs} function, which we now describe. See Bradlow \cite{Bradlow_1990, Bradlow_1991}, Bradlow and Garc{\'{\i}}a-Prada \cite{BradlowGP}, Hitchin \cite{Hitchin_1987}, Hong \cite{Hong_2001}, Li and Zhang
\cite{Li_Zhang_2011}, and Simpson \cite{Simpson_1988} for additional details and further references.

We shall follow the description by Bradlow and Garc{\'{\i}}a-Prada
\cite[Section 3]{BradlowGP}, but refer the reader to Hong \cite{Hong_2001}, Li and Zhang \cite{Li_Zhang_2011}, and the cited references for variants of the Yang--Mills--Higgs function described here. Let $E$ be a complex vector bundle with Hermitian metric $H$ over a compact K{\"a}hler manifold $(X,\omega)$. Let $\sA_H$ denote the affine space of smooth connections on $E$ that are \emph{unitary} (that is, compatible with the metric $H$), and $\Omega^0(X;E)$ denote the vector space of smooth sections of $E$, and $\tau \in \RR$.

One defines the \emph{Yang--Mills--Higgs energy function} on $(A, \Phi) \in \sA_H \times \Omega^0(X;E)$ by
\begin{equation}
\label{eq:Yang–Mills–Higgs_energy_function}
\sE_{H,\omega,\tau}(A, \Phi)
:=
\frac{1}{2}\int_X \left(|F_A|^2 + 2|\nabla_A\Phi|^2 + \left|\Phi\otimes \Phi^* - \tau\,\id_E\right|^2 \right)\,d\vol_\omega,
\end{equation}
where $\Phi^* := \langle \cdot, \Phi \rangle_H$, the dual of $\Phi$ with respect to the metric $H$.

By definition, a \emph{Yang--Mills--Higgs pair} $(A, \Phi)$ is a critical point of the Yang--Mills--Higgs function $\sE_{H,\omega,\tau}$, so $\sM_{H,\omega,\tau}(A, \Phi) = 0$, or equivalently $(A, \Phi)$ satisfies the second-order \emph{Yang--Mills--Higgs equations} (the Euler-Lagrange equations defined by the function \eqref{eq:Yang–Mills–Higgs_energy_function}). A calculation reveals that a pair is an \emph{absolute minimum} of $\sE_{H,\omega,\tau}$ if and only if it obeys the first-order \emph{vortex equations},
\begin{equation}
\label{eq:Vortex_equations}
\begin{aligned}
F_A^{0,2} &= 0,
\\
\bar\partial\Phi &=0,
\\
\Lambda F_A &= \sqrt{-1}\left(\Phi\otimes\Phi^* - \tau\,\id_E\right),
\end{aligned}
\end{equation}
where $\Lambda F_A$ denotes contraction of $F_A$ with $\omega$. Let $\fu(E)\subset \End_\CC(E)$ denote the subbundle of skew-Hermitian endomorphisms of $E$.

The proof of Theorem \ref{mainthm:Lojasiewicz-Simon_gradient_inequality_boson_Yang--Mills_energy_function} carries over \emph{mutatis mutandis} to give

\begin{mainthm}[{\L}ojasiewicz--Simon gradient inequality for the Yang--Mills--Higgs energy function]
\label{mainthm:Lojasiewicz-Simon_gradient_inequality_Yang--Mills--Higgs_energy_function}
Let $X$ be a compact, K{\"a}hler manifold of complex dimension $n \geq 1$ and $E$ be a complex vector bundle with Hermitian metric $H$ over $X$. Let $A_1$ be a smooth reference connection on the principal frame bundle for $E$. Assume that $d = 2n\geq 2$ and $p \in (1,\infty)$ obey one of the conditions in Theorem \ref{mainthm:Lojasiewicz-Simon_gradient_inequality_boson_Yang--Mills_energy_function}. Then the gradient map,
\[
\sM_{H,\omega,\tau}: (A_1,0)+W_{A_1}^{1,p}(X;\Lambda^1\otimes \fu(E) \oplus E)
\to
W_{A_1}^{-1,p}(X;\Lambda^1\otimes\fu(E) \oplus E),
\]
is \emph{real analytic} and the remaining conclusions of Theorem \ref{mainthm:Lojasiewicz-Simon_gradient_inequality_boson_Yang--Mills_energy_function} and Corollary \ref{maincor:Lojasiewicz-Simon_L2_gradient_inequality_boson_Yang-Mills_energy_function} hold \emph{mutatis mutandis} for the Yang--Mills--Higgs energy function \eqref{eq:Yang–Mills–Higgs_energy_function} near a Yang--Mills--Higgs pair, $(A_\infty, \Phi_\infty)$.
\end{mainthm}

\subsection{{\L}ojasiewicz--Simon gradient inequality for the Seiberg--Witten energy function}
\label{subsec:Lojasiewicz-Simon_gradient_inequality_Seiberg-Witten_function}
For another example of a coupled Yang--Mills energy function whose absolute minima can be readily identified, we consider the Seiberg--Witten equations.

Expositions of the Seiberg--Witten equations are now provided by many authors but, for the sake of consistency, we shall follow our development in \cite{FL2a}. Let $(\rho, W)$ denote a \spinc structure on a four-dimensional manifold, $X$, with Riemannian metric, $g$. We recall from \cite[Equation (2.55)]{FL2a} that a pair $(B, \Psi)$, comprising a \spinc connection, $B$, on $W=W^+\oplus W^-$ and a section, $\Psi$, of $W^+$ is a \emph{Seiberg--Witten monopole} if
\begin{equation}
\label{eq:Seiberg-Witten_equations}
\begin{aligned}
\tr(F_B^+) - \rho^{-1}(\Psi\otimes\Psi^*)_0 &= 0,
\\
D_B\Psi &= 0,
\end{aligned}
\end{equation}
recalling that $\rho:\Lambda^+\cong\su(W^+)$ is the isomorphism of Riemannian vector bundles induced by Clifford multiplication, $D_B:C^\infty(X;W^+)\to C^\infty(X;W^-)$ is the Dirac operator, and $(\cdot)_0$ denotes the trace-free part of $\Psi\otimes\Psi^* \in \End_\CC(W^+)$. We have restricted $B$ to $W^+$, so $F_B^+ \in  C^\infty(X;\fu(W^+)\otimes\Lambda^+) = \Omega^+(X;\fu(W^+))$ and $\tr(F_B^+) \in C^\infty(X;i\Lambda^+) = \Omega^+(X;i\RR)$,  using the fiberwise trace homomorphism, $\tr:\fu(W^+)\to i\underline{\RR}$. The Seiberg--Witten equations \eqref{eq:Seiberg-Witten_equations} are a system of first-order partial differential equations in $(B,\Psi)$ and thus cannot be the Euler-Lagrange equations of any action function. However, as we recall from \cite{Doria_2005, Doria_2006, Hong_Schabrun_2010, Jost_Peng_Wang_1996, NicolaescuSWNotes}, Seiberg--Witten monopoles have a variational interpretation by an argument which is the reverse of those provided by Bradlow and Garc{\'{\i}}a-Prada \cite[Section 3]{BradlowGP} or Hong \cite[Section 1]{Hong_2001} in their derivations of the vortex equations or Li and Zhang \cite[Section 1]{Li_Zhang_2011} for the Hermitian-Einstein equations.

Thus, from \cite[Equation (1.6)]{Hong_Schabrun_2010} or \cite[Propsition 2.1.4]{NicolaescuSWNotes}, the Seiberg--Witten energy function is
\begin{multline}
\label{eq:Seiberg-Witten_energy_function}
\sE_g(B,\Psi)
=
\int_X \left( \left|\nabla_B\Psi\right|^2 + \frac{1}{2}|\tr(F_B)|^2 + \frac{R}{4}|\Psi|^2 + \frac{1}{8}|\Psi|^2\right)\,d\vol_g
\\
+ 2\pi^2c_1(W^+)^2,
\end{multline}
where $c_1(W^+)^2 := \int_X c_1(W^+)^2$. The topological term, $2\pi^2c_1(W^+)^2$, is independent of the pair $(B,\Psi)$ and does not affect the critical points. In particular,
\[
\sE(B,\Psi) \geq 2\pi^2c_1(W^+)^2,
\]
and a pair $(B,\Psi)$ is a Seiberg--Witten monopole if and only if equality is achieved.

Hong and Schabrun derive a version of the {\L}ojasiewicz--Simon gradient inequality \cite[Lemma 5.3]{Hong_Schabrun_2010} based in part on an earlier proof due to Wilkin for the Yang--Mills--Higgs function over a Riemann surface
\cite[Proposition 3.5]{Wilkin_2008}. However, the proof of Theorem \ref{mainthm:Lojasiewicz-Simon_gradient_inequality_boson_Yang--Mills_energy_function} carries over \emph{mutatis mutandis} to give

\begin{mainthm}[{\L}ojasiewicz--Simon gradient inequality for the Seiberg--Witten energy function]
\label{mainthm:Lojasiewicz-Simon_gradient_inequality_Seiberg-Witten_energy_function}
Let $(X,g)$ be a closed, four-dimensional, oriented, Riemannian smooth manifold with \spinc structure $(\rho,W)$. Let $B_1$ be a smooth reference \spinc connection on $W$. Assume that $p \in (1,\infty)$ obeys the hypotheses of Theorem \ref{mainthm:Lojasiewicz-Simon_gradient_inequality_boson_Yang--Mills_energy_function} with $d=4$. Then the gradient map,
\[
\sM_g: (B_1,0)+W_{B_1}^{1,p}(X;i\Lambda^1 \oplus W^+)
\to
W_{B_1}^{-1,p}(X;i\Lambda^1 \oplus W^+),
\]
is \emph{real analytic} and the remaining conclusions of Theorem \ref{mainthm:Lojasiewicz-Simon_gradient_inequality_boson_Yang--Mills_energy_function} and Corollary \ref{maincor:Lojasiewicz-Simon_L2_gradient_inequality_boson_Yang-Mills_energy_function} hold \emph{mutatis mutandis} for the Seiberg--Witten energy function \eqref{eq:Seiberg-Witten_energy_function} near a Seiberg--Witten monopole, $(B_\infty,\Psi_\infty)$.
\end{mainthm}


\subsection{{\L}ojasiewicz--Simon gradient inequality for non-Abelian monopole energy functions}
\label{subsec:Lojasiewicz-Simon_gradient_inequality_non-Abelian-monopole_function}
For our final example of a coupled Yang--Mills energy function whose absolute minima can be readily identified, we have the non-Abelian monopoles arising in the work of the first author and Leness \cite{FL2a}, Okonek and Teleman \cite{OTVortex}, and Pidstrigatch and Tyurin \cite{PTLocal}.

Following \cite{FL2a}, we consider pairs $(A,\Phi)$ obeying
\begin{equation}
\label{eq:Feehan_Leness_2001_2-32}
\begin{aligned}
(F_A^+)_0 - \rho^{-1}(\Phi\otimes\Phi^*)_{00} &= 0,
\\
D_A\Phi &= 0,
\end{aligned}
\end{equation}
where $A$ is a unitary connection on a Hermitian vector bundle, $E$, with curvature $F_A \in C^\infty(X;\Lambda^2\otimes\fu(E)) = \Omega^2(X;\fu(E))$ and $(F_A^+)_0 \in C^\infty(X;\Lambda^+\otimes\su(E)) = \Omega^+(X;\su(E))$, while $\rho:\Lambda^+\otimes\su(E)\cong \su(W^+)\otimes\su(E)$ is the isomorphism of Riemannian vector bundles induced by Clifford multiplication, $D_A:C^\infty(X;W^+\otimes E)\to C^\infty(X;W^-\otimes E)$ is the Dirac operator, and $(\cdot)_{00}$ denotes the trace-free part of $\Phi\otimes\Phi^* \in \End_\CC(W^+\otimes E)$. Let $\su(E) \subset \End_\CC(E)$ denote the subbundle of skew-Hermitian, trace-free endomorphisms of $E$.

By extending the derivations of the Seiberg--Witten  energy function in \cite{Hong_Schabrun_2010} or \cite{NicolaescuSWNotes}, we find that the \emph{non-Abelian monopole energy function} is
\begin{equation}
\label{eq:Non-Abelian_monopole_energy_function}
\begin{aligned}
\sE_g(A,\Phi) &= \int_X \left( \left|\nabla_A\Phi\right|^2 + \frac{1}{2}|F_A|^2 + \frac{R}{4}|\Phi|^2 + \frac{1}{8}|\Phi|^2\right)\,d\vol_g
\\
&\quad - 4\pi^2c^2(E) + \frac{1}{2}\|F_{A_w}^+\|_{L^2(X)}^2 - 2\|F_{A_e}^+\|_{L^2(X)}^2.
\end{aligned}
\end{equation}
The connections $A_e$ on $\det E$ and $A_w$ on $\det W^+$ are fixed, with no dynamical role, so the true variables in the $\SO(3)$-monopole equations are the $\SO(3)$ connection $\hat A$ induced by $A$ on the bundle $\su(E)$ and the (spinor) section $\Phi$ of $W^+\otimes E$. The action function, $\sE_g(A,\Phi)$, again has a universal lower bound and is achieved if and only if $(A,\Phi)$ is a \emph{non-Abelian monopole}, namely a solution to \eqref{eq:Feehan_Leness_2001_2-32}. Again the proof of Theorem \ref{mainthm:Lojasiewicz-Simon_gradient_inequality_boson_Yang--Mills_energy_function} carries over \emph{mutatis mutandis} to give

\begin{mainthm}[{\L}ojasiewicz--Simon gradient inequality for the non-Abelian monopole energy function]
\label{mainthm:Lojasiewicz-Simon_gradient_inequality_non-Abelian_monopole_energy_function}
Let $(X,g)$ be a closed, four-dimensional, oriented, smooth Riemannian manifold with \spinc structure $(\rho,W)$. Let $E$ be a Hermitian vector bundle over $X$, and $A_e$ be a smooth connection on $\det E$, and $B$ be a smooth \spinc connection on $W$, and $A_1$ be a smooth reference connection on $E$ inducing $A_e$ on $\det E$. Assume that $p \in (1,\infty)$ obeys the hypotheses of Theorem \ref{mainthm:Lojasiewicz-Simon_gradient_inequality_boson_Yang--Mills_energy_function} with $d=4$. Then the gradient map,
\begin{multline*}
\sM_g: (A_1,0)+W_{A_1}^{1,p}(X;\Lambda^1\otimes\su(E) \oplus W^+\otimes E)
\\
\to
W_{A_1}^{-1,p}(X;\Lambda^1\otimes\su(E) \oplus W^+\otimes E),
\end{multline*}
is \emph{real analytic} and the remaining conclusions of Theorem \ref{mainthm:Lojasiewicz-Simon_gradient_inequality_boson_Yang--Mills_energy_function} and Corollary \ref{maincor:Lojasiewicz-Simon_L2_gradient_inequality_boson_Yang-Mills_energy_function} hold \emph{mutatis mutandis} for the non-Abelian monopole energy function \eqref{eq:Non-Abelian_monopole_energy_function} near a non-Abelian monopole, $(A_\infty,\Phi_\infty)$.
\end{mainthm}

\subsection{{\L}ojasiewicz--Simon gradient inequality for the multiple spinor Seiberg--Witten energy function}
\label{subsec:Lojasiewicz-Simon_gradient_inequality_SW_multiple_spinor_function}
In our next example of a coupled Yang--Mills energy function, the absolute minima can be identified as solutions to the multiple spinor Seiberg--Witten equations arising in the work of Haydys \cite{Haydys_2017arxiv, Haydys_2019}, and Haydys and Walpuski \cite{Haydys_Walpuski_2015}.

Let $X$ be a closed, oriented Riemannian three-manifold with fixed spin structure $(\rho, W)$, and $L$ be a Hermitian line bundle over $X$, and $E$ be a Hermitian vector bundle of rank $n\geq 1$, trivial determinant $\det E$, and a compatible fixed connection $B$ on $E$. Following \cite{Haydys_Walpuski_2015}, we consider pairs $(A,\Psi)$ obeying
\begin{equation}
\label{eq:SW_multiple_spinor}
\begin{aligned}
\rho(F_A) - \mu(\Psi) &= 0,
\\
D_{A\otimes B}\Psi &= 0,
\end{aligned}
\end{equation}
where $A$ is a unitary connection on $L$, with curvature $F_A \in \Omega^2(X;\fu(L))$, while $\rho:\Lambda^2\otimes\fu(L)\cong \su(W)\otimes\su(L)$ is the isomorphism of Riemannian vector bundles induced by Clifford multiplication,
\[
  D_{A\otimes B}:C^\infty(X; \Hom(E;W\otimes L))\to C^\infty(X;\Hom(E;W\otimes L))
\]
is the Dirac operator, and
\[
  \mu: \Hom(E;W\otimes L) \to \fu(L)\otimes \su(W)\simeq i\su(W)
\]
sends $\Psi\in \Hom(E;W\otimes L)$ to the trace-free part of $\Psi\otimes\Psi^* \in \End_\CC(W\otimes L)$.

By extending the derivations of the Seiberg--Witten energy function in \cite{Hong_Schabrun_2010} or \cite{NicolaescuSWNotes}, we find that the \emph{Seiberg--Witten multiple spinor energy function} is
\begin{multline}
\label{eq:SW_multiple_spinor_energy_function}
\sE_g(A,\Psi) = \int_X \left( \left|\nabla_A\Psi\right|^2 + \frac{1}{2}|F_A|^2 + \frac{R}{4}|\Psi|^2 + \frac{1}{2}|\mu(\Psi)|^2 \right.
\\
+ \left. \langle \rho(F_B)\Psi, \Psi\rangle\right)\,d\vol_g.
\end{multline}
The connection $B$ on $E$ is fixed, with no dynamical role, so the true variables in multiple spinor Seiberg--Witten equations \eqref{eq:SW_multiple_spinor} are the unitary connection $A$ and the (spinor) section $\Psi$ of $\End(E;W\otimes L)\simeq E^*\otimes W\otimes L$. The energy function, $\sE_g(A,\Psi)$ has universal lower bound zero and that is achieved if and only if $(A,\Psi)$ is solution to \eqref{eq:SW_multiple_spinor}. Note that the term $\langle \rho(F_B)\Psi, \Psi\rangle$ is quadratic in $\Psi$ while $|\mu(\Psi)|^2$ is fourth order in $\Psi$. Again the proof of Theorem \ref{mainthm:Lojasiewicz-Simon_gradient_inequality_boson_Yang--Mills_energy_function} carries over \emph{mutatis mutandis} to give

\begin{mainthm}[{\L}ojasiewicz--Simon gradient inequality for the Seiberg--Witten multiple spinor energy function]
\label{mainthm:Lojasiewicz-Simon_gradient_inequality_ SW_multiple_spinor_energy_function}
Let $X$ be a closed, oriented Riemannian three-manifold with fixed spin structure $(\rho, W)$, and $L$ be a Hermitian line bundle over $X$, and $E$ be a Hermitian vector bundle of rank $n\geq 1$, trivial determinant $\det E$, and a compatible fixed connection $B$ on $E$. Let $A_1$ be a fixed $C^\infty$ reference connection on $L$ and endow $W$ with the unique smooth spin connection. Assume that $p \in (1,\infty)$ obeys the hypotheses of Theorem \ref{mainthm:Lojasiewicz-Simon_gradient_inequality_boson_Yang--Mills_energy_function} with $d=3$. Then the gradient map,
\begin{multline*}
\sM_g: (A_1,0)+W_{A_1}^{1,p}(X;\Lambda^1\otimes\fu(L) \oplus E^*\otimes W\otimes L)
\\
\to
W_{A_1}^{-1,p}(X;\Lambda^1\otimes\fu(L) \oplus E^*\otimes W\otimes L),
\end{multline*}
is \emph{real analytic} and the remaining conclusions of Theorem \ref{mainthm:Lojasiewicz-Simon_gradient_inequality_boson_Yang--Mills_energy_function} and Corollary \ref{maincor:Lojasiewicz-Simon_L2_gradient_inequality_boson_Yang-Mills_energy_function} hold \emph{mutatis mutandis} for the Seiberg--Witten multiple spinor energy function \eqref{eq:SW_multiple_spinor_energy_function} near a Seiberg--Witten multi-monopole, $(A_\infty,\Psi_\infty)$.
\end{mainthm}

\subsection{{\L}ojasiewicz--Simon gradient inequality for the pure Yang--Mills energy function near $G_2$-instantons}
\label{subsec:Lojasiewicz-Simon_gradient_inequality_for_G_2_instantons}
Our final example is concerned with $G_2$-instantons, which are known to be absolute minima of the pure Yang--Mills energy function in dimension seven; for a development of the {\L}ojasiewicz--Simon gradient inequality for the Yang--Mills energy function near anti-self-dual connections over four-dimensional manifolds or flat connections over manifolds of arbitrary dimension, we refer the reader to \cite{Feehan_yang_mills_gradient_flow_v4} or \cite{Feehan_lojasiewicz_inequality_ground_state}, respectively. For an introduction to Yang--Mills gauge theory in dimension seven, see Donaldson and Segal \cite{Donaldson_Segal_2011}, Joyce \cite{Joyce_compact_manifolds_special_holonomy} and Kovalev \cite{Kovalev_2003}. For recent progress concerning $G_2$-instantons, see Walpuski \cite{WalpuskiThesis} and references therein.

Let $(X, \phi)$ be a closed $G_2$-manifold with non-degenerate three-form $\phi \in \Omega^3(X)$ (the \emph{associative form}), $G$ be a compact Lie group, and $P$ be a principal $G$-bundle over $X$.  Following \cite{WalpuskiThesis}, a connection $A$ on $P$ is a \emph{$G_2$-instanton} if obeys the anti-self-duality equation,
\begin{equation}
\label{eq:G_2_instanton}
*(F_A\wedge \phi) =  -F_A.
\end{equation}
It is proven in Walpuski \cite[Proposition 1.97]{WalpuskiThesis} that when $A$ is a $G_2$-instanton, then $A$ is an absolute minimum of the Yang-Mills energy function,
\begin{equation}
\label{eq:YM_energy_function}
\sE_g(A,\Psi) = \frac{1}{2}\int_X |F_A|^2 \,d\vol_g,
\end{equation}
with minimum value $- 4\pi^2\langle p_1(\ad P)\smile [\phi], [X]\rangle$. (Note that in our convention, the definition of the Yang-Mills energy function differs from that of \cite{WalpuskiThesis} by a factor of $1/2$.) In this case, a direct application of Theorem \ref{mainthm:Lojasiewicz-Simon_gradient_inequality_boson_Yang--Mills_energy_function} and Corollary \ref{maincor:Lojasiewicz-Simon_L2_gradient_inequality_boson_Yang-Mills_energy_function} gives

\begin{mainthm}[{\L}ojasiewicz--Simon gradient inequality for $G_2$-instantons]
\label{mainthm:Lojasiewicz-Simon_gradient_inequality_YM_energy_function}
Let $(X,\phi)$ be a closed $G_2$-manifold, $G$ be a compact Lie group and $P$ be a principal $G$-bundle over $X$ with a fixed $C^\infty$ reference connection $A_1$. Assume that $p,q \in (1,\infty)$ satisfy $7/2\leq p \leq q$ with $q>7/2$ and that $A_\infty$ is $W^{1,q}$ connection on $P$ satisfying the $G_2$-instanton equation \eqref{eq:G_2_instanton}. Then there are constants $Z \in (0, \infty)$, and $\sigma \in (0,1]$, and $\theta \in [1/2,1)$, depending on $A_1$, $A_\infty$, $g$, $G$ $p$, and $q$ with the following significance. For any connection $A$ of class $W^{1,q}$ obeying the \emph{{\L}ojasiewicz--Simon neighborhood} condition,
\begin{equation}
\label{eq:Lojasiewicz-Simon_gradient_inequality_YM_neighborhood}
\|A- A_\infty\|_{W^{1,p}_{A_1}(X)} < \sigma,
\end{equation}
then the Yang-Mills energy function \eqref{eq:YM_energy_function} obeys the \emph{{\L}ojasiewicz--Simon gradient inequality}
\begin{equation}
\label{eq:Lojasiewicz-Simon_gradient_inequality_YM_energy_function}
\|d^*_AF_A\|_{L^2(X)}
\geq
Z|\sE_g(A)) + 4\pi^2\langle p_1(\ad P)\smile [\phi], [X]\rangle |^\theta.
\end{equation}
\end{mainthm}

\section{Applications of the {\L}ojasiewicz--Simon gradient inequality for coupled Yang--Mills energy functions}
\label{sec:Lojasiewicz-Simon_gradient_inequality_coupled_Yang--Mills_energy_functions_applications}
Our interest in {\L}ojasiewicz--Simon gradient inequalities for coupled Yang--Mills and harmonic map energy functions is motivated by the wealth of potential applications. We shall survey some of those applications below.

In \cite{Feehan_yang_mills_gradient_flow_v4}, we apply the {\L}ojasiewicz--Simon gradient inequality for the pure Yang--Mills energy function \cite[Theorems 23.1 and 23.17]{Feehan_yang_mills_gradient_flow_v4} to prove global existence, convergence, convergence rate, and stability results for solutions $A(t)$ to the associated gradient flow,
\[
\frac{\partial A}{\partial t} = -\sM_g(A(t)), \quad A(0) = A_0,
\]
that is,
\[
\frac{\partial A}{\partial t} = -d_{A(t)}^{*,g}F_{A(t)}, \quad A(0) = A_0.
\]
Given our {\L}ojasiewicz--Simon gradient inequalities for the boson and fermion coupled Yang--Mills energy functions, Theorems \ref{mainthm:Lojasiewicz-Simon_gradient_inequality_boson_Yang--Mills_energy_function} and
\ref{mainthm:Lojasiewicz-Simon_gradient_inequality_fermion_Yang--Mills_energy_function}, the main conclusions in \cite{Feehan_yang_mills_gradient_flow_v4} for pure Yang--Mills gradient flow should extend easily to the more general case of coupled Yang--Mills gradient flows.

In \cite{Feehan_yangmillsenergygapflat}, we applied the {\L}ojasiewicz--Simon gradient inequality to prove an energy gap result for Yang--Mills connections with small $L^{d/2}$ energy. The proof of that result should extend without difficulty to the case of solutions to the coupled Yang--Mills equations.

\section{Automorphisms and transformation to Coulomb gauge}
\label{sec:Automorphisms and transformation to Coulomb gauge}
For some energy functions, the associated Hessian is already an elliptic second-order partial differential operator on a Sobolev space, but for others the Hessian is only elliptic when combined with a type of Coulomb gauge condition \cite{DK, FU} and it is only then that one can apply Theorem \ref{mainthm:Lojasiewicz-Simon_gradient_inequality}. For example, in the first category, one has the harmonic map energy and Yamabe functions, while in the second category one has the Yang--Mills and coupled Yang--Mills energy functions.

The Yang--Mills energy function is invariant under the action of gauge transformations (or bundle automorphisms) and so, in principle, one can always find a gauge transformation to produce the required Coulomb gauge condition with the aid of a \emph{slice theorem}. However, in order to prove the most useful version of the {\L}ojasiewicz--Simon gradient inequality, it is convenient to have a stronger version of the slice theorem for the action of the group of gauge transformations, going beyond the usual statements found in standard references such as Donaldson and Kronheimer \cite{DK} or Freed and Uhlenbeck \cite{FU} and proved by applying the Implicit Function Theorem. One stronger version of a slice theorem, valid in dimension four, was proved by the first author as \cite[Theorem 1.1]{FeehanSlice} but it nevertheless falls short of what we need for our application to the proofs of our {\L}ojasiewicz--Simon gradient inequalities (even when translated to the setting of pairs). Thus, a second purpose of this monograph is to prove a stronger version of
\cite[Theorem 1.1]{FeehanSlice} for both connections and pairs rather than just connections as in \cite{FeehanSlice}, but using standard Sobolev norms with borderline Sobolev exponents rather than the critical-exponent norms employed in \cite{FeehanSlice} and valid in all dimensions.

\subsection{Transformation to Coulomb gauge}
\label{subsec:Transformation_to_Coulomb_gauge}
We first state the desired result for connections and then its analogue for pairs.

\begin{mainthm}[Existence of $W^{2,q}$ Coulomb gauge transformations for $W^{1,q}$ connections that are $W^{1,\frac{d}{2}}$ close to a reference connection]
\label{mainthm:Feehan_proposition_3-4-4_Lp}
Let $(X,g)$ be a closed, smooth Riemannian manifold of dimension $d \geq 2$, and $G$ be a compact Lie group, and $P$ be a smooth principal $G$-bundle over $X$. If $A_1$ is a $C^\infty$ connection on $P$, and $A_0$ is a Sobolev connection on $P$ of class $W^{1,q}$ with $d/2<q<\infty$, and $p \in (1,\infty)$ obeys $d/2 \leq p \leq q$, then there exists a constant $\zeta = \zeta(A_0,A_1,g,G,p,q) \in (0,1]$ with the following significance. If $A$ is a $W^{1,q}$ connection on $P$ that obeys
\begin{equation}
\label{eq:Feehan_3-4-4_Lp_A_minus_A0_W1p_close}
\|A - A_0\|_{W_{A_1}^{1,p}(X)} < \zeta,
\end{equation}
then there exists a gauge transformation $u \in \Aut(P)$ of class $W^{2,q}$ such that
\[
d_{A_0}^*(u(A) - A_0) = 0,
\]
and
\[
\|u(A) - A_0\|_{W_{A_1}^{1,p}(X)} < 2N\|A - A_0\|_{W_{A_1}^{1,p}(X)},
\]
where $N = N(A_0,A_1,g,G,p,q) \in [1,\infty)$ is the constant in the forthcoming Proposition \ref{prop:Feehan_2001_lemma_6-6}.
\end{mainthm}

For a description of the action of the group of gauge transformations in Theorem \ref{mainthm:Feehan_proposition_3-4-4_Lp} and the definition of the Coulomb gauge condition for connections, we refer the reader to Section \ref{sec:Coulomb_gauge_slice_quotient_space_connections}, and for an explanation of the remainder of the notation in Theorem \ref{mainthm:Feehan_proposition_3-4-4_Lp}, we refer the reader to Section \ref{subsec:Lojasiewicz-Simon_gradient_inequality_boson_fermion_Yang--Mills_function}.

The essential point in Theorem \ref{mainthm:Feehan_proposition_3-4-4_Lp} is that the result holds for the critical exponent, $p = d/2$ with $d \geq 3$, when the Sobolev space $W^{2,p}(X)$ fails to embed in $C(X)$ (see \cite[Theorem 4.12]{AdamsFournier}) and a proof of Theorem \ref{mainthm:Feehan_proposition_3-4-4_Lp} by the Implicit Function Theorem in the case $p > d/2$ fails when $p = d/2$. In this situation, a $W^{2,\frac{d}{2}}$ gauge transformation $u$ of $P$ is not continuous, the set $\Aut^{2,\frac{d}{2}}(P)$ of $W^{2,\frac{d}{2}}$ gauge transformations is not a manifold, and $\Aut^{2,\frac{d}{2}}(P)$ cannot act smoothly on the affine space $\sA^{1,\frac{d}{2}}(P)$ of $W^{1,\frac{d}{2}}$ connections on $P$. When $d = 4$ and $p \geq 2$, this phenomenon is discussed by Freed and Uhlenbeck in \cite[Appendix A]{FU}.

While the Quantitative Implicit Function Theorem \ref{thm:Quantitative_implicit_function_theorem} could be applied to give a shorter proof of Theorem \ref{mainthm:Feehan_proposition_3-4-4_Lp} when $p > d/2$, it is unclear whether this method could yield the borderline case $p = d/2$. See Remark \ref{rmk:Feehan_proposition_3-4-4_Lp_Ld_close}  for an explanation. However, to illustrate use of the Quantitative Implicit Function Theorem in an application that is of interest in its own right, we shall apply Theorem \ref{thm:Quantitative_implicit_function_theorem} to prove

\begin{mainthm}[Existence of $W^{2,q}$ Coulomb gauge transformations for $W^{1,q}$ connections that are $L^r$ close to a reference connection]
\label{mainthm:Feehan_proposition_3-4-4_Lp_Lr_close}
Let $(X,g)$ be a closed, smooth Riemannian manifold of dimension $d \geq 2$, and $G$ be a compact Lie group, and $P$ be a smooth principal $G$-bundle over $X$. If $A_1$ is a $C^\infty$ connection on $P$, and $A_0$ is a Sobolev connection on $P$ of class $W^{1,q}$ with $d/2<q<\infty$, and $r$ is a constant that obeys $d < r \leq q^* = dq/(d-q)$ if $q<d$ or $d < r < \infty$ if $q\geq d$, then there are constants $\delta = \delta(A_0,A_1,g,G,r) \in (0,1]$ and $C = C(A_0,A_1,g,G,r) \in (0,\infty)$ with the following significance. If $A$ is a $W^{1,q}$ connection on $P$ that obeys
\begin{equation}
\label{eq:Feehan_3-4-4_Lp_A_minus_A0_Lr_close}
\|A - A_0\|_{L^r(X)} < \delta,
\end{equation}
then there exists a gauge transformation $u \in \Aut(P)$ of class $W^{2,q}$ such that
\[
d_{A_0}^*(u(A) - A_0) = 0,
\]
and
\[
\|u(A) - A_0\|_{L^r(X)} < C\|A - A_0\|_{L^r(X)}.
\]
\end{mainthm}

\begin{rmk}[Borderline Sobolev exponents]
\label{rmk:Feehan_proposition_3-4-4_Lp_Ld_close}  
The constant $C$ in Theorem \ref{mainthm:Feehan_proposition_3-4-4_Lp_Lr_close} and the $L^r(X;\Lambda^1\otimes\ad P)$ norm depend continuously on $r \geq d$ (see \cite[Theorem 2.17]{AdamsFournier} or \cite[Problem 7.1]{GilbargTrudinger}) and so (as in the proof of \cite[Corollary 2.2]{UhlLp}), we may take the limit as $r\downarrow d$ to give
 \[
\|u(A) - A_0\|_{L^d(X)} \leq C\|A - A_0\|_{L^d(X)}.
\]
However, it is unclear whether one can also relax the condition \eqref{eq:Feehan_3-4-4_Lp_A_minus_A0_Lr_close} by allowing $r = d$ because our direct proof in Section \ref{sec:Coulomb_gauge_slice_quotient_space_connections_Lr_close} of Theorem \ref{mainthm:Feehan_proposition_3-4-4_Lp_Lr_close} via the Quantitative Implicit Function Theorem \ref{thm:Quantitative_implicit_function_theorem} precludes that choice. It is also unclear that one could achieve such a result by applying the Method of Continuity (by analogy with our proof of Theorem \ref{mainthm:Feehan_proposition_3-4-4_Lp}) since there is no obvious analogue of Proposition \ref{prop:Feehan_2001_lemma_6-6}. The analogous comments apply to the forthcoming Theorem \ref{mainthm:Feehan_proposition_3-4-4_Lp_pairs_Lr_close}.
\end{rmk}

The proof of Theorem \ref{mainthm:Feehan_proposition_3-4-4_Lp} adapts \mutatis to establish the following refinement of \cite[Proposition 2.8]{FL1} and \cite[Theorem 4.1]{ParkerGauge}.

\begin{mainthm}[Existence of $W^{2,q}$ Coulomb gauge transformations for $W^{1,q}$ pairs that are $W^{1,\frac{d}{2}}$ close to a reference pair]
\label{mainthm:Feehan_proposition_3-4-4_Lp_pairs}
Let $(X,g)$ be a closed, smooth Riemannian manifold of dimension $d \geq 2$, and $G$ be a compact Lie group, $P$ be a smooth principal $G$-bundle over $X$, and $E = P\times_\varrho\EE$ be a smooth Hermitian vector bundle over $X$ defined by a finite-dimensional unitary representation, $\varrho: G \to \Aut_\CC(\EE)$. If $A_1$ is a $C^\infty$ connection on $P$, and $(A_0,\Phi_0)$ is a Sobolev pair on $(P,E)$ of class $W^{1,q}$ with $d/2<q<\infty$, and $p \in (1,\infty)$ obeys $d/2 \leq p \leq q$, then there exists a constant $\zeta = \zeta(A_1,A_0,\Phi_0,g,G,p,q) \in (0,1]$ with the following significance. If $(A,\Phi)$ is a $W^{1,q}$ pair on $(P,E)$ that obeys
\begin{equation}
\label{eq:Feehan_3-4-4_Lp_APhi_minus_A0Phi0_W1p_close}
\|(A,\Phi) - (A_0,\Phi_0)\|_{W_{A_1}^{1,p}(X)} < \zeta,
\end{equation}
then there exists a gauge transformation $u \in \Aut(P)$ of class $W^{2,q}$ such that
\[
d_{A_0,\Phi_0}^*(u(A,\Phi) - (A_0,\Phi_0)) = 0,
\]
and
\[
\|u(A,\Phi) - (A_0,\Phi_0)\|_{W_{A_1}^{1,p}(X)}
< 2N\|(A,\Phi) - (A_0,\Phi_0)\|_{W_{A_1}^{1,p}(X)},
\]
where $N = N(A_1,A_0,\Phi_0,g,G,p,q) \in [1,\infty)$ is the constant in the forthcoming Proposition \ref{prop:Feehan_2001_lemma_6-6_pairs}.
\end{mainthm}

For a description of the action of the group of gauge transformations in Theorem \ref{mainthm:Feehan_proposition_3-4-4_Lp_pairs} and the definition of the Coulomb gauge condition for pairs, we refer the reader to Section \ref{sec:Coulomb_gauge_slice_quotient_space_pairs}. The proof of Theorem \ref{mainthm:Feehan_proposition_3-4-4_Lp_Lr_close} adapts \mutatis to establish the following refinement of Theorem \ref{mainthm:Feehan_proposition_3-4-4_Lp_pairs}.

\begin{mainthm}[Existence of $W^{2,q}$ Coulomb gauge transformations for $W^{1,q}$ pairs that are $L^r$ close to a reference pair]
\label{mainthm:Feehan_proposition_3-4-4_Lp_pairs_Lr_close}
Let $(X,g)$ be a closed, smooth Riemannian manifold of dimension $d \geq 2$, and $G$ be a compact Lie group, $P$ be a smooth principal $G$-bundle over $X$, and $E = P\times_\varrho\EE$ be a smooth Hermitian vector bundle over $X$ defined by a finite-dimensional unitary representation, $\varrho: G \to \Aut_\CC(\EE)$. If $A_1$ is a $C^\infty$ connection on $P$, and $(A_0,\Phi_0)$ is a Sobolev pair on $(P,E)$ of class $W^{1,q}$ with $d/2<q<\infty$, and $r$ is a constant that obeys $d < r \leq q^* = dq/(d-q)$ if $q<d$ or $d < r < \infty$ if $q\geq d$, then there are constants $\delta = \delta(A_0,\Phi_0,A_1,g,G,r) \in (0,1]$ and $C = C(A_0,\Phi_0,A_1,g,G,r) \in (0,\infty)$ with the following significance. If $(A,\Phi)$ is a $W^{1,q}$ pair on $(P,E)$ that obeys
\begin{equation}
\label{eq:Feehan_3-4-4_Lp_APhi_minus_A0Phi0_L^r_close}
\|(A,\Phi) - (A_0,\Phi_0)\|_{L^r(X)} < \delta,
\end{equation}
then there exists a gauge transformation $u \in \Aut(P)$ of class $W^{2,q}$ such that
\[
d_{A_0,\Phi_0}^*(u(A,\Phi) - (A_0,\Phi_0)) = 0,
\]
and
\[
\|u(A,\Phi) - (A_0,\Phi_0)\|_{L^r(X)} < C\|(A,\Phi) - (A_0,\Phi_0)\|_{L^r(X)}.
\]
\end{mainthm}

\subsection{Real analytic Banach manifold structures on quotient spaces}
\label{subsec:Analyticity_of_quotient_spaces}
In order to establish the analyticity of the pure or coupled Yang--Mills energy functions on affine spaces of $W^{1,q}$ connections $\sA^{1,q}(P)$ or pairs $\sP^{1,q}(P,E)$, respectively, it is not necessary to know that their quotient spaces with respect to the action of the group $\Aut^{2,q}(P)$ of gauge transformations are analytic Banach manifolds. Nevertheless, because this readily follows from the proofs of Theorem \ref{mainthm:Feehan_proposition_3-4-4_Lp} and Theorem \ref{mainthm:Feehan_proposition_3-4-4_Lp_pairs}, respectively, we include the relevant statements here for the case of connections, noting that the analogous statements for pairs are similar.

Theorem \ref{mainthm:Feehan_proposition_3-4-4_Lp} provides the essential ingredient one needs to show not only that the quotient space $\sB(P) := \sA^{1,q}(P)/\Aut^{2,q}(P)$ is a $C^\infty$ but also a \emph{real analytic} Banach manifold away from orbits $[A] = \{u(A): u \in \Aut^{2,q}(P)\}$ corresponding to $W^{1,q}$ connections $A$ on $P$ whose stabilizers (or isotropy groups), $\Stab(A) := \{u \in \Aut^{2,q}(P): u(A) = A\}$, are non-minimal, that is, contain the $\Center(G)$ as a proper subgroup. To show that
\[
\sB^*(P) = \left\{A \in \sA^{1,q}(P): \Stab(A) = \Center(G) \right\}/\Aut^{2,q}(P),
\]
is a $C^\infty$ Banach manifold, one only needs the `easy case' of Theorem \ref{mainthm:Feehan_proposition_3-4-4_Lp} where $p=q$, as the condition $q>d/2$ ensures that the proofs using $H^{k+1}(X)$ Sobolev spaces (with $d=4$ and $k\geq 2$) due to Donaldson and Kronheimer \cite[Sections 4.2.1 and 4.2.2]{DK} or Freed and Uhlenbeck \cite[pp. 48-51]{FU} apply \emph{mutatis mutandis}. We have the following analogue of \cite[Proposition 4.2.9]{DK}, \cite[Corollary, p. 50]{FU}, for real analytic Banach manifolds and $X$ of dimension $d\geq 2$ rather than $C^\infty$ Hilbert manifolds and $X$ of dimension four.

\begin{maincor}[Real analytic Banach manifold structure on the quotient space of $W^{1,q}$ connections]
\label{maincor:Slice}
Let $(X,g)$ be a closed, smooth Riemannian manifold of dimension $d \geq 2$, and $G$ be a compact Lie group, and $P$ be a smooth principal $G$-bundle over $X$, and $q$ obey $d/2<q<\infty$. If $A_1$ is a $C^\infty$ reference connection on $P$ and $[A] \in \sB(P)$, then there is a constant $\eps = \eps(A_1,[A],g,G,q) \in (0,1]$ with the following significance. If
\[
\bB_A(\eps) := \left\{a \in W_{A_1}^{1,q}(X;\Lambda^1\otimes\ad P): d_A^*a = 0 \text{ and } \|a\|_{W_{A_1}^{1,q}(X)} < \eps\right\},
\]
then the map,
\[
\pi_A: \bB_A(\eps)/\Stab(A) \ni [a] \mapsto [A+a] \in \sB(P),
\]
is a homeomorphism onto an open neighborhood of $[A]\in \sB(P)$. For $a \in \bB_A(\eps)$, the stabilizer of $a$ in $\Stab(A)$ is naturally isomorphic to that of $\pi_A(a)$ in $\Aut^{2,q}(P)$. In particular, the inverse coordinate charts, $\pi_A$, determine real analytic transition functions for $\sB^*(P)$, giving it the structure of a real analytic Banach manifold, and each map $\pi_A$ is a real analytic diffeomorphism from the open subset of points $[a] \in \bB_A(\eps)/\Stab(A)$ where $\pi_A(a)$ has stabilizer isomorphic to $\Center(G)$.
\end{maincor}

As in \cite[p. 328]{TauFrame}, one may consider the quotient space of framed connections modulo gauge transformations, $\sB'(P) := (\sA^{1,q}(P)\times P|_{x_0})/\Aut^{2,q}(P)$, for some fixed base point $x_0 \in X$, and now the obvious analogue of Theorem \ref{mainthm:Feehan_proposition_3-4-4_Lp} shows that $\sB'(P)$ is a real analytic Banach manifold.

Corollary \ref{maincor:Slice} may be easily extended to the setting of pairs by applying Theorem \ref{mainthm:Feehan_proposition_3-4-4_Lp_pairs} in place of Theorem \ref{mainthm:Feehan_proposition_3-4-4_Lp}. We leave such extensions to the reader, but refer to \cite[Theorem 4.2]{ParkerGauge} and \cite[Proposition 2.8]{FL1} for statements and proofs of $C^\infty$ Banach manifold structures for quotient spaces of pairs.

\section{Outline of the monograph}
\label{sec:Outline}
To apply Theorem \ref{mainthm:Lojasiewicz-Simon_gradient_inequality} to pure or coupled Yang--Mills energy functions and obtain the best possible results in those applications, one requires the global Coulomb gauge constructions provided by Theorems \ref{mainthm:Feehan_proposition_3-4-4_Lp} or \ref{mainthm:Feehan_proposition_3-4-4_Lp_pairs}, and those results are proved in Chapter \ref{chap:Coulomb_gauge_slice_quotient_space_connections_pairs}. In Chapter \ref{chap:Lojasiewicz-Simon_gradient_inequality_coupled_Yang--Mills_energy_functions}, we derive {\L}ojasiewicz--Simon gradient inequalities for the coupled Yang--Mills energy functions, proving Theorems \ref{mainthm:Lojasiewicz-Simon_gradient_inequality_boson_Yang--Mills_energy_function}, \ref{mainthm:Lojasiewicz-Simon_gradient_inequality_fermion_Yang--Mills_energy_function}, \ref{mainthm:Lojasiewicz-Simon_gradient_inequality_Yang--Mills--Higgs_energy_function}, \ref{mainthm:Lojasiewicz-Simon_gradient_inequality_Seiberg-Witten_energy_function}, \ref{mainthm:Lojasiewicz-Simon_gradient_inequality_non-Abelian_monopole_energy_function}, and Corollary \ref{maincor:Lojasiewicz-Simon_gradient_inequality_boson_Yang--Mills_energy_function}. In Chapter \ref{chap:Lojasiewicz-Simon_gradient_inequality_Hilbert_space}, we extend the methods of Chapter \ref{chap:Lojasiewicz-Simon_gradient_inequality_coupled_Yang--Mills_energy_functions} to derive the stronger {\L}ojasiewicz--Simon $L^2$ gradient inequalities for boson and fermion coupled Yang--Mills energy functions, proving Corollaries \ref{maincor:Lojasiewicz-Simon_L2_gradient_inequality_boson_Yang-Mills_energy_function} and \ref{maincor:Lojasiewicz-Simon_L2_gradient_inequality_fermion_Yang-Mills_energy_function}. Appendix \ref{chap:Fredholm_properties_elliptic_PDEs_Sobolev_spaces} establishes the Fredholm properties and computes the index of an elliptic partial differential operator with smooth coefficients acting on Sobolev spaces, Appendix \ref{chap:Equivalence_Sobolev_norms_for_Sobolev_and_smooth_connections} discusses the equivalence of Sobolev norms defined by Sobolev and smooth connections, and Appendix \ref{chap:Fredholm_properties_Hodge_Laplacian_Sobolev_coefficients} establishes the Fredholm properties and computes the index of a Hodge Laplacian with smooth and Sobolev coefficients. Appendix \ref{chap:Convergence_gradient_flows_validity_Lojasiewicz-Simon_gradient_inequality} contains a summary of key convergence results for gradient flows under the validity of a {\L}ojasiewicz--Simon gradient inequality.  Appendix \ref{chap:Huang_2-4-2} contains a review of Huang's \cite[Theorem 2.4.2 (i)]{Huang_2006} for the {\L}ojasiewcz--Simon gradient inequality for analytic functions on Banach spaces. Because precise statements in useful generality are difficult to find in the literature, in Appendix \ref{chap:Quantitative_implicit_function_theorem} we provide statements and proofs of quantitative implicit and inverse function theorems for maps of Banach spaces.

\section{Notation and conventions}
\label{sec:Notation}
For the notation of function spaces, we follow Adams and Fournier \cite{AdamsFournier}, and for functional analysis, Brezis \cite{Brezis} and Rudin \cite{Rudin}. We let $\NN:=\left\{0,1,2,3,\ldots\right\}$ denote the set of non-negative integers. We use $C=C(*,\ldots,*)$ to denote a constant which depends at most on the quantities appearing on the parentheses. In a given context, a constant denoted by $C$ may have different values depending on the same set of arguments and may increase from one inequality to the next. If $\sX, \sY$ is a pair of Banach spaces, then $\sL(\sX,\sY)$ denotes the Banach space of all continuous linear operators from $\sX$ to $\sY$. We denote the continuous dual space of $\sX$ by $\sX^* = \sL(\sX,\RR)$. We write $\alpha(x) = \langle x, \alpha \rangle_{\sX\times\sX^*}$ for the canonical pairing between $\sX$ and its dual space, where $x \in \sX$ and $\alpha \in \sX^*$. If $T \in \sL(\sX, \sY)$, then its adjoint is denoted by $T^* \in \sL(\sY^*,\sX^*)$, where $(T^*\beta)(x) := \beta(Tx)$ for all $x \in \sX$ and $\beta \in \sY^*$. If $x_0 \in \sX$ and $r \in (0,\infty)$, we let $B_r(x_0)$ denote the open ball of radius $r$ and, if $x_0$ is the origin, denote the ball by $B_r$.

\chapter[Existence of Coulomb gauge transformations]{Existence of Coulomb gauge transformations for connections and pairs}
\label{chap:Coulomb_gauge_slice_quotient_space_connections_pairs}
In Section \ref{sec:Coulomb_gauge_slice_quotient_space_connections}, we prove our refinement, Theorem \ref{mainthm:Feehan_proposition_3-4-4_Lp}, of the standard construction of a $W^{2,q}$ Coulomb gauge transformation $u$, with $q > d/2$, for a $W^{1,q}$ connection $A$ on a principal $G$-bundle $P$ over a closed Riemannian smooth manifold of dimension $d \geq 2$. We extend this result in Section \ref{sec:Coulomb_gauge_slice_quotient_space_pairs} to the action of gauge transformations on affine spaces of $W^{1,q}$ pairs, obtaining our refinement, Theorem \ref{mainthm:Feehan_proposition_3-4-4_Lp_pairs} of the standard constructions of Coulomb gauge transformations in that context due to Parker \cite{ParkerGauge} and the first author and Leness \cite{FL1}. Finally, in Section \ref{sec:Regularity_for_pairs}, we extend known regularity results for solutions to the Yang--Mills equations in dimensions greater than or equal to two and coupled Yang--Mills equations in dimension four to the case of solutions to the coupled Yang--Mills equations in dimensions greater than or equal to two.

\section{Action of Sobolev gauge transformations on Sobolev connections}
\label{sec:Action_Sobolev_gauge_transformations}
Suppose that $P$ is a smooth principal $G$-bundle over a smooth manifold $X = P/G$ of dimension $d$, where $P\times G \to P$ is a right action of $G$ on $P$. For $q > d/2$, let $\Aut^{2,q}(P)$ denote the Banach Lie group of Sobolev $W^{2,q}$ automorphisms (or gauge transformations) of  $P$ \cite[Section 2.3.1]{DK}, \cite[Appendix A and p. 32 and pp. 45--51]{FU}, \cite[Section 3.1.2]{FrM}. We recall that there is a smooth left action,
\[
\Aut^{2,q}(P)\times P \to P,
\]
which commutes with the right action of $G$ on $P$. This induces a smooth right (affine) action on the affine space $\sA^{1,q}(P)$ of Sobolev $W^{1,q}$ connections on $P$,
\begin{equation}
\label{eq:Right_action_group_gauge_transformations_on_affine_space_connections}
\sA^{1,q}(P) \times \Aut^{2,q}(P) \ni (A,u) \to u(A) \in \sA^{1,q}(P),
\end{equation}
defined by pull-back,
\[
u(A) := u^*A, \quad\forall\, u \in \Aut^{2,q}(P) \text{ and } A \in \sA^{1,q}(P).
\]
We recall from Definition \ref{defn:Boson_and_fermion_coupled_Yang--Mills_energy_function} that we (implicitly) assume that $G$ is a matrix Lie group, equipped with a unitary representation, $\varrho: G \to \Aut_\CC(\EE)$. The constraint $q > d/2$ is required to ensure that $W^{2,q}(X) \subset C(X)$ by the Sobolev Embedding \cite[Theorem 4.12]{AdamsFournier} and thus $u \in \Aut^{2,q}(P)$ is a continuous gauge transformation of $P$ and that $W^{2,q}(X)$ is a Banach algebra by \cite[Theorem 4.39]{AdamsFournier}.

Given a $W^{1,q}$ connection $A_0$ on $P$ and a $C^\infty$ reference connection $A_1$ on $P$, the standard construction of a slice for the action of $\Aut^{2,q}(P)$ on $\sA^{1,q}(P)$ provides constants $\eps = \eps(A_0,A_1,g,P) \in (0,1]$ and  $C = C(A_0,A_1,g,P) \in [1,\infty)$ such that if $A$ is close to $A_0$ in the sense that,
\[
\|A - A_0\|_{W_{A_1}^{1,q}(X)} < \eps,
\]
then there exists $u \in \Aut^{2,q}(P)$ such that $u(A)$ is in \emph{Coulomb gauge relative to $A_0$}, that is,
\[
d_{A_0}^*(u(A) - A_0) = 0,
\]
and $u(A)$ is close to $A_0$,
\[
\|u(A) - A_0\|_{W_{A_1}^{1,q}(X)} < C\eps.
\]
For example, see \cite[Proposition 2.3.4]{DK}, \cite[Theorem 3.2]{FU}, \cite[Theorem 10.4]{Lawson} or \cite[Theorem 1.1]{FeehanSlice} for statements of the Slice Theorem and their proofs using the Implicit Function Theorem for smooth maps of Banach spaces.

Our Theorem \ref{mainthm:Feehan_proposition_3-4-4_Lp} relaxes the condition that $A$ be $W_{A_1}^{1,q}$ close to $A_0$ for $q>d/2$ to $W_{A_1}^{1,p}$ close for $p$ obeying $d/2 \leq p \leq q$ when $d \geq 3$ and provides $W^{1,p}$ bounds for $u(A)-A_0$ in terms of $A-A_0$. This is significant since $\Aut^{2,\frac{d}{2}}(P)$ is not a smooth manifold and the action \eqref{eq:Right_action_group_gauge_transformations_on_affine_space_connections} cannot be smooth when $q = d/2$, so the Implicit Function Theorem does not apply.

\section[Estimates for Laplace operators with Sobolev coefficients]{A priori estimates for Laplace operators with Sobolev coefficients and existence and uniqueness of strong solutions}
\label{sec:Apriori_estimates_Laplace_operators_Sobolev_coefficients}
Before proceeding to the proof of Theorem
\ref{mainthm:Feehan_proposition_3-4-4_Lp}, we begin with some
preparatory lemmata and remarks that have some interest in their own
right. Standard theory for existence and uniqueness of strong
solutions to (scalar) second-order elliptic partial differential
equations, such as \cite[Chapter 9]{GilbargTrudinger}, requires that
the second-order coefficients be continuous and the lower-order
coefficients be bounded. Here, we observe that one can relax those
requirements on the lower-order coefficients and accommodate the
setting we employ in this monograph.

For a smooth connection $A$ on $P$ and integers $l \geq 0$, we let
\begin{equation}
\label{eq:Hodge_Laplace_operator}
\Delta_A = d_A^*d_A + d_Ad_A^* \quad\text{on } \Omega^l(X;\ad P)
\end{equation}
denote the \emph{Hodge Laplace operator}. Our proof of Theorem \ref{mainthm:Feehan_proposition_3-4-4_Lp} will require \apriori $L^p$ estimates, existence and uniqueness results, Fredholm properties, and Hodge decompositions involving the Hodge Laplacian \eqref{eq:Hodge_Laplace_operator} when $A$ is a $W^{1,q}$ \emph{Sobolev} connection. When $A$ is a $C^\infty$ connection and we restrict our attention to $p=2$, those properties are immediate consequences of more general results (for example, see Gilkey \cite{Gilkey2}) for elliptic operators on sections of vector bundles over closed manifolds.

\begin{prop}[\Apriori $L^p$ estimate for a Laplace operator with Sobolev coefficients]
\label{prop:W2p_apriori_estimate_Delta_A_Sobolev}
Let $(X,g)$ be a closed, smooth Riemannian manifold of dimension $d \geq 2$, and $G$ be a compact Lie group, and $P$ be a smooth principal $G$-bundle over $X$, and $l\geq 0$ be an integer. If $A$ is a $W^{1,q}$ connection on $P$ with $q > d/2$, and $A_1$ is a $C^\infty$ connection on $P$, and $p$ obeys $d/2 \leq p \leq q$, then
\begin{equation}
\label{eq:DeltaA_W2p_to_W1p}
\Delta_A:W_{A_1}^{2,p}(X;\Lambda^l\otimes\ad P) \to L^p(X;\Lambda^l\otimes\ad P)
\end{equation}
is a bounded operator. If in addition $p \in (1,\infty)$, then there is a constant $C = C(A,A_1,g,G,l,p,q) \in [1,\infty)$ such that
\begin{equation}
\label{eq:W2p_apriori_estimate_Delta_A_Sobolev}
\|\xi\|_{W_{A_1}^{2,p}(X)} \leq C\left(\|\Delta_A\xi\|_{L^p(X)}
+ \|\xi\|_{L^p(X)}\right),
\quad
\forall\, \xi \in W_{A_1}^{2,p}(X;\Lambda^l\otimes\ad P).
\end{equation}
\end{prop}

\begin{rmk}[Regularity of distributional solutions to elliptic partial differential equations]
\label{rmk:Regularity_distributional_solutions_elliptic_PDEs}
Suppose as in the hypotheses of Proposition \ref{prop:W2p_apriori_estimate_Delta_A_Sobolev} that $A_1$ is a smooth connection on $P$. By analogy with
\cite[Definition 2.56]{Maly_Ziemer_1997}, we call $\xi \in L^1(X;\Lambda^l\otimes\ad P)$ a \emph{distributional solution} to the equation $\Delta_{A_1}\xi = 0$ if
\[
(\xi, \Delta_{A_1}\eta)_{L^2(X)} = 0,
\quad\forall\, \eta \in C^\infty(X;\Lambda^l\otimes\ad P).
\]
In the case of the scalar Laplace operator on functions, $C^\infty$-smoothness of distributional solutions is provided by \emph{Weyl's Lemma} \cite[Theorem 18.G]{Zeidler_nfaa_v2a}. More generally, the $C^\infty$-smoothness of a solution $u \in L_\loc^1(\Omega)$ to a scalar (second-order) elliptic equation on an open subset $\Omega \subset \RR^d$ is a consequence of regularity theory for solutions in $H_\loc^s(\Omega)$, for $s \in \RR$ \cite[Theorem 6.33]{Folland}. Such regularity results extend to the case of elliptic systems (see \cite{Feehan_yang_mills_gradient_flow_v4} and references therein) and so we conclude that if $\xi$ is a \emph{distributional solution} to the equation $\Delta_{A_1}\xi = 0$, then $\xi \in C^\infty(X;\Lambda^l\otimes\ad P)$.
\end{rmk}

Proposition \ref{prop:W2p_apriori_estimate_Delta_A_Sobolev} implies that the domain of the unbounded operator $\Delta_A$ on $L^p(X;\Lambda^l\otimes\ad P)$ is
\[
\sD_p(\Delta_A) = W_{A_1}^{2,p}(X;\Lambda^l\otimes\ad P).
\]
(We omit the subscript $p$ when that is clear from the context.) In order to give criteria for when the term $\|\xi\|_{L^p(X)}$ can be eliminated from the right-hand side of the \apriori estimate \eqref{eq:W2p_apriori_estimate_Delta_A_Sobolev}, we need to analyze the spectrum of the Hodge Laplacian with Sobolev coefficients. The forthcoming Proposition \ref{prop:Gilbarg_Trudinger_theorem_8-6} is an analogue of \cite[Theorem 8.6]{GilbargTrudinger}, for a scalar, second-order, strictly elliptic equation in divergence form with homogeneous Dirichlet boundary condition over a bounded domain $\Omega \subset \RR^d$. However, it is not a direct consequence since the first and zeroth-order coefficients of the Laplace operator $\Delta_A$ on $\Omega^l(X;\ad P)$ are not necessarily bounded unless $q > d$, which we do not wish to assume, for a $W^{1,q}$ connection $A$ on $P$.

Let $\sX$ be a Banach space and $T:\sX\to\sX$ a bounded operator. Recall from \cite[Definition 4.17 (c)]{Rudin} that the \emph{spectrum}, $\sigma(T)$, of $T$ is the set of all $\lambda\in\CC$ such that $T - \lambda$ is not invertible. Thus $\lambda \in \sigma(T)$ if and only if at least one of the following two statements is true:
\begin{inparaenum}[\itshape a\upshape)]
\item The range of $T - \lambda$ is not all of $\sX$, or
\item $T - \lambda$ is not one-to-one.
\end{inparaenum}
In the latter case, $\lambda$ is an \emph{eigenvalue} of $T$; the corresponding \emph{eigenspace} is $\Ker(T - \lambda)$; each $x \in \Ker(T - \lambda)$ (except $x = 0$) is an \emph{eigenvector} of $T$ and satisfies the equation $Tx = \lambda x$.

If $T:\sD(\sX)\subset\sX\to\sX$ is a closed operator with dense domain, $\sD(\sX) \subset \sX$, then we say that $\lambda \notin \sigma(T)$ if the operator $T-\lambda:\sD(\sX)\to\sX$ has a bounded inverse and otherwise that $\lambda \in \sigma(T)$ \cite[Section 5.6, p. 357]{KadisonRingrose1}, \cite[Section III.6.1, pp. 174--175]{Kato}.

\begin{prop}[Spectral properties of a Laplace operator with Sobolev coefficients]
\label{prop:Gilbarg_Trudinger_theorem_8-6}
Let $(X,g)$ be a closed, smooth Riemannian manifold of dimension $d \geq 2$, and $G$ be a compact Lie group, $P$ be a smooth principal $G$-bundle over $X$, and $l \geq 0$ be an integer. If $A$ is a $W^{1,q}$ connection on $P$ with $d/2 < q < \infty$, and $A_1$ is a $C^\infty$ reference connection on $P$, and $p \in (1,\infty)$ obeys $d/2 \leq p \leq q$, then the spectrum, $\sigma(\Delta_A)$, of the unbounded operator,
\[
\Delta_A: \sD(\Delta_A) \subset L^p(X;\Lambda^l\otimes\ad P) \to L^p(X;\Lambda^l\otimes\ad P),
\]
is countable without accumulation points, consisting of non-negative, real eigenvalues, $\lambda$, with finite multiplicities equal to $\dim\Ker(\Delta_A-\lambda)$.
\end{prop}

\begin{proof}
Corollary \ref{cor:Fredholmness_and_index_Laplace_operator_on_W2p_Sobolev_connection} implies that the operator,
\[
\Delta_A: W_{A_1}^{2,p}(X;\Lambda^l\otimes\ad P) \to L^p(X;\Lambda^l\otimes\ad P),
\]
is Fredholm and, setting
\[
  K := \Ker(\Delta_A: W_{A_1}^{2,p}(X;\Lambda^l\otimes\ad P) \to L^p(X;\Lambda^l\otimes\ad P)),
\]
that
\[
\Delta_A: K^\perp\cap W_{A_1}^{2,p}(X;\Lambda^l\otimes\ad P) \to K^\perp\cap L^p(X;\Lambda^l\otimes\ad P),
\]
is invertible. We denote the Green's operator of $\Delta_A$ by
\[
G_A: L^p(X;\Lambda^l\otimes\ad P) \to W_{A_1}^{2,p}(X;\Lambda^l\otimes\ad P),
\]
so that $G_A\Delta_A = 1-\Pi_A$, where
\[
  \Pi_A:W_{A_1}^{2,p}(X;\Lambda^l\otimes\ad P) \to K
\]
is $L^2$-orthogonal projection, and $\Delta_AG_A = 1-\Pi_A$, where
\[
  \Pi_A:L^p(X;\Lambda^l\otimes\ad P) \to K
\]
again denotes $L^2$-orthogonal projection.

The Sobolev embedding,
\[
  W_{A_1}^{2,p}(X;\Lambda^l\otimes\ad P) \Subset L^p(X;\Lambda^l\otimes\ad P),
\]
is compact by \cite[Theorem 6.3]{AdamsFournier} and hence the composition of $G_A$ with this embedding,
\[
G_A: L^p(X;\Lambda^l\otimes\ad P) \to L^p(X;\Lambda^l\otimes\ad P),
\]
is compact by \cite[Proposition 6.3]{Brezis}. But then \cite[Theorem 4.25]{Rudin} implies that preceding operator has the spectral properties, aside from reality and non-negativity,  described in the conclusion of Proposition \ref{prop:Gilbarg_Trudinger_theorem_8-6}.

To relate the spectra of $G_A$ and $\Delta_A$, observe that for any $\lambda \in \CC\less\{0\}$ and $\chi \in K^\perp\cap L^p(X;\Lambda^l\otimes\ad P)$, the equation,
\[
(\Delta_A-\lambda)\xi = \chi,
\]
for $\xi \in K^\perp\cap W_{A_1}^{2,p}(X;\Lambda^l\otimes\ad P)$ is equivalent to the equation,
\[
(G_A\Delta_A-\lambda G_A)\xi = G_A\chi,
\]
that is,
\[
(G_A - \lambda^{-1})\xi = -\lambda^{-1}G_A\chi.
\]
In other words, $\lambda\neq 0$ is in the spectrum of $\Delta_A$ if and only if $\lambda^{-1}$ is in the spectrum of $G_A$.

To see that the eigenvalues of $G_A$ are real as claimed, note that $\Delta_A$ has $L^2$-adjoint, $\Delta_A^* = \Delta_A$, and so $G_A$ has $L^2$-adjoint, $G_A^* = G_A$, and the operator,
\[
G_A: L^2(X;\Lambda^l\otimes\ad P) \to L^2(X;\Lambda^l\otimes\ad P),
\]
is bounded and self-adjoint. Thus, $\sigma(G_A) \subset \RR$ by \cite[Theorem VI.8]{Reed_Simon_v1} and hence $\sigma(\Delta_A) \subset \RR$.

Finally, to see that the eigenvalues of $\Delta_A$ are non-negative, observe that if $\Delta_A\xi = \lambda\xi$ for $\lambda\in\sigma(\Delta_A)\less\{0\}$ and $\xi \in W_{A_1}^{2,p}(X;\Lambda^l\otimes\ad P)\less\{0\}$, then
\begin{multline*}
  \lambda\|\xi\|_{L^2(X)}^2 = (\lambda\xi,\xi)_{L^2(X)} = (\Delta_A\xi,\xi)_{L^2(X)}
  \\
  = (d_A\xi,d_A\xi)_{L^2(X)} + (d_A^*\xi,d_A^*\xi)_{L^2(X)} \geq 0
\end{multline*}
by \eqref{eq:Hodge_Laplace_operator}, and thus $\lambda \geq 0$, as claimed.
\end{proof}

\begin{rmk}[Spectral properties of a Laplace operator with Sobolev coefficients on $L^p$ spaces and compact perturbations]
\label{rmk:Gilbarg_Trudinger_theorem_8-6_and_Weyl_compact_perturbation}
We recall from Weyl's Theorem \cite[Theorem IV.5.35]{Kato} that if $T$ is a closed operator on a Banach space $\sX$ and $K$ is an operator on $\sX$ that is compact relative to $T$, then $T$ and $T+K$ have the same essential spectrum. In particular, under the hypotheses of Corollary \ref{cor:Fredholmness_and_index_Laplace_operator_on_W2p_Sobolev_connection}, the operator,
\[
\Delta_A - \Delta_{A_1}:W_{A_1}^{2,p}(X;\Lambda^l\otimes\ad P)
\to
L^p(X;\Lambda^l\otimes\ad P),
\]
is compact by the proof of that corollary. Therefore, the essential spectrum of $\Delta_A$ as an unbounded operator on $L^p(X;\Lambda^l\otimes\ad P)$ is empty and hence the spectrum of $\Delta_A$ consists purely of real eigenvalues with finite multiplicity, since the same is true of $\Delta_{A_1}$. These observations could be used to give an alternative proof of Proposition \ref{prop:Gilbarg_Trudinger_theorem_8-6}, in place of the one that we provide.
\end{rmk}

\begin{cor}[\Apriori $L^p$ estimate for a Laplace operator with Sobolev coefficients]
\label{cor:Spectrum_Delta_A_Sobolev_and_W2p_apriori_estimate}
Let $(X,g)$ be a closed, smooth Riemannian manifold of dimension $d \geq 2$, and $G$ be a compact Lie group, and $P$ be a smooth principal $G$-bundle over $X$, and $l \geq 0$ be an integer. If $A$ is a $W^{1,q}$ connection on $P$ with $d/2<q<\infty$, and $A_1$ is a $C^\infty$ connection on $P$, and $p \in (1,\infty)$ obeys $d/2 \leq p \leq q$, then the kernel $\Ker \Delta_A\cap W_{A_1}^{2,p}(X;\Lambda^l\otimes\ad P)$ of the operator \eqref{eq:DeltaA_W2p_to_W1p} is finite-dimensional and
\begin{equation}
\label{eq:W2p_apriori_estimate_Delta_A_orthogonal_kernel}
\|\xi\|_{W_{A_1}^{2,p}(X)} \leq C\|\Delta_A\xi\|_{L^p(X)},
\quad
\forall\, \xi \in \left(\Ker \Delta_A\right)^\perp
\cap W_{A_1}^{2,p}(X;\Lambda^l\otimes\ad P),
\end{equation}
where $\perp$ denotes $L^2$-orthogonal complement and $C = C(A,A_1,g,G,l,p,q) \in [1,\infty)$.
\end{cor}

Before proceeding to the proofs of these results proper, we begin with the

\begin{lem}[\Apriori $L^p$ estimate for a Laplace operator with smooth coefficients]
\label{lem:W2p_apriori_estimate_Delta_A_smooth}
Assume the hypotheses on $A_1$, $d$, $G$, $l$, $P$, and $(X,g)$ in Proposition \ref{prop:W2p_apriori_estimate_Delta_A_Sobolev} and let $p \in (1,\infty)$. If $A$ is $C^\infty$, then there is a constant $C = C(A,A_1,g,G,l,p) \in [1,\infty)$ such that
\begin{equation}
\label{eq:W2p_apriori_estimate_Delta_A_smooth}
\|\xi\|_{W_{A_1}^{2,p}(X)} \leq C\left(\|\Delta_A\xi\|_{L^p(X)}
+ \|\xi\|_{L^p(X)}\right),
\quad \forall\, \xi \in W_{A_1}^{2,p}(X;\Lambda^l\otimes\ad P).
\end{equation}
\end{lem}

\begin{proof}
Suppose first that $\Delta_g$ is the Laplace-Beltrami operator on $C^\infty(X)$ defined by the Riemannian metric $g$. The \apriori $L^p$ estimate for $\Delta_g$ analogous to \eqref{eq:W2p_apriori_estimate_Delta_A_smooth} can be obtained from the \apriori interior $L^p$ estimate provided by \cite[Theorem 9.11]{GilbargTrudinger} for a scalar, second-order, strictly elliptic operator with $C^\infty$ coefficients defined on a bounded domain $\Omega \subset \RR^d$ with the aid of a $C^\infty$ partition of unity subordinate to a finite set of coordinate charts covering the closed manifold, $X$. For the general case, one first chooses in addition a set of local trivializations for $\Lambda^l\otimes\ad P$ corresponding to the coordinate neighborhoods, after shrinking those neighborhoods if needed. The Bochner-Weitzenb\"ock formula  \cite[Equation (C.7)]{FU}, \cite[Equation (II.1)]{Lawson} for $\Delta_A$ implies that $\Delta_A - \nabla_A^*\nabla_A$ is a first-order differential operator with $C^\infty$ coefficients and that $\Delta_A$ has principal symbol given by the $C^\infty$ Riemannian metric $g$ times the identity endomorphism of $\Lambda^l\otimes\ad P$. (In fact, $\Delta_A = \nabla_A^*\nabla_A$ when $l=0$.) The (first-order) covariant derivative of $\xi \in W_{A_1}^{2,p}(X;\Lambda^l\otimes\ad P)$ may be estimated with the following analogue of the interpolation inequality \cite[Theorem 7.27]{GilbargTrudinger}, valid for $p \in [1,\infty)$,
\begin{equation}
\label{eq:Gilbarg_Trudinger_7-59}
\|\nabla_{A_1}\xi\|_{L^p(X)}
\leq
\eps\|\xi\|_{W_{A_1}^{2,p}(X)} + C\eps^{-1}\|\xi\|_{L^p(X)},
\end{equation}
where $C = C(A_1,g,G,l,p) \in [1,\infty)$ and $\eps$ is any positive constant. The conclusion now follows by combining the preceding observations and using rearrangement with small $\eps$ to remove the term $\|\nabla_{A_1}\xi\|_{L^p(X)}$ from the right-hand side.
\end{proof}

We can now proceed to the

\begin{proof}[Proof of Proposition \ref{prop:W2p_apriori_estimate_Delta_A_Sobolev}]
We choose a $C^\infty$ connection, $A_s$, on $P$ that we regard as a smooth approximation to $A$. We write $A = A_s + a$, with $a \in W_{A_1}^{1,q}(X;\Lambda^l\otimes\ad P)$ obeying a bound $\|a\|_{W_{A_1}^{1,q}(X)} \leq \eps$ with small constant $\eps \in (0,1]$ to be chosen during the proof, and write $A_s = A_1 + a_1$, where $a_1 \in C^\infty(X;\Lambda^l\otimes\ad P)$ may be `large'. We expand $\Delta_A = \Delta_{A_s+a}$ to give
\begin{equation}
\label{eq:Delta_A1+a_expansion_prelim}
\Delta_A\xi = \Delta_{A_s}\xi + \nabla_{A_s}a \times \xi + a\times \nabla_{A_s}\xi + a\times a\times \xi,
\end{equation}
and thus, for $\xi \in W_{A_1}^{2,p}(X;\Lambda^l\otimes\ad P)$,
\begin{equation}
\label{eq:Delta_A1+a_expansion}
\Delta_A\xi = \Delta_{A_s}\xi + \nabla_{A_1}a \times \xi + a_1\times a\times \xi + a\times \nabla_{A_1}\xi + a\times a\times \xi.
\end{equation}
We define $r \in [p,\infty]$ by $1/p = 1/q + 1/r$ and recall that by \cite[Theorem 4.12]{AdamsFournier} we have
\begin{inparaenum}[\itshape i\upshape)]
\item $W^{2,p}(X) \subset L^r(X)$ for any $r \in [1,\infty)$ when $p = d/2$, and
\item $W^{2,p}(X) \subset L^\infty(X)$ when $p > d/2$.
\end{inparaenum}
The expansion \eqref{eq:Delta_A1+a_expansion} and continuous Sobolev multiplication map, $L^q(X)\times L^r(X) \to L^p(X)$, yield
\begin{multline}
\label{eq:Delta_A_Sobolev_minus_Delta_As_smooth_prelim}
\|(\Delta_A - \Delta_{A_s})\xi\|_{L^p(X)}
\leq
z\|\nabla_{A_1}a\|_{L^q(X)}\|\xi\|_{L^r(X)} + \|a \times \nabla_{A_1}\xi\|_{L^p(X)}
\\
+ z\|a_1\|_{C(X)}\|a\|_{L^{2p}(X)}\|\xi\|_{L^{2p}(X)} +  z\||a|^2\|_{L^q(X)} \|\xi\|_{L^r(X)},
\end{multline}
where $z = z(g,G,l) \in [1,\infty)$. To ensure a continuous Sobolev embedding, $W^{1,p}(X) \subset L^d(X)$, by \cite[Theorem 4.12]{AdamsFournier}, we need $p^* = dp/(d-p) \geq d$, that is, $p \geq d-p$ or $p \geq d/2$, which we assume in our hypotheses.

To ensure a continuous Sobolev embedding, $W^{1,q}(X) \subset L^{2q}(X)$, when $q < d$, we need $q^* = dq/(d-q) \geq 2q$, that is, $d \geq 2d-2q$ or $2q \geq d$ or $q \geq d/2$, which follows from our hypothesis that $q \geq p \geq d/2$; when $q \geq d$, the fact that $W^{1,q}(X) \subset L^{2q}(X)$ is a continuous Sobolev embedding is immediate from \cite[Theorem 4.12]{AdamsFournier}.

Consequently, by the applying these continuous Sobolev embeddings and the Kato Inequality \cite[Equation (6.20)]{FU} to the preceding inequality, we obtain
\begin{multline*}
\|(\Delta_A - \Delta_{A_s})\xi\|_{L^p(X)}
\leq
z\left(\|\nabla_{A_1}a\|_{L^q(X)} + \|a\|_{W_{A_1}^{1,q}(X)}^2 \right)
\|\xi\|_{W_{A_1}^{2,p}(X)}
\\
+ z\|a_1\|_{C(X)}\|a\|_{W_{A_1}^{1,q}(X)}\|\xi\|_{W_{A_1}^{1,p}(X)}
+ \|a \times \nabla_{A_1}\xi\|_{L^p(X)},
\end{multline*}
where $z = z(g,G,l,p,q) \in [1,\infty)$.

When $q < d$, we recall from \cite[Theorem 4.12]{AdamsFournier} that there is a continuous embedding $W^{1,q}(X) \subset L^{q^*}(X)$, where $q^* = dq/(d-q)$. Hence, $1/q^* = 1/q - 1/d$ or $1/q = 1/q^* + 1/d$ and so, using $p \leq q$,
\[
\|a\times \nabla_{A_1}\xi\|_{L^p(X)}
\leq
z\|a\times \nabla_{A_1}\xi\|_{L^q(X)}
\leq
z\|a\|_{L^{q^*}(X)}\|\nabla_{A_1}\xi\|_{L^d(X)},
\]
where $z = z(g,p,q) \in [1,\infty)$, and therefore, by the preceding continuous Sobolev embeddings,
\begin{equation}
\label{eq:Lp_bound_a1_times_nablaA1_xi}
\|a\times \nabla_{A_1}\xi\|_{L^p(X)}
\leq
C\|a\|_{W_{A_1}^{1,q}(X)}\|\nabla_{A_1}\xi\|_{W_{A_1}^{1,p}(X)},
\end{equation}
where $C = C(g,G,l,p,q) \in [1,\infty)$. When $q = d$ and $d/2 \leq p < d$, we can define $t \in [d,\infty)$ by $1/p = 1/t + 1/d$ and apply the continuous Sobolev embedding $W^{1,d}(X) \subset L^t(X)$ from \cite[Theorem 4.12]{AdamsFournier} to give
\[
\|a\times \nabla_{A_1}\xi\|_{L^p(X)}
\leq
z\|a\|_{L^t(X)}\|\nabla_{A_1}\xi\|_{L^d(X)}
\leq
C\|a\|_{W_{A_1}^{1,q}(X)}\|\nabla_{A_1}\xi\|_{W_{A_1}^{1,p}(X)},
\]
and \eqref{eq:Lp_bound_a1_times_nablaA1_xi} again holds; when $q = d = p$, we can simply use the embedding $W^{1,d}(X) \subset L^t(X)$ for any $t \in [1,\infty)$ and observe that \eqref{eq:Lp_bound_a1_times_nablaA1_xi} holds from
\[
\|a\times \nabla_{A_1}\xi\|_{L^d(X)}
\leq
z\|a\|_{L^{2d}(X)}\|\nabla_{A_1}\xi\|_{L^{2d}(X)}
\leq
C\|a\|_{W_{A_1}^{1,d}(X)}\|\nabla_{A_1}\xi\|_{W_{A_1}^{1,d}(X)}.
\]
Finally, when $q > d$ we have the continuous Sobolev embedding $W^{1,q}(X) \subset C(X)$ from \cite[Theorem 4.12]{AdamsFournier} and so
\[
\|a\times \nabla_{A_1}\xi\|_{L^p(X)}
\leq
z\|a\|_{C(X)}\|\nabla_{A_1}\xi\|_{L^p(X)}
\leq
C\|a\|_{W_{A_1}^{1,q}(X)}\|\nabla_{A_1}\xi\|_{L^p(X)},
\]
which also yields \eqref{eq:Lp_bound_a1_times_nablaA1_xi}.

Combining our previous $L^p$ bound for $(\Delta_A - \Delta_{A_s})\xi$ with the inequality \eqref{eq:Lp_bound_a1_times_nablaA1_xi} gives
\begin{align}
\label{eq:Delta_A_Sobolev_minus_Delta_As_smooth}
\|(\Delta_A - \Delta_{A_s})\xi\|_{L^p(X)}
&\leq
z\left( \|\nabla_{A_1}a\|_{L^q(X)}  + \|a\|_{W_{A_1}^{1,q}(X)}^2 \right) \|\xi\|_{W_{A_1}^{2,p}(X)}
\\
&\quad + z\|a_1\|_{C(X)}\|a\|_{W_{A_1}^{1,q}(X)}\|\xi\|_{W_{A_1}^{1,p}(X)} \nonumber
\\
&\quad + z\|a\|_{W_{A_1}^{1,q}(X)}\|\nabla_{A_1}\xi\|_{W_{A_1}^{1,p}(X)}, \nonumber
\end{align}
where $z = z(g,G,l,p,q) \in [1,\infty)$. Combining the preceding bound with the \apriori estimate \eqref{eq:W2p_apriori_estimate_Delta_A_smooth} for $\Delta_{A_s}$ provided by Lemma \ref{lem:W2p_apriori_estimate_Delta_A_smooth},
\[
\|\xi\|_{W_{A_1}^{2,p}(X)}
\leq C_0\left(\|\Delta_{A_s}\xi\|_{L^p(X)} + \|\xi\|_{L^p(X)}\right),
\]
with constant denoted by $C_0 = C_0(A_1,A_s,g,p) \in [1,\infty)$ for clarity, yields
\begin{align*}
\|\xi\|_{W_{A_1}^{2,p}(X)}
&\leq C_0\left(\|\Delta_A\xi\|_{L^p(X)} + \|\xi\|_{L^p(X)}\right)
  \\
  &\quad + z\left( \|a\|_{W_{A_1}^{1,q}(X)}  + \|a\|_{W_{A_1}^{1,q}(X)}^2 \right) \|\xi\|_{W_{A_1}^{2,p}(X)}
\\
&\quad + z\|a_1\|_{C(X)}\|a\|_{W_{A_1}^{1,q}(X)}\|\xi\|_{W_{A_1}^{1,p}(X)}.
\end{align*}
We can choose $a = A-A_s$ so that $\|a\|_{W_{A_1}^{1,q}(X)} \leq \eps$ for small constant $\eps \in (0,1]$, but we are not at liberty to choose $a_1 = A_s-A_1$ to be $W_{A_1}^{1,q}(X)$-small (since $A_1$ is a fixed $C^\infty$ connection on $P$ and $A$ is an arbitrary $W^{1,q}$ connection on $P$). Thus in our forthcoming rearrangement arguments we first apply the interpolation inequality \eqref{eq:Gilbarg_Trudinger_7-59},
\[
\|\nabla_{A_1}\xi\|_{L^p(X)}
\leq
\delta\|\xi\|_{W_{A_1}^{2,p}(X)} + C_1\delta^{-1}\|\xi\|_{L^p(X)},
\]
where $C_1 = C_1(A_1,g) \in [1,\infty)$ and $\delta = \delta(A_1,\|A_s-A_1\|_{C(X)},g,G,l,p,q) \in (0,1]$ is a constant chosen small enough that
\[
\delta z\|a_1\|_{C(X)} \leq 1/2,
\]
and thus, for a constant $C_2 = C_2(A_1,A_s,g,G,l,p,q) \in [1, \infty)$,
\begin{multline*}
\|\xi\|_{W_{A_1}^{2,p}(X)}
\leq 2C_0\left(\|\Delta_A\xi\|_{L^p(X)} + \|\xi\|_{L^p(X)}\right)
\\
+ 2z\left( \|a\|_{W_{A_1}^{1,q}(X)}  + \|a\|_{W_{A_1}^{1,q}(X)}^2 \right) \|\xi\|_{W_{A_1}^{2,p}(X)}
+ C_2\|\xi\|_{L^p(X)}.
\end{multline*}
Provided $\|a\|_{W_{A_1}^{1,q}(X)} \leq \eps$ and we choose $\eps \equiv \eps(g,G,l,p,q) = 1/(8z) \in (0, 1]$ in the preceding inequality, rearrangement yields the desired estimate \eqref{eq:W2p_apriori_estimate_Delta_A_Sobolev}.

Our proof of \eqref{eq:W2p_apriori_estimate_Delta_A_Sobolev} also verifies that the operator $\Delta_A$ in \eqref{eq:DeltaA_W2p_to_W1p} is bounded since $\Delta_{A_s}$ is bounded with the same domain and range spaces. This completes the proof of Proposition \ref{prop:W2p_apriori_estimate_Delta_A_Sobolev}.
\end{proof}

Next, we have the

\begin{proof}[Proof of Corollary \ref{cor:Spectrum_Delta_A_Sobolev_and_W2p_apriori_estimate}]
The fact that $\Ker\Delta_A\cap W_{A_1}^{2,p}(X;\Lambda^l\otimes\ad P)$ is finite-dimensional follows from Corollary \ref{cor:Fredholmness_and_index_Laplace_operator_on_W2p_Sobolev_connection}. We observe that, by increasing the constant $C$ as needed, the term $\|\xi\|_{L^p(X)}$ appearing on the right-hand side of the inequality \eqref{eq:W2p_apriori_estimate_Delta_A_Sobolev} can be replaced by $\|\xi\|_{L^2(X)}$. This is clear when $p \leq 2$, while if $p > 2$, we can choose $s \in (p,\infty)$ and apply the interpolation inequality \cite[Equation (7.10)]{GilbargTrudinger},
\[
\|\xi\|_{L^p(X)} \leq \delta \|\xi\|_{L^s(X)} + \delta^{-\nu}\|\xi\|_{L^2(X)},
\]
for $\nu := (1/2-1/p)/(1/p-1/s) > 0$ and arbitrary positive $\delta$. Because $p \geq d/2$, we have a continuous Sobolev embedding $W^{2,p}(X) \subset L^s(X)$ as already observed in the proof of Proposition \ref{prop:W2p_apriori_estimate_Delta_A_Sobolev}, so
\[
\|\xi\|_{L^p(X)} \leq C_1\delta \|\xi\|_{W_{A_1}^{2,p}(X)} + \delta^{-\nu}\|\xi\|_{L^2(X)},
\]
where $C_1 = C_1(A_1,g,l,p) \in [1,\infty)$. Hence, for $\delta(A_1,g,l,p) \in (0,1]$ given by $\delta = 1/(2CC_1)$, we can use rearrangement in \eqref{eq:W2p_apriori_estimate_Delta_A_Sobolev} to replace $\|\xi\|_{L^p(X)}$ by $\|\xi\|_{L^2(X)}$. Therefore, the estimate \eqref{eq:W2p_apriori_estimate_Delta_A_Sobolev} implies
\begin{equation}
\label{eq:W2p_apriori_estimate_Delta_A_Sobolev_L2_xi_rhs}
\|\xi\|_{W_{A_1}^{2,p}(X)} \leq C\left(\|\Delta_A\xi\|_{L^p(X)}
+ \|\xi\|_{L^2(X)}\right),
\quad
\forall\, \xi \in W_{A_1}^{2,p}(X;\Lambda^l\otimes\ad P).
\end{equation}
Proposition \ref{prop:Gilbarg_Trudinger_theorem_8-6} implies that the spectrum $\sigma(\Delta_A)$ of $\Delta_A$ on $L^2(X;\Lambda^l\otimes\ad P)$ consists purely of non-negative eigenvalues and is discrete with no accumulation points. Let $\mu[A]$ denote the least positive eigenvalue of $\Delta_A$ on $L^2(X;\Lambda^l\otimes\ad P)$ and recall from \cite[Rayleigh's Theorem, p. 16]{Chavel} (or more generally \cite[Theorem 6.5.1]{Stakgold_Holst}, applied to the Green's operator $G_A$ for $\Delta_A$) that
\[
\mu[A]
=
\inf_{\xi \in (\Ker\Delta_A)^\perp}
\frac{(\xi, \Delta_A\xi)_{L^2(X)}}{\|\xi\|_{L^2(X)}^2},
\]
where $(\Ker\Delta_A)^\perp$ is the $L^2$-orthogonal complement of $\Ker\Delta_A \subset H_{A_1}^1(X;\Lambda^l\otimes\ad P)$, with equality achieved in the infimum if and only if $\xi$ is an eigenvector with eigenvalue $\mu[A]$. Therefore, if $\xi \in (\Ker\Delta_A)^\perp \cap W_{A_1}^{2,p}(X;\Lambda^l\otimes\ad P)$, then
\[
\|\xi\|_{L^2(X)} \leq \mu[A]^{-1}\|\Delta_A\xi\|_{L^2(X)}.
\]
Hence, the inequality \eqref{eq:W2p_apriori_estimate_Delta_A_orthogonal_kernel} follows from the preceding eigenvalue bound and \eqref{eq:W2p_apriori_estimate_Delta_A_Sobolev_L2_xi_rhs} when $p \geq 2$.

For the case $d/2 \leq p < 2$ (which forces $d=2$ or $3$), we observe that
\[
\|\xi\|_{L^2(X)}^2
\leq
\frac{1}{\mu[A]}(\Delta_A \xi,\xi)|_{L^2(X)}
\leq
\frac{1}{\mu[A]}\|\Delta_A\xi\|_{L^p(X)} \|\xi\|_{L^{p'}(X)},
\]
where $p'$ is defined by $1/p+1/p'=1$, and so
\begin{multline*}
\|\xi\|_{L^2(X)}
\leq
\frac{1}{\sqrt{\mu[A]}}\|\Delta_A\xi\|_{L^p(X)}^{1/2} \|\xi\|_{L^{p'}(X)}^{1/2}
\\
\leq
\frac{1}{2\sqrt{\mu[A]}}\left(\delta^{-1}\|\Delta_A\xi\|_{L^p(X)} + \delta\|\xi\|_{L^{p'}(X)}\right),
\end{multline*}
for arbitrary positive $\delta$. If $d=2$ and $p\in (1,2)$ (so $p'\in(2,\infty)$, then $W^{2,p}(X) \subset C^0(X)$ by \cite[Theorem 4.12]{AdamsFournier} and thus $W^{2,p}(X) \subset L^{p'}(X)$; if $d=3$ and $p\in [3/2,2)$ (so $p'\in(2,3]$, then $W^{2,p}(X) \subset L^q(X)$ for $p = 3/2$ and any $q\in[1,\infty)$ while $W^{2,p}(X) \subset C^0(X)$ for $p > 3/2$ by \cite[Theorem 4.12]{AdamsFournier} and thus $W^{2,p}(X) \subset L^{p'}(X)$ in either case. Hence, for $d = 2,3$ and $p>1$ obeying $d/2 \leq p < 2$, we have the bound
\[
\|\xi\|_{L^{p'}(X)} \leq C_2\|\xi\|_{W_{A_1}^{2,p}(X)},
\]
where $C_2 = C_2(A_1,g,l,p) \in [1,\infty)$. Combining the preceding inequalities for the case $d/2 \leq p < 2$ gives
\[
\|\xi\|_{L^2(X)}
\leq
\frac{1}{2\sqrt{\mu[A]}}\left(\delta^{-1}\|\Delta_A\xi\|_{L^p(X)}
+ \delta C_2\|\xi\|_{W_{A_1}^{2,p}(X)}\right).
\]
Combining the preceding inequality with \eqref{eq:W2p_apriori_estimate_Delta_A_Sobolev_L2_xi_rhs}
and applying rearrangement by choosing $\delta = \sqrt{\mu[A]}/(CC_2)$ yields the desired the inequality \eqref{eq:W2p_apriori_estimate_Delta_A_orthogonal_kernel} for this case too.
\end{proof}

\section[Regularity for distributional solutions to an elliptic equation]{Regularity for distributional solutions to an elliptic equation with Sobolev coefficients}
\label{sec:Regularity_distributional_solutions_elliptic_equation_Sobolev_coefficients}
We shall need to address a complication that arises when establishing regularity for distributional solutions to an elliptic equation with Sobolev coefficients. We shall confine our discussion to the Hodge Laplace operator, though one can clearly establish more general results of this kind.

\begin{lem}[Regularity for distributional solutions to an equation defined by the Hodge Laplace operator for a Sobolev connection]
\label{lem:Regularity_distributional_solution_Hodge_Laplacian_Sobolev}
Let $(X,g)$ be a closed, smooth Riemannian manifold of dimension $d \geq 2$, and $G$ be a compact Lie group and $P$ be a smooth principal $G$-bundle over $X$. Let $A_1$ be a $C^\infty$ connection on $P$, and $A$ be a $W^{1,q}$ connection on $P$ with $d/2 < q < \infty$, and $l\geq 0$ be an integer. If $q' \in (1,\infty)$ is the dual exponent defined by $1/q+1/q'=1$ and $\eta \in L^{q'}(X;\Lambda^l\otimes\ad P)$ is a distributional solution\footnote{In the sense of Remark \ref{rmk:Regularity_distributional_solutions_elliptic_PDEs}.} to
\[
\Delta_A \eta = 0,
\]
then $\eta \in W_{A_1}^{2,q}(X;\Lambda^l\otimes\ad P)$.
\end{lem}

\begin{proof}
We recall from equation \eqref{eq:Delta_A_Sobolev_minus_Delta_A_smooth_W1u_to_Lp_bounded} in the proof of Corollary \ref{cor:Fredholmness_and_index_Laplace_operator_on_W2p_Sobolev_connection} (with $p=q$) that
\[
\Delta_A - \Delta_{A_1}: W_{A_1}^{1,u}(X;\Lambda^l\otimes\ad P)
\to L^q(X;\Lambda^l\otimes\ad P)
\]
is a bounded operator, where allowable values of $u \in (1,\infty)$ are given by
\begin{equation}
\label{eq:u_exponent_range_for_proof_DeltaA_Fredholm}
u
=
\begin{cases}
d+\eps &\text{if } d/2 < q < d,
\\
2d &\text{if } q = d,
\\
q &\text{if } q > d,
\end{cases}
\end{equation}
and $\eps \in (0,1]$ is chosen small enough that $q^* = dq/(d-q) \geq d+\eps$, which is possible when $q > d/2$. Consequently, the dual operator,
\[
\Delta_A - \Delta_{A_1}: L^{q'}(X;\Lambda^l\otimes\ad P) \to W_{A_1}^{-1,u'}(X;\Lambda^l\otimes\ad P),
\]
is bounded, where the dual exponent $u' \in (1,\infty)$ is defined by $1/u+1/u'=1$.

We write $\Delta_A = \Delta_{A_1} + (\Delta_A - \Delta_{A_1})$, set
\[
  \alpha := (\Delta_{A_1} - \Delta_A)\eta \in W_{A_1}^{-1,u'}(X;\Lambda^l\otimes\ad P),
\]
and observe that $\eta$ is a distributional solution to
\begin{equation}
\label{eq:DeltaA1_eta_equalto_alpha_distributions}
\Delta_{A_1}\eta = \alpha.
\end{equation}
We now appeal to regularity for distributional solutions to an equation (namely, \eqref{eq:DeltaA1_eta_equalto_alpha_distributions}) defined by an elliptic operator $\Delta_{A_1}$ with $C^\infty$ coefficients (see Remark \ref{rmk:Regularity_distributional_solutions_elliptic_PDEs} and \cite{Feehan_yang_mills_gradient_flow_v4}) to conclude that $\eta \in W_{A_1}^{1,u}(X;\Lambda^l\otimes\ad P)$. The range of exponents $u \in (d,\infty)$ given by \eqref{eq:u_exponent_range_for_proof_DeltaA_Fredholm} ensures that $\eta \in C(X;\Lambda^l\otimes\ad P)$, since $W^{1,u}(X) \subset C(X)$ by \cite[Theorem 4.12]{AdamsFournier} when $u > d$. Moreover, the estimate \eqref{eq:Lp_estimate_DeltaASobolev_minus_DeltaA1smooth} for $(\Delta_{A_1} - \Delta_A)\eta$ in terms of $a = A-A_1$ and $\eta$ ensures that
\[
(\Delta_{A_1} - \Delta_A)\eta \in L^q(X;\Lambda^l\otimes\ad P).
\]
In particular, $\alpha \in L^q(X;\Lambda^l\otimes\ad P)$ and regularity for solutions to an elliptic equation (that is, \eqref{eq:DeltaA1_eta_equalto_alpha_distributions}) with $C^\infty$ coefficients implies that $\eta \in W_{A_1}^{2,q}(X;\Lambda^l\otimes\ad P)$, as desired.
\end{proof}

\section{Surjectivity of a perturbed Laplace operator}
\label{sec:Surjectivity_perturbed_Laplace_operator}
We now consider surjectivity properties of a perturbation of a Laplace operator, namely

\begin{lem}[Surjectivity of a perturbed Laplace operator]
\label{lem:Surjectivity_perturbed_Laplace_operator}
Let $(X,g)$ be a closed, smooth Riemannian manifold of dimension $d \geq 2$, and $G$ be a compact Lie group and $P$ be a smooth principal $G$-bundle over $X$. Let $A_1$ be a $C^\infty$ connection on $P$ and  $A$ be a $W^{1,q}$ connection on $P$ with $d/2 < q < \infty$. Then there is a constant $\delta = \delta(A,g) \in (0,1]$ with the following significance. If $a \in W_{A_1}^{1,q}(X;\Lambda^1\otimes\ad P)$ obeys
\begin{equation}
\label{eq:Perturbation_Laplace_operator_small}
\|a\|_{L^d(X)} < \delta \quad \text{when } d \geq 3
\quad\text{or}\quad
\|a\|_{L^4(X)} < \delta \quad \text{when } d = 2,
\end{equation}
then the operator,
\begin{equation}
\label{eq:Perturbation_Laplace_operator_W2p_to_Lp}
d_A^*d_{A+a}:
\left(\Ker \Delta_A\right)^\perp \cap W_{A_1}^{2,q}(X;\ad P)
\to
\left(\Ker \Delta_A\right)^\perp \cap L^q(X;\ad P),
\end{equation}
is well-defined and surjective.
\end{lem}

\begin{rmk}[Comparison with the argument due to Donaldson and Kronheimer]
\label{rmk:Surjectivity_perturbed_Laplace_operator}
Our proof of Lemma \ref{lem:Surjectivity_perturbed_Laplace_operator} is based on a similar argument arising in Donaldson and Kronheimer \cite[p. 66]{DK}, as part of their version of the proof of Uhlenbeck's local Coulomb gauge-fixing theorem \cite[Theorem 2.3.7]{DK}. We make some adjustments to their argument for reasons which we briefly explain here. When $B$ is a Sobolev connection matrix \cite[p. 66]{DK}, one has to take into account the possibility that their operator $d^*d_B$, for a $\fg$-valued Sobolev one-form $B$ over $S^4$, could have dense range but still fail to be surjective, so one would first have to verify, for example, that $d^*d_B$ has closed range. Similarly, while elliptic regularity ensures that if $\eta \in L^2(S^4;\ad P)\cap \Ker d^*d_B$ then $\eta$ is $C^\infty(S^4;\ad P)$ when $B$ is $C^\infty$, those regularity issues become more subtle when $B$ is merely a Sobolev one-form.
\end{rmk}

\begin{rmk}[Dual spaces and direct sums of subspaces of Banach spaces]
\label{rmk:Dual_spaces}
Our proof of Lemma \ref{lem:Surjectivity_perturbed_Laplace_operator} is clarified by a few observations concerning the dual space of a finite-dimensional subspace $K$ of a Banach space $\sX$ that is continuously embedded in a Hilbert space, $\sH$. Since $K$ has finite dimension, it has a closed complement, $\sX_0 \subset \sX$, such that $\sX = \sX_0\oplus K$ (vector space direct sum) by \cite[Definition 4.20 and Lemma 4.21(a)]{Rudin}. We may also simply define $\sX_0 := K^\perp\cap\sX$, where $K^\perp \subset \sH$ is the orthogonal complement of $K \subset\sH$ and a closed subspace of $\sH$ \cite[Theorem 12.4]{Rudin}. Therefore, $K^\perp\cap\sX \subset \sX$ is also a closed subspace and a closed complement for $K \subset \sX$ in the sense of \cite[Definition 4.20]{Rudin}. In summary:
\begin{equation}
\label{eq:Orthogonal_direct_sum_finite-dimensional_and_closed_subspace}
\sX = \sX_0\oplus K, \quad\text{with } \sX_0 := K^\perp\cap\sX,
\end{equation}
is an orthogonal direct sum of a finite-dimensional and a closed subspace.

Recall that if $M \subset \sX$ is any subspace, then $M^\circ := \{\alpha\in\sX^*: \langle u, \alpha\rangle_{\sX\times\sX^*} = 0,\ \forall\, u \in M\}$ denotes the annihilator of $M$ in $\sX^*$ \cite[Section 4.6]{Rudin}. Since $\sX = \sX_0\oplus K$, we have $\sX^* = \sX_0^\circ \oplus K^\circ$, the direct sum of the annihilators in $\sX^*$ of the subspaces $\sX_0$ and $K$ of $\sX$.

Because $K \subset \sX$ and $\sX_0 \subset \sX$ are closed subspaces, we have $K^\circ \cong (\sX/K)^*$ and $\sX_0^\circ \cong K^*$ by \cite[Theorem 4.9]{Rudin}. Since $\sX = \sX_0\oplus K$, then $\sX/K \cong \sX_0$ and  $\sX/\sX_0 \cong K$, so $K^\circ \cong \sX_0^*$ and $\sX_0^\circ \cong K^* \cong K$, where the final isomorphism follows from the fact that $K$ is finite-dimensional. In particular,
$\sX^* \cong K^* \oplus \sX_0^* \cong K\oplus\sX_0^*$. Since $K^*$ is finite-dimensional, then $K^* \subset \sX^*$ is a closed subspace by \cite[Proposition 11.1]{Brezis} and so the complement $\sX_0^* \subset \sX^*$ is a closed subspace by \cite[Definition 4.20 and Lemma 4.21]{Rudin}. In summary:
\begin{equation}
\label{eq:Orthogonal_direct_sum_finite-dimensional_and_closed_subspace_dual_space}
\sX^* = \sX_0^*\oplus K,
\end{equation}
is a direct sum of a finite-dimensional and a closed subspace.
\end{rmk}

\begin{proof}[Proof of Lemma \ref{lem:Surjectivity_perturbed_Laplace_operator}]
Note that when $a=0$, the operator \eqref{eq:Perturbation_Laplace_operator_W2p_to_Lp} is invertible by Corollary \ref{cor:Fredholmness_and_index_Laplace_operator_on_W2p_Sobolev_connection}. To see that the operator \eqref{eq:Perturbation_Laplace_operator_W2p_to_Lp} is well-defined, observe that if $\xi = d_A^*d_{A+a}\chi \in L^q(X;\ad P)$ for some $\chi \in W_{A_1}^{2,q}(X;\ad P)$ and $\eta \in \Ker \Delta_A \cap L^q(X;\ad P)$, then
\[
(\xi,\eta)_{L^2(X)} = (d_A^*d_{A+a}\chi,\eta)_{L^2(X)} = (d_{A+a}\chi,d_A\eta)_{L^2(X)} = 0,
\]
since $\Delta_A\eta = d_A^*d_A\eta = 0$ and hence
\[
  (d_A^*d_A\eta,\eta)_{L^2(X)} = \|d_A\eta\|_{L^2(X)}^2 = 0,
\]
so that $d_A\eta = 0$. Thus,
\begin{multline}
\label{eq:Perturbation_Laplace_operator_W2p_to_Lp_range}
\Ran\left(d_A^*d_{A+a}: W_{A_1}^{2,q}(X;\ad P) \to L^q(X;\ad P) \right)
\\
\subset
\left(\Ker \Delta_A\right)^\perp \cap L^q(X;\ad P),
\end{multline}
and so the operator \eqref{eq:Perturbation_Laplace_operator_W2p_to_Lp} is well-defined.

It is convenient to abbreviate $\sX := W_{A_1}^{2,q}(X;\ad P)$, and $K := \Ker \Delta_A \cap W_{A_1}^{2,q}(X;\ad P)$, and $\sY := L^q(X;\ad P)$, and $\sH := L^2(X;\ad P)$, and denote $T = d_A^*d_{A+a}$ in \eqref{eq:Perturbation_Laplace_operator_W2p_to_Lp}. The kernel $K \subset \sX$ is finite-dimensional by Corollary \ref{cor:Fredholmness_and_index_Laplace_operator_on_W2p_Sobolev_connection}, and its $L^2$-orthogonal complements $K^\perp\cap\sX$ and $K^\perp\cap\sY$ and $K^\perp$ provide closed complements of $K$ in $\sX$ and $\sY$ and $\sH$, respectively. Similarly, we let $K^\perp$ and $K^\perp\cap\sY^*$ and $K^\perp\cap\sX^*$ denote the closed complements of $K^*\cong K$ in $\sH^*\cong \sH$ and $\sY^*$ and $\sX^*$, respectively, where $\sY^* = L^{q'}(X;\ad P)$, with $q' \in (1,\infty)$ defined by $1/q+1/q'=1$, and $\sX^* = W_{A_1}^{-2,q'}(X;\ad P)$.

If $N \subset \sY^*$ is any subspace, we recall that the \emph{annihilator} \cite[Section 4.6]{Rudin} of $N$ in $\sY$ is
\[
{}^\circ N
:=
\{\xi \in \sY: \langle \xi, \alpha \rangle_{\sY\times\sY^*} = 0, \ \forall\, \alpha \in \sY^* \},
\]
and $\langle \cdot, \cdot\rangle_{\sY\times\sY^*}:\sY\times \sY^* \to \RR$ is the canonical pairing. The operator, $T:\sX \to \sY$, is Fredholm by Lemma \ref{lem:Fredholm_index_zero_perturbed_Laplace_operator_Sobolev} and we can identify its range using
\begin{align*}
\Ran\left( T:\sX \to \sY \right)
&=
\overline{\Ran}\left( T:\sX \to \sY \right) \quad\text{(by closed range)}
\\
&= {}^\circ\Ker\left( T^*:\sY^* \to \sX^* \right) \quad\text{(by \cite[Corollary 2.18 (iv))]{Brezis})}.
\end{align*}
Therefore, we have shown that
\[
\Ran(T:\sX\to\sY) = {}^\circ\Ker(T^*:\sY^* \to \sX^*),
\]
and so $\Ran(T:\sX\to\sY) = \sY$ if and only if $\Ker(T^*:\sY^* \to \sX^*) = 0$. Similarly, because $T:\sX\to\sY$ has closed range and $\Ran(T:\sX\to\sY) \subset K^\perp\cap\sY$ by \eqref{eq:Perturbation_Laplace_operator_W2p_to_Lp_range}, the operator,
\[
T:K^\perp\cap\sX\to K^\perp\cap\sY,
\]
also has closed range and we obtain
\[
\Ran(T:K^\perp\cap\sX\to K^\perp\cap\sY)
= {}^\circ\Ker(T^*:K^\perp\cap\sY^* \to K^\perp\cap\sX^*).
\]
Consequently,
\begin{multline}
\label{eq:Surjectivity_of_T_iff_injectivity_T*}
\Ran(T:K^\perp\cap\sX\to K^\perp\cap\sY) = K^\perp\cap\sY
\\
\iff \Ker(T^*:K^\perp\cap\sY^* \to K^\perp\cap\sX^*) = 0.
\end{multline}
If $T^*:K^\perp\cap\sY^* \to K^\perp\cap\sX^*$ were not injective, there would be a non-zero $\eta \in K^\perp\cap\sY^*$ such that $T^*\eta = 0$. In other words, because $T^* = d_{A+a}^*d_A$, there would be a non-zero $\eta \in K^\perp \cap L^{q'}(X;\ad P)$ such that
\[
d_{A+a}^*d_A\eta = 0 \in W_{A_1}^{-2,q'}(X;\ad P),
\]
that is, $(d_{A+a}^*d_A\eta)(\xi) = 0$ for all $\xi \in K^\perp \cap W_{A_1}^{2,q}(X;\ad P)$ or equivalently,
\begin{multline*}
\langle\chi, d_{A+a}^*d_A\eta\rangle_{\sX\times\sX^*}
=
\langle \eta, d_A^*d_{A+a}\chi\rangle_{\sY\times\sY^*}
=
(\eta, d_A^*d_{A+a}\chi)_{L^2(X)} = 0,
\\
\quad\forall\, \chi \in K^\perp\cap W_{A_1}^{2,q}(X;\ad P).
\end{multline*}
Lemma \ref{lem:Regularity_distributional_solution_Hodge_Laplacian_Sobolev} implies that $\eta \in W_{A_1}^{2,q}(X;\ad P)$. Observe that, by writing $d_{A+a}\chi = d_A\chi + [a,\chi]$,
\begin{align*}
0 &= (d_A^*d_{A+a}\chi, \eta)_{L^2(X)}
\\
&= (d_{A+a}\chi, d_A\eta)_{L^2(X)}
\\
&= (d_A\chi, d_A\eta)_{L^2(X)} + ([a,\chi], d_A\eta)_{L^2(X)}.
\end{align*}
Since $\eta \perp \Ker \Delta_A \cap W_{A_1}^{2,q}(X;\ad P)$ and letting $\mu[A]$ denote the least positive eigenvalue of the Laplace operator $\Delta_A$ on $L^2(X;\ad P)$ provided by Proposition \ref{prop:Gilbarg_Trudinger_theorem_8-6}, we have
\[
\mu[A]
\leq
\frac{(\eta, d_A^*d_A\eta)_{L^2(X)}}{\|\eta\|_{L^2(X)}^2}
=
\frac{\|d_A\eta\|_{L^2(X)}^2}{\|\eta\|_{L^2(X)}^2}
\]
and thus,
\[
\|\eta\|_{L^2(X)}
\leq
\mu[A]^{-1/2}\|d_A\eta\|_{L^2(X)}.
\]
Hence, we obtain
\begin{equation}
\label{eq:Apriori_estimate_W12norm_eta_by_L2_norm_d_eta}
\|\eta\|_{W_A^{1,2}(X)} \leq C_1\|d_A\eta\|_{L^2(X)},
\end{equation}
for a constant $C_1 = C_1(A,g) = 1 + \mu[A]^{-1/2} \in [1,\infty)$. For $d \geq 3$ and using $1/2 = (d-2)/2d + 1/d$ and the continuous multiplication $L^{2d/(d-2)}(X) \times L^d(X) \to L^2(X)$, we see that
\begin{align*}
\left| ([a,\chi], d_A\eta)_{L^2(X)} \right|
&\leq
\|[a,\chi]\|_{L^2(X)} \|d_A\eta\|_{L^2(X)}
\\
&\leq c\|a\|_{L^d(X)} \|\chi\|_{L^{2d/(d-2)}(X)} \|d_A\eta\|_{L^2(X)}
\\
&\leq C_2\|a\|_{L^d(X)} \|\chi\|_{W_A^{1,2}(X)} \|d_A\eta\|_{L^2(X)},
\end{align*}
where $c\in [1,\infty)$ and the constant $C_2 = C_2(g) \in [1,\infty)$ is a multiple by $c$ of the norm of the continuous Sobolev embedding $W^{1,2}(X) \subset L^{2d/(d-2)}(X)$ provided by \cite[Theorem 4.12]{AdamsFournier}. Hence, setting $\chi = \eta$ and applying the \apriori estimate \eqref{eq:Apriori_estimate_W12norm_eta_by_L2_norm_d_eta}, the preceding identity and inequalities yield
\begin{align*}
\|d_A\eta\|_{L^2(X)}^2
&=
\left| ([a,\chi], d_A\eta)_{L^2(X)} \right|
\\
&\leq C_1C_2 \|a\|_{L^d(X)} \|d_A\eta\|_{L^2(X)}^2,
\end{align*}
and so, if $\eta \not\equiv 0$, we have
\[
\|a\|_{L^d(X)} \geq C_1C_2,
\]
contradicting our hypothesis \eqref{eq:Perturbation_Laplace_operator_small} that $\|a\|_{L^d(X)} < \delta$, with $\delta$ small.

For the case $d=2$, we instead use the continuous multiplication $L^4(X) \times L^4(X) \to L^2(X)$ and continuous Sobolev embedding $W^{1,2}(X) \subset L^r(X)$ for $1 \leq r < \infty$  provided by \cite[Theorem 4.12]{AdamsFournier}. In particular, using the embedding $W^{1,2}(X) \subset L^4(X)$ with constant $C_3 = C_3(g) \in [1,\infty)$ a multiple by $c\in[1,\infty)$ of the norm of that embedding, we obtain
\begin{align*}
\|d_A\eta\|_{L^2(X)}^2
&\leq
\left| ([a,\eta], d_A\eta)_{L^2(X)} \right|
\\
&\leq
c\|[a,\eta]\|_{L^2(X)} \|d_A\eta\|_{L^2(X)}
\\
&\leq c\|a\|_{L^4(X)} \|\eta\|_{L^4(X)} \|d_A\eta\|_{L^2(X)}
\\
&\leq C_3 \|a\|_{L^4(X)} \|\eta\|_{W_A^{1,2}(X)} \|d_A\eta\|_{L^2(X)}
\\
&\leq C_1C_3\|a\|_{L^4(X)} \|d_A\eta\|_{L^2(X)}^2
\quad\text{(by \eqref{eq:Apriori_estimate_W12norm_eta_by_L2_norm_d_eta})},
\end{align*}
which again yields a contradiction to our hypothesis \eqref{eq:Perturbation_Laplace_operator_small}, just as in the case $d \geq 3$.
\end{proof}

\section{A priori estimates for Coulomb gauge transformations}
\label{sec:Apriori_estimates_Coulomb_gauge_transformations}
We now establish a generalization of \cite[Lemma 6.6]{FeehanSlice}, which is in turn an analogue of \cite[Lemma 2.3.10]{DK}. We allow $p \geq 2$ and any $d \geq 2$ rather than assume $d = 4$, as in \cite[Lemma 6.6]{FeehanSlice}, but we use standard Sobolev norms rather than the `critical exponent' Sobolev norms employed in the statement and proof of \cite[Lemma 6.6]{FeehanSlice}, since we do not seek an explicit optimal dependence of constants on the reference connection, $A_0$.

Before stating our generalization of \cite[Lemma 6.6]{FeehanSlice}, we digress to recall from \cite[p. 231]{FrM}, that a gauge transformation, $u \in \Aut(P)$, may be viewed as a section of the fiber bundle $\Ad P := P\times_\Ad G \to X$, where we denote $\Ad(g): G \ni h \to g^{-1}hg \in G$, for all $h \in G$. With the aid of a choice of a unitary representation, $\varrho: G \subset \Aut_\CC(\EE)$, we may therefore consider $\Ad P$ to be a subbundle of the Hermitian vector bundle $P\times_\varrho \End_\CC(\EE)$. We can alternatively replace $\varrho: G \to \Aut_\CC(\EE)$ by $\Ad:G \to \Aut(\fg)$ and $\End_\CC(\EE)$ by $\End(\fg)$. A choice of connection $A$ on $P$ induces covariant derivatives on all associated vector bundles, such as $E = P\times_\varrho \EE$ and $P\times_\varrho \End_\CC(\EE)$ or $\ad P = P\times_\ad \fg$ and $P\times_\Ad \End(\fg)$.  We can thus define Sobolev norms of sections of $\Ad P$, generalizing the construction of Freed and Uhlenbeck in \cite[Appendix A]{FU}

A similar construction is described by Parker \cite[Section 4]{ParkerGauge}, but we note that while the center of $\Aut(P)$ --- which is given by $P\times_\Ad \Center(G)$ --- acts trivially on $\sA(P)$, it does not act trivially on $C^\infty(X;E)$.

\begin{prop}[\Apriori $W^{1,p}$ estimate for $u(A)-A_0$ in terms of $A-A_0$]
\label{prop:Feehan_2001_lemma_6-6}
Let $(X,g)$ be a closed, smooth Riemannian manifold of dimension $d \geq 2$, and $G$ be a compact Lie group, and $P$ be a smooth principal $G$-bundle over $X$. Let $A_1$ be a $C^\infty$ connection on $P$, and $A_0$ be a $W^{1,q}$ connection on $P$ with $d/2 < q < \infty$, and $p \in (1,\infty)$ obey $d/2 \leq p \leq q$, and\footnote{In applications of Proposition \ref{prop:Feehan_2001_lemma_6-6}, we can choose $\delta=1$ without loss of generality.}
$\delta \in (0,1]$. Then there are constants $N \in [1,\infty)$ and $\eps=\eps(A_0,A_1,g,G,p,q) \in (0,1]$ (with dependence on $p$ replaced by dependence on $\delta$ when $p = d$) with the following significance. If $A$ is a $W^{1,q}$ connection on $P$ and $u \in \Aut(P)$ is a gauge transformation of class $W^{2,q}$ such that
\begin{equation}
\label{eq:Feehan_2001_6-6_uA_minus_A0_Coulomb_gauge}
d_{A_0}^*(u(A) - A_0) = 0,
\end{equation}
then the following hold. If
\begin{equation}
\label{eq:Feehan_2001_6-6_A_minus_A0_and_uA_minus_A0_Ld_or_Ld+delta_close}
\|A - A_0\|_{L^s(X)} \leq \eps
\quad\text{and}\quad
\|u(A) - A_0\|_{L^s(X)} \leq \eps,
\end{equation}
where
\begin{equation}
\label{eq:Lp_bound_DeltaA_u0_definition_s(p)}
s(p)
:=
\begin{cases}
d &\text{if } p < d,
\\
d+\delta &\text{if } p = d,
\\
p &\text{if } p > d,
\end{cases}
\end{equation}
then
\begin{equation}
\label{eq:Feehan_2001_6-6_uA_minus_A0_Lp_bound_by_A_minus_A0}
\|u(A) - A_0\|_{L^p(X)} \leq N\|A - A_0\|_{W_{A_1}^{1,p}(X)},
\end{equation}
where $N = N(A_0,A_1,g,G,p,q)$. If in addition $A$ obeys\footnote{The use of the constant $M$ could be avoided if we replaced the right-hand-side of Inequality \eqref{eq:Feehan_2001_6-6_uA_minus_A0_W1p_bound_by_A_minus_A0} by
$N(1+\|A - A_0\|_{W_{A_1}^{1,p}(X)})\|A - A_0\|_{W_{A_1}^{1,p}(X)}$.}
\begin{equation}
\label{eq:Feehan_2001_6-6_A_minus_A0_W1p_lessthan_constant}
\|A - A_0\|_{W_{A_1}^{1,p}(X)} \leq M,
\end{equation}
for some constant $M \in [1,\infty)$, then
\begin{equation}
\label{eq:Feehan_2001_6-6_uA_minus_A0_W1p_bound_by_A_minus_A0}
\|u(A) - A_0\|_{W_{A_1}^{1,p}(X)} \leq N\|A - A_0\|_{W_{A_1}^{1,p}(X)},
\end{equation}
where $N = N(A_0,A_1,g,G,M,p,q)$.
\end{prop}

\begin{proof}
Following the convention of \cite[p. 32]{UhlLp} for the action of $u\in\Aut(P)$ on connections $A$ on $P$ and setting $B := u(A)$ for convenience, we have
\begin{equation}
\label{eq:Feehan_2001_6-13}
B-A_0 = u^{-1}(A-A_0)u + u^{-1}d_{A_0}u.
\end{equation}
Our task is thus to estimate the term $d_{A_0}u$. Rewriting the preceding equality gives a first-order, linear elliptic equation in $u$ with $W^{1,q}$ coefficients,
\begin{equation}
\label{eq:Feehan_2001_6-9}
d_{A_0}u = u(B-A_0) - (A-A_0)u.
\end{equation}
Corollary \ref{cor:Fredholmness_and_index_Laplace_operator_on_W2p_Sobolev_connection} implies that the kernel,
\[
K := \Ker\left(\Delta_{A_0}: W_{A_1}^{2,q}(X; \ad P) \to L^q(X; \ad P) \right),
\]
is finite-dimensional. Let
\[
\Pi: L^2(X; \ad P) \to K \subset L^2(X; \ad P)
\]
denote the $L^2$-orthogonal projection and denote
\begin{align*}
\gamma &:= \Pi u \in K \subset W_{A_1}^{2,q}(X; \ad P),
\\
u_0 &:= u-\gamma \in K^\perp \cap W_{A_1}^{2,q}(X; \ad P),
\end{align*}
where $\perp$ is $L^2$-orthogonal complement. We may assume without loss of generality that we have a unitary representation, $G \subset \U(n)$. Recall that (due to \cite[Equation (6.2)]{Warner}),
\begin{equation}
\label{eq:Warner_6-2}
d_{A_0}^* = (-1)^{-d(k+1)+1}*d_{A_0}* \quad\text{on } \Omega^k(X;\ad P),
\end{equation}
where $*:\Omega^k(X) \to \Omega^{d-k}$ is the Hodge star operator on $k$-forms. Because $d_{A_0}^*(B-A_0)=0$ and $d_{A_0}u=d_{A_0}u_0$, an application of $d_{A_0}^*$ to \eqref{eq:Feehan_2001_6-9} yields
\begin{align*}
d_{A_0}^*d_{A_0}u_0 &= -*(d_{A_0}u\wedge *(B-A_0)) + ud_{A_0}^*(B-A_0)
\\
&\quad  - (d_{A_0}^*(A-A_0))u  + *(*(A-A_0)\wedge d_{A_0}u)
\\
&= -*(d_{A_0}u_0\wedge *(B-A_0)) - (d_{A_0}^*(A-A_0))u  + *(*(A-A_0)\wedge d_{A_0}u_0).
\end{align*}
We shall use the bound $\|u\|_{C(X)}\leq 1$ for any $u\in\Aut(P)$ of class $W^{2,q}$, implied by the fact that the representation for $G$ is unitary. Recall that $\Delta_{A_0} = d_{A_0}^*d_{A_0}$. We first consider the case $p<d$ and setting $p^* = dp/(d-p)$, we have $1/p = 1/d+1/p^*$ and the continuous multiplication $L^d(X) \times L^{p^*}(X) \to L^p(X)$, so
\begin{align*}
\|\Delta_{A_0}u_0\|_{L^p(X)}
&\leq
\|d_{A_0}u_0\|_{L^{p^*}(X)}\|B-A_0\|_{L^d(X)}
+ \|d_{A_0}^*(A-A_0)\|_{L^p(X)}\|u\|_{C(X)}
\\
&\quad + \|A-A_0\|_{L^d(X)}\|d_{A_0}u_0\|_{L^{p^*}(X)}
\\
&\leq \left(\|B-A_0\|_{L^d(X)} + \|A-A_0\|_{L^d(X)}\right)
\|d_{A_0}u_0\|_{L^{p^*}(X)}
\\
&\quad + \|d_{A_0}^*(A-A_0)\|_{L^p(X)}.
\end{align*}
Second, for the case $p>d$, we use the continuous multiplication $L^p(X) \times L^\infty(X) \to L^p(X)$, so
\begin{align*}
\|\Delta_{A_0}u_0\|_{L^p(X)}
&\leq \left(\|B-A_0\|_{L^p(X)} + \|A-A_0\|_{L^p(X)}\right)
\|d_{A_0}u_0\|_{L^\infty(X)}
\\
&\quad + \|d_{A_0}^*(A-A_0)\|_{L^p(X)}.
\end{align*}
Third, for the case $p=d$, we can instead use $r \in (d,\infty)$ defined by $1/d = 1/(d+\delta) + 1/r$ and the resulting continuous multiplication, $L^{d+\delta}(X)\times L^r(X) \to L^d(X)$, to give
\begin{align*}
\|\Delta_{A_0}u_0\|_{L^d(X)}
&\leq \left(\|B-A_0\|_{L^{d+\delta}(X)} + \|A-A_0\|_{L^{d+\delta}(X)}\right)
\|d_{A_0}u_0\|_{L^r(X)}
\\
&\quad + \|d_{A_0}^*(A-A_0)\|_{L^d(X)}.
\end{align*}
By \cite[Theorem 4.12]{AdamsFournier}, we have continuous Sobolev embeddings,
\begin{equation}
\label{eq:Adams_Fournier_theorem_4-12_W1p_embeddings}
W^{1,p}(X)
\subset
\begin{cases}
L^{dp/(d-p)}(X) &\text{if } 1 \leq p < d,
\\
L^r(X) &\text{if } p = d \text{ and } 1 \leq r < \infty,
\\
C(X) &\text{if } p > d.
\end{cases}
\end{equation}
Therefore, the Sobolev Embedding Theorem and Kato Inequality \cite[Equation (6.20)]{FU} give, for $r\in(d,\delta)$ determined by $\delta$ as above,
\begin{equation}
\label{eq:Integral_norm_bound_dA0u0_prelim}
\begin{aligned}
\|d_{A_0}u_0\|_{L^{p^*}(X)}
&\leq
C_0\|d_{A_0}u_0\|_{W_{A_1}^{1,p}(X)}, \quad p < d,
\\
\|d_{A_0}u_0\|_{L^r(X)}
&\leq
C_0\|d_{A_0}u_0\|_{W_{A_1}^{1,d}(X)}, \quad p = d,
\\
\|d_{A_0}u_0\|_{L^\infty(X)}
&\leq
C_0\|d_{A_0}u_0\|_{W_{A_1}^{1,p}(X)}, \quad p > d,
\end{aligned}
\end{equation}
where $C_0 = C_0(g,p)$ or $C_0(g,\delta) \in [1,\infty)$ is bounded below by the norm of the Sobolev embedding \eqref{eq:Adams_Fournier_theorem_4-12_W1p_embeddings}. Writing $A_0 = A_1+a_0$, for $a_0 \in W_{A_1}^{1,q}(X;\Lambda^1\otimes\ad P)$, so $d_{A_0}u_0 = d_{A_1}u_0 + [a_0,u_0]$, we see that
\begin{align*}
\|\nabla_{A_1}d_{A_0}u_0\|_{L^p(X)}
&\leq
\|\nabla_{A_1}d_{A_1}u_0\|_{L^p(X)} + \|\nabla_{A_1}[a_0,u_0]\|_{L^p(X)}
\\
&\leq
\|\nabla_{A_1}^2u_0\|_{L^p(X)} + \|\nabla_{A_1}a_0\times u_0
+ a_0\times \nabla_{A_1}u_0\|_{L^p(X)}.
\end{align*}
By hypothesis, we have $p = d/2 < q$ or $d/2 < p \leq q$, and thus, defining $r \in [p,\infty]$ by $1/p=1/q+1/r$ and using the continuous Sobolev embeddings $W^{1,p}\subset L^{2p}(X)$ and $W^{2,p}(X) \subset L^r(X)$ (with norm bounded above by $C_0$) for $p \geq d/2$,
\begin{align*}
{}&\|\nabla_{A_1}d_{A_0}u_0\|_{L^p(X)}
\\
&\quad \leq
\|\nabla_{A_1}^2u_0\|_{L^p(X)} + z\|\nabla_{A_1}a_0\|_{L^q(X)} \|u_0\|_{L^r(X)}
+ z\|a_0\|_{L^{2p}(X)} \|\nabla_{A_1}u_0\|_{L^{2p}(X)}
\\
&\quad \leq
\|\nabla_{A_1}^2u_0\|_{L^p(X)}
+ zC_0\|\nabla_{A_1}a_0\|_{L^q(X)} \|u_0\|_{W_{A_1}^{2,p}(X)}
  \\
  &\qquad + zC_0^2\|a_0\|_{W_{A_1}^{1,p}(X)} \|\nabla_{A_1}u_0\|_{W_{A_1}^{1,p}(X)}
\\
&\quad \leq
\|\nabla_{A_1}^2u_0\|_{L^p(X)} + zC_0\|a_0\|_{W_{A_1}^{1,q}(X)} \|u_0\|_{W_{A_1}^{2,p}(X)}
+ zC_0^2\|a_0\|_{W_{A_1}^{1,q}(X)} \|u_0\|_{W_{A_1}^{2,p}(X)},
\end{align*}
where $z = z(g) \in [1,\infty)$ and now $C_0 = C_0(A_1,g,p,q) \in [1,\infty)$. By substituting the preceding bound into \eqref{eq:Integral_norm_bound_dA0u0_prelim}, we find that
\begin{align*}
\|d_{A_0}u_0\|_{L^{p^*}(X)}
&\leq
C_1\|u_0\|_{W_{A_1}^{2,p}(X)}, \quad p < d,
\\
\|d_{A_0}u_0\|_{L^r(X)}
&\leq
C_1\|u_0\|_{W_{A_1}^{2,d}(X)}, \quad p = d,
\\
\|d_{A_0}u_0\|_{L^\infty(X)}
&\leq
C_1\|u_0\|_{W_{A_1}^{2,p}(X)}, \quad p > d,
\end{align*}
for a constant $C_1 = C_1(A_1,g,p,q,\|a_0\|_{W_{A_1}^{1,q}(X)}) \in [1,\infty)$, with dependence on $p$ replaced by $\delta$ when $p=d$. By combining the preceding three cases ($p<d$, and $p=d$, and $p>d$), we obtain
\begin{multline}
\label{eq:Lp_bound_DeltaA_u0}
\|\Delta_{A_0}u_0\|_{L^p(X)}
\leq
C_1\left(\|B-A_0\|_{L^s(X)} + \|A-A_0\|_{L^s(X)}\right)
\|u_0\|_{W_{A_1}^{2,p}(X)}
\\
+ \|d_{A_0}^*(A-A_0)\|_{L^p(X)},
\end{multline}
where $s = s(p)$ is as in \eqref{eq:Lp_bound_DeltaA_u0_definition_s(p)}. From the \apriori estimate \eqref{eq:W2p_apriori_estimate_Delta_A_orthogonal_kernel} in
Corollary \ref{cor:Spectrum_Delta_A_Sobolev_and_W2p_apriori_estimate}
--- and noting that this lemma also holds for $\Ad P$ in place of $\ad P$ via the definition of Sobolev norms of $u \in \Aut(P)$ described earlier --- we have the \apriori estimate,
\begin{equation}
\label{eq:Feehan_2001_6-11}
\|u_0\|_{W_{A_1}^{2,p}(X)} \leq C_3\|\Delta_{A_0}u_0\|_{L^p(X)},
\quad\text{for } 1<p<\infty \text{ and } d/2 \leq p \leq q,
\end{equation}
where $C_3 = C_3(A_1,A_0,g,G,p,q) \in [1,\infty)$. Substituting the \apriori estimate \eqref{eq:Feehan_2001_6-11} into our $L^p$ bound \eqref{eq:Lp_bound_DeltaA_u0} for $\Delta_{A_0}u_0$ gives, for $s$ as in \eqref{eq:Lp_bound_DeltaA_u0_definition_s(p)},
\begin{align*}
\|\Delta_{A_0}u_0\|_{L^p(X)}
&\leq
C_1C_3\left(\|B-A_0\|_{L^s(X)} + \|A-A_0\|_{L^s(X)}\right)
\|\Delta_{A_0}u_0\|_{L^p(X)}
\\
&\quad + \|d_{A_0}^*(A-A_0)\|_{L^p(X)},
\quad\text{for } 1<p<\infty \text{ and } d/2 \leq p \leq q.
\end{align*}
Provided
\begin{equation}
\label{eq:Ls_norm_B_minus_A0_lessthan_small_constant}
\|B-A_0\|_{L^s(X)} \leq 1/(4C_1C_3)
\quad\text{and}\quad
\|A-A_0\|_{L^s(X)} \leq 1/(4C_1C_3),
\end{equation}
as assured by \eqref{eq:Feehan_2001_6-6_A_minus_A0_and_uA_minus_A0_Ld_or_Ld+delta_close}, then rearrangement in the preceding inequality yields
\begin{multline}
\label{eq:Feehan_2001_6-10}
\|\Delta_{A_0}u_0\|_{L^p(X)}
\leq
2\|d_{A_0}^*(A-A_0)\|_{L^p(X)},
\\
\text{for } 1<p<\infty \text{ and } d/2 \leq p \leq q.
\end{multline}
Therefore, by combining the inequalities \eqref{eq:Feehan_2001_6-10} and \eqref{eq:Feehan_2001_6-11} we find that
\begin{multline}
\label{eq:Feehan_2001_6-12}
\|u_0\|_{W_{A_1}^{2,p}(X)}
\leq
2C_3\|d_{A_0}^*(A-A_0)\|_{L^p(X)},
\\
\text{for } 1<p<\infty \text{ and } d/2 \leq p \leq q.
\end{multline}
Using $d_{A_0}u=d_{A_0}u_0$ and \eqref{eq:Feehan_2001_6-13} and the facts that $|u| \leq 1$ and $|u^{-1}| \leq 1$ on $X$, we obtain
\begin{equation}
\label{eq:Feehan_2001_6-14}
\|B-A_0\|_{L^p(X)}
\leq
\|A-A_0\|_{L^p(X)} + \|d_{A_0}u_0\|_{L^p(X)}.
\end{equation}
From \eqref{eq:Feehan_2001_6-14} and \eqref{eq:Feehan_2001_6-12}, we see that
\begin{align*}
\|B-A_0\|_{L^p(X)}
&\leq
\|A-A_0\|_{L^p(X)} + \|u_0\|_{W_{A_1}^{1,p}(X)}
\\
&\leq \|A-A_0\|_{L^p(X)} + 2C_3\|d_{A_0}^*(A-A_0)\|_{L^p(X)}.
\end{align*}
Using $A_0=A_1+a_0$, we have
\[
  d_{A_0}^*(A-A_0) = d_{A_1}^*(A-A_0) + a_0\times(A-A_0)
\]
and
\[
\|d_{A_0}^*(A-A_0)\|_{L^p(X)}
\leq z\|A-A_0\|_{W_{A_1}^{1,p}(X)} + z\|a_0\|_{L^{2p}(X)} \|A-A_0\|_{L^{2p}(X)}.
\]
Applying the continuous Sobolev embedding, $W^{1,p}(X) \subset L^{2p}(X)$, with norm $C_0=C_0(g,p) \in [1,\infty)$ and the Kato Inequality \cite[Equation (6.20)]{FU},
\begin{multline}
\label{eq:Lp_bound_dA0star_A_minus_A0}
\|d_{A_0}^*(A-A_0)\|_{L^p(X)}
\leq
z\|A-A_0\|_{W_{A_1}^{1,p}(X)}
\\
+ zC_0^2\|a_0\|_{W_{A_1}^{1,p}(X)} \|A-A_0\|_{W_{A_1}^{1,p}(X)}.
\end{multline}
Thus,
\begin{equation}
\label{eq:Feehan_2001_6-18}
\|B-A_0\|_{L^p(X)}
\leq
C_4\|A-A_0\|_{W_{A_1}^{1,p}(X)},
\quad\text{for } 1<p<\infty \text{ and } d/2 \leq p \leq q,
\end{equation}
where $C_4 = C_4(A_0,A_1,g,G,p,q) \in [1,\infty)$, giving the desired $L^p$ estimate \eqref{eq:Feehan_2001_6-6_uA_minus_A0_Lp_bound_by_A_minus_A0} for $B-A_0$.

We now estimate the $L^p$ norms of the covariant derivatives of the right-hand side  of the identity \eqref{eq:Feehan_2001_6-13}.  Considering the term $u^{-1}d_{A_0}u$ in the right-hand side of the identity \eqref{eq:Feehan_2001_6-13} and recalling that $\nabla_{A_0}u = d_{A_0}u = d_{A_0}u_0$, we have
\[
\nabla_{A_0}(u^{-1}d_{A_0}u_0)
=
-u^{-1}(\nabla_{A_0}u_0)u^{-1}\otimes d_{A_0}u_0 + u^{-1}\nabla_{A_0}d_{A_0}u_0.
\]
First, if $d/2 \leq p < d$ and using the continuous multiplication, $L^{p^*}(X)\times L^d(X) \to L^p(X)$,
\begin{align*}
\|\nabla_{A_0}(u^{-1}d_{A_0}u_0)\|_{L^p(X)}
&\leq
\|\nabla_{A_0}u_0\|_{L^{p^*}(X)} \|\nabla_{A_0}u_0\|_{L^d(X)}
+ \|\nabla_{A_0}^2u_0\|_{L^p(X)}
\\
&\leq C_0^2\|\nabla_{A_0}u_0\|_{W_{A_0}^{1,p}(X)}^2
+ \|\nabla_{A_0}^2u_0\|_{L^p(X)}
\\
&\leq C_0^2\|u_0\|_{W_{A_0}^{2,p}(X)}^2 + \|u_0\|_{W_{A_0}^{2,p}(X)}.
\end{align*}
Second, if $p = d$ and using the continuous multiplication, $L^{2d}(X)\times L^{2d}(X) \to L^d(X)$,
\begin{align*}
\|\nabla_{A_0}(u^{-1}d_{A_0}u_0)\|_{L^d(X)}
&\leq
\|\nabla_{A_0}u_0\|_{L^{2d}(X)} \|\nabla_{A_0}u_0\|_{L^{2d}(X)}
+ \|\nabla_{A_0}^2u_0\|_{L^d(X)}
\\
&\leq C_0^2\|\nabla_{A_0}u_0\|_{W_{A_0}^{1,d}(X)}^2
+ \|\nabla_{A_0}^2u_0\|_{L^d(X)}
\\
&\leq C_0^2\|u_0\|_{W_{A_0}^{2,d}(X)}^2 + \|u_0\|_{W_{A_0}^{2,d}(X)}.
\end{align*}
Third, if $p > d$,
\begin{align*}
\|\nabla_{A_0}(u^{-1}d_{A_0}u_0)\|_{L^p(X)}
&\leq
\|\nabla_{A_0}u_0\|_{L^\infty(X)} \|\nabla_{A_0}u_0\|_{L^p(X)}
+ \|\nabla_{A_0}^2u_0\|_{L^p(X)}
\\
&\leq C_0^2\|\nabla_{A_0}u_0\|_{W_{A_0}^{1,p}(X)}^2
+ \|\nabla_{A_0}^2u_0\|_{L^p(X)}
\\
&\leq C_0^2\|u_0\|_{W_{A_0}^{2,p}(X)}^2 + \|u_0\|_{W_{A_0}^{2,p}(X)}.
\end{align*}
Thus, by combining the three preceding cases and applying Lemma \ref{lem:Equivalence_Sobolev_norms_for_Sobolev_and_smooth_connections}~\eqref{item:Equivalence_W2p_A1_embed_W2pA0},
\begin{multline}
\label{eq:Lp_bound_nablaA_u-inverse_dAu}
\|\nabla_{A_0}(u^{-1}d_{A_0}u_0)\|_{L^p(X)}
\leq
C_6^2C_0^2\|u_0\|_{W_{A_1}^{2,p}(X)}^2 + C_6\|u_0\|_{W_{A_1}^{2,p}(X)},
\\
\text{for } d/2 \leq p \leq q,
\end{multline}
where $C_6 = C_6(A_0,A_1,g,p,q) \in[1,\infty)$ is the constant in Lemma \ref{lem:Equivalence_Sobolev_norms_for_Sobolev_and_smooth_connections}~\eqref{item:Equivalence_W2p_A1_embed_W2pA0}. By combining the bound \eqref{eq:Lp_bound_nablaA_u-inverse_dAu} for $\|\nabla_{A_0}(u^{-1}d_{A_0}u_0)\|_{L^p(X)}$ with \eqref{eq:Feehan_2001_6-12}, we find that
\begin{align*}
\|\nabla_{A_0}(u^{-1}d_{A_0}u_0)\|_{L^p(X)}
&\leq
                                              2C_6^2C_0^2C_3\|d_{A_0}^*(A-A_0)\|_{L^p(X)}^2
  \\
  &\quad + 2C_6C_3\|d_{A_0}^*(A-A_0)\|_{L^p(X)},
\\
&\qquad\text{for } 1<p<\infty \text{ and } d/2 \leq p \leq q.
\end{align*}
But
\begin{equation}
\label{eq:Lp_norm_dA0star_A_minus_A0_lessthan_constant_W1p_norm_A_minus_A0}
\|d_{A_0}^*(A-A_0)\|_{L^p(X)} \leq z\|A-A_0\|_{W_{A_0}^{1,p}(X)},
\end{equation}
for a generic constant $z = z(g) \in [1,\infty)$, and by Lemma \ref{lem:Equivalence_Sobolev_norms_for_Sobolev_and_smooth_connections}~\eqref{item:Equivalence_W1p_A0_and_W1pA1},
\begin{equation}
\label{eq:WA0_1p_norm_A_minus_A0_lessthan_constant_WA1_1p_norm_A_minus_A0}
\|A-A_0\|_{W_{A_0}^{1,p}(X)} \leq C_7\|A-A_0\|_{W_{A_1}^{1,p}(X)},
\end{equation}
where $C_7 = C_7(A_0,A_1,g,p) \in[1,\infty)$ is the constant in Lemma \ref{lem:Equivalence_Sobolev_norms_for_Sobolev_and_smooth_connections}~\eqref{item:Equivalence_W1p_A0_and_W1pA1},
and because $A$ is now assumed to obey \eqref{eq:Feehan_2001_6-6_A_minus_A0_W1p_lessthan_constant}, that is,
\[
\|A-A_0\|_{W_{A_1}^{1,p}(X)} \leq M,
\]
we obtain
\begin{multline*}
\|\nabla_{A_0}(u^{-1}d_{A_0}u_0)\|_{L^p(X)}
\leq
2C_6C_3(zC_6C_7C_0^2M + 1)\|d_{A_0}^*(A-A_0)\|_{L^p(X)},
\\
\text{for } 1<p<\infty \text{ and } d/2 \leq p \leq q,
\end{multline*}
and thus by \eqref{eq:Lp_norm_dA0star_A_minus_A0_lessthan_constant_W1p_norm_A_minus_A0} and \eqref{eq:WA0_1p_norm_A_minus_A0_lessthan_constant_WA1_1p_norm_A_minus_A0},
\begin{multline}
\label{eq:Feehan_2001_6-17}
\|\nabla_{A_0}(u^{-1}d_{A_0}u_0)\|_{L^p(X)}
\\
\leq
2z^2C_6C_7C_3(zC_6C_7C_0^2M + 1)\|A-A_0\|_{W_{A_1}^{1,p}(X)},
\\
\text{for } 1<p<\infty \text{ and } d/2 \leq p \leq q.
\end{multline}
Considering the term $u^{-1}(A-A_0)u$ in the right-hand side of \eqref{eq:Feehan_2001_6-13}, we discover that
\begin{align*}
\nabla_{A_0}(u^{-1}(A-A_0)u)
&=
-u^{-1}(\nabla_{A_0}u)u^{-1}\otimes(A-A_0)u
+ u^{-1}(\nabla_{A_0}(A-A_0))u
\\
&\quad + u^{-1}(A-A_0)\otimes \nabla_{A_0}u.
\end{align*}
Noting that $\nabla_{A_0}u=\nabla_{A_0}u_0$ and $\|u\|_{C(X)}\leq 1$, the preceding identity gives, for $d/2 \leq p < d$,
\begin{align*}
  {}&\|\nabla_{A_0}(u^{-1}(A-A_0)u)\|_{L^p(X)}
      \\
&\quad\leq
2\|\nabla_{A_0}u_0\|_{L^{p^*}(X)}\|A-A_0\|_{L^d(X)}
+ \|\nabla_{A_0}(A-A_0)\|_{L^p(X)}
\\
&\quad\leq 2C_0^2\|\nabla_{A_0}u_0\|_{W_{A_0}^{1,p}(X)} \|A-A_0\|_{W_{A_0}^{1,p}(X)}
+ \|\nabla_{A_0}(A-A_0)\|_{L^p(X)}
\\
&\quad\leq 2C_0^2\|u_0\|_{W_{A_0}^{2,p}(X)} \|A-A_0\|_{W_{A_0}^{1,p}(X)}
+ \|\nabla_{A_0}(A-A_0)\|_{L^p(X)}.
\end{align*}
Second, for the case $p = d$,
\begin{align*}
  {}&\|\nabla_{A_0}(u^{-1}(A-A_0)u)\|_{L^d(X)}
      \\
&\quad\leq
2\|\nabla_{A_0}u_0\|_{L^{2d}(X)}\|A-A_0\|_{L^{2d}(X)}
+ \|\nabla_{A_0}(A-A_0)\|_{L^d(X)}
\\
&\quad\leq 2C_0^2\|\nabla_{A_0}u_0\|_{W_{A_0}^{1,d}(X)}\|A-A_0\|_{W_{A_0}^{1,d}(X)}
+ \|\nabla_{A_0}(A-A_0)\|_{L^d(X)}
\\
&\quad\leq 2C_0^2\|u_0\|_{W_{A_0}^{2,p}(X)}\|A-A_0\|_{W_{A_0}^{1,d}(X)}
+ \|\nabla_{A_0}(A-A_0)\|_{L^d(X)}.
\end{align*}
Third, for the case $d<p<\infty$,
\begin{align*}
  {}&\|\nabla_{A_0}(u^{-1}(A-A_0)u)\|_{L^p(X)}
      \\
&\quad\leq
2\|\nabla_{A_0}u_0\|_{L^\infty(X)}\|A-A_0\|_{L^p(X)}
+ \|\nabla_{A_0}(A-A_0)\|_{L^p(X)}
\\
&\quad\leq 2C_0\|\nabla_{A_0}u_0\|_{W_{A_0}^{1,p}(X)}\|A-A_0\|_{L^p(X)}
+ \|\nabla_{A_0}(A-A_0)\|_{L^p(X)}
\\
&\quad\leq 2C_0\|u_0\|_{W_{A_0}^{2,p}(X)}\|A-A_0\|_{L^p(X)}
+ \|\nabla_{A_0}(A-A_0)\|_{L^p(X)}.
\end{align*}
Hence, the combination of the preceding three cases gives
\begin{multline*}
\|\nabla_{A_0}(u^{-1}(A-A_0)u)\|_{L^p(X)}
\\
\leq
2C_0^2\|u_0\|_{W_{A_0}^{2,p}(X)}\|A-A_0\|_{W_{A_0}^{1,p}(X)}
+ \|\nabla_{A_0}(A-A_0)\|_{L^p(X)},
\\
\quad\text{for } p<\infty \text{ and } d/2 \leq p \leq q.
\end{multline*}
Applying Lemma \ref{lem:Equivalence_Sobolev_norms_for_Sobolev_and_smooth_connections}, Items \eqref{item:Equivalence_W1p_A0_and_W1pA1} and \eqref{item:Equivalence_W2p_A1_embed_W2pA0}, yields
\begin{multline}
\label{eq:Lp_bound_nablaA_u-inverse(A-A0)u}
\|\nabla_{A_0}(u^{-1}(A-A_0)u)\|_{L^p(X)}
\\
\leq
2C_6C_7C_0^2\|u_0\|_{W_{A_1}^{2,p}(X)}\|A-A_0\|_{W_{A_1}^{1,p}(X)}
\\
+ C_7\|A-A_0\|_{W_{A_1}^{1,p}(X)}, \quad\text{for } p<\infty \text{ and } d/2 \leq p \leq q.
\end{multline}
Therefore, combining the inequalities \eqref{eq:Feehan_2001_6-12} and \eqref{eq:Lp_bound_nablaA_u-inverse(A-A0)u} yields
\begin{multline}
\label{eq:Feehan_2001_6-16_prelim}
\|\nabla_{A_0}(u^{-1}(A-A_0)u)\|_{L^p(X)}
\\
\leq
4C_6C_7C_0^2C_3\|d_{A_0}^*(A-A_0)\|_{L^p(X)}\|A-A_0\|_{W_{A_1}^{1,p}(X)}
\\
+ C_7\|A-A_0\|_{W_{A_1}^{1,p}(X)},
\\
\text{for } 1<p<\infty \text{ and } d/2 \leq p \leq q.
\end{multline}
From \eqref{eq:Feehan_2001_6-6_A_minus_A0_W1p_lessthan_constant} and \eqref{eq:Lp_norm_dA0star_A_minus_A0_lessthan_constant_W1p_norm_A_minus_A0}, we see that \eqref{eq:Feehan_2001_6-16_prelim} simplifies to give
\begin{multline}
\label{eq:Feehan_2001_6-16}
\|\nabla_{A_0}(u^{-1}(A-A_0)u)\|_{L^p(X)}
\\
\leq
(4zC_6C_7C_0^2C_3M + 1)C_7\|A-A_0\|_{W_{A_1}^{1,p}(X)},
\\
\text{for } 1<p<\infty \text{ and } d/2 \leq p \leq q.
\end{multline}
From the identity \eqref{eq:Feehan_2001_6-13} (noting again that $d_{A_0}u = d_{A_0}u_0$) we have
\[
\|\nabla_{A_0}(B-A_0)\|_{L^p(X)}
\leq
\|\nabla_{A_0}(u^{-1}(A-A_0)u\|_{L^p(X)} + \|\nabla_{A_0}(u^{-1}d_{A_0}u_0)\|_{L^p(X)}.
\]
Combining the preceding estimate with the inequalities \eqref{eq:Feehan_2001_6-17} and  \eqref{eq:Feehan_2001_6-16} gives
\begin{multline}
\label{eq:nablaA0_B_minus_A0_Lp_bound}
\|\nabla_{A_0}(B-A_0)\|_{L^p(X)}
\leq
C_5\|A-A_0\|_{W_{A_1}^{1,p}(X)},
\\
\quad\text{for } 1<p<\infty \text{ and } d/2 \leq p \leq q,
\end{multline}
with a constant $C_5 = C_5(A_0,A_1,g,G,M,p,q) \in [1,\infty)$. Finally, from \eqref{eq:Feehan_2001_6-18} and \eqref{eq:nablaA0_B_minus_A0_Lp_bound} and Lemma \ref{lem:Equivalence_Sobolev_norms_for_Sobolev_and_smooth_connections}~\eqref{item:Equivalence_W1p_A0_and_W1pA1} we obtain the desired $W_{A_1}^{1,p}$ bound \eqref{eq:Feehan_2001_6-6_uA_minus_A0_W1p_bound_by_A_minus_A0} for $u(A)-A_0$ in terms of $A-A_0$, recalling that $B=u(A)$, with large enough constant $N=N(A_0,A_1,g,G,M,p,q) \in [1,\infty)$ and under the hypothesis \eqref{eq:Feehan_2001_6-6_A_minus_A0_and_uA_minus_A0_Ld_or_Ld+delta_close} with small enough constant $\eps=\eps(A_0,A_1,g,G,p,q) \in (0,1]$.
\end{proof}

The proof of Proposition \ref{prop:Feehan_2001_lemma_6-6} also yields the following useful

\begin{lem}[\Apriori $W^{2,p}$ estimate for a $W^{2,q}$ gauge transformation $u$ intertwining two $W^{1,q}$ connections]
\label{lem:connections_control_gauge}
Assume the hypotheses of Proposition \ref{prop:Feehan_2001_lemma_6-6}, excluding those on the connection $A$. Then there is a constant $C = C(A_0,A_1,g,G,p,q) \in [1,\infty)$ with the following significance. If $A$ obeys the hypotheses of Proposition \ref{prop:Feehan_2001_lemma_6-6} and $u \in \Aut^{2,q}(P)$ is the resulting gauge transformation, depending on $A$ and $A_0$, such that
\[
d_{A_0}^*(u(A)-A_0) = 0,
\]
and furthermore
\begin{equation}
\label{eq:Feehan_2001_6-6_A_minus_A0_W1p_close}
\|A - A_0\|_{W_{A_1}^{1,p}(X)} \leq \eps,
\end{equation}
then
\[
\|u\|_{W_{A_1}^{2,p}(X)} \leq C.
\]
\end{lem}

\begin{proof}
Write $u = u_0+\gamma$ as in the proof of Proposition \ref{prop:Feehan_2001_lemma_6-6}, with $u_0 \in (\Ker \Delta_{A_0})^\perp$ and $\gamma \in \Ker \Delta_{A_0}$, and observe that
\begin{align*}
\|u\|_{W_{A_1}^{2,p}(X)} &\leq C\left(\|\Delta_{A_0}u\|_{L^p(X)} + \|u\|_{L^p(X)}\right)
\quad \text{(by Proposition \ref{prop:W2p_apriori_estimate_Delta_A_Sobolev})}
\\
&=
C\left(\|\Delta_{A_0}u_0\|_{L^p(X)} + \|u\|_{L^p(X)}\right)
\\
&\leq
C\left(\|d_{A_0}^*(A-A_0)\|_{L^p(X)} + \|u\|_{L^p(X)}\right)
\quad \text{(by \eqref{eq:Feehan_2001_6-10})}
\\
&\leq
C\left(1 + \|A_0-A_1\|_{W^{1,p}_{A_1}(X)}\right)\|A-A_0\|_{W^{1,p}_{A_1}(X)} + C\|u\|_{L^p(X)}
\quad \text{(by \eqref{eq:Lp_bound_dA0star_A_minus_A0})}
\\
&\leq
C\left(1 + \|A_0-A_1\|_{W^{1,p}_{A_1}(X)}\right)\varepsilon + C\Vol_g(X)^{1/p},
\end{align*}
where the last inequality follows from \eqref{eq:Feehan_2001_6-6_A_minus_A0_W1p_close} and the fact that $|u|\leq 1$ pointwise. This completes the proof.
\end{proof}

\section[Coulomb gauge transformations for $W^{1,d/2}$-close connections]{Existence of Coulomb gauge transformations for connections that are $W^{1,d/2}$-close to a reference connection}
\label{sec:Coulomb_gauge_slice_quotient_space_connections}
Finally, we can proceed to the

\begin{proof}[Proof of Theorem \ref{mainthm:Feehan_proposition_3-4-4_Lp}]
We shall apply the method of continuity, modeled on the proofs of \cite[Theorem 2.1]{UhlLp} due to Uhlenbeck and \cite[Proposition 2.3.13]{DK} due to Donaldson and Kronheimer. For a related application of the method of continuity, see the proof \cite[Theorem 1.1]{FeehanSlice} due to the first author.

We begin by defining a one-parameter family of $W^{1,q}$ connections by setting
\begin{equation}
\label{eq:Definition_A_t}
A_t := A_0 + t(A-A_0), \quad\forall\, t \in [0, 1],
\end{equation}
and observe that their curvatures are given by
\[
F(A_t) = F_{A_0} + td_{A_0}(A-A_0) + \frac{t^2}{2}[A-A_0, A-A_0],
\]
and they obey the bounds,
\begin{align*}
\|F(A_t)\|_{L^q(X)}
&\leq
\|F_{A_0}\|_{L^q(X)} + \|d_{A_0}(A-A_0)\|_{L^q(X)} + \frac{c_0}{2}\|A-A_0\|_{L^{2q}(X)}^2
\\
&\leq
\|F_{A_0}\|_{L^q(X)} + \|d_{A_0}(A-A_0)\|_{L^q(X)} + c\|A-A_0\|_{W_{A_1}^{1,q}(X)}^2,
\end{align*}
with $c_0 \in [1,\infty)$ and $c = c(g,q) \in [1,\infty)$ and where we use the Sobolev embedding, $W^{1,q}(X) \subset L^{2q}(X)$ \cite[Theorem 4.12]{AdamsFournier} and the Kato Inequality \cite[Equation (6.20)]{FU} to obtain the last inequality. Note that we have a continuous embedding, $W^{1,q}(X) \subset L^{2q}(X)$, when
\begin{inparaenum}[\itshape a\upshape)]
\item $q<d$ and $2q \leq q^* := dq/(d-q)$, that is $2d - 2q \leq d$ or $q \geq d/2$, as implied by our hypotheses, or
\item $q \geq d$.
\end{inparaenum}
Therefore,
\begin{equation}
\label{eq:Curvature_uniform_Lq_bound_path}
\|F(A_t)\|_{L^q(X)} \leq K, \quad\forall\, t \in [0, 1],
\end{equation}
for $K = K(A,A_0,G,g,q) \in [1,\infty)$, noting that $F_{A_0} \in L^q(X;\Lambda^2\otimes\ad P)$ since $A_0$ is of class $W^{1,q}$.

Let $S$ denote the set of $t \in [0,1]$ such that there exists a $W^{2,q}$ gauge transformation $u_t \in \Aut(P)$ with the property that
\[
d_{A_0}^*(u_t(A_t) - A_0) = 0
\quad\text{and}\quad
\|u_t(A_t) - A_0\|_{W_{A_1}^{1,p}(X)} < 2N\|A_t - A_0\|_{W_{A_1}^{1,p}(X)},
\]
where $N$ is the constant in Proposition \ref{prop:Feehan_2001_lemma_6-6}. Clearly, $0 \in S$ since the identity automorphism of $P$ is the required gauge transformation in that case, so $S$ is non-empty. As usual, we need to show that $S$ is an open and closed subset of $[0,1]$.

\setcounter{step}{0}
\begin{step}[$S$ is open]
\label{eq:Openness}
To prove openness, we shall adapt the argument of Donaldson and Kronheimer in \cite[Section 2.3.8]{DK}. We apply the Implicit Function Theorem to the gauge fixing equation,
\[
d_{A_0}^*(u_t(A_t)-A_0)
=
d_{A_0}^*\left(u_t^{-1}(A_t-A_0)u_t + u_t^{-1}d_{A_0}u_t\right) = 0.
\]
As usual, we denote $B_t = u_t(A_t)$ for convenience, for $t\in S$. Let $t_0 \in S$, so we have
\[
d_{A_0}^*\left(u_{t_0}(A_{t_0}) - A_0\right) = 0
\quad\text{and}\quad
\|u_{t_0}(A_{t_0}) - A_0\|_{W_{A_1}^{1,p}(X)}
<
2N\|A_{t_0} - A_0\|_{W_{A_1}^{1,p}(X)}.
\]
Our task is to show that $t_0+s\in S$ for $|s|$ sufficiently small, that is, there exists $u_{t_0+s} \in \Aut^{2,q}(P)$ such that the preceding two properties hold with $t_0$ replaced by $t_0+s$. By the hypothesis \eqref{eq:Feehan_3-4-4_Lp_A_minus_A0_W1p_close} of Theorem \ref{mainthm:Feehan_proposition_3-4-4_Lp}, we have $\|A - A_0\|_{W_{A_1}^{1,p}(X)} < \zeta$ and thus, since
\[
  A_{t_0} - A_0 = A_0 + t_0(A-A_0) - A_0 = t_0(A-A_0)
\]
and $t_0\in [0,1]$, we see that
\[
\|A_{t_0} - A_0\|_{W_{A_1}^{1,p}(X)} < \zeta,
\]
and so
\[
\|u_{t_0}(A_{t_0}) - A_0\|_{W_{A_1}^{1,p}(X)} < 2N\zeta.
\]
It will be convenient to define $a \in W_{A_1}^{1,q}(X;\Lambda^1\otimes\ad P)$ by
\begin{equation}
\label{eq:Definition_a}
u_{t_0}(A_{t_0}) =: A_0 + a,
\end{equation}
The preceding inequality ensures that
\begin{equation}
\label{eq:S_open_W1p_norm_a_small}
\|a\|_{W_{A_1}^{1,p}(X)} < 2N\zeta.
\end{equation}
We shall seek a solution $u_{t_0+s} \in \Aut^{2,q}(P)$ to the gauge-fixing equation,
\[
d_{A_0}^*(u_{t_0+s}(A_{t_0+s}) - A_0) = 0.
\]
In particular, we shall seek a solution in the form
\[
u_{t_0+s} = e^{\chi_s}u_{t_0},
\quad\text{for } \chi_s \in W_{A_1}^{2,q}(X;\ad P),
\]
so the gauge-fixing equation becomes
\begin{equation}
\label{eq:Gauge_fixing}
d_{A_0}^*(e^{\chi_s}u_{t_0}(A_{t_0+s}) - A_0) = 0.
\end{equation}
For $s \in \RR$, it will be convenient to define $b_s \in W_{A_1}^{1,q}(X;\Lambda^1\otimes\ad P)$ by
\begin{equation}
\label{eq:Definition_b_s}
u_{t_0}(A_{t_0+s}) =: A_0 + a + b_s.
\end{equation}
We can determine $b_s$ explicitly using $A_{t_0+s} = A_0 + (t_0+s)(A-A_0)$, so that
\begin{align*}
u_{t_0}(A_{t_0+s}) - A_0
&=
u_{t_0}^{-1}((t_0+s)(A-A_0))u_{t_0} + u_{t_0}^{-1}d_{A_0}u_{t_0}
\\
&=
u_{t_0}^{-1}(t_0(A-A_0))u_{t_0} + u_{t_0}^{-1}d_{A_0}u_{t_0}
+ u_{t_0}^{-1}(s(A-A_0))u_{t_0}
\\
&=
u_{t_0}^{-1}(A_{t_0}-A)u_{t_0} + u_{t_0}^{-1}d_{A_0}u_{t_0}
+ u_{t_0}^{-1}(s(A-A_0))u_{t_0}
\\
&=
u_{t_0}(A_{t_0}) - A_0 + s u_{t_0}^{-1}(A-A_0)u_{t_0}
\\
&= a + s u_{t_0}^{-1}(A-A_0)u_{t_0} \quad\text{(by \eqref{eq:Definition_a})},
\end{align*}
and thus
\begin{equation}
\label{eq:Expression_b_s}
b_s = s u_{t_0}^{-1}(A-A_0)u_{t_0}, \quad s \in \RR.
\end{equation}
Note that $u_{t_0} \in \Aut^{2,q}(P)$ and so we have the estimate,
\begin{equation}
\label{eq:W1q_norm_bound_b_s}
\|b_s\|_{W_{A_1}^{1,q}(X)}
\leq
|s| C_0\|A-A_0\|_{W_{A_1}^{1,q}(X)}, \quad s \in \RR,
\end{equation}
for $C_0 = C_0(A_0,g,q,t_0) \in [1,\infty)$. In particular, $b_s \to 0$ in $W_{A_1}^{1,q}(X;\Lambda^1\otimes\ad P)$ strongly as $s \to 0$.

The gauge fixing equation \eqref{eq:Gauge_fixing} takes the form
\begin{align*}
e^{\chi_s}u_{t_0}(A_{t_0+s}) - A_0
&=
e^{\chi_s}(A_0 + a + b_s) - A_0
\\
&= e^{-\chi_s}(a + b_s)e^{\chi_s} + e^{-\chi_s}d_{A_0}(e^{\chi_s}).
\end{align*}
The equation to be solved is then $H(\chi_s, b_s) = 0$, where
\begin{equation}
\label{eq:Donaldson_Kronheimer_2-3-5a}
H(\chi, b) := d_{A_0}^*\left( e^{-\chi}(a + b)e^\chi + e^{-\chi}d_{A_0}(e^\chi) \right).
\end{equation}
For any $q > d/2$, the expression \eqref{eq:Donaldson_Kronheimer_2-3-5a} for $H$ defines a smooth map,
\begin{multline}
\label{eq:Donaldson_Kronheimer_2-3-5b}
H: (\Ker \Delta_{A_0})^\perp\cap W_{A_1}^{2,q}(X;\ad P)
\times W_{A_1}^{1,q}(X;\Lambda^1\otimes\ad P)
\\
\to
(\Ker \Delta_{A_0})^\perp\cap L^q(X;\ad P).
\end{multline}
Here, we note that if $\xi = d_{A_0}^*a_\xi$ for some $a_\xi \in W_{A_1}^{1,q}(X;\Lambda^1\otimes\ad P)$, then $\xi \perp \Ker \Delta_{A_0}$, as implied by the preceding expression for $H$. Indeed, for any $\gamma \in \Ker \Delta_{A_0}$,
\[
(\xi, \gamma)_{L^2(X)} = (d_{A_0}^*a_\xi, \gamma)_{L^2(X)}
= (a_\xi, d_{A_0}\gamma)_{L^2(X)} = 0.
\]
Hence, the image of $H$ is contained in $(\Ker \Delta_{A_0})^\perp\cap L^q(X;\ad P)$.

The Implicit Function Theorem asserts that if the partial derivative,
\[
(D_1 H)_{(0,0)}: (\Ker \Delta_{A_0})^\perp\cap W_{A_1}^{2,q}(X;\ad P)
\to (\Ker \Delta_{A_0})^\perp\cap L^q(X;\ad P),
\]
is surjective, then for small $b \in W_{A_1}^{1,q}(X;\Lambda^1\otimes\ad P)$ there is a small solution $\chi \in W_{A_1}^{2,q}(X;\ad P)$ to $H(\chi, b) = 0$ and that is $L^2$-orthogonal to $\Ker \Delta_{A_0}$.

Now the linearization, $(D_1 H)_{(0,0)}$, of the map $H$ at the origin $(0,0)$ with respect to variations in $\chi$ is given by
\[
(D_1 H)_{(0,0)}\chi = d_{A_0}^*d_{A_0+a}\chi.
\]
But $\|a\|_{W_{A_1}^{1,p}(X)} < 2N\zeta$ by \eqref{eq:S_open_W1p_norm_a_small} and because of the continuous Sobolev embeddings provided by \cite[Theorem 4.12]{AdamsFournier},
\begin{gather*}
W^{1,p}(X) \subset L^d(X), \quad\text{for } d \geq 3 \text{ and } p \geq d/2,
\\
W^{1,p}(X) \subset L^4(X), \quad\text{for } d = 2 \text{ and } p \geq 2,
\end{gather*}
we obtain,
\[
\|a\|_{L^d(X)} < 2C_1N\zeta \quad\text{when } d \geq 3
\quad\text{and}\quad
\|a\|_{L^4(X)} < 2C_1N\zeta \quad\text{when } d = 2,
\]
where $C_1 = C_1(g,p) \in [1, \infty)$ is the norm of the Sobolev embedding employed.
By the hypothesis of Theorem \ref{mainthm:Feehan_proposition_3-4-4_Lp}, we can choose $\zeta \in (0,1]$ as small as desired. Hence, the operator $d_{A_0}^*d_{A_0+a}$ is surjective by Lemma \ref{lem:Surjectivity_perturbed_Laplace_operator}.

To summarize, we have shown that if $|s|$ is small and $t_0 \in S$, then there exists $u_{t_0+s} \in \Aut^{2,q}(P)$ such that
\[
d_{A_0}^*(u_{t_0+s}(A_{t_0+s}) - A_0) = 0.
\]
It remains to check that the following norm condition holds,
\begin{equation}
\label{eq:Open_W1p_norm_condition_ut0+s(At0+s)_minus_A0}
\|u_{t_0+s}(A_{t_0+s}) - A_0\|_{W_{A_1}^{1,p}(X)}
<
2N\|A_{t_0+s} - A_0\|_{W_{A_1}^{1,p}(X)},
\end{equation}
for small enough $|s|$, to conclude that $t_0+s \in S$. To see this, we first note that since $A_{t_0+s} = A_0 + (t_0+s)(A-A_0)$, we have
\[
\|A_{t_0+s} - A_0\|_{W_{A_1}^{1,p}(X)}
=
(t_0+s)\|A - A_0\|_{W_{A_1}^{1,p}(X)}
<
(t_0+s)\zeta \quad\text{(by \eqref{eq:Feehan_3-4-4_Lp_A_minus_A0_W1p_close})},
\]
and thus, for $t_0+s \leq 1$ and $\zeta \leq \eps/C_1$,
\[
\|A_{t_0+s} - A_0\|_{W_{A_1}^{1,p}(X)} \leq \eps/C_1,
\]
and thus,
\begin{align*}
\|A_{t_0+s} - A_0\|_{L^d(X)} &\leq \eps,
\quad\text{if } p < d,
\\
\|A_{t_0+s} - A_0\|_{L^{d+\delta}(X)} &\leq \eps,
\quad\text{if } p = d,
\\
\|A_{t_0+s} - A_0\|_{L^p(X)} &\leq \eps,
\quad\text{if } p > d,
\end{align*}
where $\eps$ is the constant in Proposition \ref{prop:Feehan_2001_lemma_6-6} and $C_1 = C_1(g,p)$ or $C_1(\delta,p) \in [1,\infty)$ is the norm (provided by \cite[Theorem 4.12]{AdamsFournier}) of the continuous Sobolev embedding, $W^{1,p}(X) \subset L^d(X)$ when $d/2\leq p < d$ and $W^{1,p}(X) \subset L^{d+\delta}(X)$ when $p = d$. This verifies the hypotheses \eqref{eq:Feehan_2001_6-6_A_minus_A0_and_uA_minus_A0_Ld_or_Ld+delta_close} and \eqref{eq:Feehan_2001_6-6_A_minus_A0_W1p_lessthan_constant} of Proposition \ref{prop:Feehan_2001_lemma_6-6} for $A_{t_0+s} - A_0$ (in place of $A-A_0$ in the statement of that proposition).

On the other hand,
\begin{align*}
{}& \|u_{t_0+s}(A_{t_0+s}) - A_0\|_{W_{A_1}^{1,p}(X)}
\\
&\quad =
\|e^{\chi_s}u_{t_0}(A_{t_0+s}) - A_0\|_{W_{A_1}^{1,p}(X)}
\\
&\quad \leq \|e^{\chi_s}u_{t_0}(A_{t_0+s})
- u_{t_0}(A_{t_0+s})\|_{W_{A_1}^{1,p}(X)}
+ \|u_{t_0}(A_{t_0+s}) - A_0\|_{W_{A_1}^{1,p}(X)}
\\
&\quad = \|e^{\chi_s}u_{t_0}(A_{t_0+s})
- u_{t_0}(A_{t_0+s})\|_{W_{A_1}^{1,p}(X)}
+ \|a + b_s\|_{W_{A_1}^{1,p}(X)}
\quad\text{(by \eqref{eq:Definition_b_s})}.
\end{align*}
The inequalities \eqref{eq:Feehan_3-4-4_Lp_A_minus_A0_W1p_close}, \eqref{eq:S_open_W1p_norm_a_small}, and \eqref{eq:W1q_norm_bound_b_s} (which also holds with $p$ in place of $q$) yield the bound
\begin{align*}
\|a + b_s\|_{W_{A_1}^{1,p}(X)}
&\leq
\|a\|_{W_{A_1}^{1,p}(X)} + \|b_s\|_{W_{A_1}^{1,p}(X)}
\\
&< 2N\zeta + |s| C_0\|A-A_0\|_{W_{A_1}^{1,p}(X)}
\\
&\leq 2N\zeta + |s| C_0\zeta
\\
&\leq \eps/(2C_1),
\end{align*}
for small enough $\zeta$. (We could also have used $\|b_s\|_{W_{A_1}^{1,q}(X)} \leq C_0\|A-A_0\|_{W_{A_1}^{1,q}(X)}$ and the continuous embedding $W^{1,q}(X) \subset W^{1,p}(X)$ and rely on our freedom to also choose $|s|$ small.) But if $|s|$ is small then so is $\|b_s\|_{W_{A_1}^{1,q}(X)}$ by \eqref{eq:W1q_norm_bound_b_s} and hence $\|\chi_s\|_{W_{A_1}^{2,q}(X)}$ is small by the Implicit Function Theorem\footnote{The Quantitative Implicit Function Theorem \ref{thm:Quantitative_implicit_function_theorem} can be used to give a precise bound on $\chi_s$ given a bound on $b_s$.}
and so we may assume that
\[
\|e^{\chi_s}u_{t_0}(A_{t_0+s}) - u_{t_0}(A_{t_0+s})\|_{W_{A_1}^{1,p}(X)}
\leq
\eps/(2C_1),
\]
for small enough $|s|$. Collecting the preceding inequalities gives
\[
\|u_{t_0+s}(A_{t_0+s}) - A_0\|_{W_{A_1}^{1,p}(X)} \leq \eps/C_1,
\]
and thus,
\begin{align*}
\|u_{t_0+s}(A_{t_0+s}) - A_0\|_{L^d(X)} &\leq \eps,
\quad\text{if } p < d,
\\
\|u_{t_0+s}(A_{t_0+s}) - A_0\|_{L^{d+\delta}(X)} &\leq \eps,
\quad\text{if } p = d,
\\
\|u_{t_0+s}(A_{t_0+s}) - A_0\|_{L^p(X)} &\leq \eps,
\quad\text{if } p > d,
\end{align*}
verifying the hypotheses  \eqref{eq:Feehan_2001_6-6_A_minus_A0_and_uA_minus_A0_Ld_or_Ld+delta_close} and \eqref{eq:Feehan_2001_6-6_A_minus_A0_W1p_lessthan_constant}
of Proposition \ref{prop:Feehan_2001_lemma_6-6} for $u_{t_0+s}(A_{t_0+s}) - A_0$ (in place of $u(A)-A_0$ in the statement of that proposition).

Hence, Proposition \ref{prop:Feehan_2001_lemma_6-6} yields the bound
\begin{align*}
\|u_{t_0+s}(A_{t_0+s}) - A_0\|_{W_{A_1}^{1,p}(X)}
&\leq
N\|A_{t_0+s} - A_0\|_{W_{A_1}^{1,p}(X)}
  \\
  &<
2N\|A_{t_0+s} - A_0\|_{W_{A_1}^{1,p}(X)}.
\end{align*}
This verifies the norm condition \eqref{eq:Open_W1p_norm_condition_ut0+s(At0+s)_minus_A0} and we conclude that $t_0+s \in S$ and so $S$ is open.
\end{step}

\begin{step}[$S$ is closed]
\label{eq:Closedness}
For closedness, we adapt the argument in \cite[Section 2.3.7]{DK}. Set $B_t := u_t(A_t)$ for $t \in S$ and observe that the inequality \eqref{eq:Curvature_uniform_Lq_bound_path} yields
\begin{multline*}
\|F(B_t)\|_{L^q(X)} = \|F(u_t(A_t))\|_{L^q(X)} = \|u_t(F(A_t))\|_{L^q(X)}
\\
= \|F(A_t)\|_{L^q(X)} \leq K, \quad\forall\, t \in S.
\end{multline*}
Let $\{t_m\}_{m\in\NN} \subset S$ be a sequence and suppose that $t_m \to t_\infty \in [0,1]$ as $m\to\infty$. Since $q > d/2$ by hypothesis, the Uhlenbeck Weak Compactness \cite[Theorem 1.5 = 3.6]{UhlLp} (see also \cite[Theorem 7.1]{Wehrheim_2004} for a recent exposition) implies that there exists a subsequence $\{m'\} \subset \{m\}$ and, after relabeling, a sequence of $W^{2,q}$ gauge transformations, $\{u_{t_m}\}_{m\in\NN} \subset \Aut(P)$ and a $W^{1,q}$ connection $B_\infty$ on $P$ such that, as $m \to \infty$, we have
\[
B_{t_m} \rightharpoonup B_\infty
\quad\text{weakly in } W_{A_1}^{1,q}(X;\Lambda^1\otimes\ad P).
\]
Hence, for
\begin{inparaenum}[\itshape a\upshape)]
\item $q<d$ and $r < q* := dq/(d-q)$, or
\item $q \geq d$ and $r < \infty$,
\end{inparaenum}
the Rellich--Kondrachov Theorem \cite[Theorem 6.3]{AdamsFournier} implies that there is a compact embedding of Sobolev spaces, $W^{1,q}(X) \Subset L^r(X)$, and hence there exists a subsequence $\{m''\} \subset \{m\}$ such that, after again relabeling, as $m \to \infty$ we have
\[
B_{t_m} \to B_\infty \quad\text{strongly in } L^r(X;\Lambda^1\otimes\ad P).
\]
The $W^{2,q}$ gauge transformations $u_t$ intertwine the $W^{1,q}$ connections $A_t$ and $B_t$ via the relation \eqref{eq:Feehan_2001_6-13},
\[
B_t = A_0 + t u_t^{-1}(A-A_0)u_t + u_t^{-1}d_{A_0}u_t,
\]
and thus
\[
d_{A_0}u_t = u_tB_t - u_tA_0 - t(A-A_0)u_t, \quad\forall\, t \in [0,1].
\]
Because $A_{t_m} \to A_\infty := A_0 + t_\infty(A-A_0)$ strongly in $W^{1,q}$ (see \eqref{eq:Definition_A_t}) and $B_{t_m} \rightharpoonup B_\infty$ weakly in $W^{1,q}$ and $B_{t_m} \to B_\infty$ strongly in $L^r$ as $m\to\infty$, there exists a $W^{2,q}$ gauge transformation $u_\infty \in \Aut(P)$ such that, as $m\to\infty$,
\begin{multline*}
u_{t_m} \rightharpoonup u_\infty \quad\text{weakly in } W_{A_1}^{2,q}(X;\Ad P)
\quad\text{and}
\\
u_{t_m} \to u_\infty \quad\text{strongly in } W_{A_1}^{1,r}(X;\Ad P).
\end{multline*}
In particular,
\[
B_\infty = u_\infty(A_\infty)
\quad\text{and}\quad
d_{A_0}^*(u_\infty(A_\infty) - A_0) = 0.
\]
The Coulomb gauge condition follows from the fact that, for any $\xi \in C^\infty(X;\ad P)$, we have
\begin{align*}
0 &= \lim_{m\to\infty} (d_{A_0}^*(u_{t_m}(A_{t_m}) - A_0), \xi)_{L^2(X)}
\\
&= \lim_{m\to\infty} (u_{t_m}(A_{t_m}) - A_0, d_{A_0}\xi)_{L^2(X)}
\\
&= (u_\infty(A_\infty) - A_0, d_{A_0}\xi)_{L^2(X)}
\\
&= (d_{A_0}^*(u_\infty(A_\infty) - A_0), \xi)_{L^2(X)}.
\end{align*}
Similarly, for any $a \in \Omega^1(X;\ad P)$,
\begin{align*}
\lim_{m\to\infty}(B_{t_m} - A_0, a)_{L^2(X)}
&=
\lim_{m\to\infty}(u_{t_m}(A_{t_m}) - A_0, a)_{L^2(X)}
\\
&= \lim_{m\to\infty}(u_{t_m}^{-1}(A_{t_m} - A_0)u_{t_m}
+ u_{t_m}^{-1}d_{A_0}u_{t_m}, a)_{L^2(X)}
\\
&= \lim_{m\to\infty}((A_{t_m} - A_0)u_{t_m} + d_{A_0}u_{t_m}, u_{t_m}a)_{L^2(X)}
\\
&= ((A_\infty - A_0)u_\infty + d_{A_0}u_\infty, u_\infty a)_{L^2(X)}
\\
&= (u_\infty^{-1}(A_\infty - A_0)u_\infty + u_\infty^{-1}d_{A_0}u_\infty, a)_{L^2(X)}
\\
&= (u_\infty(A_\infty) - A_0, a)_{L^2(X)}.
\end{align*}
We now wish to apply Proposition \ref{prop:Feehan_2001_lemma_6-6} to bound $\|u_{t_\infty}(A_{t_\infty}) - A_0\|_{W_{A_1}^{1,p}(X)}$ and establish the remaining norm condition required to show that $t_\infty \in S$. First, we note that
\begin{align*}
\|A_{t_m} - A_0\|_{W_{A_1}^{1,p}(X)}
&=
\|A_0 + t_m(A-A_0) - A_0\|_{W_{A_1}^{1,p}(X)}
\\
&= t_m\|A-A_0\|_{W_{A_1}^{1,p}(X)}
\\
&< t_m\zeta \quad\text{(by hypothesis \eqref{eq:Feehan_3-4-4_Lp_A_minus_A0_W1p_close})}
\\
&\leq \zeta, \quad\forall\, m \in \NN.
\end{align*}
Because $t_m \in S$, we have
\[
\|u_{t_m}(A_{t_m}) - A_0\|_{W_{A_1}^{1,p}(X)}
<
2N\|A_{t_m} - A_0\|_{W_{A_1}^{1,p}(X)}, \quad\forall\, m \in \NN,
\]
and combining this inequality with the preceding inequality yields
\[
\|u_{t_m}(A_{t_m}) - A_0\|_{W_{A_1}^{1,p}(X)}
<
2N\zeta, \quad\forall\, m \in \NN.
\]
Since $u_\infty(A_\infty)$ is the weak limit of $u_{t_m}(A_{t_m})$ in $W_{A_1}^{1,q}(X;\Lambda^1\otimes\ad P)$, we have (see \cite[Appendix D.4]{Evans2})
\[
\|u_\infty(A_\infty) - A_0\|_{W_{A_1}^{1,p}(X)}
\leq
\liminf_{m\to\infty}2N\|A_{t_m} - A_0\|_{W_{A_1}^{1,p}(X)}.
\]
But $A_{t_m} \to A_{t_\infty}$ strongly in $W_{A_1}^{1,q}(X;\Lambda^1\otimes\ad P)$ by construction \eqref{eq:Definition_A_t} of the path $A_t$ and as $p \leq q$ by hypothesis and  $\|A_{t_m} - A_0\|_{W_{A_1}^{1,p}(X)} < \zeta$ for all $m\in \NN$, then
\[
\|A_{t_\infty} - A_0\|_{W_{A_1}^{1,p}(X)}
=
\lim_{m\to\infty}\|A_{t_m} - A_0\|_{W_{A_1}^{1,p}(X)}
\leq
\zeta.
\]
Combining the preceding inequalities yields
\[
\|u_\infty(A_\infty) - A_0\|_{W_{A_1}^{1,p}(X)}
\leq
2N\zeta.
\]
We choose $\zeta \in (0,1]$ small enough that $\zeta \leq \eps/C_1$ and $2N\zeta \leq \eps/C_1$, where $C_1$ is the norm of the Sobolev embedding $W^{1,p}(X) \subset L^d(X)$ when $p\neq d$ or $L^{d+\delta}(X)$ when $p = d$ as in Step \ref{eq:Openness}, and then observe that
\[
\|A_{t_\infty} - A_0\|_{W_{A_1}^{1,p}(X)} \leq \eps/C_1
\quad\text{and}\quad
\|u_\infty(A_\infty) - A_0\|_{W_{A_1}^{1,p}(X)} < \eps/C_1.
\]
The first inequality above verifies the hypothesis \eqref{eq:Feehan_2001_6-6_A_minus_A0_W1p_lessthan_constant} of Proposition \ref{prop:Feehan_2001_lemma_6-6}. Moreover,
\begin{gather*}
\|A_{t_\infty} - A_0\|_{L^d(X)} \leq \eps
\quad\text{and}\quad
\|u_\infty(A_{t_\infty}) - A_0\|_{L^d(X)} \leq \eps, \quad\text{if } p < d,
\\
\|A_{t_\infty} - A_0\|_{L^{d+\delta}(X)} \leq \eps
\quad\text{and}\quad
\|u_\infty(A_{t_\infty}) - A_0\|_{L^{d+\delta}(X)} \leq \eps, \quad\text{if } p = d,
\\
\|A_{t_\infty} - A_0\|_{L^p(X)} \leq \eps
\quad\text{and}\quad
\|u_\infty(A_{t_\infty}) - A_0\|_{L^p(X)} \leq \eps, \quad\text{if } p > d,
\end{gather*}
which verifies the hypothesis \eqref{eq:Feehan_2001_6-6_A_minus_A0_and_uA_minus_A0_Ld_or_Ld+delta_close} of Proposition \ref{prop:Feehan_2001_lemma_6-6} on norms (for $A_{t_\infty} - A_0$ in place of $A - A_0$ and $u_\infty(A_{t_\infty}) - A_0$ in place of $u(A) - A_0$ in the statement of that proposition). Since $d_{A_0}^*(u_\infty(A_\infty) - A_0) = 0$, as required by \eqref{eq:Feehan_2001_6-6_uA_minus_A0_Coulomb_gauge}, Proposition \ref{prop:Feehan_2001_lemma_6-6} implies that
\[
\|u_\infty(A_\infty) - A_0\|_{W_{A_1}^{1,p}(X)}
\leq
N\|A_{t_\infty} - A_0\|_{W_{A_1}^{1,p}(X)}
<
2N\|A_{t_\infty} - A_0\|_{W_{A_1}^{1,p}(X)}.
\]
Thus, $t_\infty \in S$ and so $S$ is closed.
\end{step}

Consequently, $S \subset [0,1]$ is non-empty and open and closed by the preceding two steps, so $S = [0,1]$ and this completes the proof of Theorem \ref{mainthm:Feehan_proposition_3-4-4_Lp}.
\end{proof}

\section[Coulomb gauge transformations for $L^r$-close connections]{Existence of Coulomb gauge transformations for connections that are $L^r$-close to a reference connection}
\label{sec:Coulomb_gauge_slice_quotient_space_connections_Lr_close}
Rather than modify the proof of Theorem \ref{mainthm:Feehan_proposition_3-4-4_Lp}, we shall prove Theorem \ref{mainthm:Feehan_proposition_3-4-4_Lp_Lr_close} directly for the case $r>d$ using the Quantitaive Implicit Function Theorem \ref{thm:Quantitative_implicit_function_theorem} for smooth maps on Banach spaces. While the resulting proof is much easier than that of Theorem \ref{mainthm:Feehan_proposition_3-4-4_Lp} and avoids the Method of Continuity, it does not obviously yield the full borderline case $r=d$ involving $W^{1,d}$ gauge transformations. See Remark \ref{rmk:Feehan_proposition_3-4-4_Lp_Ld_close} for related explanations. The issues here are well-known, namely, that $W^{1,d}$ gauge transformations need not be continuous and the map $H$ in \eqref{eq:Donaldson_Kronheimer_2-3-5a} need not be smooth. While there versions of the Implicit Function Theorem for Lipschitz functions on Euclidean space, it is unclear whether they hold on Banach spaces: see Clarke \cite{Clarke_1976} for a well-known example, Papi \cite{Papi_2005} for a more recent example, and references cited therein and references citing them for examples of Implicit Function Theorems for non-smooth functions. 

\begin{proof}[Proof of Theorem \ref{mainthm:Feehan_proposition_3-4-4_Lp_Lr_close}]
For $a := A-A_0 \in L^r(X;\Lambda^1\otimes\ad P)$ and $H$ as in \eqref{eq:Donaldson_Kronheimer_2-3-5a}, 
\[
H(\chi, a) := d_{A_0}^*\left( e^{-\chi}ae^\chi + e^{-\chi}d_{A_0}(e^\chi) \right),
\]
we see that for any $r > d$, we have $W^{1,r}(X) \subset C(X)$ by \cite[Theorem 4.12]{AdamsFournier} and so the expression \eqref{eq:Donaldson_Kronheimer_2-3-5a} for $H$ defines a smooth map,
\begin{multline}
\label{eq:Donaldson_Kronheimer_2-3-5b_Lr}
H: (\Ker \Delta_{A_0})^\perp\cap W_{A_1}^{1,r}(X;\ad P)
\times L^r(X;\Lambda^1\otimes\ad P)
\\
\to
(\Ker \Delta_{A_0})^\perp\cap W_{A_1}^{-1,r}(X;\ad P).
\end{multline}
The proof that the image of $H$ in \eqref{eq:Donaldson_Kronheimer_2-3-5b_Lr} is contained in $(\Ker \Delta_{A_0})^\perp\cap W_{A_1}^{-1,r}(X;\ad P)$ and not just $W_{A_1}^{-1,r}(X;\ad P)$ is the same as the proof that the image of $H$ in \eqref{eq:Donaldson_Kronheimer_2-3-5b} is contained in $(\Ker \Delta_{A_0})^\perp\cap L^q(X;\ad P)$ and not just $L^q(X;\ad P)$.

Given $a$, we will solve the equation $H(\chi,a) = 0$ for $\chi$ using Theorem \ref{thm:Quantitative_implicit_function_theorem}, so that $d_{A_0}^*(u(A)-A_0)=0$ (as asserted by Theorem \ref{mainthm:Feehan_proposition_3-4-4_Lp_Lr_close}) with $u=e^\chi$, using
\begin{align*}
  \sX &= (\Ker \Delta_{A_0})^\perp\cap W_{A_1}^{1,r}(X;\ad P),
  \\
  \sY &= L^r(X;\Lambda^1\otimes\ad P),
  \\
  \sZ &= (\Ker \Delta_{A_0})^\perp\cap W_{A_1}^{-1,r}(X;\ad P),
\end{align*}
and $(\chi_0,a_0) = (0,0)$ and interchanging the roles of the first and second variables.

From Remark \ref{rmk:Partial_derivatives_H}, we compute the partial derivative of $H(\chi,a)$ with respect to $\chi$ at the origin to be $(D_1 H)_{(0,0)} = d_{A_0}^*d_{A_0}$ and so we must first verify that the operator
\begin{equation}
  \label{eq:D_1H_origin_domain_range}
(D_1 H)_{(0,0)}: (\Ker \Delta_{A_0})^\perp\cap W_{A_1}^{1,r}(X;\ad P)
\to (\Ker \Delta_{A_0})^\perp\cap W_{A_1}^{-1,r}(X;\ad P)
\end{equation}
is an isomorphism of Banach spaces. We temporarily assume that $A_0$ is $C^\infty$ and observe that
\[
  d_{A_0}^*d_{A_0}:W_{A_1}^{1,r}(X;\ad P) \to W_{A_1}^{-1,r}(X;\ad P)
\]
is Fredholm with index zero by standard elliptic theory with the usual kernel and range for any $r\in (1,\infty)$: see the proof of Proposition \ref{prop:Fredholmness_and_index_Laplace_operator_on_W2p_smooth_connection} via Theorem \ref{thm:Gilkey_1-4-5_Sobolev}. (While we might allow $A_0$ to be a $W^{1,q}$ Sobolev connection in this step with $q \in (d/2,\infty)$ and $r \in (d,\infty)$, the proof of the analogue of Corollary \ref{cor:Fredholmness_and_index_Laplace_operator_on_W2p_Sobolev_connection}, the proof of such a refinement would be quite technical.) In particular, Theorem \ref{thm:Gilkey_1-4-5_Sobolev} implies that \eqref{eq:D_1H_origin_domain_range} is an isomorphism.

In order to verify the remaining hypotheses \eqref{eq:D_1_and_D_2f_conditions_implicit_function_theorem} of Theorem \ref{thm:Quantitative_implicit_function_theorem} and prove the estimate in the conclusion of Theorem \ref{mainthm:Feehan_proposition_3-4-4_Lp_Lr_close} via \eqref{eq:Lipschitz_constant_implied_function_g}, we shall need to identify the dependencies of the constants $M$ and $\beta$. Since
\[
  \Exp:W_{A_1}^{1,r}(X;\ad P) \to W_{A_1}^{1,r}(X;\ad P)
\]
is the smooth map defined by the smooth exponential map, $\exp:\fg\to G$, for the Lie group $G$ with Lie algebra $\fg = T_1G$, and the domain of $\exp$ is all of $\fg$ by \cite[Proposition 9.2.5]{Hilgert_Neeb_structure_geometry_lie_groups}, then the domain of $\Exp$ is all of $W_{A_1}^{1,r}(X;\ad P)$. The constant $M = M(A_0,A_1,g,G,r)$ in \eqref{eq:D_2f_x0y_0_inverse} is given by
\[
  M = \|D_1H(0,0)\|_{\sL(K^\perp\cap W^{1,r}, K^\perp\cap W^{-1,r})} = \|(d_{A_0}^*d_{A_0})^{-1}\|_{\sL(K^\perp\cap W^{1,r}, K^\perp\cap W^{-1,r})},
\]
where we abbreviate
\begin{align*}
  K &= W_{A_1}^{1,r}(X;\ad P)\cap \Ker d_{A_0}^*d_{A_0},
  \\
  W^{1,r} &= W_{A_1}^{1,r}(X;\ad P), \quad\text{and}
  \\
  W^{-1,r} &= W_{A_1}^{-1,r}(X;\ad P). 
\end{align*}
To satisfy the hypothesis \eqref{eq:Uniform_Lipschitz_bound_D_2fxy}, we choose $\zeta = \zeta(A_0,A_1,g,G,r) \in (0,1]$ small enough that
\begin{multline*}
  \sup_{(\chi,a)\in U \times V}\|D_1H(\chi,a) - D_1H(0,0)\|_{\sL^2(K^\perp\cap W^{1,r}\times L^r,K^\perp\cap W^{-1,r})}
  \\
  \leq \sup_{(\chi,a)\in U \times V}\|D_1H(\chi,a) - D_1H(0,0)\|_{\sL^2(W^{1,r}\times L^r,W^{-1,r})} \leq \frac{1}{2M},
\end{multline*}
where we abbreviate
\begin{align*}
  L^r &= L^r(X;\Lambda^1\otimes \ad P),
  \\
  U &= \left\{\chi\in W_{A_1}^{1,r}(X;\ad P): \|\chi\|_{W_{A_1}^{1,r}(X)} < \zeta \right\},
  \\
  V &= \left\{a\in L^r(X;\Lambda^1\otimes \ad P): \|a\|_{L^r(X)} < \zeta \right\}.   
\end{align*}
Finally, the constant $\beta = \beta(A_0,A_1,g,G,r)$ in \eqref{eq:D1f_uniform_bound} is given by
\[
\beta = \sup_{(\chi,a)\in U \times V}\|D_2H(\chi,a)\|_{\sL(L^r, K^\perp\cap W^{-1,r})}.
\]
Theorem \ref{thm:Quantitative_implicit_function_theorem} now provides constants $\delta = \delta(A_0,A_1,g,G,r) \in (0,1/(2\beta M)]$ and $C = C(A_0,A_1,g,G,r) \in (0, 1]$ and a smooth map,
\begin{multline*}
  \bchi: \left\{b \in L^r(X;\Lambda^1\otimes\ad P): \|b\|_{L^r(X)} < \delta\right\} \ni a
  \\
  \mapsto \bchi(a) \in (\Ker \Delta_{A_0})^\perp\cap \left\{\chi \in W_{A_1}^{1,r}(X;\ad P): \|\chi\|_{W_{A_1}^{1,r}(X)} < \zeta \right\},
\end{multline*}
such that
\[
H(\bchi(a), a) = 0, \quad\forall\, a \in L^r(X;\Lambda^1\otimes\ad P) \text{ with } \|a\|_{L^r(X)} < \delta,
\]
that is, the $W^{1,r}$ gauge transformation $u = e^\chi$ defined by the $W^{1,r}$ section $\chi = \bchi(a)$ obeys
\begin{gather*}
  d_{A_0}^*(u(A)-A_0) = 0,
  \\
  \|u(A) - A_0\|_{L^r(X)} < C\|A - A_0\|_{L^r(X)},
\end{gather*}
as required by Theorem \ref{mainthm:Feehan_proposition_3-4-4_Lp_Lr_close}. We can approximate any $W^{1,q}$ connection $A_0$ by a $C^\infty$ connection $\tilde A_0$ (see \cite[Theorem 3.17]{AdamsFournier}) and so the preceding conclusions follow by approximation for $W^{1,q}$ connections $A_0$ and continuity of the constants $M$ and $\beta$ with respect to variations of $A_0$ in the $W^{1,q}$ norm. (The proof of the latter continuity would be very similar to the proof of Corollary \ref{cor:Fredholmness_and_index_Laplace_operator_on_W2p_Sobolev_connection}, but simpler since we only require continuity and not compactness of Sobolev embeddings and thus is omitted.)

Finally, because $A$ is a $W^{1,q}$ connection, then $a = A-A_0$ is in $W^{1,q}$ and to show that $\chi$ and $u$ are in $W^{2,q}$, we can apply the proof of Lemma \ref{lem:Regularity_distributional_solution_Hodge_Laplacian_Sobolev} to the nonlinear, second-order elliptic equation,
\[
d_{A_0}^*\left( u^{-1}au + u^{-1}d_{A_0}u \right) = 0,
\]
to obtain the desired regularity for $u$. (Details of the proof are omitted since they will be very similar to those in the proof of Lemma \ref{lem:Regularity_distributional_solution_Hodge_Laplacian_Sobolev}.) This completes the proof of Theorem \ref{mainthm:Feehan_proposition_3-4-4_Lp_Lr_close}.
\end{proof}

\begin{rmk}[Computation of partial derivatives of $H(\chi,a)$ in \eqref{eq:Donaldson_Kronheimer_2-3-5a}]
  \label{rmk:Partial_derivatives_H}
While unnecessary for the proof of Theorem \ref{mainthm:Feehan_proposition_3-4-4_Lp_Lr_close}, one can explicitly compute the partial derivatives, $D_iH(\chi,a)$, for $i=1,2$. Observe that the expression \eqref{eq:Donaldson_Kronheimer_2-3-5a} yields
\begin{subequations}
 \label{eq:DH}  
\begin{align}
  \label{eq:DHDchi}
  D_1H(\chi,a)\xi &= d_{A_0}^*\left( - ((De^{-\chi})\xi) ae^\chi + e^{-\chi}a (De^{\chi})\xi \right.
                    \\
                    &\qquad \left. - ((De^{-\chi})\xi) d_{A_0}(e^\chi) + e^{-\chi}d_{A_0}(De^{\chi})\xi)\right),\nonumber
  \\
  \label{eq:DHDa}
  D_2H(\chi,a)b &:= d_{A_0}^*\left( e^{-\chi}be^\chi \right),
\end{align}
\end{subequations}
where $(De^\chi)\xi := (D\Exp)(\chi)\xi \in W_{A_1}^{1,r}(X;\ad P)$. Recall that $(D\exp)(0) = \id_\fg$ by \cite[Proposition 9.2.5]{Hilgert_Neeb_structure_geometry_lie_groups} and so $(D\Exp)(0)  = \id_P$ when $\chi=0$ and we see that \eqref{eq:DHDchi} becomes
\[
  D_1H(0,a)\xi = d_{A_0}^*(-\xi a + a\xi + d_{A_0}\xi) = d_{A_0}^*(d_{A_0}\xi + [a,\xi]) = d_{A_0}^*d_{A_0+a}\xi,
\]
using $d_{A_0}\id_P = 0$, and thus $D_1H(0,0) = d_{A_0}^*d_{A_0}$ while \eqref{eq:DHDa} becomes $D_2H(0,a)b = d_{A_0}^*b$, for any $a \in L^r(X;\Lambda^1\otimes\ad P)$.
\end{rmk}

\section{Real analytic Banach manifold structure on the quotient space of connections}
\label{sec:Real_analytic_Banach_manifold_structure_on_quotient_space_connections}
The statements and proofs of Lemmata \ref{lem:Continuous_operators_on_Lp_spaces_and_L2-orthogonal_decompositions} and \ref{lem:Closed_range_operators_on_Lp_spaces} would follow standard lines (see Gilkey \cite[Theorem 1.5.2]{Gilkey2}, for example) if the operators
\[
d_A:\Omega^l(X;\ad P) \to \Omega^{l+1}(X;\ad P), \quad l \geq 0,
\]
had $C^\infty$ coefficients, rather than Sobolev coefficients as we allow here, and formed an elliptic complex, rather than only satisfying $d_A\circ d_A = F_A$.

\begin{lem}[Continuous operators on $L^p$ spaces and $L^2$-orthogonal decompositions]
\label{lem:Continuous_operators_on_Lp_spaces_and_L2-orthogonal_decompositions}
Let $(X,g)$ be a closed, smooth Riemannian manifold of dimension $d \geq 2$, and $G$ be a compact Lie group, and $P$ be a smooth principal $G$-bundle over $X$. If $A$ is a connection on $P$ of class $W^{1,q}$ with $q \geq d/2$, and $A_1$ is a $C^\infty$ is reference connection on $P$, and $l\geq 1$ is an integer, and $p$ obeys $d/2\leq p \leq q$, then the operator
\[
d_A^*:W_{A_1}^{1,p}(X;\Lambda^l\otimes\ad P) \to L^p(X;\Lambda^{l-1}\otimes\ad P),
\]
is continuous and, if in addition $q > d/2$, then the operator
\[
d_A:W_{A_1}^{2,p}(X;\Lambda^{l-1}\otimes\ad P) \to W_{A_1}^{1,p}(X;\Lambda^l\otimes\ad P),
\]
is also continuous, and there is an $L^2$-orthogonal decomposition,
\begin{align*}
W_{A_1}^{1,p}(X;\Lambda^l\otimes\ad P)
&=
\Ker\left(d_A^*: W_{A_1}^{1,p}(X;\Lambda^l\otimes\ad P)\to L^p(X;\Lambda^{l-1}\otimes\ad P)\right)
\\
&\qquad \oplus \Ran\left(d_A: W_{A_1}^{2,p}(X;\Lambda^{l-1}\otimes\ad P)\to W_{A_1}^{1,p}(X;\Lambda^l\otimes\ad P)\right).
\end{align*}
\end{lem}

\begin{proof}
If $\xi \in W_{A_1}^{2,p}(X;\Lambda^{l-1}\otimes\ad P)$ and we write $A = A_1+a$, with $a \in  W_{A_1}^{1,q}(X;\Lambda^1\otimes\ad P)$, then $d_A\xi = d_{A_1}\xi + [a,\xi]$ and using the fact that $W^{1,p}(X) \subset L^{2p}(X)$ for any $p\geq d/2$ by \cite[Theorem 4.12]{AdamsFournier} and applying the Kato Inequality \cite[Equation (6.20]{FU},
\begin{align*}
\|d_A\xi\|_{L^p(X)} &\leq z\left(\|\nabla_{A_1}\xi\|_{L^p(X)} + \|a\|_{L^{2p}(X)} \|\xi\|_{L^{2p}(X)} \right)
\\
&\leq z\left(\|\nabla_{A_1}\xi\|_{L^p(X)} + \|a\|_{W_{A_1}^{1,p}(X)} \|\xi\|_{W_{A_1}^{1,p}(X)} \right)
\\
&\leq z\left(1 + \|a\|_{W_{A_1}^{1,q}(X)} \right) \|\xi\|_{W_{A_1}^{1,p}(X)},
\end{align*}
where $z = z(g,G,p,q) \in [1,\infty)$ and we use the fact that $q\geq p$. Similarly,
\[
\|d_A^*\xi\|_{L^p(X)} \leq z\left(1 + \|a\|_{W_{A_1}^{1,q}(X)} \right) \|\xi\|_{W_{A_1}^{1,p}(X)}
\]
and so the operator $d_A^*:W_{A_1}^{1,p}(X;\Lambda^l\otimes\ad P)\to L^p(X;\Lambda^{l-1}\otimes\ad P)$ is continuous.

Moreover, defining $r \in [p,\infty]$ by $1/p = 1/q + 1/r$, we recall that by \cite[Theorem 4.12]{AdamsFournier} we have
\begin{inparaenum}[\itshape i\upshape)]
\item $W^{2,p}(X) \subset L^r(X)$ for any $r \in [1,\infty)$ when $p = d/2$, and
\item $W^{2,p}(X) \subset L^\infty(X)$ when $p > d/2$.
\end{inparaenum}
Thus, using
\[
\nabla_{A_1}d_A\xi = \nabla_{A_1}d_{A_1}\xi + \nabla_{A_1}a\times\xi + a\times\nabla_{A_1}\xi,
\]
we see that
\begin{align*}
  {}&\|\nabla_{A_1}d_A\xi\|_{L^p(X)}
  \\
    &\leq z\left(\|\nabla_{A_1}^2\xi\|_{L^p(X)} + \|\nabla_{A_1}a\|_{L^q(X)} \|\xi\|_{L^r(X)} + \|a\|_{L^{2p}(X)} \|\nabla_{A_1}\xi\|_{L^{2p}(X)} \right)
\\
&\leq z\left(\|\nabla_{A_1}^2\xi\|_{L^p(X)} + \|\nabla_{A_1}a\|_{L^q(X)} \|\xi\|_{W_{A_1}^{2,p}(X)} + \|a\|_{W_{A_1}^{1,p}(X)} \|\nabla_{A_1}\xi\|_{W_{A_1}^{1,p}(X)} \right)
\\
&\leq z\left(1 + \|a\|_{W_{A_1}^{1,q}(X)} \right) \|\xi\|_{W_{A_1}^{2,p}(X)}.
\end{align*}
We conclude that the operator
\[
  d_A:W_{A_1}^{2,p}(X;\Lambda^{l-1}\otimes\ad P)\to W_{A_1}^{1,p}(X;\Lambda^l\otimes\ad P)
\]
is also continuous.

Note that $W^{1,p}(X) \subset L^2(X)$ when $p \geq 2$ or, when $1 \leq p < 2$, if $p^* := dp/(d-p) \geq 2$, that is, $dp \geq 2d-2p$ or $p \geq 2d/(d+2)$. But $p \geq d/2$ by hypothesis and $d/2 \geq 2d/(d+2)$ for all $d \geq 2$, so we have $W^{1,p}(X) \subset L^2(X)$ for all $p \geq d/2$ and $d \geq 2$. Using $\perp$ to denote $L^2$-orthogonal complement and $\oplus$ to denote $L^2$-orthogonal decomposition, we have
\begin{align*}
{}&W_{A_1}^{1,p}(X;\Lambda^l\otimes\ad P)
    \\
&\quad = \left(\Ran\left(d_A: W_{A_1}^{2,p}(X;\Lambda^{l-1}\otimes\ad P)\to W_{A_1}^{1,p}(X;\Lambda^l\otimes\ad P)\right)\right)^\perp
\\
&\qquad\qquad \oplus \Ran\left(d_A: W_{A_1}^{2,p}(X;\Lambda^{l-1}\otimes\ad P)\to W_{A_1}^{1,p}(X;\Lambda^l\otimes\ad P)\right).
\end{align*}
For all $\eta \in W_{A_1}^{1,p}(X;\Lambda^l\otimes\ad P)$ and $\xi \in W_{A_1}^{2,p}(X;\Lambda^{l-1}\otimes\ad P)$ we have
\[
(\eta,d_A\xi)_{L^2(X)} = (d_A^*\eta,\xi)_{L^2(X)}
\]
and so
\[
  \eta \perp \Ran(d_A: W_{A_1}^{2,p}(X;\Lambda^{l-1}\otimes\ad P)\to W_{A_1}^{1,p}(X;\Lambda^l\otimes\ad P))
\]
if and only if
\[
  \eta \in \Ker(d_A^*: W_{A_1}^{1,p}(X;\Lambda^l\otimes\ad P)\to L^p(X;\Lambda^{l-1}\otimes\ad P)).
\]
This concludes the proof of the lemma.
\end{proof}

Although not required by the proofs of Lemma \ref{lem:Continuous_operators_on_Lp_spaces_and_L2-orthogonal_decompositions} or Corollary \ref{maincor:Slice}, it is useful to note that the operator $d_A$ in that statement has closed range.

\begin{lem}[Closed range operators on $L^p$ spaces]
\label{lem:Closed_range_operators_on_Lp_spaces}
Let $(X,g)$ be a closed, smooth Riemannian manifold of dimension $d \geq 2$, and $G$ be a compact Lie group, and $P$ be a smooth principal $G$-bundle over $X$. If $A$ is a connection on $P$ of class $W^{1,q}$ with $d/2 < q < \infty$, and $A_1$ is a $C^\infty$ reference connection on $P$, and $p$ obeys $d/2\leq p \leq q$, then the operator
\[
d_A:W_{A_1}^{2,p}(X;\ad P) \to W_{A_1}^{1,p}(X;\Lambda^1\otimes\ad P),
\]
has closed range.
\end{lem}

\begin{proof}
Note that the operator in the statement of the lemma is bounded by Lemma \ref{lem:Continuous_operators_on_Lp_spaces_and_L2-orthogonal_decompositions}. Let $\{\chi_n\}_{n\in\NN} \subset W_{A_1}^{2,p}(X;\ad P)$ and suppose that $d_A\chi_n \to \xi \in  W_{A_1}^{1,p}(X;\Lambda^1\otimes\ad P)$ as $n \to \infty$. Thus $d_A^*d_A\chi_n = \Delta_A\chi_n \to d_A^*\xi \in  L^p(X;\ad P)$ as $n \to \infty$. We may assume without loss of generality that $\{\chi_n\}_{n\in\NN} \subset (\Ker \Delta_A)^\perp$, where $\perp$ denotes $L^2$-orthogonal complement, and so the \apriori estimate \eqref{eq:W2p_apriori_estimate_Delta_A_orthogonal_kernel} in Corollary \ref{cor:Spectrum_Delta_A_Sobolev_and_W2p_apriori_estimate} then implies that
\[
\|\chi_n - \chi_m\|_{W_{A_1}^{2,p}(X)} \leq C\|\Delta_A(\chi_n - \chi_m)\|_{L^p(X)}, \quad \forall\, n, m \in \NN.
\]
Hence, the sequence $\{\chi_n\}_{n\in\NN}$ is Cauchy in $W_{A_1}^{2,p}(X;\ad P)$ and thus $\chi_n \to \chi \in W_{A_1}^{2,p}(X;\ad P)$ and $d_A\chi_n \to d_A\chi \in  W_{A_1}^{1,p}(X;\Lambda^1\otimes\ad P)$ as $n \to \infty$. Therefore, $d_A$ on $W_{A_1}^{2,p}(X;\ad P)$ has closed range.
\end{proof}

We are now ready to complete the

\begin{proof}[Proof of Corollary \ref{maincor:Slice}]
Every compact Lie group has a compatible structure of a real analytic manifold \cite[Section III.4, Exercise 1]{BrockertomDieck} and this structure is unique by \cite[Theorem 2.11.3]{Varadarajan}. In particular, the exponential map is a real analytic diffeomorphism from an open neighborhood of the origin in the Lie algebra $\fg$ onto an open neighborhood of the identity in $G$. We recall from \cite[Proposition A.2]{FU} that $\Aut^{k+1,2}(P)$ may be given the structure of a Hilbert Lie group when $k \geq 2$ and, because $W^{2,q}(X)$ (with $q>d/2$) and $H^{k+1}(X) = W^{k+1,2}(X)$ are Banach algebras and contained in $C(X)$ (the Banach algebra of continuous functions on $X$), the same arguments show that $\Aut^{2,q}(P)$ may be given the structure of a $C^\infty$ Banach Lie group and that both $\Aut^{2,q}(P)$ and $\Aut^{k,2}(P)$ may be given the structure of real analytic manifolds.

According to \cite[Proposition A.3]{FU}, the (right) action of $\Aut^{k+1,2}(P)$ on $\sA^{k,2}(P)$ is $C^\infty$ when $k \geq 2$ and the same proof applies \mutatis to show that this action is real analytic and that the action \eqref{eq:Right_action_group_gauge_transformations_on_affine_space_connections} of $\Aut^{2,q}(P)$ on $\sA^{1,q}(P)$ is not only $C^\infty$ but also real analytic.

The only additional ingredient one needs to show that $\sB^*(P)$ is real analytic is the observation that the map $H$ defined in \eqref{eq:Donaldson_Kronheimer_2-3-5a} and \eqref{eq:Donaldson_Kronheimer_2-3-5b} is real analytic and thus, rather than apply the customary $C^\infty$ Inverse Function Theorem one can instead apply its real analytic counterpart \cite[Section 2.1.1]{Feehan_Maridakis_Lojasiewicz-Simon_harmonic_maps_v6} to show that for each $A_0 \in \sA^{1,q}(P)$, the map defined in the statement of \cite[Theorem 3.2]{DK},
\begin{multline}
\label{eq:Freed_Uhlenbeck_3-2map}
\Xi_{A_0}: \sA^{1,q}(P) \supset \sO_{A_0} \ni A \mapsto (u(A) - A_0, u)
\\
\in \Ker\left\{d_{A_0}^*:W_{A_1}^{1,q}(X;\Lambda^1\otimes\ad P)\to L^q(X;\ad P)\right\} \times \Aut^{2,q}(P),
\end{multline}
is a real analytic diffeomorphism onto an open neighborhood of $(0,\id_P)$, for a small enough open neighborhood $\sO_{A_0}$ of a $W^{1,q}$ connection $A_0$ on $P$ and the gauge transformation $u$ is produced by Theorem \ref{mainthm:Feehan_proposition_3-4-4_Lp}, so $u(A)$ is in Coulomb gauge with respect to $A_0$. The open neighborhood $\sO_{A_0}$ may be chosen to be $\Aut^{2,q}(P)$-invariant and the map $\Xi_{A_0}$ is $\Aut^{2,q}(P)$-equivariant. The proof that the quotient $\sB(P)$ is a Hausdorff topological space follows \mutatis either by adapting the proof of \cite[Corollary, p. 50]{FU} or by adapting the proof of \cite[Lemma 4.2.4]{DK}, using the observation the $L^2$ distance function,
\begin{equation}
\label{eq:Donaldson_Kronheimer_4-2-3}
\dist_{L^2}([A], [B]) := \inf_{u \in \Aut^{2,q}(P)} \|u(A) - B\|_{L^2(X)},
\end{equation}
is a metric on $\sB(P)$ and, in particular, that the quotient topology is metrizable. This completes the proof of Corollary \ref{maincor:Slice}.
\end{proof}

\section{Existence of Coulomb gauge transformations for pairs}
\label{sec:Coulomb_gauge_slice_quotient_space_pairs}
We now adapt the construction of Section \ref{sec:Coulomb_gauge_slice_quotient_space_connections} to the case of pairs. In \cite[p. 280]{FL1}, we employed a left action of $\Aut(P)$ on the affine space of pairs, $\sA(P) \times C^\infty(X;E)$, so $\Aut(P)$ acts on $\sA(P)$ by pushforward (consistent with Donaldson and Kronheimer \cite{DK}) and on $C^\infty(X;E)$ in the usual way, which is a left action. Here, to be consistent with Section \ref{sec:Coulomb_gauge_slice_quotient_space_connections} we shall use the opposite convention and continue to let $\Aut(P)$ act on $\sA(P)$ by pullback (consistent with Freed and Uhlenbeck \cite{FU} and Uhlenbeck \cite{UhlLp}) and use inversion to define a right action on $C^\infty(X;E)$, so that
\begin{multline}
  \label{eq:Action_gauge_transformation_on_pairs}
u(A,\Phi) := (u^*A, u^{-1}\Phi),
\\
\forall\, A \in \sA(P),\ \Phi \in C^\infty(X;E), \text{ and } u \in \Aut(P),
\end{multline}
giving a smooth (affine) right action,
\[
\sA(P) \times C^\infty(X;E) \times \Aut(P) \to \sA(P) \times C^\infty(X;E).
\]
Passing to Banach space completions, but temporarily suppressing the $W^{1,q}$ reference connection $A_0$ (for $q > d/2$) from our notation, the differential of the smooth map,
\[
\Aut^{2,q}(P) \ni u \mapsto u(A,\Phi) \in \sA^{1,q}(P) \times W^{1,q}(X;E),
\]
at $\id_P \in \Aut^{2,q}(P)$ is given by
\begin{multline}
\label{eq:Differential_gauge_transformation_action_on_pair_at_identity}
W_{A_1}^{2,q}(X;\ad P) \ni \xi
\\
\mapsto d_{A,\Phi}\xi := (d_A\xi, -\xi\Phi)
\in
W_{A_1}^{1,q}(X;\ad P) \oplus W_{A_1}^{1,q}(X;E),
\end{multline}
using $u = e^{\xi}$ for $u$ near $\id_P$; compare \cite[Proposition 2.1]{FL1}. We say that a $W^{1,q}$ pair $(A,\Phi)$ is in \emph{Coulomb gauge relative to $(A_0,\Phi_0)$} if
\begin{equation}
\label{eq:Pair_Coulomb_gauge_relative_to_reference_pair}
d_{A_0,\Phi_0}^*((A,\Phi) - (A_0,\Phi_0)) = 0,
\end{equation}
As in the case of Theorem \ref{mainthm:Feehan_proposition_3-4-4_Lp}, the proof of Theorem \ref{mainthm:Feehan_proposition_3-4-4_Lp_pairs} is facilitated by preparatory lemmata and a proposition, which we now state. For convenience, we define
\[
\Delta_{A_0,\Phi_0}
:=
d_{A_0,\Phi_0}^*d_{A_0,\Phi_0} \quad\text{on } W_{A_1}^{2,q}(X;\ad P).
\]
The proofs of Propositions \ref{prop:W2p_apriori_estimate_Delta_A_Sobolev} and \ref{prop:Gilbarg_Trudinger_theorem_8-6} and Corollary \ref{cor:Spectrum_Delta_A_Sobolev_and_W2p_apriori_estimate} adapt \mutatis to establish the following analogues for pairs, specialized to the case $l=0$.

\begin{prop}[\Apriori $L^p$ estimate for a Laplace operator with Sobolev coefficients]
\label{prop:W2p_apriori_estimate_Delta_APhi_Sobolev}
Let $(X,g)$ be a closed, smooth Riemannian manifold of dimension $d \geq 2$, and $G$ be a compact Lie group, $P$ be a smooth principal $G$-bundle over $X$, and $E = P\times_\varrho\EE$ be a smooth Hermitian vector bundle over $X$ defined by a finite-dimensional unitary representation, $\varrho: G \to \Aut_\CC(\EE)$. If $(A,\Phi)$ is a $W^{1,q}$ pair on $(P,E)$ with $q > d/2$, and $A_1$ is a $C^\infty$ connection on $P$, and $p$ obeys $d/2 \leq p \leq q$, then
\begin{equation}
\label{eq:DeltaAPhi_W2p_to_W1p}
\Delta_{A,\Phi}:W_{A_1}^{2,p}(X;\ad P) \to L^p(X;\ad P)
\end{equation}
is a bounded operator. If in addition $p \in (1,\infty)$, then there is a constant $C = C(A,\Phi,A_1,g,G,p,q) \in [1,\infty)$ such that
\begin{equation}
\label{eq:W2p_apriori_estimate_Delta_APhi_Sobolev}
\|\xi\|_{W_{A_1}^{2,p}(X)} \leq C\left(\|\Delta_{A,\Phi}\xi\|_{L^p(X)}
+ \|\xi\|_{L^p(X)}\right),
\quad
\forall\, \xi \in W_{A_1}^{2,p}(X;\ad P).
\end{equation}
\end{prop}

\begin{prop}[Spectral properties of a Laplace operator with Sobolev coefficients]
\label{prop:Gilbarg_Trudinger_theorem_8-6_pairs}
Let $(X,g)$ be a closed, smooth Riemannian manifold of dimension $d \geq 2$, and $G$ be a compact Lie group, $P$ be a smooth principal $G$-bundle over $X$, and $E = P\times_\varrho\EE$ be a smooth Hermitian vector bundle over $X$ defined by a finite-dimensional unitary representation, $\varrho: G \to \Aut_\CC(\EE)$. If $(A,\Phi)$ is a $W^{1,q}$ pair on $(P,E)$ with $d/2 < q < \infty$, and $A_1$ is a $C^\infty$ reference connection on $P$, and $p \in (1,\infty)$ obeys $d/2 \leq p \leq q$, then the spectrum, $\sigma(\Delta_{A,\Phi})$, of the unbounded operator,
\[
\Delta_{A,\Phi}: \sD(\Delta_{A,\Phi}) \subset L^p(X;\ad P) \to L^p(X;\ad P),
\]
is countable without accumulation points, consisting of non-negative, real eigenvalues, $\lambda$, with finite multiplicities, $\dim\Ker(\Delta_{A,\Phi}-\lambda)$.
\end{prop}

\begin{cor}[\Apriori $L^p$ estimate for a Laplace operator with Sobolev coefficients]
\label{cor:Spectrum_Delta_A_Sobolev_and_W2p_apriori_estimate_pairs}
Let $(X,g)$ be a closed, smooth Riemannian manifold of dimension $d \geq 2$, and $G$ be a compact Lie group and $P$ be a smooth principal $G$-bundle over $X$, and $E = P\times_\varrho\EE$ be a smooth Hermitian vector bundle over $X$ defined by a finite-dimensional unitary representation, and $\varrho: G \to \Aut_\CC(\EE)$. If $(A,\Phi)$ is a $W^{1,q}$ pair on $(P,E)$ with $d/2 < q < \infty$, and $A_1$ is a $C^\infty$ connection on $P$, and $p \in (1,\infty)$ obeys $d/2 \leq p \leq q$, then the kernel $\Ker \Delta_{A,\Phi}\cap W_{A_1}^{2,p}(X;\ad P)$ of the operator \eqref{eq:DeltaAPhi_W2p_to_W1p} is finite-dimensional. Moreover,
\begin{equation}
\label{eq:W2p_apriori_estimate_Delta_APhi_orthogonal_kernel}
\|\xi\|_{W_{A_1}^{2,p}(X)} \leq C\|\Delta_{A,\Phi}\xi\|_{L^p(X)},
\quad
\forall\, \xi \in \left(\Ker \Delta_{A,\Phi}\right)^\perp
\cap W_{A_1}^{2,p}(X;\ad P),
\end{equation}
where $\perp$ denotes $L^2$-orthogonal complement and $C = C(A,\Phi,A_1,g,G,p,q) \in [1,\infty)$.
\end{cor}

The proof of Lemma \ref{lem:Surjectivity_perturbed_Laplace_operator} adapts \mutatis to establish the following analogue for pairs.

\begin{lem}[Surjectivity of a perturbed Laplace operator]
\label{lem:Surjectivity_perturbed_Laplace_operator_pairs}
Let $(X,g)$ be a closed, smooth Riemannian manifold of dimension $d \geq 2$, and $G$ be a compact Lie group, $P$ be a smooth principal $G$-bundle over $X$, and $E = P\times_\varrho\EE$ be a smooth Hermitian vector bundle over $X$ defined by a finite-dimensional unitary representation, $\varrho: G \to \Aut_\CC(\EE)$. Let $A_1$ be a $C^\infty$ connection on $P$ and $(A,\Phi)$ be a $W^{1,q}$ pair on $(P,E)$ with $d/2 < q < \infty$. Then there is a constant $\delta = \delta(A,\Phi,g) \in (0,1]$ with the following significance. If $(a,\phi) \in W_{A_1}^{1,q}(X;\Lambda^1\otimes\ad P\oplus E)$ obeys
\begin{equation}
\label{eq:Perturbation_Laplace_operator_small_pairs}
\|(a,\phi)\|_{L^d(X)} < \delta \quad \text{when } d \geq 3
\quad\text{or}\quad
\|(a,\phi)\|_{L^4(X)} < \delta \quad \text{when } d = 2,
\end{equation}
then the operator,
\[
d_{A,\Phi}^*d_{A+a,\Phi+\phi}:
\left(\Ker \Delta_{A,\Phi}\right)^\perp \cap W_{A_1}^{2,q}(X;\ad P)
\to
\left(\Ker \Delta_{A,\Phi}\right)^\perp \cap L^q(X;\ad P),
\]
is surjective.
\end{lem}

Finally, the proof of Proposition \ref{prop:Feehan_2001_lemma_6-6} adapts \mutatis to establish the following (simplified) analogue for pairs.

\begin{prop}[\Apriori $W^{1,p}$ estimate for $u(A,\Phi)-(A_0,\Phi_0)$ in terms of $(A,\Phi)-(A_0,\Phi_0)$]
\label{prop:Feehan_2001_lemma_6-6_pairs}
Let $(X,g)$ be a closed, smooth Riemannian manifold of dimension $d \geq 2$, and $G$ be a compact Lie group, $P$ be a smooth principal $G$-bundle over $X$, and $E = P\times_\varrho\EE$ be a smooth Hermitian vector bundle over $X$ defined by a finite-dimensional unitary representation, $\varrho: G \to \Aut_\CC(\EE)$. Let $A_1$ be a $C^\infty$ connection on $P$, and $(A_0,\Phi_0)$ be a $W^{1,q}$ pair on $(P,E)$ with $d/2 < q < \infty$ and $p \in (1,\infty)$ obey $d/2 \leq p \leq q$. Then there are constants $N = N(A_1,A_0,\Phi_0,g,G,p,q) \in [1,\infty)$ and $\eps=\eps(A_1,A_0,\Phi_0,g,G,p,q) \in (0,1]$ with the following significance. If $(A,\Phi)$ is a $W^{1,q}$ pair on $(P,E)$ and $u \in \Aut(P)$ is a gauge transformation of class $W^{2,q}$ such that
\begin{equation}
\label{eq:Feehan_2001_6-6_uAPhi_minus_A0Phi0_Coulomb_gauge}
d_{A_0,\Phi_0}^*(u(A,\Phi) - (A_0,\Phi_0)) = 0,
\end{equation}
then the following hold. If $(A,\Phi)$ and $u(A,\Phi)$ obey
\begin{equation}
\label{eq:Feehan_2001_6-6_APhi_minus_A0Phi0_and_uAPhi_minus_A0Phi0_W1p_close}
\|(A,\Phi) - (A_0,\Phi_0)\|_{W_{A_1}^{1,p}(X)} \leq \eps
\quad\text{and}\quad
\|u(A,\Phi) - (A_0,\Phi_0)\|_{W_{A_1}^{1,p}(X)} \leq \eps,
\end{equation}
then
\begin{equation}
\label{eq:Feehan_2001_6-6_uAPhi_minus_A0Phi0_W1p_bound_by_APhi_minus_A0Phi0}
\|u(A,\Phi) - (A_0,\Phi_0)\|_{W_{A_1}^{1,p}(X)}
\leq
N\|(A,\Phi) - (A_0,\Phi_0)\|_{W_{A_1}^{1,p}(X)}.
\end{equation}
\end{prop}

\begin{proof}[Proof of Theorem \ref{mainthm:Feehan_proposition_3-4-4_Lp_pairs}]
Given these preliminaries, Corollary \ref{cor:Spectrum_Delta_A_Sobolev_and_W2p_apriori_estimate_pairs}, Lemma \ref{lem:Surjectivity_perturbed_Laplace_operator_pairs}, and Proposition \ref{prop:Feehan_2001_lemma_6-6_pairs}, the proof of Theorem \ref{mainthm:Feehan_proposition_3-4-4_Lp_pairs} follows \mutatis from that of Theorem \ref{mainthm:Feehan_proposition_3-4-4_Lp}.
\end{proof}

The proof of Proposition \ref{prop:Feehan_2001_lemma_6-6_pairs} yields the following analogue of Lemma \ref{lem:connections_control_gauge} for pairs.

\begin{lem}[\Apriori $W^{2,p}$ estimate for a $W^{2,q}$ gauge transformation $u$ intertwining two $W^{1,q}$ pairs]
\label{lem:connections_control_gauge_pairs}
Assume the hypotheses of Proposition \ref{prop:Feehan_2001_lemma_6-6_pairs}, excluding those on the pair $(A,\Phi)$. Then there is a constant $C = C(A_0,\Phi_0,A_1,g,G,p,q) \in [1,\infty)$ with the following significance. If $(A,\Phi)$ obeys the hypotheses of Proposition \ref{prop:Feehan_2001_lemma_6-6_pairs} and $u \in \Aut^{2,q}(P)$ is the resulting gauge transformation, depending on $(A,\Phi)$ and $(A_0,\Phi_0)$, such that
\[
d_{A_0,\Phi_0}^*(u(A,\Phi)-(A_0,\Phi_0)) = 0,
\]
then
\[
\|u\|_{W_{A_1}^{2,p}(X)} \leq C.
\]
\end{lem}

While not required for the proof of Theorem \ref{mainthm:Feehan_proposition_3-4-4_Lp_pairs}, this is a convenient point at which to note that the proof of Lemma \ref{lem:Continuous_operators_on_Lp_spaces_and_L2-orthogonal_decompositions} (specialized to the case $l=1$) adapts \mutatis to give the following analogue for pairs.

\begin{lem}[Continuous operators on $L^p$ spaces and $L^2$-orthogonal decompositions]
\label{lem:Continuous_operators_on_Lp_spaces_and_L2-orthogonal_decompositions_pairs}
Let $(X,g)$ be a closed, smooth Riemannian manifold of dimension $d \geq 2$, and $G$ be a compact Lie group, $P$ be a smooth principal $G$-bundle over $X$, and $E = P\times_\varrho\EE$ be a smooth Hermitian vector bundle over $X$ defined by a finite-dimensional unitary representation, $\varrho: G \to \Aut_\CC(\EE)$. If $(A,\Phi)$ is a Sobolev pair on $(P,E)$ of class $W^{1,q}$ with $q \geq d/2$, and $A_1$ is a $C^\infty$ reference connection on $P$, and $p$ obeys $d/2\leq p \leq q$, then the operator
\[
d_{A,\Phi}^*:W_{A_1}^{1,p}(X;\Lambda^1\otimes\ad P\oplus E) \to L^p(X;\ad P\oplus E),
\]
is continuous and, if in addition $q > d/2$, then the operator
\[
d_{A,\Phi}:W_{A_1}^{2,p}(X;\ad P\oplus E) \to W_{A_1}^{1,p}(X;\Lambda^1\otimes\ad P\oplus E),
\]
is also continuous and there is an $L^2$-orthogonal decomposition,
\begin{multline*}
W_{A_1}^{1,p}(X;\Lambda^1\otimes\ad P\oplus E)
\\
=
\Ker\left(d_{A,\Phi}^*: W_{A_1}^{1,p}(X;\Lambda^1\otimes\ad P\oplus E)\to L^p(X;\ad P\oplus E)\right)
\\
\oplus \Ran\left(d_{A,\Phi}: W_{A_1}^{2,p}(X;\ad P\oplus E)\to W_{A_1}^{1,p}(X;\Lambda^1\otimes\ad P\oplus E)\right).
\end{multline*}
\end{lem}

\section[Coulomb gauge transformations for $L^r$-close pairs]{Existence of Coulomb gauge transformations for pairs that are $L^r$-close to a reference pair}
\label{sec:Coulomb_gauge_slice_quotient_space_pairs_Lr_close}
We now adapt the construction of Section \ref{sec:Coulomb_gauge_slice_quotient_space_connections_Lr_close} to the case of pairs, following the conventions, definitions, and notation introduced in Section \ref{sec:Coulomb_gauge_slice_quotient_space_pairs} and complete the

\begin{proof}[Proof of Theorem \ref{mainthm:Feehan_proposition_3-4-4_Lp_pairs_Lr_close}]
By analogy with the definition of $H$ in \eqref{eq:Donaldson_Kronheimer_2-3-5a}, now with $(a,\phi) := (A-A_0,\Phi-\Phi_0) \in L^r(X;\Lambda^1\otimes\ad P\oplus E)$, our definition \eqref{eq:Action_gauge_transformation_on_pairs} of the action of $u \in \Aut(P)$ on pairs $(A,\Phi)$, our definition of $d_{A_0,\Phi_0}$ in \eqref{eq:Differential_gauge_transformation_action_on_pair_at_identity}, and our definition \eqref{eq:Pair_Coulomb_gauge_relative_to_reference_pair} of Coulomb gauge for pairs, we set
\begin{equation}
\label{eq:Donaldson_Kronheimer_2-3-5a_pairs}
H(\chi, (a,\phi)) := d_{A_0,\Phi_0}^*\left( e^{-\chi}ae^\chi + e^{-\chi}d_{A_0}(e^\chi), e^{-\chi}\phi\right),
\end{equation}
and the expression \eqref{eq:Donaldson_Kronheimer_2-3-5a_pairs} for $H$ again defines a smooth map,
\begin{multline}
\label{eq:Donaldson_Kronheimer_2-3-5b_Lr_pairs}
H: (\Ker \Delta_{A_0,\Phi_0})^\perp\cap W_{A_1}^{1,r}(X;\ad P)
\times L^r(X;\Lambda^1\otimes\ad P\oplus E)
\\
\to
(\Ker \Delta_{A_0,\Phi_0})^\perp\cap W_{A_1}^{-1,r}(X;\ad P).
\end{multline}
The proof that the image of $H$ in \eqref{eq:Donaldson_Kronheimer_2-3-5b_Lr_pairs} is contained in $(\Ker \Delta_{A_0})^\perp\cap W_{A_1}^{-1,r}(X;\ad P)$ and not just $W_{A_1}^{-1,r}(X;\ad P)$ is the same as the proof that the image of $H$ in \eqref{eq:Donaldson_Kronheimer_2-3-5b} is contained in $(\Ker \Delta_{A_0})^\perp\cap L^q(X;\ad P)$ and not just $L^q(X;\ad P)$.

Given $(a,\phi)$, we solve the equation $H(\chi,(a,\phi)) = 0$ for $\chi$ using Theorem \ref{thm:Quantitative_implicit_function_theorem}, so that
\[
  d_{A_0,\Phi_0}^*(u(A,\Phi)-(A_0,\Phi_0))=0,
\]
as asserted by Theorem \ref{mainthm:Feehan_proposition_3-4-4_Lp_Lr_close}, with $u=e^\chi$, using
\begin{align*}
  \sX &= (\Ker \Delta_{A_0,\Phi_0})^\perp\cap W_{A_1}^{1,r}(X;\ad P),
  \\
  \sY &= L^r(X;\Lambda^1\otimes\ad P\oplus E),
  \\
  \sZ &= (\Ker \Delta_{A_0,\Phi_0})^\perp\cap W_{A_1}^{-1,r}(X;\ad P),
\end{align*}
and $(\chi_0,(a_0,\phi_0)) = (0,(0,0))$, and interchanging the roles of the first and second variables.

The partial derivative of $H(\chi,(a,\phi))$ with respect to $\chi$ at the origin is
\[
  (D_1 H)_{(0,(0,0))} = d_{A_0,\Phi_0}^*d_{A_0,\Phi_0},
\]
and so we must first verify that the operator
\begin{multline}
  \label{eq:D_1H_origin_domain_range_pairs}
(D_1 H)_{(0,(0,0))}: (\Ker \Delta_{A_0,\Phi_0})^\perp\cap W_{A_1}^{1,r}(X;\ad P)
\\
\to (\Ker \Delta_{A_0,\Phi_0})^\perp\cap W_{A_1}^{-1,r}(X;\ad P)
\end{multline}
is an isomorphism of Banach spaces. We temporarily assume that $(A_0,\Phi_0)$ is $C^\infty$ and observe that
\[
  d_{A_0,\Phi_0}^*d_{A_0,\Phi_0}:W_{A_1}^{1,r}(X;\ad P) \to W_{A_1}^{-1,r}(X;\ad P)
\]
is Fredholm with index zero by standard elliptic theory with the usual kernel and range for any $r\in (1,\infty)$: see the proof of Proposition \ref{prop:Fredholmness_and_index_Laplace_operator_on_W2p_smooth_connection} via Theorem \ref{thm:Gilkey_1-4-5_Sobolev}. In particular, Theorem \ref{thm:Gilkey_1-4-5_Sobolev} implies that \eqref{eq:D_1H_origin_domain_range_pairs} is an isomorphism.

In order to verify the remaining hypotheses \eqref{eq:D_1_and_D_2f_conditions_implicit_function_theorem} of Theorem \ref{thm:Quantitative_implicit_function_theorem} and prove the estimate in the conclusion of Theorem \ref{mainthm:Feehan_proposition_3-4-4_Lp_Lr_close} via \eqref{eq:Lipschitz_constant_implied_function_g}, we shall need to identify the dependencies of the constants $M$ and $\beta$. The constant $M = M(A_0,\Phi_0,A_1,g,G,r)$ in \eqref{eq:D_2f_x0y_0_inverse} is given by
\begin{multline*}
  M = \|D_1H(0,(0,0))\|_{\sL(K^\perp\cap W^{1,r}, K^\perp\cap W^{-1,r})}
  \\
  = \|(d_{A_0,\Phi_0}^*d_{A_0,\Phi_0})^{-1}\|_{\sL(K^\perp\cap W^{1,r}, K^\perp\cap W^{-1,r})},
\end{multline*}
where we abbreviate
\begin{align*}
  K &= W_{A_1}^{1,r}(X;\ad P)\cap \Ker d_{A_0,\Phi_0}^*d_{A_0,\Phi_0},
  \\
  W^{1,r} &= W_{A_1}^{1,r}(X;\ad P), \quad\text{and}
  \\
  W^{-1,r} &= W_{A_1}^{-1,r}(X;\ad P).
\end{align*}
To satisfy the hypothesis \eqref{eq:Uniform_Lipschitz_bound_D_2fxy}, we choose $\zeta = \zeta(A_0,A_1,g,G,r) \in (0,1]$ small enough that
\begin{multline*}
  \sup_{(\chi,(a,\phi))\in U \times V}\|D_1H(\chi,(a,\phi)) - D_1H(0,(0,0))\|_{\sL^2(K^\perp\cap W^{1,r}\times L^r,K^\perp\cap W^{-1,r})}
  \\
  \leq \sup_{(\chi,(a,\phi))\in U \times V}\|D_1H(\chi,(a,\phi)) - D_1H(0,(0,0))\|_{\sL^2(W^{1,r}\times L^r,W^{-1,r})} \leq \frac{1}{2M},
\end{multline*}
where we abbreviate
\begin{align*}
  L^r &= L^r(X;\Lambda^1\otimes \ad P\oplus E),
  \\
  U &= \left\{\chi\in W_{A_1}^{1,r}(X;\ad P): \|\chi\|_{W_{A_1}^{1,r}(X)} < \zeta \right\}, \quad\text{and}
  \\
  V &= \left\{a\in L^r(X;\Lambda^1\otimes \ad P\oplus E): \|(a,\phi)\|_{L^r(X)} < \zeta \right\}.
\end{align*}
Finally, the constant $\beta = \beta(A_0,\Phi_0,A_1,g,G,r)$ in \eqref{eq:D1f_uniform_bound} is given by
\[
\beta = \sup_{(\chi,(a,\phi))\in U \times V}\|D_2H(\chi,(a,\phi))\|_{\sL(L^r, K^\perp\cap W^{-1,r})}.
\]
Theorem \ref{thm:Quantitative_implicit_function_theorem} provides constants $\delta = \delta(A_0,\Phi_0,A_1,g,G,r) \in (0,1/(2\beta M)]$ and $C = C(A_0,\Phi_0,A_1,g,G,r) \in (0, 1]$ and a smooth map,
\begin{multline*}
  \bchi: \left\{(b,\psi) \in L^r(X;\Lambda^1\otimes\ad P\oplus E): \|(b,\psi)\|_{L^r(X)} < \delta\right\} \ni a
  \\
  \mapsto \bchi(a) \in (\Ker \Delta_{A_0,\Phi_0})^\perp\cap \left\{\chi \in W_{A_1}^{1,r}(X;\ad P): \|\chi\|_{W_{A_1}^{1,r}(X)} < \zeta \right\},
\end{multline*}
such that
\begin{multline*}
  H(\bchi(a,\phi), (a,\phi)) = 0,
  \\
  \forall\, (a,\phi) \in L^r(X;\Lambda^1\otimes\ad P\oplus E) \text{ with } \|(a,\phi)\|_{L^r(X)} < \delta,
\end{multline*}
that is, the $W^{1,r}$ gauge transformation $u = e^\chi$ defined by the $W^{1,r}$ section $\chi = \bchi(a,\phi)$ obeys
\begin{gather*}
  d_{A_0,\Phi_0}^*(u(A,\Phi)-(A_0,\Phi_0)) = 0,
  \\
  \|u(A,\Phi) - (A_0,\Phi_0)\|_{L^r(X)} < C\|(A,\Phi) - (A_0,\Phi_0)\|_{L^r(X)},
\end{gather*}
as required by Theorem \ref{mainthm:Feehan_proposition_3-4-4_Lp_pairs_Lr_close}. We can approximate any $W^{1,q}$ pair $(A_0,\Phi_0)$ by a $C^\infty$ pair $(\tilde A_0,\tilde\Phi_0)$ (see \cite[Theorem 3.17]{AdamsFournier}) and so the preceding conclusions follow by approximation for $W^{1,q}$ pairs $(A_0,\Phi_0)$ and continuity of the constants $M$ and $\beta$ with respect to variations of $(A_0,\Phi_0)$ in the $W^{1,q}$ norm. (The proof of the latter continuity would be very similar to the proof of Corollary \ref{cor:Fredholmness_and_index_Laplace_operator_on_W2p_Sobolev_connection}, but simpler since we only require continuity and not compactness of Sobolev embeddings and thus is omitted.)

Finally, because $(A,\Phi)$ is a $W^{1,q}$ pair, then $(a,\phi) = (A-A_0,\Phi-\Phi_0)$ is in $W^{1,q}$ and to show that $\chi$ and $u$ are in $W^{2,q}$, we can apply the proof of Lemma \ref{lem:Regularity_distributional_solution_Hodge_Laplacian_Sobolev} to the nonlinear, second-order elliptic equation,
\[
d_{A_0,\Phi_0}^*\left( u^{-1}au + u^{-1}d_{A_0}u, u^{-1}\phi\right) = 0,
\]
to obtain the desired regularity for $u$. (Details of the proof are omitted since they will be very similar to those in the proof of Lemma \ref{lem:Regularity_distributional_solution_Hodge_Laplacian_Sobolev}.) This completes the proof of Theorem \ref{mainthm:Feehan_proposition_3-4-4_Lp_pairs_Lr_close}.
\end{proof}

\section[Regularity for the pure and coupled Yang--Mills equations]{Regularity for solutions to the pure and coupled Yang--Mills equations}
\label{sec:Regularity_for_pairs}
It is well-known that techniques due to Uhlenbeck \cite{UhlRem, UhlLp} can be used to show that, given a weak solution to the Yang--Mills equation, there exists a gauge transformation such that the gauge-transformed solution is smooth. We give a proof of a similar fact here that generalizes easily to the case of coupled Yang--Mills equations.

We have the following generalization of \cite[Theorem 5.3]{ParkerGauge}, due to Parker, and \cite[Proposition 3.7]{FL1}, due to the author and Leness, from the case of $d = 4$ to arbitrary $d \geq 2$.

\begin{thm}[Regularity for solutions to the Yang--Mills equation]
\label{thm:Donaldson_Kronheimer_2-3-11}
Let $(X,g)$ be a closed, smooth Riemannian manifold of dimension $d \geq 2$, and $G$ be a compact Lie group, and $P$ be a smooth principal $G$-bundle over $X$. If $q \in (4/3,\infty)$ obeys $q > d/2$ and $A$ is a $W^{1,q}$ connection on $P$ that is a weak solution to the Yang--Mills equation with respect to the Riemannian metric $g$, then there exists a $W^{2,q}$ gauge transformation $u \in \Aut(P)$ such that $u(A)$ is a $C^\infty$ Yang--Mills connection on $P$.
\end{thm}

\begin{proof}
We proceed as in the proof of \cite[Proposition 3.7]{FL1} and note that the affine space $\sA(P)$ of $C^\infty$ connections on $P$ is dense in the affine space $\sA^{1,q}(P)$ of $W^{1,q}$ connections on $P$ and so there exists a $C^\infty$ connection $A_0$ on $P$ such that
\[
\|A - A_0\|_{W_{A_1}^{1,q}(X)} < \zeta,
\]
where $\zeta = \zeta(A_0,A_1,g,G,q) \in (0,1]$ is the constant in Theorem \ref{mainthm:Feehan_proposition_3-4-4_Lp} and $A_1$ is any fixed $C^\infty$ reference connection on $P$. Hence, there is a $W^{2,q}$ gauge transformation $u \in \Aut(P)$ such that $u(A)$ obeys
\[
d_{A_0}^*(u(A) - A_0) = 0,
\]
and
\[
\|u(A) - A_0\|_{W_{A_1}^{1,q}(X)} < 2N\|A - A_0\|_{W_{A_1}^{1,q}(X)} < 2N\zeta,
\]
where $N = N(A_0,A_1,g,G,q) \in [1,\infty)$ is the constant in Proposition \ref{prop:Feehan_2001_lemma_6-6}. Hence, we may assume without loss of generality that $A$ is in Coulomb gauge with respect to $A_0$ and that
\[
  a := A - A_0 \in W_{A_1}^{1,q}(X;\Lambda^1\otimes \ad P)
\]
is a weak solution to
\[
d_{A_0+a}^*F_{A_0+a} = 0 \quad\text{and}\quad d_{A_0}^*a = 0,
\]
and thus a weak solution to the quasi-linear, second-order elliptic system,
\begin{equation}
\label{eq:Yang--Mills_and_Coulomb_gauge_elliptic_system}
(\Delta_{A_0} + \lambda_0)a + a\times \nabla_{A_0}a + a\times a\times a = \lambda_0 a - F_{A_0},
\end{equation}
where $\Delta_{A_0} = d_{A_0}^*d_{A_0} + d_{A_0}d_{A_0}^*$ is the usual Hodge Laplacian on $\Omega^1(X;\ad P)$ and $\lambda_0>0$ is any positive constant. We recall from \cite{Feehan_yang_mills_gradient_flow_v4} that, for any $q \in (1,\infty)$ and integer $k\geq 0$, the operator
\begin{equation}
\label{eq:Delta_A0smooth_plus_lambda:Wk+2p_to_Wkp}
\Delta_{A_0} + \lambda_0:W_{A_1}^{k+2,q}(X;\Lambda^1\otimes\ad P) \to W_{A_1}^{k,q}(X;\Lambda^1\otimes\ad P)
\end{equation}
is invertible. The right-hand side of Equation \eqref{eq:Yang--Mills_and_Coulomb_gauge_elliptic_system} belongs to $W_{A_1}^{1,q}(X;\Lambda^1\otimes\ad P)$ by hypothesis on $A$ and the fact that $A_0$ is $C^\infty$. If $q > d$, then $W^{1,q}(X) \subset C(X)$ by \cite[Theorem 4.12]{AdamsFournier} and so
\[
a\times \nabla_{A_0}a + a\times a\times a \in L^q(X;\Lambda^1\otimes\ad P).
\]
But then existence and uniqueness of solutions in $W_{A_1}^{2,q}(X;\Lambda^1\otimes\ad P)$ to \eqref{eq:Yang--Mills_and_Coulomb_gauge_elliptic_system}, given a source term in $L^q(X;\Lambda^1\otimes\ad P)$, implies that $a \in W_{A_1}^{2,q}(X;\Lambda^1\otimes\ad P)$. Using the fact that $W^{k,q}(X)$ is a Banach algebra when $kq>d$ and invertibility of \eqref{eq:Delta_A0smooth_plus_lambda:Wk+2p_to_Wkp}, we can iterate in the usual way to show that $a \in W_{A_1}^{k+2,q}(X;\Lambda^1\otimes\ad P)$ for all integers $k \geq 0$ and hence that $a \in C^\infty(X;\Lambda^1\otimes\ad P)$.

Therefore, it suffices to consider the case $d/2<q<d$. Recall that $q^* = dq/(d-q)$ and so $q^* \in (d,\infty)$. Suppose that we can choose\footnote{For example, when $d \geq 4$, we have $1/q^* + 1/q < 1/d + 2/d = 3/d < 1$ for $d \geq 4$ and thus, when $d \geq 4$, it is easy to see that we can choose $r$ as stated.} $r \in (1,q)$ obeying $1/q^* + 1/q \leq 1/r$. Then there is a continuous multiplication map $L^{q^*}(X) \times L^q(X) \to L^r(X)$ and a continuous Sobolev embedding $W^{1,q}(X) \subset L^{q^*}(X)$ by \cite[Theorem 4.12]{AdamsFournier}. Thus, $a\times \nabla_{A_0}a \in L^r(X;\Lambda^1\otimes\ad P)$ if
\[
1/r \geq 1/q^* + 1/q = (d-q)/(dq) + 1/q = (2d-q)/(dq),
\]
that is, $r \leq dq/(2d-q)$. Now $(2d-q)/(dq)< 1$ (which we need to permit a choice of $r>1$) and thus $dq/(2d-q) > 1 \iff dq > 2d-q \iff q > 2d/(d+1)$. But $q>d/2$ by hypothesis (when $d\geq 3$) and $d/2 \geq 2d/(d+1) \iff d\geq 3$; for the case $d=2$, we need $q > 4/3$ to ensure that we can choose $r>1$ and this explains the restriction in our hypotheses that $q > 4/3$ when $d=2$.

Similarly, we have a continuous multiplication map $L^{3s}(X)\times L^{3s}(X)\times L^{3s}(X)\to L^s(X)$ for any $s\in[1,\infty]$ and a continuous Sobolev embedding $W^{1,q}(X) \subset L^{3s}(X)$ provided $3s \leq q^*$, that is, $s \leq dq/(3(d-q))$ and for this choice of $s$, we have $a\times a\times a \in L^s(X;\Lambda^1\otimes\ad P)$.

We now observe that $r\leq s$ when $q\geq d/2$, as we assume by hypothesis, if we choose $r = dq/(2d-q)$ and $s = dq/(3(d-q))$. Indeed, we then have
\[
r\leq s \iff dq/(2d-q) \leq dq/(3(d-q)) \iff 3d-3q \leq 2d-q \iff d \leq 2q.
\]
Hence, for $r = dq/(2d-q)$, we have
\[
a\times \nabla_{A_0}a + a\times a\times a \in L^r(X;\Lambda^1\otimes\ad P),
\]
and thus elliptic regularity theory for \eqref{eq:Yang--Mills_and_Coulomb_gauge_elliptic_system} implies that $a \in W_{A_1}^{2,r}(X;\Lambda^1\otimes\ad P)$, similar to the case $q>d$. By \cite[Theorem 4.12]{AdamsFournier}, when $r<d$, we have a continuous Sobolev embedding, $W^{2,r}(X) \subset W^{1,r^*}(X)$, where $r^* = dr/(d-r)$, and thus we obtain
\[
a \in W_{A_1}^{1,r^*}(X;\Lambda^1\otimes\ad P).
\]
In the limiting case $q=d/2$ with $d\geq 3$ we would have $r = dq/(2d-q) = (d^2/2)/(2d-d/2) = d/3$ and in the limiting case $q=d$ we would have $r=d$, so $r \in (d/3,d)$ and thus $r^* = dr/(d-r) \in (d/2,\infty)$. In particular,
\[
r^* = \frac{dr}{d-r} = \frac{d^2q/(2d-q)}{d-dq/(2d-q)} = \frac{dq/(2d-q)}{1-q/(2d-q)}
= \frac{dq}{2(d-q)}.
\]
We may write $r^* = q + \delta$, where $\delta = \delta(d,q)$ is defined by
\[
\delta := r^*-q = \frac{dq}{2(d-q)} - q = \frac{dq - 2(d-q)q}{2(d-q)}
 = \frac{2q^2 - dq}{2(d-q)} = \frac{q(2q - d)}{2(d-q)},
\]
and thus $\delta(d,q) > 0$ for $d/2<q<d$ and $d \geq 3$. Consequently, we see a regularity improvement and because $\delta(d,q)$ is an increasing function of $q$, only finitely many iterations of this regularity improvement are required to give $r^* > d$, at which point we can apply the regularity argument for the case $q>d$ to again conclude that $a \in C^\infty(X;\Lambda^1\otimes\ad P)$.

In the limiting case $q=4/3$ with $d=2$, we would have $r = dq/(2d-q) = (8/3)/(4-4/3) = (8/3)/(8/3) = 1$, and in the limiting case $q=2$ with $d=2$ we would have $r=4/(4-2)=2$, so $r \in (1,2)$ and thus $r^* = dr/(d-r) \in (2,\infty)$. In particular, $r^* > 4/3 + 2/3$. Consequently, we see a regularity improvement just as we did in the case $d \geq 3$ and the remainder of the argument follows \mutatis as in the case $d \geq 3$.
\end{proof}

The proof of Theorem \ref{thm:Donaldson_Kronheimer_2-3-11} adapts \mutatis to give the following generalization of \cite[Theorem 5.3]{ParkerGauge}, due to Parker, and \cite[Proposition 3.7]{FL1}, due to the author and Leness, from the case of $d = 4$ to arbitrary $d \geq 2$.

\begin{thm}[Regularity for solutions to the boson coupled Yang--Mills equations]
\label{thm:Parker_1982_5-3}
Let $(X,g)$ be a closed, smooth Riemannian manifold of dimension $d \geq 2$, and $G$ be a compact Lie group, $P$ be a smooth principal $G$-bundle over $X$, and $E = P\times_\varrho\EE$ be a smooth Hermitian vector bundle over $X$ defined by a finite-dimensional unitary representation, $\varrho: G \to \Aut_\CC(\EE)$. If $d/2 < q < \infty$ and $(A_\infty,\Phi_\infty)$ is a $W^{1,q}$ pair on $(P,E)$ that is a critical point of the boson coupled Yang--Mills energy function \eqref{eq:Boson_Yang--Mills_energy_function}, then there exists a $W^{2,q}$ gauge transformation $u \in \Aut(P)$ such that $u(A_\infty,\Phi_\infty)$ is a $C^\infty$ pair on $(P,E)$.
\end{thm}

\begin{proof}
We proceed as in the proofs of \cite[Theorem 5.3]{ParkerGauge} and \cite[Proposition 3.7]{FL1} and note that the affine space $\sA(P)\times C^\infty(E)$ of $C^\infty$ pairs is dense in the affine space $\sA^{1,q}(P)\times W^{1,q}(X;E)$ of $W^{1,q}$ pairs and so there exists a $C^\infty$ pair $(A_0,\Phi_0)$ on $(P,E)$ such that
\[
\|(A_\infty,\Phi_\infty) - (A_0,\Phi_0)\|_{W_{A_1}^{1,q}(X)} < \zeta,
\]
where $\zeta = \zeta(A_1,A_0,\Phi_0,g,G,q) \in (0,1]$ is the constant in Theorem \ref{mainthm:Feehan_proposition_3-4-4_Lp_pairs} and $A_1$ is any fixed $C^\infty$ reference connection on $P$. Hence, there is a gauge transformation $u \in \Aut(P)$ of class $W^{2,q}$ such that $u(A_\infty,\Phi_\infty)$ obeys
\[
d_{A_0,\Phi_0}^*\left(u(A_\infty,\Phi_\infty) - (A_0,\Phi_0)\right) = 0,
\]
and
\[
\|u(A_\infty,\Phi_\infty) - (A_0,\Phi_0)\|_{W_{A_1}^{1,p}(X)}
< 2N\|(A_\infty,\Phi_\infty) - (A_0,\Phi_0)\|_{W_{A_1}^{1,p}(X)} < 2N\zeta,
\]
where $N = N(A_1,A_0,\Phi_0,g,G,q) \in [1,\infty)$ is the constant in Proposition \ref{prop:Feehan_2001_lemma_6-6_pairs}. Hence, we may assume without loss of generality that $(A_\infty,\Phi_\infty)$ is in Coulomb gauge with respect to $(A_0,\Phi_0)$ and that
\[
  (a,\phi) := (A_\infty,\Phi_\infty) - (A_0,\Phi_0) \in W_{A_1}^{1,q}(X;\Lambda^1\otimes \ad P \oplus E)
\]
is a weak solution to the quasi-linear, second-order elliptic system,
\begin{align*}
{}& d_{A_0}^*d_{A_0}a + \nabla_{A_0}^*\nabla_{A_0}\phi + a\times \nabla_{A_0}a + a\times \nabla_{A_0}\phi + \phi\times \nabla_{A_0}\phi
\\
&\quad + a\times a\times a + a\times a\times \phi + a\times \phi\times \phi
\\
&\quad + m\phi + s\Phi_0\times \Phi_0\times \phi + s\Phi_0\times \phi\times \phi + s\phi\times \phi\times \phi = f(m, s, A_0,\Phi_0),
\\
{}& d_{A_0,\Phi_0}d_{A_0,\Phi_0}^*(a,\phi) = 0,
\end{align*}
where we employ the expression \eqref{eq:Gradient_boson_coupled_Yang--Mills_energy_function} for the gradient $\sM(A_\infty,\Phi_\infty)$ of $\sE$ at $(A_\infty,\Phi_\infty)$ and $f(m, s, A_0,\Phi_0) \in C^\infty(X;\Lambda^1\otimes \ad P \oplus E)$ is the tautologically defined right-hand source term. Ellipticity of the preceding system follows by expanding the expression \eqref{eq:Differential_gauge_transformation_action_on_pair_at_identity} for $d_{A_0,\Phi_0}$ on $\Omega^0(X;\ad P)$ to extract the second-order term $d_{A_0}d_{A_0}^*a$ and recalling that the operator
$d_{A_0}^*d_{A_0} + d_{A_0}d_{A_0}^*$ is clearly elliptic, being the usual Hodge Laplacian on $\Omega^1(X;\ad P)$. The remainder of the proof now follows the pattern of the proof of Theorem \ref{thm:Donaldson_Kronheimer_2-3-11}.
\end{proof}

\chapter[{\L}ojasiewicz--Simon inequalities for coupled Yang--Mills energies]{{\L}ojasiewicz--Simon gradient inequalities for coupled Yang--Mills energy functions}
\label{chap:Lojasiewicz-Simon_gradient_inequality_coupled_Yang--Mills_energy_functions}
In this section, we apply our Coulomb-gauge transformation result, Theorem \ref{mainthm:Feehan_proposition_3-4-4_Lp_pairs}, and abstract {\L}ojasiewicz--Simon gradient inequality, Theorem \ref{mainthm:Lojasiewicz-Simon_gradient_inequality}, to prove the corresponding {\L}ojasiewicz--Simon gradient inequalities for the boson and fermion coupled Yang--Mills energy functions, Theorem \ref{mainthm:Lojasiewicz-Simon_gradient_inequality_boson_Yang--Mills_energy_function} in Section \ref{sec:Lojasiewicz-Simon_gradient_inequality_boson_Yang--Mills_energy_function}, and Theorem \ref{mainthm:Lojasiewicz-Simon_gradient_inequality_fermion_Yang--Mills_energy_function} in Section \ref{sec:Lojasiewicz-Simon_gradient_inequality_fermion_Yang--Mills_energy_function}.

\section[{\L}ojasiewicz--Simon inequality for boson coupled Yang--Mills]{{\L}ojasiewicz--Simon gradient inequality for the boson coupled Yang--Mills energy function}
\label{sec:Lojasiewicz-Simon_gradient_inequality_boson_Yang--Mills_energy_function}
For any $C^\infty$ reference connection $A_1$ on $P$, let
\begin{equation}
\label{eq:Affine_space_boson_Sobolev_pairs}
\sP^{k,p}(P,E) := \sA^{k,p}(P)\times  W_{A_1}^{k,p}(X; E),
\end{equation}
denote the affine space of $W^{k,p}$ pairs on $(P,E)$, where $k\in\ZZ$ is a positive integer and $p \in (1,\infty)$; we write $\sP(P,E) := \sA(P)\times C^\infty(X; E)$ for the affine space of $C^\infty$ pairs on $(P,E)$.

\subsection{Analyticity of the boson coupled Yang--Mills energy function and its gradient map on the Sobolev space of pairs}
\label{subsec:Analyticity_boson_Yang--Mills_energy_function}
We begin by establishing the following result on analyticity of the
boson coupled Yang--Mills energy function,
$\sE:x_\infty+\sX\to\RR$ (where $x_\infty = (A_\infty,\Phi_\infty)$ is
a critical point); this also serves as a stepping stone towards the
proof that its gradient map, $\sM:x_\infty+\sX\to\tilde\sX$ is real
analytic for suitable choices of the Banach spaces, $\sX$ and $\tilde\sX$, in Theorem~\ref{mainthm:Lojasiewicz-Simon_gradient_inequality}.

\begin{prop}[Analyticity of the boson coupled Yang--Mills energy function on the affine space of $W^{1,p}$ pairs]
\label{prop:analyticbos}
Let $(X,g)$ be a closed, smooth Riemannian manifold of dimension $d \geq 2$, and $G$ be a compact Lie group, $P$ be a smooth principal $G$-bundle over $X$, and $E = P\times_\varrho\EE$ be a smooth Hermitian vector bundle over $X$ defined by a finite-dimensional unitary representation, $\varrho: G \to \Aut_\CC(\EE)$, and $A_1$ be a $C^\infty$ connection on $P$, and $m, s \in C^\infty(X)$. If $4d/(d+4) \leq p < \infty$ when\footnote{Observe that $d/2\geq 4d/(d+4)$ for all $d\geq 4$ and $2 \geq 4d/(d+4)$ when $d=2,3$, so if $p$ obeys the hypotheses of Theorem \ref{mainthm:Lojasiewicz-Simon_gradient_inequality_boson_Yang--Mills_energy_function}, then it obeys the hypotheses of Proposition \ref{prop:analyticbos}.}
$d \geq 2$, then the function,
\[
\sE: \sP^{1,p}(P,E) \to \RR,
\]
is real analytic, where $\sE$ is as in \eqref{eq:Boson_Yang--Mills_energy_function}.
\end{prop}

\begin{proof}
We fix a pair $(A,\Phi) \in \sP^{1,p}(P,E)$ and write $(A, \Phi) = (A_1,\Phi_1) + (a_1, \phi_1)$, where $(a_1, \phi_1)\in W_{A_1}^{1,p}(X; \Lambda^1 \otimes \ad P \oplus E)$. We will show that $\sE$ is analytic at $(A,\Phi)$. For $(a,\phi) \in W_{A_1}^{1,p}(X; \Lambda^1 \otimes \ad P \oplus E)$, we write $A+a = A_1+a_1+a$ and expand
\[
F_{A+a} = F_{A_1+a_1+a} = F_{A_1} + d_{A_1}(a_1 + a) + (a_1+a)\times (a_1+a)
\]
and
\[
\nabla_{A+a} (\Phi + \phi) = \nabla_{A+a_1+a} (\Phi + \phi)
= \nabla_{A_1} (\Phi+\phi) + \varrho(a_1 + a) (\Phi + \phi).
\]
Using the definition~\eqref{eq:Boson_Yang--Mills_energy_function} of $\sE$, we compute
\[
2\sE (A +a , \Phi+ \phi) = T_1 + T_2 + T_3,
\]
where the terms $T_i := T_i(a,\phi)$, for $i=1,2,3$, are given by
\begin{align*}
T_1
&:=
\|F_{A_1}\|^2_{L^2(X)} + \| d_{A_1}(a_1 + a)\|_{L^2(X)}^2
\\
&\quad + \|(a_1+a)\times( a_1+a)\|_{L^2(X)}^2 + 2( F_{A_1}, d_{A_1}(a_1 + a))_{L^2(X)}
\\
&\quad + 2( F_{A_1}, (a_1 + a)\times (a_1 +a))_{L^2(X)}
\\
&\quad + 2( d_{A_1}(a_1 + a), (a_1 + a)\times (a_1 +a))_{L^2(X)},
\end{align*}
and
\begin{align*}
T_2
&:=
\|\nabla_{A_1}(\Phi + \phi)\|_{L^2(X)}^2 + \|\varrho(a_1 + a) (\Phi + \phi)\|_{L^2(X)}^2
\\
  &\quad + (\nabla_{A_1}(\Phi + \phi), \varrho(a_1 + a) (\Phi + \phi))_{L^2(X)}
  \\
  &\quad + (\varrho(a_1 + a) (\Phi + \phi), \nabla_{A_1}(\Phi + \phi))_{L^2(X)},
\\
&=
\|\nabla_{A_1}(\Phi + \phi)\|_{L^2(X)}^2 + \|\varrho(a_1 + a) (\Phi + \phi)\|_{L^2(X)}^2
\\
&\quad + 2\Real(\nabla_{A_1}(\Phi + \phi), \varrho(a_1 + a) (\Phi + \phi))_{L^2(X)},
\end{align*}
and
\[
T_3 := -\int_X \left(m|\Phi+\phi|^2 + s|\Phi+\phi|^4\right)\, d\vol_g.
\]
Hence, we can write the difference as
\[
2\sE (A +a , \Phi+ \phi) - 2\sE (A, \Phi) =  T_1' +  T_2'  + T_3',
\]
where the difference terms $T_i' := T_i(a,\phi) - T_i(0,0)$, for $i = 1,2,3$, are given by
\begin{align*}
T_1'
&=
       \|d_{A_1}a\|_{L^2(X)}^2 + 2( d_{A_1} a_1,  d_{A_1} a)_{L^2(X)}
  \\
  &\quad + ( a\times (a_1+a) , (a_1+a)\times (a_1+a))_{L^2(X)}
\\
&\quad + 2( F_{A_1}, d_{A_1}a)_{L^2(X)}  + ( F_{A_1}, a\times (a_1+a))_{L^2(X)}
\\
&\quad + ( d_{A_1}a, (a_1+a)\times (a_1 + a))_{L^2(X)} +    ( d_{A_1}a_1, a\times (a_1+a))_{L^2(X)},
\end{align*}
and
\begin{equation}
\label{eq:Boson_energy_difference_T2prime}
\begin{aligned}
T_2'
&=
\|\nabla_{A_1} \phi\|^2_{L^2(X)} + 2 \Real(\nabla_{A_1} \Phi, \nabla_{A_1}\phi)_{L^2(X)}
\\
&\quad + \|\varrho(a) \Phi\|^2_{L^2(X)} + \|\varrho(a_1+a)\phi\|^2_{L^2(X)}
\\
&\quad + 2\Real( \varrho(a_1) \Phi, \varrho(a) \Phi)_{L^2(X)}  + 2\Real(\varrho(a_1+a) \Phi, \varrho(a_1+a)\phi)_{L^2(X)}
\\
&\quad + 2\Real(\nabla_{A_1}\phi, \varrho(a_1+a)(\Phi + \phi))_{L^2(X)}
\\
&\quad + 2\Real(\nabla_{A_1}\Phi, \varrho(a)\Phi+ \varrho(a_1+a)\phi)_{L^2(X)},
\end{aligned}
\end{equation}
and
\begin{align*}
T_3'
&=
\int_X \left(m|\phi|^2\, +\, 2m\Real\langle \Phi, \phi\rangle  + s |\phi|^4 +   4s(\Real \langle \Phi, \phi\rangle)^2 \right)\, d\vol_g
\\
&\quad + \int_X\left( 4s (|\Phi|^2+ |\phi|^2)\Real\langle \Phi, \phi\rangle + 2s |\Phi|^2 |\phi|^2 \right)\, d\vol_g.
\end{align*}
To see the origin of the expression \eqref{eq:Boson_energy_difference_T2prime} for $T_2'$, we observe that
\begin{align*}
T'_2 &:= \|\nabla_{A_1+a_1+a}(\Phi+\phi)\|_{L^2(X)}^2 - \|\nabla_{A_1+a_1}\Phi\|_{L^2(X)}^2
\\
&\,= \|\nabla_{A_1}(\Phi+\phi) + \varrho(a_1+a)(\Phi+\phi)\|_{L^2(X)}^2 - \|\nabla_{A_1}\Phi+ \varrho(a_1)\Phi)\|_{L^2(X)}^2
\\
&\,= \|\nabla_{A_1} (\Phi+\phi)\|_{L^2(X)}^2 - \|\nabla_{A_1} \Phi\|_{L^2(X)}^2
\\
&\quad\, + \|\varrho(a_1+a)(\Phi+\phi)\|_{L^2(X)}^2 - \|\varrho(a_1)\Phi\|_{L^2(X)}^2
\\
&\quad\, + 2\Real(\nabla_{A_1}(\Phi+\phi), \varrho(a_1+a)(\Phi+\phi))_{L^2(X)}
- 2\Real (\nabla_{A_1}\Phi, \varrho(a_1)\Phi)_{L^2(X)}
\\
&\,= T_{21}' + T_{22}' + T_{23}'.
\end{align*}
For the first term, we have
\begin{multline*}
T_{21}' := \|\nabla_{A_1} (\Phi+\phi)\|_{L^2(X)}^2 - \|\nabla_{A_1} \Phi\|_{L^2(X)}^2
\\
= \|\nabla_{A_1}\phi\|_{L^2(X)}^2 + 2\Real(\nabla_{A_1}\Phi, \nabla_{A_1}\phi)_{L^2(X)}.
\end{multline*}
For the second term, we see that
\begin{align*}
T_{22}' &:= \|\varrho(a_1+a)(\Phi+\phi)\|_{L^2(X)}^2 - \|\varrho(a_1)\Phi\|_{L^2(X)}^2
\\
&\,= \|\varrho(a_1)\Phi + \varrho(a)\Phi+ \varrho(a_1+a)\phi\|_{L^2(X)}^2 - \|\varrho(a_1)\Phi\|_{L^2(X)}^2
\\
&\,= \|\varrho(a)\Phi\|_{L^2(X)}^2 + \|\varrho(a_1+a)\phi\|_{L^2(X)}^2
+ 2\Real(\varrho(a_1)\Phi, \varrho(a)\Phi)_{L^2(X)}
\\
&\quad\, + 2\Real(\varrho(a_1)\Phi, \varrho(a_1+a)\phi\rangle + 2\Real(\varrho(a)\Phi, \varrho(a_1+a)\phi)_{L^2(X)}
\\
&\,= \|\varrho(a)\Phi\|_{L^2(X)}^2 + \|\varrho(a_1+a)\phi\|_{L^2(X)}^2 + 2\Real(\varrho(a_1)\Phi, \varrho(a)\Phi)_{L^2(X)}
\\
&\quad\, + 2\Real(\varrho(a_1+a)\Phi, \varrho(a_1+a)\phi)_{L^2(X)}.
\end{align*}
For the third term, we have
\begin{align*}
T_{23}' &:= 2\Real(\nabla_{A_1}(\Phi+\phi), \varrho(a_1+a)(\Phi+\phi))_{L^2(X)}
- 2\Real (\nabla_{A_1}\Phi, \varrho(a_1)\Phi)
\\
        &\,= 2\Real(\nabla_{A_1}\Phi, \varrho(a_1)\Phi+ \varrho(a)\Phi+ \varrho(a_1+a)\phi)_{L^2(X)}
  \\
  &\quad\, + 2\Real(\nabla_{A_1}\phi, \varrho(a_1+a)(\Phi+\phi))_{L^2(X)} - 2\Real(\nabla_{A_1}\Phi, \varrho(a_1)\Phi)_{L^2(X)}
\\
        &\,= 2\Real(\nabla_{A_1}\Phi, \varrho(a)\Phi+ \varrho(a_1+a)\phi)_{L^2(X)}
          \\
&\quad\, + 2\Real(\nabla_{A_1}\phi, \varrho(a_1+a)(\Phi+\phi))_{L^2(X)}.
\end{align*}
By adding the preceding terms, we obtain the expression \eqref{eq:Boson_energy_difference_T2prime} for $T_2'$.

To complete the proof of analyticity of $\sE$ at $(A,\Phi)$, we observe that by \cite[Theorem 4.12]{AdamsFournier} there is a continuous Sobolev embedding, $W^{1,p}(X) \subset L^4(X)$, if $p$ obeys
\[
\begin{cases}
p^* = dp/(d-p) \geq 4 &\text{when } 1 \leq p < d, \text{ or}
\\
p \geq d.
\end{cases}
\]
When $p < d$, then $p^* = dp/(d-p) \geq 4$ if and only if $dp \geq 4(d-p)$, that is, $p(d+4) \geq 4d$ or equivalently $p \geq 4d/(d+4)$. Moreover, $4d/(d+4) \leq d$ for all $d \geq 2$ and thus it suffices to choose $p \geq 4d/(d+4)$, as in our hypothesis, in order to ensure a continuous Sobolev embedding, $W^{1,p}(X) \subset L^4(X)$.

Hence, we obtain a continuous multilinear map, $\otimes_{i=1}^{4} W^{1,p}(X) \rightarrow L^1(X)$, by combining the Sobolev embedding $W^{1,p}(X) \rightarrow L^4(X)$ with the continuous multiplication map,
$\otimes_{i=1}^{4} L^4(X) \rightarrow L^1(X)$. Combining these observations with the Kato Inequality \cite[Equation (6.20)]{FU}, we obtain an estimate of the form,
\[
|\sE (A +a , \Phi+ \phi) - \sE(A, \Phi)| \leq |T_1'| +   |T_2'| +  |T_3'|,
\]
where, for a constant $C = C(g, G) \in [1,\infty)$,
\begin{align*}
C^{-1}|T_1'|
&\leq
               \|a\|_{W^{1,p}_{A_1}(X)}^2 + \|a_1\|_{W^{1,p}_{A_1}(X)}\|a\|_{W^{1,p}_{A_1}(X)}
  \\
  &\quad + \|a\|_{W^{1,p}_{A_1}(X)}\left(\|a_1\|_{W^{1,p}_{A_1}(X)} + \|a\|_{W^{1,p}_{A_1}(X)}\right)^3
\\
  &\quad + \|F_{A_1}\|_{L^2(X)} \|a\|_{W^{1,p}_{A_1}(X)}
  \\
  &\quad + \|F_{A_1}\|_{L^2(X)} \|a\|_{W^{1,p}_{A_1}(X)}\left(\|a_1\|_{W^{1,p}_{A_1}(X)} + \|a\|_{W^{1,p}_{A_1}(X)}\right)
\\
&\quad + \|a\|_{W^{1,p}_{A_1}(X)}\left(\|a_1\|_{W^{1,p}_{A_1}(X)} + \|a\|_{W^{1,p}_{A_1}(X)}\right)^2
\\
&\quad + \|a_1\|_{W^{1,p}_{A_1}(X)}\|a\|_{W^{1,p}_{A_1}(X)}\left(\|a_1\|_{W^{1,p}_{A_1}(X)} + \|a\|_{W^{1,p}_{A_1}(X)}\right),
\end{align*}
noting that $W^{1,p}(X) \subset L^2(X)$ if $dp/(d-p) \geq 2$, that is, $dp \geq 2(d-p)$ or $p(d+2) \geq 2d$ or $p \geq 2d/(d+2)$, and
\begin{align*}
C^{-1}|T_2'|
&\leq  \|\phi\|_{W^{1,p}_{A_1}(X)}^2 +  \|\Phi\|_{W^{1,p}_{A_1}(X)}\|\phi\|_{W^{1,p}_{A_1}(X)} + \|a\|^2_{W^{1,p}_{A_1}(X)}\|\Phi\|^2_{W^{1,p}_{A_1}(X)}
\\
  &\quad + \left(\|a_1\|_{W^{1,p}_{A_1}(X)} + \|a\|_{W^{1,p}_{A_1}(X)}\right)^2\|\phi\|^2_{W^{1,p}_{A_1}(X)}
  \\
  &\quad + \|a_1\|_{W^{1,p}_{A_1}(X)}\|a\|_{W^{1,p}_{A_1}(X)}\|\Phi\|^2_{W^{1,p}_{A_1}(X)}
\\
&\quad + \left(\|a_1\|_{W^{1,p}_{A_1}(X)} + \|a\|_{W^{1,p}_{A_1}(X)}\right)^2_{W^{1,p}_{A_1}(X)}\|\Phi\|_{W^{1,p}_{A_1}(X)}\|\phi\|_{W^{1,p}_{A_1}(X)}
\\
&\quad + \|\phi\|_{W^{1,p}_{A_1}(X)}\left(\|a_1\|_{W^{1,p}_{A_1}(X)} + \|a\|_{W^{1,p}_{A_1}(X)}\right)\left(\|\Phi\|_{W^{1,p}_{A_1}(X)} + \|\phi\|_{W^{1,p}_{A_1}(X)}\right)
\\
  &\quad + \|\Phi\|^2_{W^{1,p}_{A_1}(X)}\|a\|_{W^{1,p}_{A_1}(X)}
  \\
  &\quad + \|\Phi\|_{W^{1,p}_{A_1}(X)}\|\phi\|_{W^{1,p}_{A_1}(X)}\left(\|a_1\|_{W^{1,p}_{A_1}(X)} + \|a\|_{W^{1,p}_{A_1}(X)}\right),
\end{align*}
and
\begin{align*}
C^{-1}|T_3'|
&\leq
               \|\phi\|^2_{W^{1,p}_{A_1}(X)} + \|\Phi\|_{W^{1,p}_{A_1}(X)}\|\phi\|_{W^{1,p}_{A_1}(X)}
  \\
  &\quad + \|\phi\|^4_{W^{1,p}_{A_1}(X)} + \|\Phi\|^2_{W^{1,p}_{A_1}(X)}\|\phi\|^2_{W^{1,p}_{A_1}(X)}
\\
&\quad + \|\Phi\|^3_{W^{1,p}_{A_1}(X)}\|\phi\|_{W^{1,p}_{A_1}(X)} + \|\Phi\|_{W^{1,p}_{A_1}(X)}\|\phi\|^3_{W^{1,p}_{A_1}(X)}.
\end{align*}
Note that $4d/(d+4) \geq 2d/(d+2)$, so the condition $p \geq 2d/(d+2)$ is assured by the stronger $p \geq 4d/(d+4)$. Thus, $\sE (A +a , \Phi+ \phi)$ is a polynomial of degree four in the variable $(a,\phi) \in W_{A_1}^{1,p}(X; \Lambda^1 \otimes \ad P \oplus E)$. This completes the proof of Proposition \ref{prop:analyticbos}.
\end{proof}

We now verify the formula \eqref{eq:Gradient_boson_coupled_Yang--Mills_energy_function} for the differential $\sE'(A,\Phi)$ and gradient $\sM(A,\Phi)$.

\begin{lem}[Differential and gradient of the boson coupled Yang--Mills energy function]
\label{lem:Differential_gradient_boson_coupled_Yang--Mills_energy_function}
Assume the hypotheses of Proposition \ref{prop:analyticbos} with the dual H\"older exponent $p'$ in the range $1 < p'\leq 4d/(3d-4)$ determined by $4d/(d+4) \leq p < \infty$ and $1/p+1/p'=1$. Then the expression for $\sE'(A,\Phi) \in (W_{A_1}^{1,p}(X; \Lambda^1 \otimes \ad P \oplus E))^*$ and $\sM(A,\Phi) \in W_{A_1}^{-1,p'}(X; \Lambda^1 \otimes \ad P \oplus E)$ is given by \eqref{eq:Gradient_boson_coupled_Yang--Mills_energy_function}, namely
\begin{align*}
\sE'(A,\Phi)(a, \phi) &= \left(\sM(A,\Phi),(a, \phi)\right)_{L^2(X)}
\\
&=
(d_A^*F_A, a)_{L^2(X)} + \Real (\nabla_A^*\nabla_A \Phi, \phi)_{L^2(X)}
+ \Real(\nabla_A\Phi, \varrho(a)\Phi)_{L^2(X)}
\\
&\quad - \Real(m\Phi,\phi)_{L^2(X)}
- 2\Real\int_X s|\Phi|^2\langle \Phi,\phi\rangle\, d\vol_g,
\\
&\qquad \forall\, (a,\phi) \in W_{A_1}^{1,p}(X; \Lambda^1 \otimes \ad P \oplus E).
\end{align*}
\end{lem}

\begin{proof}
We establish \eqref{eq:Gradient_boson_coupled_Yang--Mills_energy_function} by extracting the terms that are linear in $(a,\phi)$ from the expressions for $T_i'$, for $i=1,2,3$, arising in the proof of Proposition \ref{prop:analyticbos}. We compute $\sE(A+ta, \Phi+ t\phi)$ using the identities,
\begin{align*}
  F_{A+ta} &= F_{A} + td_A a + \frac{t^2}{2}[a,a],
  \\
  \nabla_{A+ta} (\Phi + t\phi) &= \nabla_A (\Phi+t\phi) + \varrho(ta) (\Phi + t\phi),
\end{align*}
to obtain
\begin{align*}
\sE(A+ta, \Phi+ t\phi)
&=
\frac{1}{2}\int_X |F_A +  td_A a + \frac{t^2}{2}[a,a]|^2\, d\vol_g
\\
&\quad +
\frac{1}{2}\int_X | \nabla_A \Phi+t(\nabla_A\phi + \varrho(a)\Phi) + t^2 \varrho(a)\phi)|^2\, d\vol_g
\\
&\quad -
\frac{1}{2}\int_X \left(m|\Phi + t\phi|^2 + s|\Phi + t\phi|^4\right)\, d\vol_g,
\end{align*}
that is,
\begin{align*}
\sE(A+ta, \Phi+ t\phi)
&=
\frac{1}{2}\int_X \left( |F_A|^2 +  2t \langle F_A, d_A a\rangle\right)\,  d\vol_g
\\
&\quad + \frac{1}{2}\int_X \left(|\nabla_A \Phi|^2 + 2t\Real \langle\nabla_A\Phi, \nabla_A\phi + \varrho(a)\Phi\rangle_{L^2}\right) \, d\vol_g
\\
&\quad - \frac{1}{2}\int_X m\left(|\Phi|^2 + 2t\Real \langle\Phi, \phi\rangle\right)\, d\vol_g
\\
&\quad - \frac{1}{2}\int_X s\left(|\Phi|^4  + 4t|\Phi|^2 \Real \langle\Phi, \phi\rangle\right)\, d\vol_g + \text{higher powers of $t$}.
\end{align*}
This yields the expression \eqref{eq:Gradient_boson_coupled_Yang--Mills_energy_function} for $\sE'(A,\Phi)(a, \phi)$ and completes the proof of Lemma \ref{lem:Differential_gradient_boson_coupled_Yang--Mills_energy_function}.
\end{proof}

It is convenient to define, for $p \in (1,\infty)$ and dual exponent $p'\in (1,\infty)$ determined by $1/p+1/p'=1$, the following Banach spaces,
\begin{equation}
\label{eq:Banach_space_WA1p_pairs_and_dual_space}
\begin{gathered}
\fX := W_{A_1}^{1,p}(X; \Lambda^1 \otimes \ad P \oplus E)
\quad\text{and}\quad
\fX^* = W_{A_1}^{-1,p'}(X; \Lambda^1 \otimes \ad P \oplus E),
\\
\tilde\fX := W_{A_1}^{-1,p}(X; \Lambda^1 \otimes \ad P \oplus E).
\end{gathered}
\end{equation}
Note that for $p \geq 2$, the inclusion, $\tilde\fX \subset \fX^*$, is a continuous embedding of Banach spaces.

The expression
\eqref{eq:Gradient_boson_coupled_Yang--Mills_energy_function} defines
the gradient as a map, $\sM:(A_1,0)+\fX \to \tilde\fX$, where
$\sM(A,\Phi) \in \tilde\fX$ acts on $(b, \varphi) \in \fX$ by $L^2$
inner product via the inclusion $\tilde\fX \subset \fX^*$. We now have the

\begin{prop}[Analyticity of the gradient map for the boson coupled Yang--Mills energy function]
\label{prop:analytic_bos_gradient_map}
Assume the hypotheses of Proposition \ref{prop:analyticbos} and, in addition, that $p \geq d/2$ when $d \geq 3$ and $p > 4/3$ when $d = 2$.  Then
\[
\sM: (A_1,0)+\fX \to \tilde\fX
\]
is a real analytic map, where $\sM$ is as in \eqref{eq:Gradient_boson_coupled_Yang--Mills_energy_function} and $\fX$ and $\tilde\fX$ are as in \eqref{eq:Banach_space_WA1p_pairs_and_dual_space}.
\end{prop}

\begin{proof}
As usual, given $p \in (1,\infty)$, we let $p' \in (1,\infty)$ denote the dual H\"older exponent defined by $1/p+1/p'=1$. By \cite[Theorem 4.12]{AdamsFournier}, we have a continuous Sobolev embedding, $W^{1,p'}(X;\CC) \subset L^r(X;\CC)$, when
\begin{enumerate}
\item
\label{item:1_lessthan_pprime_lessthan_d}
$1 < p' < d$ and $1\leq r \leq (p')^* := dp'/(d-p')$, or
\item
\label{item:pprime_equals_d}
$p' = d$ and $1 \leq r < \infty$, or
\item
\label{item:d_lessthan_pprime_lessthan_infty}
$d<p'<\infty$ and $1 \leq r \leq \infty$.
\end{enumerate}
In each case, we obtain a continuous Sobolev embedding, $L^{r'}(X;\CC) \subset W^{-1,p}(X;\CC)$, by duality and density of the continuous embedding, $W^{1,p'}(X;\CC) \subset L^r(X;\CC)$, where $1/r + 1/r' = 1$ and $1/p+1/p'=1$. We shall require $p', r < \infty$ in order to appeal to the dualities, $(W^{1,p'}(X;\CC))^* = W^{-1,p}(X;\CC)$ and $(L^r(X;\CC))^* = L^{r'}(X;\CC)$; these dualities fail when $p' = \infty$ (and thus $p=1$) or $r=\infty$ (and thus $r'=1$).

If we write $A = A_1 +a_1$, for $a_1\in W_{A_1}^{1,p}(X;\Lambda^1\otimes\ad P)$, then from  the formula \eqref{eq:Gradient_boson_coupled_Yang--Mills_energy_function} for $\sM(A,\Phi)$ we have the formal expression,
\begin{equation}
\begin{aligned}
\label{eq:Gradient_boson_coupled_Yang--Mills_energy_function_expanded}
\sM(A_1+a_1,\Phi)
&=
d_{A_1}^*d_{A_1} a_1  + \nabla_{A_1}^*\nabla_{A_1}\Phi + d^*_{A_1}F_{A_1} + F_{A_1}\times a_1
\\
&\quad + \nabla_{A_1}\Phi\times a_1 + \nabla_{A_1} a_1\times \Phi + \nabla_{A_1}a_1\times a_1
\\
&\quad + \nabla_{A_1}\Phi\times\Phi - (m + 2s|\Phi|^2)\Phi
\\
&\quad + a_1\times a_1\times \Phi + a_1\times \Phi\times\Phi + a_1\times a_1\times a_1.
\end{aligned}
\end{equation}
Observe that
\begin{multline*}
  W_{A_1}^{1,p}(X; \Lambda^1 \otimes \ad P \oplus E)\ni (a_1, \Phi)
  \\
  \to  d_{A_1}^*d_{A_1} a_1  + \nabla_{A_1}^*\nabla_{A_1}\Phi \in W_{A_1}^{-1,p}(X; \Lambda^1 \otimes \ad P \oplus E),
\end{multline*}
is analytic and the term $d^*_{A_1}F_{A_1}$ is constant with respect to $(a_1, \Phi)$. Therefore, to prove that $\sM$ is analytic, it suffices to prove the

\begin{claim}
\label{claim:Analytic_difference_Gradient_and_Laplacian_W1p_to_Lrprime}
Continue the preceding notation. Then
\begin{multline}
\label{eq:Difference_Gradient_and_Laplacian_W1p_to_Lrprime}
W_{A_1}^{1,p}(X; \Lambda^1 \otimes \ad P \oplus E) \ni (a_1,\Phi)
\\
\mapsto \sM(A_1+a_1,\Phi) - d_{A_1}^*d_{A_1} a_1  + \nabla_{A_1}^*\nabla_{A_1}\Phi - d^*_{A_1}F_{A_1}
\\
\in L^{r'}(X; \Lambda^1 \otimes \ad P \oplus E),
\end{multline}
is a cubic polynomial in $(a_1,\Phi)$ and its first-order covariant derivatives with respect to $\nabla_{A_1}$, with universal coefficients (depending at most on $g$ and $G$).
\end{claim}

\begin{proof}[Proof of Claim \ref{claim:Analytic_difference_Gradient_and_Laplacian_W1p_to_Lrprime}]
We compute an $L^{r'}$ bound for each term in equation \eqref{eq:Gradient_boson_coupled_Yang--Mills_energy_function_expanded} for $\sM(A,\Phi)$; we consider the cases $p'<d$, $p'=d$, and $p'>d$ separately, recalling that $A=A_1+a_1$.

\setcounter{step}{0}
\begin{step}[$L^{r'}$ estimates for $F_{A_1}\times a_1$ and $\nabla_{A_1}\Phi\times a_1$]
\label{step:Lrprime_estimate_FA1_times_a1_and_nablaA1Phi_times_a1}
We claim that
\begin{align}
\label{eq:Lrprime_bound_FA1_times_a1}
\|F_{A_1}\times a_1\|_{L^{r'}(X)}
&\leq
z\|F_{A_1}\|_{L^p(X)} \|a_1\|_{W_{A_1}^{1,p}(X)},
\\
\label{eq:Lrprime_bound_nablaA1Phi_times_a1}
\|\nabla_{A_1}\Phi\times a_1\|_{L^{r'}(X)}
&\leq z\|\Phi\|_{W_{A_1}^{1,p}(X)} \|a_1\|_{W_{A_1}^{1,p}(X)},
\end{align}
where $z = z(g,G,p) \in [1,\infty)$.

\setcounter{case}{0}
\begin{case}[$p' < d$]
\label{case:pprime_lessthan_d}
We choose $r = (p')^* = dp'/(d-p')$, where $p' = p/(p-1)$, so $r = (dp/(p-1)) / (d - p/(p-1)) = dp/(d(p-1) - p)$. Since $p' = p/(p-1) < d$ for Case \ref{case:pprime_lessthan_d}, then $p < dp-d$ or $p(d-1)>d$ or $p>d/(d-1)$. By \cite[Theorem 4.12]{AdamsFournier} we have continuous Sobolev embeddings,
\begin{inparaenum}[\itshape a\upshape)]
\item
\label{subcase:pprime_lessthan_d_and_p_lessthan_d}
$W^{1,p}(X) \subset L^{p^*}(X)$, for $p^*=dp/(d-p)$ if $p<d$;
\item
\label{subcase:pprime_lessthan_d_and_p_equals_d}
$W^{1,p}(X) \subset L^q(X)$ for any $q\in[1,\infty)$ if $p=d$; and
\item
\label{subcase:pprime_lessthan_d_and_p_greaterthan_d}
$W^{1,p}(X) \subset L^\infty(X)$ if $p>d$.
\end{inparaenum}
We consider each of these three subcases in turn.

Consider Subcase (\ref{case:pprime_lessthan_d}\ref{subcase:pprime_lessthan_d_and_p_lessthan_d}), so $d/(d-1)<p<d$, which forces $d\geq 3$. In order to have a continuous multiplication map, $L^p(X)\times L^{p^*}(X) \to L^{r'}(X)$, we require that the inequality $1/p + 1/p^* \leq 1/r'$ holds, that is
\[
1/p + 1/(dp/(d-p)) \leq 1/r' = 1-1/r = 1-(dp-d-p)/dp,
\]
or equivalently,
\[
1/p + (d-p)/(dp) \leq 1-(1 - 1/p - 1/d),
\]
namely,
\[
1/p + 1/p - 1/d \leq 1/p + 1/d.
\]
Therefore, $1/p \leq 2/d$ or $p \geq d/2$, as we assumed in our hypotheses. Thus,
\[
\|F_{A_1}\times a_1\|_{L^{r'}(X)} \leq z\|F_{A_1}\|_{L^p(X)} \|a_1\|_{L^{p^*}(X)} \leq z\|F_{A_1}\|_{L^p(X)} \|a_1\|_{W_{A_1}^{1,p}(X)},
\]
which yields the first inequality in \eqref{eq:Lrprime_bound_FA1_times_a1} for all $d \geq 3$ in this subcase.

Consider subcase (\ref{case:pprime_lessthan_d}\ref{subcase:pprime_lessthan_d_and_p_equals_d}), so $p>d/(d-1)$ and $p=d$. Because $p>d/(d-1)$ in this subcase, then $d$ must obey $d>d/(d-1)$ or $d-1>1$ or $d\geq 3$ for this subcase. (If $d=2$, then $p=d=2$ forces $p'=2$, so the subcase $p'<d$ and $p=d$ cannot occur.) We now only need a continuous multiplication map, $L^d(X)\times L^q(X) \to L^{r'}(X)$, for large enough $q \in [1,\infty)$ and this requires only that $r'<d$, that is, $1/r' = 1-1/r > 1/d$ or
\[
1-(dp-d-p)/dp = 1/p + 1/d  = 2/d > 1/p = 1/d,
\]
which holds for all positive $d$. Thus,
\[
\|F_{A_1}\times a_1\|_{L^{r'}(X)} \leq z\|F_{A_1}\|_{L^p(X)} \|a_1\|_{L^q(X)} \leq z\|F_{A_1}\|_{L^p(X)} \|a_1\|_{W_{A_1}^{1,p}(X)},
\]
which yields \eqref{eq:Lrprime_bound_FA1_times_a1} for all $d \geq 3$ in this subcase.

Consider Subcase (\ref{case:pprime_lessthan_d}\ref{subcase:pprime_lessthan_d_and_p_greaterthan_d}), so $p>d/(d-1)$ and $p>d$. But $p>d/(d-1)$ holds for any $d\geq 2$ for this subcase since $d \geq d/(d-1)$ for any $d\geq 2$. We now only need a continuous multiplication map, $L^p(X)\times L^\infty(X) \to L^{r'}(X)$, and this requires only that $r'\leq p$, that is, $1/r' = 1-1/r \geq 1/p$ or
\[
1-(dp-d-p)/dp = 1/p + 1/d  \geq 1/p,
\]
which holds for all positive $d$. Thus,
\[
\|F_{A_1}\times a_1\|_{L^{r'}(X)} \leq z\|F_{A_1}\|_{L^p(X)} \|a_1\|_{L^\infty(X)} \leq z\|F_{A_1}\|_{L^p(X)} \|a_1\|_{W_{A_1}^{1,p}(X)},
\]
which yields \eqref{eq:Lrprime_bound_FA1_times_a1} for all $d \geq 2$ in this subcase.
\end{case}

\begin{case}[$p' = d$]
Since $p' = p/(p-1) = d$ for this case, then $p = dp-d$ or $p(d-1)=d$ and so $p=d/(d-1)$. But our hypotheses in Proposition \ref{prop:analytic_bos_gradient_map} requires $p \geq 4d/(d+4)$ when $d \geq 3$. Observe that
\[
4d/(d+4) \geq d/(d-1) \iff 4(d-1) \geq d+4 \iff 3d \geq 8 \iff d \geq 8/3
\]
that is, $p \geq 4d/(d+4) > d/(d-1)$ when $d \geq 3$. Hence, the case $p'=d$ and $p=d/(d-1)$ is excluded by our hypotheses unless $d=2$ and thus $p' = 2 = p$.

When $d=2$, we have a continuous Sobolev embedding, $W^{1,2}(X) \subset L^q(X)$ for any $q\in[1,\infty)$ by \cite[Theorem 4.12]{AdamsFournier}. Hence, we obtain a continuous multiplication map, $L^2(X)\times L^q(X) \to L^{r'}(X)$, for any $r'>2$ and large enough $q\in[1,\infty)$. Thus,
\[
\|F_{A_1}\times a_1\|_{L^{r'}(X)} \leq z\|F_{A_1}\|_{L^2(X)} \|a_1\|_{L^q(X)} \leq z\|F_{A_1}\|_{L^2(X)} \|a_1\|_{W_{A_1}^{1,2}(X)},
\]
which yields \eqref{eq:Lrprime_bound_FA1_times_a1} for the case $p' = d = 2 = p$.
\end{case}

\begin{case}[$p' > d$]
We can again choose $r \in [1,\infty)$ arbitrarily large in the continuous Sobolev embedding, $W^{1,p'}(X)\subset L^r(X)$, and so we may choose $r' \in (1,\infty]$ arbitrarily small or equivalently, $1/r' \in [0,1)$ arbitrarily close to $1$. (We refrain from choosing $r=\infty$ because continuity of the Sobolev embedding, $W^{1,p'}(X)\subset L^\infty(X)$, does not imply continuity of $L^1(X) \subset W^{-1,p}(X)$.)

Since $p' = p/(p-1) > d$ for this case, then $p > dp-d$ or $p(d-1)<d$ and so $p<d/(d-1) \leq d$. We therefore have $p^* = dp/(d-p)$ and a continuous Sobolev embedding, $W^{1,p}(X) \subset L^{p^*}(X)$. In order to have a continuous multiplication map, $L^p(X)\times L^{p^*}(X) \to L^{r'}(X)$, we require that the inequality $1/p + 1/p^* \leq 1/r'$ holds for some $r \in [1,\infty)$ and $r' = r/(r-1) \in (1,\infty]$. Such a choice of $r'$ will be possible if and only if the strict inequality below holds,
\[
1/p + 1/p^* = 1/p + (d-p)/(dp) = 2/p - 1/d < 1,
\]
that is, if and only if $2/p < 1 + 1/d = (d+1)/d$ or equivalently, $p > 2d/(d+1)$. But our hypotheses in Proposition \ref{prop:analytic_bos_gradient_map} requires $p \geq 4d/(d+4)$ when $d \geq 3$ and $p > 4/3$ when $d=2$. Now $4d/(d+4)>2d/(d+1)$ when $d \geq 3$ and $4/3 = 2d/(d+1)$ when $d=2$, so our hypotheses in Proposition \ref{prop:analytic_bos_gradient_map} ensure that $p > 2d/(d+1)$ all $d \geq 2$. Hence, we may choose a large enough $r \in [1,\infty)$ and thus small enough $r' \in (1,\infty]$ to give $1/p + 1/p^* \leq 1/r'$.

Hence, the map, $L^p(X)\times L^{p^*}(X) \to L^{r'}(X)$, is continuous for small enough $r'>1$ and
\[
\|F_{A_1}\times a_1\|_{L^{r'}(X)} \leq z\|F_{A_1}\|_{L^p(X)} \|a_1\|_{L^{p^*}(X)} \leq z\|F_{A_1}\|_{L^p(X)} \|a_1\|_{W_{A_1}^{1,p}(X)},
\]
which again yields \eqref{eq:Lrprime_bound_FA1_times_a1} for this case.
\end{case}

An argument identical to that for \eqref{eq:Lrprime_bound_FA1_times_a1} gives
\[
\|\nabla_{A_1}\Phi\times a_1\|_{L^{r'}(X)}
\leq
z\|\nabla_{A_1}\Phi\|_{L^p(X)} \|a_1\|_{W_{A_1}^{1,p}(X)}.
\]
proving \eqref{eq:Lrprime_bound_nablaA1Phi_times_a1}. This completes Step \ref{step:Lrprime_estimate_FA1_times_a1_and_nablaA1Phi_times_a1}.
\end{step}

\begin{step}[$L^{r'}$ estimates for $\nabla_{A_1}\Phi\times \Phi$ and $\nabla_{A_1}a_1\times \Phi$ and $\nabla_{A_1}a_1\times a_1$]
\label{step:Lrprime_estimate_Phi_times_nablaA1Phi_and_two_similar_terms}
We claim that
\begin{align}
\label{eq:Lrprime_bound_Phi_times_nablaA1Phi}
\|\nabla_{A_1}\Phi\times \Phi\|_{L^{r'}(X)}
&\leq z\|\Phi\|_{W_{A_1}^{1,p}(X)}^2,
\\
\label{eq:Lrprime_bound_Phi_times_nablaA1a1}
\|\nabla_{A_1}a_1\times \Phi\|_{L^{r'}(X)}
&\leq z\|\Phi\|_{W_{A_1}^{1,p}(X)}\|a_1\|_{W_{A_1}^{1,p}(X)},
\\
\label{eq:Lrprime_bound_a1_times_nablaA1a1}
\|\nabla_{A_1}a_1\times a_1\|_{L^{r'}(X)}
&\leq z\|a_1\|_{W_{A_1}^{1,p}(X)}^2
\end{align}
where $z = z(g,G,p) \in [1,\infty)$.

From the proof of \eqref{eq:Lrprime_bound_FA1_times_a1} in Step \ref{step:Lrprime_estimate_FA1_times_a1_and_nablaA1Phi_times_a1}, we have
\[
\|\Phi\times \nabla_{A_1}\Phi\|_{L^{r'}(X)} \leq z\|\Phi\|_{W_{A_1}^{1,p}(X)} \|\nabla_{A_1}\Phi\|_{L^p(X)},
\]
and this gives \eqref{eq:Lrprime_bound_Phi_times_nablaA1Phi}; identical arguments give \eqref{eq:Lrprime_bound_Phi_times_nablaA1a1} and \eqref{eq:Lrprime_bound_a1_times_nablaA1a1}. This completes Step \ref{step:Lrprime_estimate_Phi_times_nablaA1Phi_and_two_similar_terms}.
\end{step}

Note that $\|m\|_{C(X)}$ and $\|s\|_{C(X)}$ are finite by hypothesis.

\begin{step}[$L^{r'}$ estimate for $m\Phi$]
We have
\[
\|m\Phi\|_{L^{r'}(X)} \leq \|m\|_{C(X)}\|\Phi\|_{L^{r'}(X)},
\]
and because $r'\leq p$ by inspection of each of the three cases, $p'<d$ and $p'=d$ and $p'>d$, and subcases ($p<d$ and $p=d$ and $p>d$ where applicable) we obtain
\begin{equation}
\label{eq:Lrprime_bound_mphi}
\|m\Phi\|_{L^{r'}(X)} \leq z\|m\|_{C(X)}\|\Phi\|_{L^p(X)},
\end{equation}
as desired.
\end{step}

\begin{step}[$L^{r'}$ estimates for $s|\Phi|^2\Phi$]
We claim that
\begin{equation}
\label{eq:Lrprime_bound_s_times_Phi_times_Phi_phi}
\|s|\Phi|^2\Phi\|_{L^{r'}(X)}
\leq
z\|s\|_{C(X)}\|\Phi\|_{W_{A_1}^{1,p}(X)}^3,
\end{equation}
where $z = z(g,p) \in [1,\infty)$.

From the proof of \eqref{eq:Lrprime_bound_FA1_times_a1} in Step \ref{step:Lrprime_estimate_FA1_times_a1_and_nablaA1Phi_times_a1}, we have
\[
\|s|\Phi|^2\Phi\|_{L^{r'}(X)} \leq z\|s\|_{C(X)}\||\Phi|^2\|_{L^p(X)} \|\Phi\|_{W_{A_1}^{1,p}(X)}
\]
Moreover, we have
\[
\||\Phi|^2\|_{L^p(X)} = \|\Phi\|_{L^{2p}(X)}^2 \leq z\|\Phi\|_{W_{A_1}^{1,p}(X)}^2.
\]
Combining the preceding inequalities yields \eqref{eq:Lrprime_bound_s_times_Phi_times_Phi_phi}.
\end{step}

\begin{step}[$L^{r'}$ estimates for $a_1\times a_1\times a_1$ and $a_1\times a_1\times \Phi$ and $a_1\times \Phi\times\Phi$]
\label{step:Lrprime_bound_a1_times_Phi_times_Phi_and_two_similar_terms}
We claim that
\begin{align}
\label{eq:Lrprime_bound_a1_times_Phi_times_Phi}
\|a_1\times \Phi\times\Phi\|_{L^{r'}(X)}
&\leq
z\|\Phi\|_{W_{A_1}^{1,p}(X)}^2 \|a_1\|_{W_{A_1}^{1,p}(X)},
\\
\label{eq:Lrprime_bound_a1_times_a1_times_Phi}
\|a_1\times a_1 \times\Phi\|_{L^{r'}(X)}
&\leq
z\|\Phi\|_{W_{A_1}^{1,p}(X)} \|a_1\|_{W_{A_1}^{1,p}(X)}^2,
\\
\label{eq:Lrprime_bound_a1_times_a1_times_a1}
\|a_1\times a_1\times a_1\|_{L^{r'}(X)}
&\leq
z\|a_1\|_{W_{A_1}^{1,p}(X)}^3,
\end{align}
where $z = z(g,G,p) \in [1,\infty)$.

From the proof of \eqref{eq:Lrprime_bound_FA1_times_a1} in Step \ref{step:Lrprime_estimate_FA1_times_a1_and_nablaA1Phi_times_a1}, we have
\[
\|a_1\times \Phi\times\Phi\|_{L^{r'}(X)} \leq z\|a_1\|_{W_{A_1}^{1,p}(X)} \|\Phi\times \Phi\|_{L^p(X)}.
\]
Continuity of the multiplication, $L^{2p}(X)\times L^{2p}(X) \to L^p(X)$, and continuity of the Sobolev embedding, $W^{1,p}(X) \subset L^{2p}(X)$, valid for any $p\geq d/2$ (which we assume by hypothesis), gives
\[
\|\Phi\times \Phi\|_{L^p(X)} \leq z\|\Phi\|_{L^{2p}(X)} \leq z\|\Phi\|_{W_{A_1}^{1,p}(X)}^2.
\]
Combining the preceding inequalities yields \eqref{eq:Lrprime_bound_a1_times_Phi_times_Phi}. The proofs of the estimates \eqref{eq:Lrprime_bound_a1_times_a1_times_Phi} and \eqref{eq:Lrprime_bound_a1_times_a1_times_a1} are the same as that of \eqref{eq:Lrprime_bound_a1_times_Phi_times_Phi}. This completes Step \ref{step:Lrprime_bound_a1_times_Phi_times_Phi_and_two_similar_terms}.
\end{step}

The estimates obtained in each of the preceding steps show that the map \eqref{eq:Difference_Gradient_and_Laplacian_W1p_to_Lrprime}
has the properties asserted in the statement of Claim \ref{claim:Analytic_difference_Gradient_and_Laplacian_W1p_to_Lrprime} and this completes its proof.
\end{proof}

In particular, Claim \ref{claim:Analytic_difference_Gradient_and_Laplacian_W1p_to_Lrprime} implies that
\[
\sM(A_1+a_1,\Phi) \in L^{r'}(X; \Lambda^1 \otimes \ad P \oplus E),
\]
is a continuous cubic polynomial in $(a_1,\Phi) \in W_{A_1}^{1,p}(X; \Lambda^1 \otimes \ad P \oplus E)$, with universal coefficients depending at most on $g$ and $G$.  Because
\[
L^{r'}(X; \Lambda^1 \otimes \ad P \oplus E) \subset W_{A_1}^{-1,p}(X; \Lambda^1 \otimes \ad P \oplus E),
\]
is a continuous Sobolev embedding by our choice of $r'$, we see that
\[
\sM(A_1+a_1,\Phi) \in W_{A_1}^{-1,p}(X; \Lambda^1 \otimes \ad P \oplus E),
\]
is a cubic polynomial in $(a_1,\Phi) \in W_{A_1}^{1,p}(X; \Lambda^1 \otimes \ad P \oplus E)$, again with universal coefficients depending at most on $g$ and $G$. This completes the proof of Proposition \ref{prop:analytic_bos_gradient_map}.
\end{proof}

\subsection{Fredholm and index properties of the Hessian operator for the boson coupled Yang--Mills energy function on the Sobolev space of pairs}
\label{subsec:Fredholm_index_Hessian_boson_coupled_Yang--Mills_energy_function}
Consider the Hessian map, $\sM'(A,\Phi):\fX \to \fX^*$.

\begin{lem}[Hessian and Hessian operator for the boson coupled Yang--Mills energy function]
\label{lem:Hessian_boson_energy_function}
Assume the hypotheses of Proposition \ref{prop:analyticbos}. Then we have the schematic formula for $\sM'(A,\Phi) \in \sL(\fX,\fX^*)$ and $\sE''(A,\Phi) \in (\fX\times\fX)^*$ given by
\begin{equation}
\label{eq:Hessian_boson_energy_function}
\begin{aligned}
\sM'(A,\Phi) (a, \phi)
&=
d_A^*d_A a  + \nabla_A^*\nabla_A\phi  + F_A\times a +  \nabla_A^*(\varrho(a)\Phi)   + \Phi\times \nabla_A\phi
\\
&\quad - \rho(a)^*\nabla_A\Phi + \nabla_A\Phi\times\phi + \varrho(a)\Phi \times\Phi
\\
&\quad - (m + 2s|\Phi|^2)\phi - 4s\langle\Phi, \phi\rangle\Phi,
\quad\forall\, (a,\phi) \in \fX,
\end{aligned}
\end{equation}
where
\begin{multline}
\label{eq:Relation_Hessian_and_Hessian_operator}
\sE''(A,\Phi) (a, \phi)(b, \varphi) = ((b, \varphi), \sM'(A,\Phi) (a, \phi))_{L^2(X)},
\\
\forall\, (a,\phi), (b, \varphi) \in \fX.
\end{multline}
\end{lem}

\begin{proof}
Let $(a_i, \phi_i) \in \fX$, for $i=1,2$. We compute the terms in
\[
  \left(\sM(A+ta_2,\Phi+t\phi_2), (a_1, \phi_1)\right)_{L^2(X)}
\]
that are linear in $t$ using the expression \eqref{eq:Gradient_boson_coupled_Yang--Mills_energy_function} for the gradient. First,
\begin{align*}
(F_{A+ta_2},  d_{A+ta_2} a_1)_{L^2(X)}
&= \left(F_A + td_Aa_2 + \frac{t^2}{2}[a_2, a_2],  d_Aa_1 + t[a_2, a_1]\right)_{L^2(X)}
\\
&= (F_A, d_A a_1)_{L^2(X)} + t(F_A , [a_2, a_1])_{L^2(X)}
\\
&\quad + t(d_Aa_2, d_Aa_1)_{L^2(X)} + O(t^2).
\end{align*}
Second,
\begin{align*}
{}&(\nabla_{A+ta_2} (\Phi+t\phi_2), \nabla_{A+ta_2}\phi_1)_{L^2(X)}
\\
&\quad = \left((\nabla_A+t \varrho(a_2)) (\Phi+t\phi_2), \nabla_A \phi_1 + t \varrho(a_2)\phi_1\right)_{L^2(X)}
\\
&\quad = (\nabla_A \Phi, \nabla_A\phi_1)_{L^2(X)}
+ t(\nabla_A\phi_2, \nabla_A\phi_1)_{L^2(X)} + t(\varrho(a_2)\Phi, \nabla_A\phi_1)_{L^2(X)}
\\
&\qquad + t(\nabla_A\Phi, \varrho(a_2)\phi_1)_{L^2(X)} + O(t^2).
\end{align*}
Third,
\begin{align*}
{}& (\nabla_{A+ta_2}(\Phi+t\phi_2), \varrho(a_1)(\Phi+ t\phi_2))_{L^2(X)}
\\
&= \left((\nabla_A+t \varrho(a_2)) (\Phi+t\phi_2), \varrho(a_1)(\Phi+ t\phi_2)\right)_{L^2(X)}
\\
&=
(\nabla_A\Phi, \varrho(a_1)\Phi)_{L^2(X)} + t(\nabla_A\phi_2, \varrho(a_1)\Phi)_{L^2(X)}
+ t(\varrho(a_2)\Phi, \varrho(a_1)\Phi)_{L^2(X)}
\\
&\quad + t(\nabla_A \Phi, \varrho(a_1)\phi_2)_{L^2(X)} + O(t^2).
\end{align*}
Fourth,
\begin{align*}
(m (\Phi + t\phi_2), \phi_1)_{L^2(X)} = (m \Phi , \phi_1)_{L^2(X)} + t(m\phi_2, \phi_1)_{L^2(X)}.
\end{align*}
Fifth,
\begin{align*}
  {}&\int_X s|\Phi+t \phi_2|^2\langle \Phi+t\phi_2,\phi_1\rangle\, d\vol_g
      \\
&\quad =
\int_X s\left(|\Phi|^2+2t\Real\langle\Phi, \phi_2\rangle + t^2|\phi_2|^2\right)\langle \Phi+t\phi_2,\phi_1\rangle\, d\vol_g
\\
&\quad =
\int_X s|\Phi|^2\langle \Phi,\phi_1\rangle\, d\vol_g
\\
&\qquad + t\int_X\left( s|\Phi|^2\langle \phi_2,\phi_1\rangle + 2s \Real \langle \Phi, \phi_2\rangle \langle \Phi, \phi_1\rangle\right)\, d\vol_g + O(t^2).
\end{align*}
By subtracting $(\sM'(A,\Phi)(a_2, \phi_2),(a_1, \phi_1))_{L^2(X)}$, collecting all the first-order terms in $t$, and reversing the roles of $(a_1,\phi_1)$ and $(a_2, \phi_2)$, we see that
\begin{align*}
{}&\left(\sM'(A,\Phi) (a_1, \phi_1),(a_2, \phi_2)\right)_{L^2(X)}
  \\
  &\quad =
(d_A a_1, d_A a_2)_{L^2(X)} + 2(F_A, [a_1, a_2])_{L^2(X)}
\\
&\qquad + \Real (\nabla_A\phi_1, \nabla_A \phi_2)_{L^2(X)}
\\
&\qquad + \Real \left((\rho(a_1)\Phi,\nabla_A\phi_2)_{L^2(X)}
+ (\rho(a_2)\Phi,\nabla_A\phi_1)_{L^2(X)}\right)
\\
&\qquad + \Real(\nabla_A\Phi, \rho(a_1)\phi_2 + \rho(a_2)\phi_1)_{L^2(X)}
\\
&\qquad + \Real (\rho(a_1)\Phi, \rho(a_2)\Phi)_{L^2(X)}
\\
&\qquad - \Real\int_X \left((m + 2s|\Phi|^2)\langle\phi_1, \phi_2\rangle
+ 4s\langle\Phi, \phi_1\rangle\langle\Phi, \phi_2\rangle\right)\,d\vol_g,
\end{align*}
By now viewing $\sM'(A,\Phi)(a_1,\phi_1)$ as an element of $\fX^*$, we obtain the expression \eqref{eq:Hessian_boson_energy_function}.
\end{proof}

When $(A,\Phi)$ is a $C^\infty$ pair, we shall need to compare $\sM'(A,\Phi)$ with the $L^2$-self-adjoint, second-order partial differential operator,
\begin{equation}
\label{eq:LaplacianAPhi_to_compare_Hessian_EAPhi}
\sM'(A,\Phi) + d_{A, \Phi}d_{A, \Phi}^*: C^\infty(X; \Lambda^1 \otimes \ad P \oplus E)
\rightarrow C^\infty(X; \Lambda^1 \otimes \ad P \oplus E),
\end{equation}
in order to prove that $\sM'(A,\Phi)$ is Fredholm with index zero upon restriction to
\begin{equation}
\label{eq:Coulomb_gauge_slice_pairs}
\sX := \Ker\left(d_{A, \Phi}^*: W_{A_1}^{1,p}(X; \Lambda^1 \otimes \ad P \oplus E)
\to L^p(X;\ad P) \right).
\end{equation}
We recall from \eqref{eq:Differential_gauge_transformation_action_on_pair_at_identity} that
\[
d_{A,\Phi}\xi = (d_A\xi, -\xi\Phi), \quad \forall\, \xi \in C^\infty(X;\ad P),
\]
with $L^2$-adjoint,
\begin{equation}
\label{eq:Differential_gauge_transformation_action_on_pair_at_identity_L2adjoint}
d_{A,\Phi}^*(a, \phi) = d_A^*a - \langle\phi, \cdot\Phi\rangle^*,
\quad \forall\, (a,\phi) \in C^\infty(\Lambda^1\otimes \ad P\oplus E),
\end{equation}
for every $(a,\phi) \in C^\infty(\Lambda^1\otimes \ad P\oplus E)$, where the section $\langle\phi, \cdot\Phi\rangle^*$ of $\ad P$ is defined by
\[
(\langle\phi, \cdot\Phi\rangle^*, \xi)_{L^2(X)}
=
(\phi,\xi\Phi)_{L^2(X)}, \quad\forall\, \xi \in C^\infty(X;\ad P).
\]
According to Lemma \ref{lem:Continuous_operators_on_Lp_spaces_and_L2-orthogonal_decompositions_pairs}, the operator,
\[
d_{A, \Phi}^*: W_{A_1}^{1,p}(X; \Lambda^1 \otimes \ad P \oplus E)
\to L^p(X;\ad P),
\]
is bounded when $(A,\Phi)$ is a $W^{1,q}$ pair with $q \geq d/2$ and $p$ obeys $d/2\leq p \leq q$; therefore $\sX$ in \eqref{eq:Coulomb_gauge_slice_pairs} is a Banach space since it is a closed subspace of the Banach space $\fX = W_{A_1}^{1,p}(X; \Lambda^1 \otimes \ad P \oplus E)$.

When $(a,\phi) \in C^\infty(X; \Lambda^1 \otimes \ad P \oplus E)$, the expression \eqref{eq:Hessian_boson_energy_function} yields, after formally expanding $\nabla_A^*(\varrho(a)\Phi) = \nabla_A a\times \Phi + a\times \nabla_A\Phi + a\times \Phi$,
\begin{equation}
\begin{aligned}
\label{eq:Hessian_boson_energy_function_minus_Laplacian}
{}&\sM'(A,\Phi)(a, \phi)
\\
&\quad =
d_A^*d_A a  + \nabla_A^*\nabla_A\phi  + F_A\times a + \nabla_A a\times \Phi + a\times \nabla_A\Phi + \Phi\times \nabla_A\phi
\\
&\qquad + \varrho(a)\Phi + \nabla_A\Phi\times\phi + \varrho(a)\Phi \times\Phi
- (m + 2s|\Phi|^2)\phi - 4s\langle\Phi, \phi\rangle\Phi.
\end{aligned}
\end{equation}
To determine the Fredholm property and index of $\sM'(A,\Phi)$ upon restriction to a Coulomb-gauge slice, we shall need the following consequence of Theorem \ref{thm:Gilkey_1-4-5_Sobolev}.

\begin{prop}[Fredholm and index zero property of the augmented Hessian operator]
\label{prop:Fredholmness_and_index_Laplace_operator_on_W1p_pairs_slice}
Let $(X,g)$ be a closed, smooth Riemannian manifold of dimension $d \geq 2$, and $G$ be a compact Lie group, $P$ be a smooth principal $G$-bundle over $X$, and $E = P\times_\varrho\EE$ be a smooth Hermitian vector bundle over $X$ defined by a finite-dimensional unitary representation, $\varrho: G \to \Aut_\CC(\EE)$, and $A_1$ be a $C^\infty$ reference connection on $P$.  If $(A,\Phi)$ is a $C^\infty$ pair on $(P,E)$ and $k \in \ZZ$ is an integer and $1 < p < \infty$, then the following operator is Fredholm with index zero,
\begin{multline}
\label{eq:PiAPhi_MAPhi_PiAPhi_W^1pslice_to_Wminus1p_slice}
\sM'(A,\Phi) + d_{A,\Phi}d_{A,\Phi}^*: W_{A_1}^{k+2,p}(X; \Lambda^1 \otimes \ad P \oplus E)
\\
\to
W_{A_1}^{k,p}(X; \Lambda^1 \otimes \ad P \oplus E).
\end{multline}
\end{prop}

\begin{proof}
We can compare the principal symbol of the connection Laplacian, $\nabla_A^*\nabla_A$, and Hodge Laplacian, $\Delta_A = d_A^*d_A + d_Ad_A^*$ \eqref{eq:Hodge_Laplace_operator}, on $C^\infty(X;\Lambda^1\otimes\ad P) = \Omega^1(X;\ad P)$ using the Bochner-Weitzenb\"ock formula \cite{Bourguignon_1981, Bourguignon_1990}, \cite[Appendix C]{FU}, \cite[Appendix II]{Lawson} and \cite{Wu_1988}. From \cite[Corollary II.2]{Lawson}, one has
\begin{equation}
\label{eq:Lawson_corollary_II-2}
\Delta_A a = \nabla_A^*\nabla_Aa + \Ric_g \times a + F_A \times a, \quad\forall\, a \in \Omega^1(X;\ad P),
\end{equation}
where $\Ric_g$ denotes the Ricci curvature tensor of the Riemannian metric $g$ on the manifold $X$ of dimension $d \geq 2$ and here we employ `$\times$' to denote any universal bilinear expression with constant coefficients depending at most on the Lie group, $G$, or metric $g$. In particular, $\Delta_A$ is a second-order, elliptic partial differential operator on $C^\infty(X; \Lambda^1 \otimes \ad P)$ with $C^\infty$ coefficients and scalar principal symbol given by the Riemannian metric $g$ on $T^*M$.

From the expressions \eqref{eq:Hessian_boson_energy_function_minus_Laplacian} for $\sM'(A,\Phi)$ and \eqref{eq:Differential_gauge_transformation_action_on_pair_at_identity} for $d_{A,\Phi}$ and \eqref{eq:Differential_gauge_transformation_action_on_pair_at_identity_L2adjoint} for $d_{A,\Phi}^*$, we see that
\[
\sM'(A,\Phi) + d_{A,\Phi}d_{A,\Phi}^* = \Delta_A \oplus \nabla_A^*\nabla_A + \text{Lower-order terms}.
\]
Thus, $\sM'(A,\Phi) + d_{A,\Phi}d_{A,\Phi}^*$ is an elliptic, second-order partial differential operator on $C^\infty(X; \Lambda^1 \otimes \ad P \oplus E)$ with $C^\infty$ coefficients and principal symbol given by the Riemannian metric $g$ on $T^*M$.

Equation \eqref{eq:Hessian_boson_energy_function} for $\sM'(A,\Phi)$ implies, for $(a, \phi) \in C^\infty(X; \Lambda^1 \otimes \ad P \oplus E)$, that
\begin{align*}
{}&\left(\sM'(A,\Phi) + d_{A,\Phi}d_{A,\Phi}^* - \left(\sM'(A,\Phi) + d_{A,\Phi}d_{A,\Phi}^*\right)^*\right)(a, \phi)
\\
&\quad = F_A\times a +  \nabla_A^*(\varrho(a)\Phi)   + \Phi\times \nabla_A\phi
- \rho(a)^*\nabla_A\Phi + \nabla_A\Phi\times\phi + \varrho(a)\Phi \times\Phi
\\
&\qquad - (m + 2s|\Phi|^2)\phi - 4s\langle\Phi, \phi\rangle\Phi.
\end{align*}
Consequently, the following expression defines a first-order partial differential operator,
\[
\sM'(A,\Phi) + d_{A,\Phi}d_{A,\Phi}^* - \left(\sM'(A,\Phi) + d_{A,\Phi}d_{A,\Phi}^*\right)^*
\]
and Theorem \ref{thm:Gilkey_1-4-5_Sobolev} implies that the operator \eqref{eq:PiAPhi_MAPhi_PiAPhi_W^1pslice_to_Wminus1p_slice} is Fredholm with index zero.
\end{proof}

\subsection{Gradient map and Hessian operator for the boson coupled Yang--Mills energy on a Coulomb-gauge slice}
\label{subsec:Gradient_map_Hessian_boson_Yang--Mills_energy_function_slice}
Suppose that $(A_\infty,\Phi_\infty)$ is a $C^\infty$ pair on $(P,E)$, recall that $\fX$ is as in \eqref{eq:Banach_space_WA1p_pairs_and_dual_space} and $\sX$ is as in \eqref{eq:Coulomb_gauge_slice_pairs}, and define
\begin{align}
\label{eq:Banach_space_WA1p_pairs_and_dual_space_range}
\tilde\fX &:= W_{A_1}^{-1,p}(X; \Lambda^1 \otimes \ad P \oplus E),
\\
\label{eq:Coulomb_gauge_slice_pairs_range}
  \tilde\sX &:= \Ker\left(d_{A_\infty,\Phi_\infty}^*:W_{A_1}^{-1,p}(X; \Lambda^1 \otimes \ad P \oplus E) \right.
              \\
             &\qquad\qquad \left. \to W_{A_1}^{-2,p}(X; \Lambda^1 \otimes \ad P \oplus E)\right). \nonumber
\end{align}
Because $(W^{1,p}(X;\CC))^* = W^{-1,p'}(X;\CC)$ and $p\geq p'$ for all $p \geq 2$, we see that
\[
\tilde\sX \subset \sX^*, \quad\forall\, p \geq 2.
\]
For a $C^\infty$ pair $(A, \Phi)$ on $(P,E)$, the Hessian operator, $\sM'(A, \Phi):\fX \to \fX^*$, is defined by the schematic expression \eqref{eq:Hessian_boson_energy_function} and related to the Hessian, $\sE''(A,\Phi):\fX\times\fX \to \RR$, by
\begin{align*}
\sE''(A,\Phi)\left((a, \phi), (b,\varphi)\right) &= \langle(b,\varphi), \sM'(A,\Phi)(a,\phi)\rangle_{\fX\times\fX^*}
\\
&= \left((b,\varphi), \sM'(A,\Phi)(a,\phi)\right)_\sH, \quad\forall\, (a,\phi), (b,\varphi) \in \fX,
\end{align*}
where we define
\begin{equation}
\label{eq:Hilbert_space_pairs}
\sH: = L^2(X; \Lambda^1\otimes \ad P \oplus E).
\end{equation}
According to Theorem \ref{thm:Gilkey_1-4-5_Sobolev}, the elliptic, linear, second-order partial differential operator,
\[
d_{A,\Phi}^*d_{A,\Phi}:W_{A_1}^{k+2,p}(X;\ad P)
\to W_{A_1}^{k,p}(X;\ad P),
\]
is Fredholm for any $k\in\ZZ$ and $p \in (1,\infty)$ with kernel,
\[
K := \Ker\left( d_{A,\Phi}^*d_{A,\Phi}:C^\infty(X;\ad P) \to C^\infty(X;\ad P) \right),
\]
and range,
\[
\Ran\left( d_{A,\Phi}^*d_{A,\Phi}: W_{A_1}^{k+2,p}(X;\ad P)
\to W_{A_1}^{k,p}(X;\ad P) \right)
=
K^\perp \cap W_{A_1}^{k,p}(X;\ad P),
\]
where $\perp$ denotes $L^2$-orthogonal complement. Hence, the operator,
\[
d_{A,\Phi}^*d_{A,\Phi}:K^\perp \cap W_{A_1}^{k+2,p}(X;\ad P)
\to K^\perp \cap W_{A_1}^{k,p}(X;\ad P),
\]
is invertible, with inverse
\[
\left(d_{A,\Phi}^*d_{A,\Phi}\right)^{-1}: K^\perp \cap W_{A_1}^{k,p}(X;\ad P) \to
K^\perp \cap W_{A_1}^{k+2,p}(X;\ad P).
\]
We define the \emph{Green's operator},
\[
G_{A,\Phi}: W_{A_1}^{k,p}(X;\ad P) \to W_{A_1}^{k+2,p}(X;\ad P),
\]
for the Laplacian, $d_{A,\Phi}^*d_{A,\Phi}$, by setting
\[
G_{A,\Phi}\xi
:=
\begin{cases}
(d_{A,\Phi}^*d_{A,\Phi})^{-1}\xi,
&\forall\, \xi \in K^\perp \cap W_{A_1}^{k,p}(X;\ad P),
\\
0, &\forall\, \xi \in K.
\end{cases}
\]
For any $k \in \ZZ$ and $p \in (1,\infty)$, we now let
\begin{multline}
\label{eq:L2-orthogonal_projection_onto_slice}
\Pi_{A_\infty,\Phi_\infty}:W_{A_1}^{k,p}(X;\Lambda^1\otimes \ad P \oplus E)
\\
\to \Ker d_{A_\infty,\Phi_\infty}^* \cap W_{A_1}^{k,p}(X;\Lambda^1\otimes \ad P \oplus E)
\end{multline}
denote the $L^2$-orthogonal projection onto the slice through $(A_\infty,\Phi_\infty)$. Because
\begin{align*}
{}&W_{A_1}^{k,p}(X;\Lambda^1\otimes \ad P \oplus E)
\\
&\quad = \Ker d_{A_\infty,\Phi_\infty}^* \cap W_{A_1}^{k,p}(X;\Lambda^1\otimes \ad P \oplus E)
\\
&\qquad \oplus \Ran\left(d_{A_\infty,\Phi_\infty}:W_{A_1}^{k+1,p}(X;\ad P \oplus E)
\to W_{A_1}^{k,p}(X;\Lambda^1\otimes \ad P \oplus E) \right)
\end{align*}
is an $L^2$-orthogonal direct sum, we see that
\begin{equation}
\label{eq:L2-orthogonal_projection_onto_slice_intermsof_Greens_operator}
\Pi_{A_\infty,\Phi_\infty} = \id - d_{A_\infty,\Phi_\infty}G_{A_\infty,\Phi_\infty}d_{A_\infty,\Phi_\infty}^*.
\end{equation}
Because of the invariance of the energy function, $\sE: \sP(P,E) \to\RR$, in \eqref{eq:Boson_Yang--Mills_energy_function} with respect to the action of $\Aut(P)$ on $C^\infty$ pairs, $\sP(P,E) = \sA(P) \times C^\infty(X;E)$, we have the identity
\begin{equation}
\label{eq:Gauge_invariance_energy_function}
\sE(u(A,\Phi)) = \sE(A,\Phi), \quad\forall\, u \in \Aut(P) \text{ and } (A,\Phi) \in \sP(P,E).
\end{equation}
Note that if $u(t) \in \Aut(P)$ is a family of gauge transformations depending smoothly on $t \in \RR$ such that $u(0) = \id_P$, then the identity \eqref{eq:Gauge_invariance_energy_function} implies that
\[
\sE'(A,\Phi)(d_{A,\Phi}\xi) = \left.\frac{d}{dt}\sE(u(t)(A,\Phi))\right|_{t=0} = 0, \quad\forall\, (A,\Phi) \in \sP(P,E),
\]
where
\[
d_{A,\Phi}\xi = \left.\frac{d}{dt} u(t)(A,\Phi)\right|_{t=0} \in C^\infty(X;\Lambda^1\otimes\ad P\oplus E),
\]
with
\[
  \xi  = \dot{u}(0)\in C^\infty(X;\ad P) = T_{\id_P}\Aut(P) = T_{\id_P}C^\infty(X;\Ad P).
\]
Before considering higher-order derivatives of the identity \eqref{eq:Gauge_invariance_energy_function} with respect to $u \in \Aut(P)$, we digress to discuss the \emph{Chain Rule} for maps of Banach spaces.

If $F:\cY \to \cZ$ and $G:\cX\to\cY$ are $C^\infty$ maps of Banach spaces, then the Chain Rule gives, for all $x \in \cX$,
\begin{align}
\label{eq:Chain_rule_gradients}
(F\circ G)'(x) &= F'(G(x)) \circ G'(x) \in \sL(\cX,\cZ),
\\
\label{eq:Chain_rule_hessians}
(F\circ G)''(x) &= F''(G(x)) \circ G'(x)^2 + F'(G(x))\circ G''(x) \in \sL(\cX\times\cX,\cZ),
\end{align}
More explicitly, if $u,v \in \cX$, then
\begin{align*}
(F\circ G)'(x)(u) &= F'(G(x))(G'(x)(u)),
\\
(F\circ G)''(x)(u,v) &= F''(G(x))(G'(x)(u),G'(x)(v)) + F'(G(x))(G''(x)(u,v)).
\end{align*}
The expression for the Hessian of the composition simplifies when $F'(y) = 0$ at $y = G(x)$ to give
\begin{equation}
\label{eq:Chain_rule_Hessians_critical_point}
(F\circ G)''(x) = F''(G(x)) \circ G'(x)^2 \in \sL(\cX\times\cX,\cZ).
\end{equation}
This ends our digression on the Chain Rule for $C^\infty$ maps of Banach spaces.

By computing the first-order differential with respect to $u$ of the expression \eqref{eq:Gauge_invariance_energy_function} at $\id_P \in \Aut(P)$ in directions $\xi \in T_{\id_P}\Aut(P) = C^\infty(X;\ad P)$ and recalling the definition \eqref{eq:Pair_Coulomb_gauge_relative_to_reference_pair} of $d_{A,\Phi}\xi$, we see that the first-order differential of $\sE$ and its gradient map obey
\begin{equation}
\label{eq:Gradient_operator_zero_L2-orthogonal_complement_slice}
\sE'(A,\Phi)(d_{A,\Phi}\xi) = 0 = (d_{A,\Phi}\xi, \sM(A,\Phi))_{L^2(X)}, \quad\forall\, \xi \in C^\infty(X;\ad P),
\end{equation}
and thus
\[
d_{A,\Phi}^*\sM(A,\Phi) = 0, \quad\forall\, (A,\Phi) \in \sP(P,E).
\]
Next, we observe that \eqref{eq:Gauge_invariance_energy_function} implies that the differential of the energy function satisfies
\begin{multline}
\label{eq:Gauge_invariance_energy_function_gradient}
\sE'(u(A,\Phi))(u(a,\phi)) = \sE'(A,\Phi)(a,\phi),
\\
\forall\, u \in \Aut(P) \text{ and } (A,\Phi) \in \sP(P,E)
\\
\text{and } \forall\, (a,\phi) \in C^\infty(X; \Lambda^1 \otimes \ad P \oplus E).
\end{multline}
By computing the first-order differentials with respect to $u$ of the expression \eqref{eq:Gauge_invariance_energy_function_gradient} at $\id_P \in \Aut(P)$ in directions $\xi \in T_{\id_P}\Aut(P) = C^\infty(X;\ad P)$ and recalling the definition \eqref{eq:Pair_Coulomb_gauge_relative_to_reference_pair} of $d_{A,\Phi}\xi$, we see that the gradient map and second-order differential obey
\[
\sE''(A,\Phi)\left(d_{A,\Phi}\xi, (a, \phi)\right) + \sE'(A,\Phi)(d_{A,\Phi}\xi) = 0,
\]
and thus, by \eqref{eq:Gradient_operator_zero_L2-orthogonal_complement_slice} and the fact that the Hessian $\sE''(A,\Phi)$ is a symmetric operator,
\begin{multline*}
\sE''(A,\Phi)\left(d_{A,\Phi}\xi, (a, \phi)\right) = 0 = \sE''(A,\Phi)\left((a, \phi),d_{A,\Phi}\xi\right),
\\
\forall\, (a,\phi) \in C^\infty(X; \Lambda^1 \otimes \ad P \oplus E).
\end{multline*}
Hence, by the relation \eqref{eq:Relation_Hessian_and_Hessian_operator} between $\sE''(A,\Phi)$ and $\sM'(A,\Phi)$,
\begin{multline*}
\left(d_{A,\Phi}\xi, \sM'(A,\Phi)(a, \phi)\right)_{L^2(X)} = 0 = \left((a, \phi), \sM'(A,\Phi)d_{A,\Phi}\xi\right)_{L^2(X)},
\\
\forall\, (a,\phi) \in C^\infty(X; \Lambda^1 \otimes \ad P \oplus E).
\end{multline*}
Consequently, the Hessian operator satisfies
\begin{subequations}
\label{eq:Hessian_operator_slice_domain_range}
\begin{align}
\label{eq:Hessian_operator_range_in_slice}
d_{A,\Phi}^*\sM'(A,\Phi)(a, \phi) &= 0, \quad\forall\, (a,\phi) \in C^\infty(X; \Lambda^1 \otimes \ad P \oplus E).
\\
\label{eq:Hessian_operator_zero_on_slice_orthocomplement}
\sM'(A,\Phi)d_{A,\Phi}\xi &= 0, \quad \forall\, \xi \in C^\infty(X;\ad P).
\end{align}
\end{subequations}
Thus, for any $k\in\ZZ$ and $p\in (1,\infty)$ when $(A_\infty,\Phi_\infty)$ is a $C^\infty$ pair, the Hessian operator,
\[
\sM'(A_\infty,\Phi_\infty): W_{A_1}^{k+2,p}(X; \Lambda^1 \otimes \ad P \oplus E)
\to W_{A_1}^{k,p}(X; \Lambda^1 \otimes \ad P \oplus E),
\]
naturally restricts to the Coulomb-gauge slice domain and range defined by $(A_\infty,\Phi_\infty)$,
\begin{multline}
\label{eq:Boson_Hessian_operator_naturally_restricts_to_slice}
\sM'(A_\infty,\Phi_\infty):\Ker d_{A_\infty,\Phi_\infty}^* \cap W_{A_1}^{k+2,p}(X; \Lambda^1 \otimes \ad P \oplus E)
\\
\to \Ker d_{A_\infty,\Phi_\infty}^* \cap W_{A_1}^{k,p}(X; \Lambda^1 \otimes \ad P \oplus E),
\end{multline}
when the Hessian operator is also computed at the pair $(A_\infty,\Phi_\infty)$. The gradient map,
\[
\sM: (A_\infty,\Phi_\infty)+W_{A_1}^{k+2,p}(X; \Lambda^1 \otimes \ad P \oplus E)
\to W_{A_1}^{k,p}(X; \Lambda^1 \otimes \ad P \oplus E),
\]
restricts to the slice domain and range with the $L^2$-orthogonal projection, $\Pi_{A_\infty,\Phi_\infty}$,
\begin{multline}
\label{eq:Boson_gradient_operator_naturally_restricts_to_slice}
\hat\sM \equiv \Pi_{A_\infty,\Phi_\infty}\sM:
\\
(A_\infty,\Phi_\infty)+\Ker d_{A_\infty,\Phi_\infty}^* \cap W_{A_1}^{k+2,p}(X; \Lambda^1 \otimes \ad P \oplus E)
\\
\to \Ker d_{A_\infty,\Phi_\infty}^* \cap W_{A_1}^{k,p}(X; \Lambda^1 \otimes \ad P \oplus E).
\end{multline}
The definition \eqref{eq:L2-orthogonal_projection_onto_slice} of the $L^2$-orthogonal projection, $\Pi_{A,\Phi}$, and the relation \eqref{eq:Gradient_operator_zero_L2-orthogonal_complement_slice} yield
\[
\Pi_{A,\Phi}\sM(A,\Phi) = \sM(A,\Phi), \quad\forall\, (A,\Phi) \in \sP(P,E).
\]
In the definition \eqref{eq:Boson_gradient_operator_naturally_restricts_to_slice}, we suppress dependence on the pair $(A_\infty,\Phi_\infty)$ in the choice of Coulomb-gauge slice from the notation for $\hat\sM$, which we may regard as a lift to a coordinate chart on an open neighborhood of the point $[A,\Phi] \in \sC(P,E) := \sP(P,E)/\Aut(P)$ of the gradient map on the quotient.

We now apply Proposition \ref{prop:Fredholmness_and_index_Laplace_operator_on_W1p_pairs_slice} to prove the Fredholm and index zero properties of the Hessian operator.

\begin{prop}[Fredholm and index zero properties of the Hessian operator for the boson coupled Yang--Mills energy function on a Coulomb-gauge slice]
\label{prop:fredbos}
Let $(X,g)$ be a closed, smooth Riemannian manifold of dimension $d \geq 2$, and $G$ be a compact Lie group, $P$ be a smooth principal $G$-bundle over $X$, and $E = P\times_\varrho\EE$ be a smooth Hermitian vector bundle over $X$ defined by a finite-dimensional unitary representation, $\varrho: G \to \Aut_\CC(\EE)$, and $A_1$ be a $C^\infty$ reference connection on $P$, and $m, s \in C^\infty(X)$. If $(A,\Phi)$ is a $C^\infty$ pair on $(P,E)$ and $k\in\ZZ$ is an integer and $1 < p < \infty$, then the following operator is Fredholm with index zero,
\begin{multline*}
\sM'(A,\Phi): \Ker d_{A,\Phi}^* \cap W_{A_1}^{k+2,p}(X; \Lambda^1 \otimes \ad P \oplus E)
\\
\to \Ker d_{A,\Phi}^* \cap W_{A_1}^{k,p}(X; \Lambda^1 \otimes \ad P \oplus E).
\end{multline*}
\end{prop}

\begin{proof}
To reduce notational clutter, we abbreviate the domain and range Sobolev spaces by
\[
W^{k+2,p} = W_{A_1}^{k+2,p}(X; \Lambda^1 \otimes \ad P \oplus E)
\quad\text{and}\quad
W^{k,p} = W_{A_1}^{k,p}(X; \Lambda^1 \otimes \ad P \oplus E).
\]
We shall adapt the argument of R\r{a}de \cite[p. 148]{Rade_1992}. Observe that $\sM'(A,\Phi) = 0$ on $\Ran d_{A,\Phi}\cap W^{k+2,p}$ and $d_{A,\Phi}^*\sM'(A,\Phi) = 0$ on $W^{k+2,p}$ by \eqref{eq:Hessian_operator_slice_domain_range} and so we have a well-defined restriction,
\[
\sM'(A,\Phi): \Ker d_{A,\Phi}^*\cap W^{k+2,p} \to \Ker d_{A,\Phi}^*\cap W^{k,p}.
\]
Clearly, $d_{A,\Phi}d_{A,\Phi}^* = 0$ on $\Ker d_{A,\Phi}^*\cap W^{k+2,p}$ and
\[
\Ran(d_{A,\Phi}d_{A,\Phi}^*: W^{k+2,p} \to W^{k,p}) \subset \Ran d_{A,\Phi}\cap W^{k,p}.
\]
Hence, we obtain a decomposition of the operator,
\begin{equation}
\label{eq:Ellipticized_Hessian_operator_Wk+2p}
\sM'(A,\Phi) + d_{A,\Phi}d_{A,\Phi}^*: W^{k+2,p} \to W^{k,p},
\end{equation}
as a direct sum,
\begin{multline}
\label{eq:Ellipticized_Hessian_operator_Wk+2p_matrix}
\sM'(A,\Phi) + d_{A,\Phi}d_{A,\Phi}^*
=
\begin{pmatrix}\sM'(A,\Phi) & 0 \\ 0 & d_{A,\Phi}d_{A,\Phi}^* \end{pmatrix}:
\\
\left(\Ker d_{A,\Phi}^* \oplus \Ran d_{A,\Phi}\right)\cap W^{k+2,p}
\to
\left(\Ker d_{A,\Phi}^* \oplus \Ran d_{A,\Phi}\right)\cap W^{k,p}.
\end{multline}
We now claim that the following operator is invertible,
\begin{equation}
\label{eq:Half_Laplacian_on_range_dAPhi}
d_{A,\Phi}d_{A,\Phi}^*: \Ran d_{A,\Phi}\cap W^{k+2,p} \to \Ran d_{A,\Phi}\cap W^{k,p}.
\end{equation}
Indeed, the operator \eqref{eq:Half_Laplacian_on_range_dAPhi} is clearly bounded. For injectivity, if $\eta = d_{A,\Phi}\xi \in \Ran d_{A,\Phi}\cap W^{k+2,p}$ for $\xi \in W^{k+3,p}$ and $d_{A,\Phi}d_{A,\Phi}^*d_{A,\Phi}\xi = 0$, then $(d_{A,\Phi}^*d_{A,\Phi})^2\xi = 0$. We may assume without loss of generality that $\xi \perp \Ker d_{A,\Phi}^*d_{A,\Phi}$ and because $d_{A,\Phi}^*d_{A,\Phi}:W^{k+2,p} \to W^{k,p}$ is Fredholm with index zero by Theorem \ref{thm:Gilkey_1-4-5_Sobolev}, the following operator is invertible,
\[
d_{A,\Phi}^*d_{A,\Phi}:(\Ker d_{A,\Phi}^*d_{A,\Phi})^\perp \cap W^{k+2,p} \to (\Ker d_{A,\Phi}^*d_{A,\Phi})^\perp \cap W^{k,p}.
\]
Thus, $\xi = 0$ and the operator \eqref{eq:Half_Laplacian_on_range_dAPhi} is injective. For surjectivity, suppose $\chi = d_{A,\Phi}\zeta \in \Ran d_{A,\Phi}\cap W^{k,p}$ for $\zeta \in W^{k+1,p}$. We may again assume without loss of generality that $\zeta \perp \Ker d_{A,\Phi}^*d_{A,\Phi}$. If $\chi\perp d_{A,\Phi}d_{A,\Phi}^*d_{A,\Phi}\xi$ for all $\xi \in W^{k+3,p}$, then $(d_{A,\Phi}^*d_{A,\Phi})^2\zeta = 0$, and we again find that $\zeta = 0$ and so $\chi=0$. Since $d_{A,\Phi}d_{A,\Phi}^*: W^{k+2,p} \to W^{k,p}$ has closed range (because it is a Fredholm operator), this implies that the operator \eqref{eq:Half_Laplacian_on_range_dAPhi} is surjective and thus invertible by the Open Mapping Theorem.

According to Proposition \ref{prop:Fredholmness_and_index_Laplace_operator_on_W1p_pairs_slice}, the operator \eqref{eq:Ellipticized_Hessian_operator_Wk+2p} is Fredholm. Consequently,
\[
  \sM'(A,\Phi): \Ker d_{A,\Phi}^*\cap W^{k+2,p} \to \Ker d_{A,\Phi}^*\cap W^{k,p}
\]
is Fredholm by virtue of the direct sum decomposition \eqref{eq:Ellipticized_Hessian_operator_Wk+2p_matrix} and invertibility of the operator \eqref{eq:Half_Laplacian_on_range_dAPhi}. We compute indices,
\begin{align*}
{}& \Ind\left\{\sM'(A,\Phi) + d_{A,\Phi}d_{A,\Phi}^*:W^{k+2,p} \to W^{k,p}\right\}
\\
&\quad = \Ind\left\{\begin{pmatrix}\sM'(A,\Phi) & 0 \\ 0 & d_{A,\Phi}d_{A,\Phi}^* \end{pmatrix}:
\left(\Ker d_{A,\Phi}^* \oplus \Ran d_{A,\Phi}\right) \cap W^{k+2,p}\right.
\\
&\qquad \to \left.\left(\Ker d_{A,\Phi}^* \oplus \Ran d_{A,\Phi}\right)\cap W^{k,p}\right\}
\\
&\quad = \Ind\left\{\sM'(A,\Phi): \Ker d_{A,\Phi}^*\cap W^{k+2,p} \to \Ker d_{A,\Phi}^*\cap W^{k,p}\right\}
\\
&\qquad + \Ind\left\{d_{A,\Phi}d_{A,\Phi}^*: \Ran d_{A,\Phi}\cap W^{k+2,p} \to \Ran d_{A,\Phi}\cap W^{k,p}\right\}.
\end{align*}
Therefore, because
\[
  \Ind\{d_{A,\Phi}d_{A,\Phi}^*: \Ran d_{A,\Phi}\cap W^{k+2,p} \to \Ran d_{A,\Phi}\cap W^{k,p}\} = 0,
\]
we have
\begin{multline*}
\Ind\left\{\sM'(A,\Phi): \Ker d_{A,\Phi}^*\cap W^{k+2,p} \to \Ker d_{A,\Phi}^*\cap W^{k,p}\right\}
\\
= \Ind\left\{\sM'(A,\Phi) + d_{A,\Phi}d_{A,\Phi}^*:W^{k+2,p} \to W^{k,p}\right\}.
\end{multline*}
But the latter index is zero by Proposition \ref{prop:Fredholmness_and_index_Laplace_operator_on_W1p_pairs_slice} and this completes the proof of Proposition \ref{prop:fredbos}.
\end{proof}

\subsection{Analyticity of the gradient map for the boson coupled Yang--Mills energy function on a Coulomb-gauge slice}
\label{subsec:Analyticity_of_gradient_for_boson_coupled_Yang--Mills_energy_on_Coulomb_slice}
Suppose $(A_\infty,\Phi_\infty)$ and $(A,\Phi)$ are $C^\infty$ pairs on $(P,E)$ and recall from \eqref{eq:Definition_gradient_boson_coupled_Yang--Mills_energy_function} that the first-order differential and gradient map of the boson coupled Yang--Mills energy function \eqref{eq:Boson_Yang--Mills_energy_function} are related by
\[
\sE'(A,\Phi)(a,\phi) = ((a,\phi), \sM(A,\Phi))_{L^2(X)}, \quad\forall\, (a,\phi) \in C^\infty(X; \Lambda^1 \otimes \ad P \oplus E).
\]
If we now restrict $(a,\phi)$ to be a pair in $\Ker d_{A_\infty,\Phi_\infty}^* \cap C^\infty(X; \Lambda^1 \otimes \ad P \oplus E)$, then the preceding relation yields
\begin{multline*}
\sE'(A,\Phi)(a,\phi) = ((a,\phi), \Pi_{A_\infty,\Phi_\infty}\sM(A,\Phi))_{L^2(X)}
 = ((a,\phi), \hat\sM(A,\Phi))_{L^2(X)},
\\
\forall\, (a,\phi) \in \Ker d_{A_\infty,\Phi_\infty}^* \cap C^\infty(X; \Lambda^1 \otimes \ad P \oplus E),
\end{multline*}
where $\Pi_{A_\infty,\Phi_\infty}$ is the $L^2$-orthogonal projection \eqref{eq:L2-orthogonal_projection_onto_slice} onto the Coulomb-gauge slice through $(A_\infty,\Phi_\infty)$ and we appeal to the definition \eqref{eq:Boson_gradient_operator_naturally_restricts_to_slice} of $\hat\sM$. Consequently,
\begin{multline}
\label{eq:Gradient_map_boson_energy_function_slice}
\hat\sM = \Pi_{A_\infty,\Phi_\infty}\sM:(A_\infty,\Phi_\infty)+\Ker d_{A_\infty,\Phi_\infty}^* \cap W_{A_1}^{1,p}(X; \Lambda^1 \otimes \ad P \oplus E)
\\
\to \Ker d_{A_\infty,\Phi_\infty}^* \cap W_{A_1}^{-1,p}(X; \Lambda^1 \otimes \ad P \oplus E)
\end{multline}
is the gradient map for the restriction,
\begin{equation}
\label{eq:Boson_Yang--Mills_energy_function_slice}
\sE: (A_\infty,\Phi_\infty)+\Ker d_{A_\infty,\Phi_\infty}^* \cap W_{A_1}^{1,p}(X; \Lambda^1 \otimes \ad P \oplus E) \to \RR,
\end{equation}
of the boson coupled Yang--Mills energy function \eqref{eq:Boson_Yang--Mills_energy_function} to the Coulomb-gauge slice through $(A_\infty,\Phi_\infty)$. Note that the relations \eqref{eq:Hessian_operator_slice_domain_range} imply that the Hessian operator (namely, the derivative of the gradient map, $\hat\sM$) at $(A_\infty,\Phi_\infty)$ simplifies to
\begin{multline*}
\hat\sM'(A_\infty,\Phi_\infty)
=
(\Pi_{A_\infty,\Phi_\infty}\sM)'(A_\infty,\Phi_\infty)
\\
=
\sM'(A_\infty,\Phi_\infty):\Ker d_{A_\infty,\Phi_\infty}^* \cap W_{A_1}^{1,p}(X; \Lambda^1 \otimes \ad P \oplus E)
\\
\to \Ker d_{A_\infty,\Phi_\infty}^* \cap W_{A_1}^{-1,p}(X; \Lambda^1 \otimes \ad P \oplus E).
\end{multline*}
Proposition \ref{prop:analytic_bos_gradient_map} yields the

\begin{cor}[Analyticity of the gradient map for the boson coupled Yang--Mills energy function on a Coulomb-gauge slice]
\label{cor:Analyticity_Hessian_boson_energy_function_slice}
Assume the hypotheses of Proposition \ref{prop:analytic_bos_gradient_map} and let $(A_\infty,\Phi_\infty)$ be a $C^\infty$ pair on $(P,E)$. Then the following map is analytic,
\begin{multline*}
\Ker d_{A_\infty,\Phi_\infty}^* \cap W^{1, p} \ni (a_1,\phi_1)
\\
\mapsto \hat\sM(A_\infty+a_1,\Phi_\infty+\phi_1) = \Pi_{A_\infty,\Phi_\infty}\sM(A_\infty+a_1,\Phi_\infty+\phi_1)
\\
\in \Ker d_{A_\infty,\Phi_\infty}^* \cap W^{-1, p},
\end{multline*}
where we abbreviate
\[
W^{\pm 1, p} = W_{A_1}^{\pm 1,p}(X; \Lambda^1 \otimes \ad P \oplus E).
\]
\end{cor}

\begin{proof}
The conclusion follows from Proposition \ref{prop:analytic_bos_gradient_map} and the fact that the operators, $\Pi_{A_\infty,\Phi_\infty}$, define continuous projections.
\end{proof}

\subsection{Estimates for gauge transformations intertwining two pairs}
\label{subsec:Estimates_gauge_transformations_intertwining_two_pairs}
We shall require the

\begin{lem}[Estimate for the action of a $W^{2,q}$ gauge transformation intertwining two $W^{1,q}$ pairs]
\label{lem:gauged}
Let $(X,g)$ be a closed, smooth Riemannian manifold of dimension $d \geq 2$, and $G$ be a compact Lie group, $P$ be a smooth principal $G$-bundle over $X$, and $E = P\times_\varrho\EE$ be a smooth Hermitian vector bundle over $X$ defined by a finite-dimensional unitary representation, $\varrho: G \to \Aut_\CC(\EE)$, and $A_1$ be a $C^\infty$ reference connection on $P$, $q>d/2$ and $p$ obeys $d/2\leq p\leq q$, then there is a constant $C = C(g,G,p) \in [1,\infty)$ with the following significance. If $(A, \Phi)$ and $(A', \Phi')$ are $W^{1,q}$ pairs on $(P,E)$ and $u\in \Aut^{2,q}(P)$, then
\begin{align*}
\|u(A,\Phi) - u(A', \Phi')\|_{W_{A_1}^{1,p}(X)}
&\leq C\left(1 + \|u\|_{W_{A_1}^{2,p}(X)}\right) \|(A,\Phi) - (A', \Phi')\|_{W_{A_1}^{1,p}(X)},
\\
\|(A,\Phi) - (A', \Phi')\|_{W_{A_1}^{1,p}(X)}
&\leq C\left(1 + \|u\|_{W_{A_1}^{2,p}(X)}\right) \|u(A,\Phi) - u(A', \Phi')\|_{W_{A_1}^{1,p}(X)}.
\end{align*}
If in addition $p > 3/2$ if $d=3$ and $p \geq 4/3$ if $d=2$ and $p' \in [1,\infty)$ is the dual H\"older exponent defined by $1/p+1/p'=1$, and $(a,\phi) \in W^{1,p'}_{A_1}(X;\Lambda^1\otimes\ad P)$, then
\begin{align*}
\|u(a,\phi)\|_{W_{A_1}^{1,p'}(X)}
&\leq C\left(1 + \|u\|_{W_{A_1}^{2,p}(X)}\right) \|(a,\phi)\|_{W_{A_1}^{1,p'}(X)},
\\
\|(a,\phi)\|_{W_{A_1}^{1,p'}(X)}
&\leq C\left(1 + \|u\|_{W_{A_1}^{2,p}(X)}\right) \|u(a,\phi)\|_{W_{A_1}^{1,p'}(X)}.
\end{align*}
\end{lem}

\begin{proof}
Recall that $u(A)-A_1 = u^{-1}(A-A_1)u + u^{-1}d_{A_1}u$ by \eqref{eq:Feehan_2001_6-13} and similarly for $A'$, so
\[
u(A) - u(A') = u^{-1}(A-A_1)u - u^{-1}(A'-A_1)u = u^{-1}(A-A')u,
\]
and thus,
\[
u(A,\Phi) - u(A', \Phi') = (u^{-1}(A-A')u, u(\Phi - \Phi')).
\]
Therefore, writing $a:= A-A' \in W^{1,q}_{A_1}(X;\Lambda^1\otimes\ad P)$ for convenience,
\[
\nabla_{A_1}(u(A) - u(A')) = -u^{-1}(\nabla_{A_1}u)u^{-1} a u + u^{-1}(\nabla_{A_1}a)u + u^{-1}a (\nabla_{A_1}u).
\]
By taking $L^p$ norms and using the pointwise bound $|u|\leq 1$, the Sobolev embedding $W^{1,p}(X) \subset L^{2p}(X)$ (valid for $p \geq d/2$), and the Kato Inequality \cite[Equation (6.20)]{FU}, we obtain
\begin{align*}
\|\nabla_{A_1}(u(A - A'))\|_{L^p(X)} &\leq 2\|\nabla_{A_1}u\|_{L^{2p}(X)} \|a\|_{L^{2p}(X)} + \|\nabla_{A_1}a\|_{L^p(X)}
\\
&\leq C\|\nabla_{A_1}u\|_{W_{A_1}^{1,p}(X)} \|a\|_{W_{A_1}^{1,p}(X)} + \|\nabla_{A_1}a\|_{L^p(X)}.
\end{align*}
Similarly, $\nabla_{A_1}(u(\Phi - \Phi')) = (\nabla_{A_1}u)(\Phi - \Phi') + u(\nabla_{A_1}(\Phi - \Phi'))$ and
\begin{align*}
{}&\|\nabla_{A_1}(u(\Phi - \Phi'))\|_{L^p(X)}
  \\
  &\quad \leq
\|\nabla_{A_1}u\|_{L^{2p}(X)}\|\Phi - \Phi'\|_{L^{2p}(X)} + \|\nabla_{A_1}(\Phi - \Phi')\|_{L^p(X)}
\\
&\quad \leq
\|\nabla_{A_1}u\|_{W_{A_1}^{1,p}(X)}\|\Phi - \Phi'\|_{W_{A_1}^{1,p}(X)} + \|\nabla_{A_1}(\Phi - \Phi')\|_{L^p(X)}.
\end{align*}
By combining the preceding estimates, we obtain the first inequality; the second inequality is proved by a symmetric argument.

For the third inequality, we note that $u(a,\phi) = (u(a),u(\phi)) = (u^{-1}au, u^{-1}\phi)$ by our convention \eqref{eq:Action_gauge_transformation_on_pairs} and use the identity,
\[
\nabla_{A_1}(u(a)) = -u^{-1}(\nabla_{A_1}u)u^{-1} a u + u^{-1}(\nabla_{A_1}a)u + u^{-1}a (\nabla_{A_1}u).
\]
We make the

\begin{claim}
\label{claim:Continuous_Sobolev_multiplication_W1pprime_times_W1p_to_Lpprime}
Let $(X,g)$ be a closed, smooth Riemannian manifold of dimension $d \geq 2$ and $p \in [d/2,\infty]$, with $p > 3/2$ if $d=3$ and $p \geq 4/3$ if $d=2$. If $p' \in [1,\infty)$ is the dual H\"older exponent defined by $1/p+1/p'=1$, then there is a continuous Sobolev multiplication map,
\begin{equation}
\label{eq:Continuous_Sobolev_multiplication_W1pprime_times_W1p_to_Lpprime}
W^{1,p'}(X) \times  W^{1,p}(X) \to L^{p'}(X).
\end{equation}
\end{claim}

Given Claim \ref{claim:Continuous_Sobolev_multiplication_W1pprime_times_W1p_to_Lpprime}, we have
\[
\|\nabla_{A_1}(u(a))\|_{L^{p'}(X)}
\leq
z\|\nabla_{A_1}u\|_{W_{A_1}^{1,p}(X)}\|a\|_{W_{A_1}^{1,p'}(X)} + \|\nabla_{A_1}a\|_{L^{p'}(X)},
\]
for $z = z(g,p) \in [1,\infty)$, while
\[
\|u(a)\|_{L^{p'}(X)}
=
\|a\|_{L^{p'}(X)}.
\]
Thus,
\[
\|u(a)\|_{W_{A_1}^{1,p'}(X)} \leq z\left(1 + \|\nabla_{A_1}u\|_{W_{A_1}^{1,p}(X)}\right)\|a\|_{W_{A_1}^{1,p'}(X)},
\]
with the analogous estimate for $\|u(\phi)\|_{W_{A_1}^{1,p'}(X)}$. This yields the third inequality and symmetry yields the fourth inequality.

\begin{proof}[Proof of Claim \ref{claim:Continuous_Sobolev_multiplication_W1pprime_times_W1p_to_Lpprime}]
A multiplication, $L^r(X)\times L^s(X) \to L^{p'}(X)$, is continuous provided $r,s \in [p',\infty]$ obey $1/r + 1/s \leq 1/p'$. If these exponents $r,s$ also yield continuous Sobolev embeddings,
\[
W^{1,p'}(X) \subset L^r(X) \quad\text{and}\quad W^{1,p}(X) \subset L^s(X),
\]
then we obtain the desired continuous Sobolev multiplication map \eqref{eq:Continuous_Sobolev_multiplication_W1pprime_times_W1p_to_Lpprime}. To confirm the existence of suitable exponents, $r,s$, we shall consider the cases $p<d$, $p=d$, and $p>d$ separately.

\setcounter{case}{0}
\begin{case}[$p<d$]
\label{case:Sobolev_mult_p_lessthan_d}
Choose $s = p^* = dp/(d-p)$ to give the continuous Sobolev embedding, $W^{1,p}(X) \subset L^{p^*}(X)$, provided by \cite[Theorem 4.12]{AdamsFournier}. We required that $s \in [p',\infty]$, so $p$ must obey $p^* \geq p'$, that is
\[
dp/(d-p) \geq p/(p-1),
\]
or $p>1$ and $d(p-1) \geq d-p$ or $p(d+1) \geq 2d$ and so we require that $p \geq 2d/(d+1)$. Note that $d/2 \geq 2d/(d+1) \iff d+1 \geq 4$, that is, $d \geq 3$, and $d/2 > 2d/(d+1) \iff d \geq 4$. (This is why for $d=2,3$ we augment the hypothesis $p\geq d/2$ in Claim \ref{claim:Continuous_Sobolev_multiplication_W1pprime_times_W1p_to_Lpprime} with additional conditions in certain subcases as noted below.) Because $p<d$, we also have $p' = p/(p-1) > d/(d-1)$. We consider three subcases, depending on whether
\begin{inparaenum}[\itshape a\upshape)]
\item
\label{subcase:Sobolev_mult_p_lessthan_d_and_pprime_lessthan_d}
$p'<d$,
\item
\label{subcase:Sobolev_mult_p_lessthan_d_and_pprime_equals_d}
$p'=d$, or
\item
\label{subcase:Sobolev_mult_p_lessthan_d_and_pprime_greaterthan_d}
$p'>d$.
\end{inparaenum}

Consider Subcase (\ref{case:Sobolev_mult_p_lessthan_d}\ref{subcase:Sobolev_mult_p_lessthan_d_and_pprime_greaterthan_d}), so $p<d$ and $p' > d$. Then we have a continuous Sobolev embedding, $W^{1,p'}(X) \subset L^\infty(X)$, by \cite[Theorem 4.12]{AdamsFournier} and because $p^* \geq p'$ when $p \geq 2d/(d+1)$ for $d\geq 3$ and $p\geq 4/3$ for $d=2$, we may choose any $r \in (1,\infty]$ large enough that $1/r + 1/p^* \leq 1/p'$.

Consider Subcase (\ref{case:Sobolev_mult_p_lessthan_d}\ref{subcase:Sobolev_mult_p_lessthan_d_and_pprime_equals_d}), so $p<d$ and $p'=d$. Then we have a continuous Sobolev embedding, $W^{1,p'}(X) \subset L^r(X)$, by \cite[Theorem 4.12]{AdamsFournier} for any $r\in [1,\infty)$ and because $p^* > p'$ when $p > 2d/(d+1)$ for $d\geq 4$ and $p > 3/2$ for $d=3$, we may choose $r\in (1,\infty)$ large enough that $1/r + 1/p^* \leq 1/p'$. (The condition $p'>d/(d-1)$ and $p'=d$ forces $d\geq 3$, so the subcase $p<d$ and $p'=d$ cannot occur for $d=2$.)

Consider Subcase (\ref{case:Sobolev_mult_p_lessthan_d}\ref{subcase:Sobolev_mult_p_lessthan_d_and_pprime_lessthan_d}), so $p<d$ and $p'<d$. Then the condition $p'>d/(d-1)$ implies that $d/(d-1) < d$, that is, $d-1>1$ or $d\geq 3$. (The subcase $p<d$ and $p'<d$ cannot occur for $d=2$.) We choose $r = (p')^* = dp'/(d-p') = d(p/(p-1))/(d-p/(p-1)) = dp/(dp-d-p)$ and use the continuous Sobolev embedding, $W^{1,p'}(X) \subset L^{(p')^*}(X)$, provided by \cite[Theorem 4.12]{AdamsFournier}. To have a continuous multiplication, $L^{(p')^*}(X)\times L^{p^*}(X) \to L^{p'}(X)$, we see that $p>1$ must obey
\[
1/(p')^* + 1/p^* \leq 1/p',
\]
that is,
\[
(dp-d-p)/dp + (d-p)/dp \leq (p-1)/p,
\]
or
\[
d(p-1)-p + d-p \leq d(p-1),
\]
or $d \leq 2p$, that is, $p \geq d/2$. Combining the conclusions of the three subcases verifies the case $p<d$.
\end{case}

\begin{case}[$p=d$]
\label{case:Sobolev_mult_p_equals_d}
For any $s \in [1,\infty)$, we have a continuous Sobolev embedding, $W^{1,p}(X) \subset L^s(X)$. We also required that $s \in [p',\infty]$, so $p$ and $s$ must obey $p' \leq s < \infty$. Because $p=d$, we have $p' = p/(p-1) = d/(d-1)$, so
\begin{inparaenum}[\itshape a\upshape)]
\item
\label{subcase:Sobolev_mult_p_equals_d_and_pprime_equals_d_equals_2}
$p' = d$ for $d=2$, or
\item
\label{subcase:Sobolev_mult_p_equals_d_and_pprime_lessthan_d_and d_geq_3}
$1 < p' < d$ for $d\geq 3$.
\end{inparaenum}

Consider Subcase (\ref{case:Sobolev_mult_p_equals_d}\ref{subcase:Sobolev_mult_p_equals_d_and_pprime_equals_d_equals_2}), so $p'=d$ and $d=2$. Then $W^{1,p'}(X) \subset L^r(X)$ is a continuous Sobolev embedding for any $r\in [1,\infty)$. We also required that $r \in [p',\infty]$ and further restrict to $r \in (p',\infty)$ since $s$ is finite. In particular, we may choose $r, s \in (p',\infty)$ such that $1/r + 1/s \leq 1/p'$.

Consider Subcase (\ref{case:Sobolev_mult_p_equals_d}\ref{subcase:Sobolev_mult_p_equals_d_and_pprime_lessthan_d_and d_geq_3}), so $p'<d$ and $d\geq 3$. Then $W^{1,p'}(X) \subset L^r(X)$ is a continuous Sobolev embedding for $r = (p')^* = dp/(dp-d-p)$. We also required that $r \in [p',\infty]$ and further restrict to $r \in (p',\infty)$ since $s$ is finite, so $p$ must obey $p' < (p')^*$, that is
\[
p/(p-1) < dp/(dp-d-p),
\]
or $p>1$ and $dp-d-p < d(p-1)$, or simply $p > 1$ (automatic since $p=d\geq 3$). In particular, we may choose $r = (p')^* \in (p',\infty)$ and then $s \in (p',\infty)$ large enough that $1/r + 1/s \leq 1/p'$.

Combining the conclusions of each of the two subcases verifies the case $p=d$.
\end{case}

\begin{case}[$p>d$]
For any $s \in [1,\infty]$, we have a continuous Sobolev embedding, $W^{1,p}(X) \subset L^s(X)$. We also required that $s \in [p',\infty]$, so $p$ and $s$ must obey $p' \leq s \leq \infty$. Because $p>d$, we have $p' = p/(p-1) < d/(d-1)$, so $1 < p' < d$ for $d\geq 2$. Thus, $W^{1,p'}(X) \subset L^r(X)$ is a continuous Sobolev embedding for $r = (p')^* = dp/(dp-d-p)$.  We also required that $r \in [p',\infty]$, so $p$ must obey $p' \leq (p')^*$, which holds for any $p > 1$ from our analysis of the case $p=d$. In particular, we may choose $r = (p')^* \in [p',\infty)$ and then $s \in (p',\infty]$ large enough that $1/r + 1/s \leq 1/p'$. This verifies the case $p>d$.
\end{case}

Combining these three cases completes the proof of Claim \ref{claim:Continuous_Sobolev_multiplication_W1pprime_times_W1p_to_Lpprime}.
\end{proof}

The third and fourth inequalities follow from Claim \ref{claim:Continuous_Sobolev_multiplication_W1pprime_times_W1p_to_Lpprime} as described earlier, so this completes the proof of Lemma \ref{lem:gauged}.
\end{proof}

\subsection{Completion of the proof of Theorem \ref{mainthm:Lojasiewicz-Simon_gradient_inequality_boson_Yang--Mills_energy_function}}
\label{subsec:Completion_proof_Lojasiewicz-Simon_gradient_inequality_boson_Yang--Mills_energy_function}
We can now proceed to the

\begin{proof}[Proof of Theorem \ref{mainthm:Lojasiewicz-Simon_gradient_inequality_boson_Yang--Mills_energy_function}]
We first consider the simpler case where the pair $(A,\Phi)$ is in Coulomb gauge relative to the critical point $(A_\infty, \Phi_\infty)$ and then consider the general case

\setcounter{case}{0}
\begin{case}[$(A,\Phi)$ in Coulomb gauge relative to $(A_\infty, \Phi_\infty)$]
\label{case:APhi_Coulomb_gauge_wrt_APhi_infty}
By hypothesis of the theorem, $(A_\infty, \Phi_\infty)$ is a $W^{1,q}$ pair that is a critical point for the function $\sE$ in \eqref{eq:Boson_Yang--Mills_energy_function}. By the regularity Theorem~\ref{thm:Parker_1982_5-3}, there exists a $W^{2,q}$ gauge transformation $u_\infty$ such that $u_\infty(A_\infty,\Phi_\infty)$ is a $C^\infty$ pair. In particular, $u_\infty(A_\infty, \Phi_\infty)$ is a $W^{2,q}$ pair and $u_\infty(A, \Phi)$ is in Coulomb gauge relative to $u_\infty(A_\infty, \Phi_\infty)$. Following \eqref{eq:Coulomb_gauge_slice_pairs} and \eqref{eq:Coulomb_gauge_slice_pairs_range}, we choose the Banach spaces,
\begin{align*}
\sX &= \Ker\left(d_{u_\infty(A_\infty, \Phi_\infty)}^*: W_{A_1}^{1,p}(X; \Lambda^1 \otimes \ad P \oplus E)
\to L^p(X;\ad P) \right),
\\
\tilde\sX &= \Ker\left(d_{u_\infty(A_\infty, \Phi_\infty)}^*: W_{A_1}^{-1,p}(X; \Lambda^1 \otimes \ad P \oplus E)
\to W_{A_1}^{-2,p}(X;\ad P) \right).
\end{align*}
Hence, $\sX \subset \tilde\sX$ is a continuous embedding of Banach spaces and
\[
\sX^* = \Ker\left(d_{u_\infty(A_\infty, \Phi_\infty)}^*: L^{p'}(X;\ad P) \to W_{A_1}^{-1,p'}(X; \Lambda^1 \otimes \ad P \oplus E) \right).
\]
We observe that $\tilde\sX \subset \sX^*$ is a continuous embedding of Banach spaces when $W^{-1,p}(X;\CC) \subset W^{-1,p'}(X;\CC)$ is a continuous Sobolev embedding and thus when $p$ obeys $p \geq 2$.

Proposition \ref{prop:fredbos} implies that the Hessian operator with $x_\infty = u_\infty(A_\infty, \Phi_\infty)$,
\[
\hat\sM'(x_\infty):\sX \to \tilde\sX,
\]
is Fredholm with index zero while Corollary \ref{cor:Analyticity_Hessian_boson_energy_function_slice} implies that the gradient map,
\[
\hat\sM:x_\infty+\sX \to \tilde\sX,
\]
is analytic, where we recall from \eqref{eq:Boson_gradient_operator_naturally_restricts_to_slice} that
\[
\hat\sM = \Pi_{u_\infty(A_\infty, \Phi_\infty)}\sM.
\]
Hence, Theorem~\ref{mainthm:Lojasiewicz-Simon_gradient_inequality} implies that there exist constants $Z'\in (0,\infty)$ and $\sigma' \in (0,1]$ and $\theta \in [1/2, 1)$ (depending on $(A_1, \Phi_1)$, and  $u_\infty(A_\infty,\Phi_\infty)$, and $g$, $G$, $p$, $P$) such that if
\begin{equation}
\label{eq:Lojasiewicz-Simon_gradient_inequality_boson_Yang--Mills_pair_neighborhood_gauged}
\|u_\infty(A,\Phi) -  u_\infty(A_\infty, \Phi_\infty)\|_{W_{A_1}^{1,p}(X)} < \sigma',
\end{equation}
then
\begin{equation}
\label{eq:gauged_LS}
|\sE(u_\infty(A,\Phi)) - \sE(u_\infty(A_\infty, \Phi_\infty))|^\theta \leq Z'\|\hat\sM(u_\infty(A,\Phi))\|_{W^{-1, p}_{A_1}(X)}.
\end{equation}
By Lemma~\ref{lem:gauged}, there exists $C_1 = C_1(g,G,p,u_\infty) = C_1(A_\infty,\Phi_\infty,g,p) \in [1,\infty)$ so that
\[
\|u_\infty(A,\Phi) - u_\infty(A_\infty, \Phi_\infty)\|_{W_{A_1}^{1,p}(X)}
\leq
C_1 \|(A,\Phi) - (A_\infty, \Phi_\infty)\|_{W_{A_1}^{1,p}(X)}.
\]
More explicitly, Lemma~\ref{lem:gauged} gives $C_1 = C(1 + \|u_\infty\|_{W_{A_1}^{2,p}(X)})$, where $C = C(g,G,p) \in [1,\infty)$. Therefore, setting $\sigma := C_1^{-1}\sigma'$, we see that if $(A,\Phi)$ obeys the {\L}ojasiewicz--Simon neighborhood condition \eqref{eq:Lojasiewicz-Simon_gradient_inequality_boson_Yang--Mills_pair_neighborhood}, namely
\[
\|(A,\Phi) - (A_\infty, \Phi_\infty)\|_{W_{A_1}^{1,p}(X)} <\sigma,
\]
then \eqref{eq:Lojasiewicz-Simon_gradient_inequality_boson_Yang--Mills_pair_neighborhood_gauged} holds and thus also \eqref{eq:gauged_LS}. Moreover,
\begin{align*}
{}&\|\hat\sM(u_\infty(A,\Phi))\|_{W^{-1, p}_{A_1}(X)} = \|\hat\sM(u_\infty(A,\Phi))\|_{(W^{1, p'}_{A_1}(X))^*}
\\
&\quad = \sup\left\{|\hat\sM(u_\infty(A,\Phi)) (u_\infty(a,\phi))|: \, \|u_\infty(a,\phi)\|_{W_{A_1}^{1,p'}(X)}\leq 1 \right\}
\\
&\quad = \sup\left\{|\hat\sM(A,\Phi)(a,\phi)|: \, \|u_\infty(a,\phi)\|_{W_{A_1}^{1,p'}(X)}\leq 1\right\}
\quad \text{(by gauge invariance)},
\end{align*}
where the supremum is over all pairs, $(a,\phi)\in W^{1,p}_{A_1}(X;\Lambda^1\otimes\ad P)$, obeying the inequality. But
\[
\|(a,\phi)\|_{W_{A_1}^{1,p'}(X)}
\leq
C_1\|u_\infty(a,\phi)\|_{W_{A_1}^{1,p'}(X)}
\]
by Lemma~\ref{lem:gauged} and therefore,
\begin{multline*}
\left\{(a,\phi)\in W^{1,p}_{A_1}(X;\Lambda^1\otimes\ad P): \, \|u_\infty(a,\phi)\|_{W_{A_1}^{1,p'}(X)}\leq 1\right\}
\\
\subset \left\{(a,\phi)\in W^{1,p}_{A_1}(X;\Lambda^1\otimes\ad P): \, C_1^{-1}\|(a,\phi)\|_{W_{A_1}^{1,p'}(X)}\leq 1\right\}.
\end{multline*}
Combining the preceding equality and inequality yields,
\begin{align*}
\|\hat\sM(u_\infty(A,\Phi))\|_{W^{-1, p}_{A_1}(X)}
&\leq
\sup\left\{|\hat\sM(A,\Phi)(a,\phi)|: \, C_1^{-1}\|(a,\phi)\|_{W_{A_1}^{1,p'}(X)}\leq 1\right\}
\\
&= C_1 \|\hat\sM(A,\Phi)\|_{W^{-1, p}_{A_1}(X)}.
\end{align*}
Substituting the preceding inequality into \eqref{eq:gauged_LS} yields
\begin{align*}
|\sE(A,\Phi) - \sE(A_\infty, \Phi_\infty)|^\theta
&= |\sE(u_\infty(A,\Phi)) - \sE(u_\infty(A_\infty, \Phi_\infty))|^\theta
  \\
  &\qquad \text{(by gauge invariance)}
\\
&\leq Z'\|\hat\sM(u_\infty(A,\Phi))\|_{W^{-1, p}_{A_1}(X)} \quad\text{(by \eqref{eq:gauged_LS})}
\\
&\leq Z'C_1\|\hat\sM(A,\Phi)\|_{W^{-1, p}_{A_1}(X)}.
\end{align*}
But $\hat\sM(A,\Phi) = \Pi_{u_\infty(A_\infty, \Phi_\infty)}\sM$ and because the projection,
\[
\Pi_{u_\infty(A_\infty, \Phi_\infty)}:W^{k, p}_{A_1}(X; \Lambda^1 \otimes \ad P \oplus E) \to W^{k, p}_{A_1}(X; \Lambda^1 \otimes \ad P \oplus E),
\]
is bounded with norm one (for any $k\in\ZZ$ and $1<p<\infty$), then
\[
\|\hat\sM(A,\Phi)\|_{W^{-1, p}_{A_1}(X)} \leq \|\sM(A,\Phi)\|_{W^{-1, p}_{A_1}(X)}.
\]
Hence, the {\L}ojasiewicz--Simon gradient inequality \eqref{eq:Lojasiewicz-Simon_gradient_inequality_boson_Yang--Mills_energy_function} holds for the pairs $(A, \Phi)$ and $(A_\infty, \Phi_\infty)$ with constants $(Z, \theta, \sigma)$, where $Z := C_1Z'$.
\end{case}

\begin{case}[$(A,\Phi)$ not in Coulomb gauge relative to $(A_\infty, \Phi_\infty)$]
\label{case:APhi_notin_Coulomb_gauge_wrt_APhi_infty}
Let
\begin{align*}
  \zeta &= \zeta(A_1, A_\infty, \Phi_\infty, g, G,p,q)\in (0,1] \quad \text{ and}
  \\
  N &= N(A_1,A_\infty, \Phi_\infty, g, G,p,q) \in [1,\infty)
\end{align*}
denote the constants in Theorem~\ref{mainthm:Feehan_proposition_3-4-4_Lp_pairs} and choose $\zeta_1 \in (0,\zeta]$ small enough that $2N\zeta_1 < \sigma_1$, where we now use $\sigma_1$ to denote the {\L}ojasiewicz--Simon constant from Case \ref{case:APhi_Coulomb_gauge_wrt_APhi_infty}.  If $(A,\Phi)$ obeys
\[
\|(A,\Phi) - (A_\infty, \Phi_\infty)\|_{W_{A_1}^{1,p}(X)} < \zeta_1,
\]
then Theorem~\ref{mainthm:Feehan_proposition_3-4-4_Lp_pairs} provides $u \in \Aut^{2,q}(P)$, depending on the pair $(A, \Phi)$, such that
\begin{gather*}
d_{A_\infty, \Phi_\infty}^*(u(A, \Phi) - (A_\infty, \Phi_\infty)) = 0,
\\
\|u(A,\Phi) - (A_\infty, \Phi_\infty)\|_{W_{A_1}^{1,p}(X)} < 2N\zeta_1 < \sigma.
\end{gather*}
By applying Case \ref{case:APhi_Coulomb_gauge_wrt_APhi_infty} to the pairs $u(A,\Phi)$ and $(A_\infty, \Phi_\infty)$, we obtain
\[
|\sE(u(A,\Phi)) - \sE(A_\infty, \Phi_\infty)|^\theta \leq C_1Z'\|\sM(u(A,\Phi))\|_{W^{-1, p}_{A_1}(X)}.
\]
Estimating as in Case \ref{case:APhi_Coulomb_gauge_wrt_APhi_infty}, with $u$ replacing $u_\infty$, we see that
\[
\|\sM(u(A,\Phi))\|_{W^{-1, p}_{A_1}(X)} \leq C_2\|\sM(A,\Phi)\|_{W^{-1, p}_{A_1}(X)},
\]
where $C_2 = C(1 + \|u\|_{W_{A_1}^{2,p}(X)})$ and $C = C(g,p) \in [1,\infty)$. According to Lemma \ref{lem:connections_control_gauge_pairs}, we have
\[
\|u\|_{W_{A_1}^{2,p}(X)} \leq C_3,
\]
where $C_3 = C_3(A_\infty,\Phi_\infty,A_1,g,G,p,q) \in [1,\infty)$. By combining the preceding inequalities, we obtain
\begin{align*}
|\sE(A,\Phi) - \sE(A_\infty, \Phi_\infty)|^\theta
&= |\sE(u(A,\Phi)) - \sE(A_\infty, \Phi_\infty)|^\theta
  \\
  &\qquad \text{(by gauge invariance)}
\\
&\leq C_1Z'\|\sM(u(A,\Phi))\|_{W^{-1, p}_{A_1}(X)}
\\
&\leq C_1C(1+C_3)Z'\|\sM(A,\Phi)\|_{W^{-1, p}_{A_1}(X)}.
\end{align*}
Hence, we obtain the {\L}ojasiewicz--Simon gradient inequality \eqref{eq:Lojasiewicz-Simon_gradient_inequality_boson_Yang--Mills_energy_function} with constants $(Z, \theta, \sigma)$, where we now choose $Z = C_1C(1+C_3)Z'$ and $\sigma = \zeta_1$.
\end{case}

This completes the proof of Theorem \ref{mainthm:Lojasiewicz-Simon_gradient_inequality_boson_Yang--Mills_energy_function}.
\end{proof}

\begin{rmk}[On the proof of Theorem \ref{mainthm:Lojasiewicz-Simon_gradient_inequality_boson_Yang--Mills_energy_function} for $p=2$ and $d=2,3,4$]
\label{rmk:Proof_Lojasiewicz-Simon_gradient_inequality_boson_Yang--Mills_energy_small_dim}
As we discussed prior to the statement of Theorem \ref{mainthm:Lojasiewicz-Simon_gradient_inequality_dualspace}, that version of the abstract {\L}ojasiewicz--Simon gradient inequality, while considerably more restrictive, has the advantage that, when applicable, its hypotheses are much easier to verify than those of Theorem \ref{mainthm:Lojasiewicz-Simon_gradient_inequality}. Thus, for $2\leq d \leq 4$, we may choose $p=2$ and observe that the condition $d/2 \leq p \leq q$ is met for $q>2$. With the notation used in the proof of Theorem \ref{mainthm:Lojasiewicz-Simon_gradient_inequality}, we can then choose
\[
\sX = \Ker\left(d_{u_\infty(A_\infty, \Phi_\infty)}^*: W_{A_1}^{1,2}(X; \Lambda^1 \otimes \ad P \oplus E) \to L^2(X; \Lambda^1 \otimes \ad P \oplus E) \right),
\]
with dual space,
\begin{multline*}
  \sX^* =  \Ker\left(d_{u_\infty(A_\infty, \Phi_\infty)}^*: W_{A_1}^{-1,2}(X; \Lambda^1 \otimes \ad P \oplus E) \right.
  \\
  \left. \to W_{A_1}^{-2,2}(X; \Lambda^1 \otimes \ad P \oplus E) \right).
\end{multline*}
Proposition \ref{prop:analyticbos} implies that the restriction of the energy function to the Coulomb-gauge slice,
\[
\sE:x_\infty + \sX \to \RR,
\]
is real analytic, where $x_\infty := u_\infty(A_\infty, \Phi_\infty)$. Proposition \ref{prop:fredbos} implies that the Hessian operator,
\[
\sE''(x_\infty) \in \sL(\sX, \sX^*),
\]
is Fredholm with index zero. Case \ref{case:APhi_Coulomb_gauge_wrt_APhi_infty} of the proof of Theorem \ref{mainthm:Lojasiewicz-Simon_gradient_inequality_boson_Yang--Mills_energy_function} now follows easily, where the gauge transformation, $u_\infty$, is again chosen via Theorem~\ref{thm:Parker_1982_5-3} so that $u_\infty(A_\infty, \Phi_\infty)$ is a $C^\infty$ pair. The difficult third and fourth inequalities in the statement of Lemma \ref{lem:gauged} are not required. Case \ref{case:APhi_notin_Coulomb_gauge_wrt_APhi_infty} of the proof of Theorem \ref{mainthm:Lojasiewicz-Simon_gradient_inequality_boson_Yang--Mills_energy_function} is unchanged.
\end{rmk}

\section[{\L}ojasiewicz--Simon inequality for fermion coupled Yang--Mills]{{\L}ojasiewicz--Simon gradient inequality for the fermion coupled Yang--Mills energy function}
\label{sec:Lojasiewicz-Simon_gradient_inequality_fermion_Yang--Mills_energy_function}
We assume the notation and conventions of Section \ref{sec:Lojasiewicz-Simon_gradient_inequality_boson_Yang--Mills_energy_function}. By analogy with Proposition~\ref{prop:analyticbos} we establish the forthcoming Proposition \ref{prop:analyticferm} for the analyticity of the fermion coupled Yang--Mills energy function, $\sF$. Again, this serves as a stepping stone towards the proof that its gradient map, $\hat\sM:x_\infty+\sX\to\tilde\sX$, is real analytic for suitable choices of Banach spaces as in Theorem \ref{mainthm:Lojasiewicz-Simon_gradient_inequality}.

\begin{prop}[Analyticity of the fermion coupled Yang--Mills energy function]
\label{prop:analyticferm}
Let $(X,g)$ be a closed, smooth Riemannian manifold of dimension $d \geq 2$, and $(\rho,W)$ be a \spinc structure on $X$, and $G$ be a compact Lie group, $P$ be a smooth principal $G$-bundle over $X$, and $E = P\times_\varrho\EE$ be a smooth Hermitian vector bundle over $X$ defined by a finite-dimensional unitary representation, $\varrho: G \to \Aut_\CC(\EE)$, and $A_1$ be a smooth reference connection on $P$, and $m \in C^\infty(X)$. If $4d/(d+4) \leq p < \infty$, then the function \eqref{eq:Fermion_Yang--Mills_energy_function},
\[
\sF:\sA^{1,p}(P) \times W_{A_1}^{1,p}(X; W\otimes E) \to \RR,
\]
is real analytic.
\end{prop}

\begin{proof}
We prove analyticity at a point $(A,\Psi)$ and write $A = A_1 + a_1$, where $a \in W_{A_1}^{1,q}(X; \Lambda^1 \otimes \ad P)$. For any $(a,\psi) \in \fX$, we have
\begin{align*}
F_{A+a} &= F_{A_1} + d_{A_1}(a_1 + a) + \frac{1}{2}[a_1+a, a_1+a],
\\
D_{A+a}(\Psi + \psi) &= D_{A_1}(\Psi + \psi) + \rho(a_1+a) (\Psi + \psi).
\end{align*}
By definition~\eqref{eq:Fermion_Yang--Mills_energy_function} of the fermion coupled Yang--Mills energy function, we obtain
\[
2\sF (A + a , \Psi + \psi) = T_1 + T_2 + T_3,
\]
where the curvature term,
\[
T_1 := (F_{A+a}, F_{A+a})_{L^2(X)},
\]
has the same expansion as the corresponding term $T_1$ in Proposition~\ref{prop:analyticbos} for the boson coupled Yang--Mills energy function, while
\begin{align*}
T_2
&:=
(\Psi, D_{A_1}\Psi)_{L^2(X)} + (\Psi, D_{A_1}\psi)_{L^2(X)} + (\psi, D_{A_1}\Psi)_{L^2(X)} + (\psi, D_{A_1}\psi)_{L^2(X)}
\\
&\quad\, + (\Psi, \rho(a_1+a)\Psi)_{L^2(X)} + (\Psi, \rho(a_1+a)\psi)_{L^2(X)} + (\psi, \rho(a_1+a)\Psi)_{L^2(X)}
\\
&\quad\, + (\psi, \rho(a_1+a)\psi)_{L^2(X)},
\end{align*}
and thus,
\begin{align*}
T_2
&= (\Psi, D_{A_1}\Psi)_{L^2(X)} + 2\Real (\Psi, D_{A_1}\psi)_{L^2(X)} + (\psi, D_{A_1}\psi)_{L^2(X)}
\\
&\quad\, + (\Psi, \rho(a_1+a)\Psi)_{L^2(X)} + 2\Real (\Psi, \rho(a_1+a)\psi)_{L^2(X)} + (\psi, \rho(a_1+a)\psi)_{L^2(X)},
\\
T_3 &:= \int_X m\left(|\Psi|^2 + \langle\Psi,\psi\rangle + \langle\psi,\Psi\rangle
+ |\psi|^2\right)\, d\vol_g.
\end{align*}
The terms in the expression for the difference,
\[
2\sF (A +a , \Psi+ \psi) - 2\sF (A, \Psi) =  T_1'  +  T'_2  + T_3',
\]
are organized in such a way that
\[
T_1' := (F_{A+a}, F_{A+a})_{L^2(X)} - (F_A, F_A)_{L^2(X)}
\]
has the same expansion as the corresponding term $T_1'$ for the boson coupled Yang--Mills energy function in the proof of Proposition \ref{prop:analyticferm} and the remaining terms are given by
\begin{align*}
T_2' &:= 2\Real (\Psi, D_{A_1}\psi)_{L^2(X)} + (\psi, D_{A_1}\psi)_{L^2(X)}
\\
&\quad\, + 2\Real (\Psi, \rho(a_1+a)\psi)_{L^2(X)} + (\psi, \rho(a_1+a)\psi)_{L^2(X)},
\\
T'_3 &:= \int_X m\left(2\Real \langle \Psi, \psi\rangle + |\psi|^2\right)\, d\vol_g.
\end{align*}
The proof of analyticity of $\sF$ at $(A,\Psi)$ now follows by adapting \mutatis the arguments used to prove Proposition~\ref{prop:analyticbos}.
\end{proof}

We now verify the formula \eqref{eq:Gradient_fermion_coupled_Yang--Mills_energy_function} for the differential $\sE'(A,\Phi)$ and gradient $\sM(A,\Phi)$.

\begin{lem}[Differential and gradient of the fermion coupled Yang--Mills energy function]
\label{lem:Differential_gradient_fermion_coupled_Yang--Mills_energy_function}
Assume the hypotheses of Proposition \ref{prop:analyticferm} with the dual H\"older exponent $p'$ in the range $1 < p'\leq 4d/(3d-4)$ determined by $4d/(d+4) \leq p < \infty$ and $1/p+1/p'=1$. Then the expression for $\sE'(A,\Psi) \in (W_{A_1}^{1,p}(X; \Lambda^1 \otimes \ad P \oplus W\otimes E))^*$ and $\sM(A,\Psi) \in W_{A_1}^{-1,p'}(X; \Lambda^1 \otimes \ad P \oplus W\otimes E)$ is given by \eqref{eq:Gradient_fermion_coupled_Yang--Mills_energy_function}, namely
\begin{align*}
\sE'(A,\Psi)(a, \psi) &= \left((a, \psi), \sM(A,\Psi)\right)_{L^2(X)}
\\
&= (d_A^*F_A, a)_{L^2(X)} + \Real(D_A\Psi - m\Psi, \psi)_{L^2(X)} + \frac{1}{2}(\Psi, \rho(a) \Psi)_{L^2(X)},
\\
&\qquad\forall\, (a,\psi)\in) \in W_{A_1}^{1,p}(X; \Lambda^1 \otimes \ad P \oplus W\otimes E).
\end{align*}
\end{lem}

\begin{proof}
It suffices to extract the terms that are linear in $(a,\psi)$ from the expressions for
\[
T_1' = 2(F_A, d_Aa)_{L^2(X)} + (d_Aa, d_Aa)_{L^2(X)}
\]
and $T_2'$ and $T_3'$ arising in the proof of Proposition \ref{prop:analyticferm}.
\end{proof}

\begin{rmk}[Pointwise self-adjointness and reality]
\label{rmk:Pointwise_self-adjointness}
The fact that the term $\langle\Psi, \rho(a)\Psi\rangle$ appearing in \eqref{eq:Gradient_fermion_coupled_Yang--Mills_energy_function} is real could be inferred indirectly by noting the origin of this term and the fact that the Dirac operator, $D_A$, is self-adjoint. To see directly that $\langle\Psi, \rho(a)\Psi\rangle$ is real, recall that Clifford multiplication is skew-Hermitian, so $c(\alpha)^* = -c(\alpha) \in \End_\CC(W)$ for all $\alpha \in \Omega^1(X)$ (for example, see \cite[p. 49]{Guentner_1993}) while if $\xi \in \fg$, then $\varrho_*(\xi)^* = - \varrho_*(\xi)$ since we assume that Lie structure group, $G$, of $P$ acts on the complex, finite-dimensional vector space $\EE$ via a unitary representation, $\varrho:G\to\End_\CC(\EE)$, and $\varrho_*:\fg\to\End_\CC(\EE)$ is the induced representation of the Lie algebra, $\fg$. Hence, given $\alpha\otimes\xi \in C^\infty(T^*X\otimes\ad P) = \Omega^1(X;\ad P)$ and recalling that $E = P\times_\varrho\EE$, then $\rho(\alpha\otimes\xi) = c(\alpha)\otimes\varrho_*(\xi) \in \End_\CC(W\otimes E)$ obeys
\[
\rho(\alpha\otimes\xi)^* = c(\alpha)^*\otimes\varrho_*(\xi)^*
= c(\alpha)\otimes\varrho_*(\xi) = \rho(\alpha\otimes\xi).
\]
In particular, $\rho(a) \in \End_\CC(W\otimes E)$ satisfies $\rho(a)^* = \rho(a)$ for all $a \in \Omega^1(X;\ad P)$.
\end{rmk}

We now compute the Hessian operator, $\sM'(A,\Psi)$, at a $C^\infty$ pair $(A,\Psi) \in \sA(P)\times C^\infty(X;W\otimes E)$. The gradient, $\sM(A, \Psi) \in C^\infty(X; \Lambda^1 \otimes \ad P \oplus W\otimes E)$ in \eqref{eq:Gradient_fermion_coupled_Yang--Mills_energy_function}, may be written as
\begin{equation}
\label{eq:Gradient_fermion_coupled_Yang--Mills_energy_function_preL2pairing}
\sM(A,\Psi)
=
d_A^*F_A + \frac{1}{2} \left( (D_A-m)\Psi\cdot\, + \,\cdot(D_A-m)\Psi \right)
+ \frac{1}{2}\rho^{-1}(\Psi\otimes\Psi^*),
\end{equation}
where the terms involving $\Psi$ in this expression for $\sM(A,\Psi)$ are defined by the $L^2$-pairings,
\begin{align*}
\Real(D_A\Psi - m\Psi, \psi)_{L^2(X)}
&=
\frac{1}{2} \left( ((D_A-m)\Psi, \psi)_{L^2(X)} + (\psi, (D_A-m)\Psi)_{L^2(X)}\right),
\\
\frac{1}{2}(\Psi, \rho(a) \Psi)_{L^2(X)}
&=
\frac{1}{2}(\Psi\otimes\Psi^*, \rho(a))_{L^2(X)}
=
\frac{1}{2}(\rho^{-1}(\Psi\otimes\Psi^*), a)_{L^2(X)},
\\
&\qquad\forall\, (a,\psi) \in C^\infty(X; \Lambda^1 \otimes \ad P \oplus W\otimes E).
\end{align*}
Taking the derivative of the gradient $\sM(A, \Psi)$ in \eqref{eq:Gradient_fermion_coupled_Yang--Mills_energy_function_preL2pairing} with respect to $(A,\Psi)$ in the direction $(a,\psi)$ yields
\begin{multline}
\label{eq:Hessian_fermion_coupled_Yang--Mills_energy_function_preL2pairing}
\sM'(A,\Psi)(a,\psi)
=
d_A^*d_A a + \frac{1}{2} \left( D_A\psi\cdot\, + \,\cdot D_A\psi \right)
- \frac{1}{2} \left( m\psi\cdot\, + \,\cdot m\psi \right)
\\
+ (a\wedge\cdot)^*F_A
+
\frac{1}{2} \left( \rho(a)\Psi\cdot\, + \,\cdot\rho(a)\Psi \right)
+
\frac{1}{2}\rho^{-1}(\Psi\otimes\psi^* + \psi\otimes\Psi^*).
\end{multline}
By virtue of \eqref{eq:Hessian_fermion_coupled_Yang--Mills_energy_function_preL2pairing} we may view the Hessian operator for $\sF$ at a $C^\infty$ pair $(A,\Psi)$ as a linear, second-order partial differential operator,
\[
\sM'(A,\Psi): C^\infty(X; \Lambda^1 \otimes \ad P \oplus W\otimes E) \to C^\infty(X; \Lambda^1 \otimes \ad P \oplus W\otimes E).
\]
The differential,
\[
d_{A,\Psi}:C^\infty(X; \ad P) \to C^\infty(X; \Lambda^1 \otimes \ad P \oplus W\otimes E),
\]
is defined just as in \eqref{eq:Differential_gauge_transformation_action_on_pair_at_identity} except that we replace $E$ by $W\otimes E$. As in \eqref{eq:Ellipticized_Hessian_operator_Wk+2p}, the operator $\sM'(A,\Psi)+d_{A,\Psi}d_{A,\Psi}^*$ is elliptic and thus Fredholm on the analogous Sobolev spaces. We can now simply make the following observation to conclude the proof of Theorem \ref{mainthm:Lojasiewicz-Simon_gradient_inequality_fermion_Yang--Mills_energy_function}.

\begin{proof}[Proof of Theorem \ref{mainthm:Lojasiewicz-Simon_gradient_inequality_fermion_Yang--Mills_energy_function}]
The argument applies \mutatis the corresponding steps used to prove Theorem~\ref{mainthm:Lojasiewicz-Simon_gradient_inequality_boson_Yang--Mills_energy_function} and all required intermediate results established in Section \ref{sec:Lojasiewicz-Simon_gradient_inequality_boson_Yang--Mills_energy_function}.
\end{proof}

\chapter[{\L}ojasiewicz--Simon $W^{-1,2}$ inequalities for coupled Yang--Mills energies]{{\L}ojasiewicz--Simon $W^{-1,2}$ gradient inequalities for coupled Yang--Mills energy functions}
\label{chap:Lojasiewicz-Simon_gradient_inequality_Hilbert_space}
In this chapter, we shall prove Corollaries \ref{maincor:Lojasiewicz-Simon_L2_gradient_inequality_boson_Yang-Mills_energy_function} and \ref{maincor:Lojasiewicz-Simon_L2_gradient_inequality_fermion_Yang-Mills_energy_function}, using a method analogous to that of our proof of \cite[Corollary 6]{Feehan_Maridakis_Lojasiewicz-Simon_harmonic_maps_v6} for the harmonic map energy function. We shall accomplish this by applying the more general Theorem \ref{mainthm:Lojasiewicz-Simon_gradient_inequality2} instead of Theorem \ref{mainthm:Lojasiewicz-Simon_gradient_inequality}, which we used to prove Theorem \ref{mainthm:Lojasiewicz-Simon_gradient_inequality_boson_Yang--Mills_energy_function}.

\section[Technical preparation for proof of $W^{-1,2}$ gradient inequalities]{Technical preparation for the proof of the $W^{-1,2}$ gradient inequalities for coupled Yang--Mills energy functions}
\label{sec:Technical_results_L2_gradient_inequality}
Before proceeding to the proof of Corollaries \ref{maincor:Lojasiewicz-Simon_L2_gradient_inequality_boson_Yang-Mills_energy_function} and \ref{maincor:Lojasiewicz-Simon_L2_gradient_inequality_fermion_Yang-Mills_energy_function}, we first establish several technical lemmas and corollaries needed to verify the hypotheses of Theorem \ref{mainthm:Lojasiewicz-Simon_gradient_inequality2}. In the forthcoming Lemma \ref{lem:Analyticity_of_extension_to_boson_YM_augmented_Hessian_operator}, it would suffice to have continuity of the extension $\sM_1(A, \Phi)$ as a function of $(A,\Phi)$ in the affine Banach space $\sP^{1,q}(P,E)$ of $W^{1,q}$ pairs on $(P,E)$ in order to satisfy a key hypothesis of Theorem \ref{mainthm:Lojasiewicz-Simon_gradient_inequality2}, but the proof yields analyticity as a function of $(A,\Phi)$ in the affine Banach space $\sP^{1,p}(P,E)$ of $W^{1,p}$ pairs on $(P,E)$ (compare Proposition \ref{prop:analyticbos}).

\begin{lem}[Analyticity of extended boson coupled Yang-Mills augmented Hessian operator map] 
\label{lem:Analyticity_of_extension_to_boson_YM_augmented_Hessian_operator}
Assume the hypotheses of Theorem \ref{mainthm:Lojasiewicz-Simon_gradient_inequality_boson_Yang--Mills_energy_function}. If $(A,\Phi)$ is a $W^{1,p}$ pair on $(P,E)$ then the operator \eqref{eq:PiAPhi_MAPhi_PiAPhi_W^1pslice_to_Wminus1p_slice},
\[
  \sM'(A,\Phi):W_{A_1}^{1,p}(X;\Lambda^1\otimes\ad P\oplus E)
  \to  W_{A_1}^{-1,p}(X;\Lambda^1\otimes\ad P\oplus E)
\]
has a bounded extension 
\begin{equation}
\label{eq:PiAPhi_MAPhi_PiAPhi_W^12slice_to_Wminus12_slice}
\sM_1(A, \Phi): W_{A_1}^{1,2}(X;\Lambda^1\otimes\ad P\oplus E)
\to  W_{A_1}^{-1,2}(X;\Lambda^1\otimes\ad P\oplus E),
\end{equation}
that is an analytic function of $(A, \Phi)$ in the affine Banach space $\sP^{1,p}(P,E)$ of $W^{1,p}$ pairs on $(P,E)$ .
\end{lem}

\begin{proof}
The fact that the augmented Hessian operator \eqref{eq:PiAPhi_MAPhi_PiAPhi_W^1pslice_to_Wminus1p_slice} has a well-defined bounded extension \eqref{eq:PiAPhi_MAPhi_PiAPhi_W^12slice_to_Wminus12_slice} follows from Proposition \ref{prop:Fredholmness_and_index_Laplace_operator_on_W1p_pairs_slice} when $(A, \Phi)$ is a $C^\infty$ pair by taking $k=-1$ and $p=2$.

For a general pair $(A,\Phi)$ of class $W^{1,p}$, we write $(A, \Phi)=(A_1 + a, \Phi)$, where $A_1$ is $C^\infty$ and $a\in  W_{A_1}^{1,p}(X;\Lambda^1\otimes\ad P)$. Then from the expression \eqref{eq:Hessian_boson_energy_function} (with $(a,\phi)$ replaced by $(b, \varphi)$), we have the formal expansion
\begin{equation}
\label{eq:Hessian_boson_energy_function_formal_expansion}
\begin{aligned}
\sM'(A,\Phi) (b, \varphi)
&=
d_A^*d_A b  + \nabla_A^*\nabla_A\varphi  + F_A\times b +  \nabla_A^*(\varrho(b)\Phi)   + \Phi\times \nabla_A\varphi
\\
&\quad - \rho(b)^*\nabla_A\Phi + \nabla_A\Phi\times\varphi + \varrho(b)\Phi \times\Phi
\\
&\quad - (m + 2s|\Phi|^2)\varphi - 4s\langle\Phi, \varphi\rangle\Phi,
\quad\forall\, (b,\varphi) \in \fX,
\end{aligned}
\end{equation}
so that, using $\nabla_{A_1}^*(\varrho(b)\Phi) = b\times \nabla_{A_1}\Phi + \nabla_{A_1}b \times \Phi$,
\begin{equation}
\begin{aligned}
\label{eq:Extended_Hessian_operator_boson_coupled_Yang--Mills_energy_function_expanded}
{}&\sM'(A_1 +a,\Phi) (b,\varphi)
\\
&\quad = 
d^*_{A_1} d_{A_1} b + \nabla^*_{A_1}\nabla_{A_1} \varphi +  F_{A_1} \times b  
\\
&\qquad + a\times d_{A_1} b + d_{A_1} a \times b + \nabla_{A_1} a \times \varphi  +  a \times \nabla_{A_1} \varphi 
\\
&\qquad +  a\times a \times b + a\times a \times \varphi. 
\\
&\qquad + \nabla_{A_1}b \times \Phi  + \Phi \times \nabla_{A_1} \varphi 
\\
&\qquad + b\times \nabla_{A_1}\Phi + \nabla_{A_1}\Phi\times\varphi
\\
&\qquad +  a \times \Phi \times b  + a \times \Phi \times \varphi 
\\
&\qquad +\varrho(b)\Phi \times\Phi  - (m + 2s|\Phi|^2)\varphi - 4s\langle\Phi, \varphi\rangle\Phi,
\end{aligned}
\end{equation}
for every $(b,\varphi) \in W_{A_1}^{1,2}(X;\Lambda^1\otimes\ad P\oplus E)$. We make the 

\begin{claim}
\label{claim:Analytic_difference_extended_Hessian_operator_and_Laplacian_W12_to_W-12}
Continue the preceding notation but abbreviate 
\[
W^{\pm1,2} = W_{A_1}^{\pm1,2} (X; \Lambda^1 \otimes \ad P \oplus E). 
\]
If $ \sM_1(A_1+a,\Phi)$ denotes the bounded extension of $\sM'(A_1+a,\Phi)$, then
\begin{multline}
\label{eq:Difference_Hessian_operator_and_Laplacian_W12_to_W-12}
W_{A_1}^{1,p}(X; \Lambda^1 \otimes \ad P \oplus E) \ni (a,\Phi)
\\
\mapsto \sM_1(A_1+a,\Phi) \in \sL(W^{1,2}, W^{-1,2}),
\end{multline}
is a cubic polynomial in $(a,\Phi)$ and its first-order covariant derivatives with respect to $\nabla_{A_1}$, with universal coefficients (depending at most on $g$ and $G$).
\end{claim}

Given Claim \ref{claim:Analytic_difference_extended_Hessian_operator_and_Laplacian_W12_to_W-12}, it follows that $\sM_1(A, \Phi)$ is a well-defined analytic function of $(A,\Phi)$ in an open $W^{1,p}$ neighborhood of $(A_1, \Phi_1)$, thus completing the proof of Lemma~\ref{lem:Analyticity_of_extension_to_boson_YM_augmented_Hessian_operator}.
\end{proof}

We now turn to the 

\begin{proof}[Proof of Claim \ref{claim:Analytic_difference_extended_Hessian_operator_and_Laplacian_W12_to_W-12}]
We compute a $W^{-1,2}$ bound for each term in equation \eqref{eq:Extended_Hessian_operator_boson_coupled_Yang--Mills_energy_function_expanded}.

\setcounter{step}{0}
\begin{step}[$W^{-1,2}$ estimates for $a\times d_{A_1}b$ and $a \times \nabla_{A_1} \varphi$ and $\Phi\times \nabla_{A_1} \varphi$ and $\nabla_{A_1}b\times\Phi$]
\label{step:W-12_estimates_for_a_times_dAinfty_b_and_a_times_nablaAinfty_varphi_and_Phi_times_nabla_Ainfty_ varphi_nablaAinfty_varrho_(b)_Phi}
We claim that
\begin{subequations}
\begin{align}
\label{eq:W{-1,2}_bound_a_times_dAinfty_b}
\|a\times d_{A_1}b\|_{W_{A_1}^{-1,2}(X)} 
&\leq
z\|a\|_{W_{A_1}^{1,p}(X)} \|b\|_{W^{1,2}_{A_1}(X)},
\\
\label{eq:W{-1,2}_bound_a_times_nablaAinfty_varphi}
\|a \times \nabla_{A_1} \varphi\|_{W^{-1,2}_{A_1}(X)} 
&\leq 
z\|a\|_{W_{A_1}^{1,p}(X)} \|\varphi\|_{W^{1,2}_{A_1}(X)},
\\
\label{eq:W{-1,2}_bound_Phi_times_nablaAinfty_varphi}
\|\Phi\times \nabla_{A_1} \varphi\|_{W^{-1,2}_{A_1}(X)} 
&\leq 
z\|\Phi\|_{W_{A_1}^{1,p}(X)} \|\varphi\|_{W^{1,2}_{A_1}(X)},
\\
\label{eq:W{-12,}_nablaAinfty_varrho_(b)_Phi}
\|\nabla_{A_1}b\times\Phi\|_{W_{A_1}^{-1,2}(X)} 
&\leq 
z\|\Phi\|_{W_{A_1}^{1,p}(X)} \|b\|_{W^{1,2}_{A_1}(X)},
\end{align}
\end{subequations}
where $z = z(g,G,p) \in [1,\infty)$.

We first prove estimate \eqref{eq:W{-1,2}_bound_a_times_dAinfty_b}.  Using duality, together with the H\"{o}lder inequality with exponents $r, r'$ satisfying $1/r+1/r'=1$, we estimate
 \begin{align*}
\|a\times d_{A_1}b\|_{W_{A_1}^{-1,2}(X)} 
&= 
\sup_{0<\|\beta\|_{W_{A_1}^{1,2}(X)}\leq 1} (a\times d_{A_1}b,\beta)_{L^2(X)} 
\\
&\leq 
\|a\times d_{A_1}b\|_{L^{r'}(X)} \sup_{0<\|\beta\|_{W_{A_1}^{1,2}(X)}\leq 1} \|\beta\|_{L^r(X)}
\end{align*}
Estimate \eqref{eq:W{-1,2}_bound_a_times_dAinfty_b} will follow by showing that we can always choose $r\in(1,\infty)$ so that
\begin{subequations} 
\label{eq:Lrprime_bound_a_times_dAinfty_b}
\begin{align}
\label{eq:Lrprime_bound_a_times_dAinfty_b1}
\|a\times d_{A_1}b\|_{L^{r'}(X)} 
&\leq
z\|a\|_{W^{1,p}_{A_1}(X)} \|b\|_{W_{A_1}^{1,2}(X)},
\\
\label{eq:Lrprime_bound_a_times_dAinfty_b2}
\|\beta\|_{L^r(X)} 
&\leq 
z\|\beta \|_{W^{1,2}_{A_1}(X)}
\end{align}
\end{subequations}
where $z = z(g,G,p) \in [1,\infty)$. We separately consider the cases $d\geq 3$ with $p\geq d$ or $d/2 \leq p <d$ and $d=2$ with $p\geq 2$.  

\setcounter{case}{0}
\begin{case}[$d\geq 3$ and $d/2 \leq  p< d$]
\label{case:d_geq_3_and_p_lessthan_d}
We choose $r = 2^* = 2d/(d-2) \in (2, 6]$, so $r' = r/(r-1) = 2d/(d+2) \in (1, 2)$, and $s= p^* = dp/(d-p)\in [d, \infty)$.  By \cite[Theorem 4.12]{AdamsFournier}, we have continuous Sobolev embeddings $W^{1,2}(X) \subset L^r(X)$ and $W^{1,p}(X) \subset L^s(X)$. Also $s \in (2, \infty)$ and obeys $1/s + 1/2 \leq 1/r'$, that is 
\[
1/s + 1/2 = (d-p)/(dp) + 1/2  = 1/p- 1/d +1/2 \leq  1/d + 1/2 = 1/r',
\]
since $1/p \leq 2/d $ from $p\geq d/2$, and which we assume by hypothesis on $p$. Therefore, 
 \[
 \|a\times d_{A_1}b\|_{L^{r'}(X)} \leq z\|a\|_{L^s(X)}\|d_{A_1}b\|_{L^2(X)} \leq z\|a\|_{W_{A_1}^{1,p}(X)}\|b\|_{W_{A_1}^{1,2}(X)},
 \]
and $\|\beta\|_{L^r(X)} \leq z\|\beta \|_{W^{1,2}_{A_1}(X)}$. This yields the estimates  \eqref{eq:Lrprime_bound_a_times_dAinfty_b} for this case.
\end{case}

\begin{case}[$d\geq 3$ and $p \geq d$]
We again choose $r = 2^* = 2d/(d-2) \in (2, 6]$, so $r' = r/(r-1) = 2d/(d+2) \in (1,2)$ and $s\in (2, \infty)$ is any sufficiently large number that $1/s + 1/2 \leq 1/r'$. By \cite[Theorem 4.12]{AdamsFournier}, we have a continuous Sobolev embedding $W^{1,2}(X) \subset L^r(X)$ and continuous Sobolev embeddings $W^{1,p}(X) \subset L^s(X)$ when $p=d$ and $W^{1,p}(X) \subset C(X) \subset L^s(X)$ when $p > d$. The estimates \eqref{eq:Lrprime_bound_a_times_dAinfty_b} now follow just as in Case~\ref{case:d_geq_3_and_p_lessthan_d}.
\end{case}

\begin{case}[$d=2$ and $p\geq 2$]
  We choose $r \in (2,\infty)$ large enough that $r' = r/(r-1) \in(1,2)$ is sufficiently close to one and choose $s\in (1, \infty)$ large enough that $1/s+ 1/2 \leq 1/r'$. By \cite[Theorem 4.12]{AdamsFournier}, we have a continuous Sobolev embedding $W^{1,2}(X) \subset L^r(X)$ and continuous Sobolev embeddings $W^{1,p}(X) \subset L^s(X)$ when $p=2$ and $W^{1,p}(X) \subset C(X) \subset L^s(X)$ when $p > 2$.  The estimates \eqref{eq:Lrprime_bound_a_times_dAinfty_b} now follow just as in Case~\ref{case:d_geq_3_and_p_lessthan_d}.  
\end{case}

This completes the proof of estimate  \eqref{eq:W{-1,2}_bound_a_times_dAinfty_b}. The proofs of estimates \eqref{eq:W{-1,2}_bound_a_times_nablaAinfty_varphi}, \eqref{eq:W{-1,2}_bound_Phi_times_nablaAinfty_varphi}, and \eqref{eq:W{-12,}_nablaAinfty_varrho_(b)_Phi} are analogous to that of estimate  \eqref{eq:W{-1,2}_bound_a_times_dAinfty_b}.  This completes Step~\ref{step:W-12_estimates_for_a_times_dAinfty_b_and_a_times_nablaAinfty_varphi_and_Phi_times_nabla_Ainfty_ varphi_nablaAinfty_varrho_(b)_Phi}.
\end{step}

\begin{step}[$W^{-1,2}$ estimates for $d_{A_1} a \times b$ and $\nabla_{A_1} a \times \varphi$ and $b\times\nabla_{A_1} \Phi$ and $\nabla_{A_1} \Phi \times \varphi$]
\label{step:W-12_estimates_for_dAinfty_a_times_b_and_nablaAinfty_a_times_varphi_and_varrho(b)_nablaAinfty
_Phi_and_nablaAinfty_Phi_times_varphi}
We claim that
\begin{subequations}
\begin{align}
\label{eq:W{-1,2}_bound_dAinfty_a_times_b}
\|d_{A_1} a \times b\|_{W_{A_1}^{-1,2}(X)} 
&\leq
z\|a\|_{W_{A_1}^{1,p}(X)} \|b\|_{W^{1,2}_{A_1}(X)},
\\
\label{eq:W{-1,2}_bound_nablaAinfty_a_times_varphi}
\|\nabla_{A_1} a \times \varphi\|_{W_{A_1}^{-1,2}(X)} 
&\leq
z\|a\|_{W_{A_1}^{1,p}(X)} \|\varphi\|_{W^{1,2}_{A_1}(X)},
\\
\label{eq:W{-1,2}_bound_nablaAinfty_Phi_times_b}
\| b\times \nabla_{A_1} \Phi\|_{W_{A_1}^{-1,2}(X)} 
&\leq
z\|\Phi\|_{W_{A_1}^{1,p}(X)} \|b\|_{W^{1,2}_{A_1}(X)},
\\
\label{eq:W{-1,2}_bound_nablaAinfty_Phi_times_varphi}
\|\nabla_{A_1} \Phi \times \varphi\|_{W_{A_1}^{-1,2}(X)} 
&\leq
z\|\Phi\|_{W_{A_1}^{1,p}(X)} \|\varphi\|_{W^{1,2}_{A_1}(X)},
\end{align}
\end{subequations}
where $z = z(g,G,p) \in [1,\infty)$.

We first prove estimate \eqref{eq:W{-1,2}_bound_dAinfty_a_times_b}. Again, we use duality together with the H\"{o}lder inequality, first with exponents $r, r'$ satisfying $1/r+1/r'=1$ and second with exponents $s,r$ satisfying $1/p+1/r \leq 1/r'$. We estimate
\begin{align*}
\| d_{A_1}a \times b\|_{W_{A_1}^{-1,2}(X)} 
&= 
\sup_{0<\|\beta\|_{W_{A_1}^{1,2}(X)}\leq 1} (d_{A_1} a\times b,\beta)_{L^2(X)} 
\\
&\leq 
\|d_{A_1} a \times b\|_{L^{r'}(X)} \sup_{0<\|\beta\|_{W_{A_1}^{1,2}(X)}\leq 1} \|\beta\|_{L^r(X)}
\end{align*}
We separately consider the cases $d\geq 3$ and $d=2$.

\setcounter{case}{0}
\begin{case}[$d\geq 3$]
\label{case:d_geq_3}
As in Step~\ref{step:W-12_estimates_for_a_times_dAinfty_b_and_a_times_nablaAinfty_varphi_and_Phi_times_nabla_Ainfty_ varphi_nablaAinfty_varrho_(b)_Phi} for the two cases with $d\geq 3$, we choose $r = 2^* = 2d/(d-2) \in (2, 6]$, so $r' = r/(r-1) = 2d/(d+2) \in (1, 2)$ and again we have a continuous Sobolev embedding $W^{1,2}(X) \subset L^r(X)$. Moreover,
\begin{multline*}
  1/p+1/r = 1/p+(d-2)/(2d) \leq 1/(d/2)+(d-2)/(2d)
  \\
  = 2/d+1/2-1/d = 1/d+1/2 = (2+d)/(2d) = 1/r'.
\end{multline*}
Therefore, using $1/p+1/r \leq 1/r'$,
\[
\|d_{A_1} a \times b\|_{L^{r'}(X)} \leq z \|d_{A_1} a\|_{L^p(X)}\|b\|_{L^r(X)} \leq z \|a\|_{W^{1,p}_{A_1}(X)} \|b\|_{W^{1,2}_{A_1}(X)}.
\]
This proves \eqref{eq:W{-1,2}_bound_dAinfty_a_times_b} in this case.
\end{case}

\begin{case}[$d=2$ and $p\geq 2$]
  Noting that $p < \infty$, we may choose $r \in (2,\infty)$ large enough and thus $r' = r/(r-1) \in (1,2)$ close enough to one so $1/p+1/r \leq 1/r'$. By \cite[Theorem 4.12]{AdamsFournier}, we have a continuous Sobolev embedding $W^{1,2}(X) \subset L^r(X)$ and thus estimate \eqref{eq:W{-1,2}_bound_dAinfty_a_times_b} follows just as in Case~\ref{case:d_geq_3}.
\end{case}

The proofs of estimates \eqref{eq:W{-1,2}_bound_nablaAinfty_a_times_varphi}, \eqref{eq:W{-1,2}_bound_nablaAinfty_Phi_times_b}, and \eqref{eq:W{-1,2}_bound_nablaAinfty_Phi_times_varphi} are analogous to that of estimate \eqref{eq:W{-1,2}_bound_dAinfty_a_times_b}. This completes Step~\ref{step:W-12_estimates_for_dAinfty_a_times_b_and_nablaAinfty_a_times_varphi_and_varrho(b)_nablaAinfty
_Phi_and_nablaAinfty_Phi_times_varphi}.
\end{step}

\begin{step}[$W^{-1,2}$ estimates for $a\times \Phi \times b$ and $a\times \Phi \times \varphi$ and $ a\times a \times b$ and $a\times a \times \varphi$ and $\Phi\times \Phi \times b$ and $\Phi\times \Phi \times \varphi$]
\label{step:W-12_estimates_for_a_times_Phi_times_b_and_a_times_Phi_times_varphi_and_a_times_a_times_b_and_
a_times_a_times_varphi_and_Phi_times_Phi_times_b_and_Phi_times_Phi_times_varphi}
We claim that
\begin{subequations}
\begin{align}
\label{eq:W{-1,2}_bound_a_times_Phi_times_b}
\|a\times \Phi \times b\|_{W_{A_1}^{-1,2}(X)} 
&\leq
z\|a\|_{W_{A_1}^{1,p}(X)}\|\Phi\|_{W^{1,p}_{A_1}(X)} \|b\|_{W^{1,2}_{A_1}(X)},
\\
\label{eq:W{-1,2}_bound_a_times_Phi_times_varphi}
\|a\times \Phi \times \varphi\|_{W_{A_1}^{-1,2}(X)} 
&\leq
z\|a\|_{W_{A_1}^{1,p}(X)}\|\Phi\|_{W^{1,p}_{A_1}(X)} \|\varphi\|_{W^{1,2}_{A_1}(X)},
\\
\label{eq:W{-1,2}_bound_a_times_a_times_b}
\|a\times a \times b\|_{W_{A_1}^{-1,2}(X)} 
&\leq
z\|a\|^2_{W_{A_1}^{1,p}(X)} \|b\|_{W^{1,2}_{A_1}(X)},
\\
\label{eq:W{-1,2}_bound_a_times_a_times_varphi}
\|a\times a \times \varphi\|_{W_{A_1}^{-1,2}(X)} 
&\leq
z\|a\|^2_{W_{A_1}^{1,p}(X)} \|\varphi\|_{W^{1,2}_{A_1}(X)},
\\
\label{eq:W{-1,2}_bound_Phi_times_Phi_times_b}
\|\Phi\times \Phi \times b\|_{W_{A_1}^{-1,2}(X)} 
&\leq
z\|\Phi\|^2_{W^{1,p}_{A_1}(X)} \|b\|_{W^{1,2}_{A_1}(X)},
\\
\label{eq:W{-1,2}_bound_Phi_times_Phi_times_varphi}
\|\Phi\times \Phi \times \varphi\|_{W_{A_1}^{-1,2}(X)} 
&\leq
z\|\Phi\|^2_{W^{1,p}_{A_1}(X)} \|\varphi\|_{W^{1,2}_{A_1}(X)},
\end{align}
\end{subequations}
where $z = z(g,G,p) \in [1,\infty)$.

We first consider the estimate \eqref{eq:W{-1,2}_bound_a_times_Phi_times_b}. We claim that we can always find $r \in (2,\infty)$, with $r' = r/(r-1) \in (1,2)$, and $s \in (1,\infty)$ such that $1/s + 1/r \leq 1/r'$, and such that we have continuous Sobolev embeddings, $W^{1,p}(X) \subset L^{2s}(X)$ and $W^{1,2}(X) \subset L^r(X)$. Assuming that claim, we have
\begin{align*}
\|a\times \Phi \times b\|_{W_{A_1}^{-1,2}(X)} 
&= 
\sup_{0<\|\beta\|_{W_{A_1}^{1,2}(X)}\leq 1} (a\times \Phi \times b,\beta)_{L^2(X)} 
\\
&\leq \|a\times \Phi \times b\|_{L^{r'}(X)} \sup_{0<\|\beta\|_{W_{A_1}^{1,2}(X)}\leq 1} \|\beta\|_{L^r(X)}
\\
&\leq z\|a\times \Phi\|_{L^s(X)}\|b\|_{L^r(X)} \sup_{0<\|\beta\|_{W_{A_1}^{1,2}(X)}\leq 1} \|\beta\|_{W_{A_1}^{1,2}(X)}
\\
&\leq z\|a\|_{L^{2s}(X)} \|\Phi\|_{L^{2s}(X)}\|b\|_{L^r(X)}
\\
&\leq z\|a\|_{W_{A_1}^{1,p}(X)} \|\Phi\|_{W_{A_1}^{1,p}(X)}\|b\|_{W_{A_1}^{1,2}(X)}, 
\end{align*}
and thus yielding \eqref{eq:W{-1,2}_bound_a_times_Phi_times_b}. Our task now is to verify our claim that such $r,s$ always exist. We separately consider the cases $d\geq 3$ with $p\geq d$ or $d/2 \leq p <d$ and $d=2$ with $p \geq 2$.

\setcounter{case}{0}
\begin{case}[$d\geq 3$ and $d/2 \leq  p< d$]
\label{case:d_geq_3_p_lessthan_d}
We choose $r = 2^* = 2d/(d-2) \in (2, \infty)$, so $r' = r/(r-1) = 2d/(d+2) \in (1, 2)$ and choose $s \in (1,\infty)$ by setting $2s= p^* = dp/(d-p)\in (2, \infty)$.  By \cite[Theorem 4.12]{AdamsFournier}, we have continuous Sobolev embeddings, $W^{1,2}(X) \subset L^r(X)$ and $W^{1,p}(X) \subset L^{2s}(X)$. Moreover,
\begin{multline*}
  1/s + 1/r = 2(d-p)/(dp) + (d-2)/(2d) = 2/p-2/d + 1/2-1/d
  \\
  = 2/p + 1/2 - 3/d \leq 4/d + 1/2 - 3/d = 1/d + 1/2 = 1/r',
\end{multline*}
using $p\geq d/2$, which we assume by hypothesis, and thus $1/s + 1/r \leq 1/r'$. This justifies our choice of $r,s$ in this case.
\end{case}

\begin{case}[$d\geq 3$ and $p \geq d$]
We again choose $r = 2^* = 2d/(d-2) \in (2, \infty)$, so $r' = r/(r-1) = 2d/(d+2) \in (1,2)$, and choose $s\in (1, \infty)$ large enough that $1/s + 1/r \leq 1/r'$. By \cite[Theorem 4.12]{AdamsFournier}, we have again have continuous Sobolev embeddings, $W^{1,2}(X) \subset L^r(X)$ and $W^{1,p}(X) \subset L^{2s}(X)$. This justifies our choice of $r,s$ in this case.
\end{case}

\begin{case}[$d=2$ and $p\geq 2$]
We choose $r \in (2,\infty)$ large enough that $r' = r/(r-1) \in (1,2)$ is close enough to one and choose $s\in (1, \infty)$ large enough that $1/s+ 1/r \leq 1/r'$.  By \cite[Theorem 4.12]{AdamsFournier}, we still have continuous Sobolev embeddings, $W^{1,2}(X) \subset L^r(X)$ and $W^{1,p}(X) \subset L^{2s}(X)$. This justifies our choice of $r,s$ in this case.
\end{case}

The proofs of estimates \eqref{eq:W{-1,2}_bound_a_times_Phi_times_varphi}, \eqref{eq:W{-1,2}_bound_a_times_a_times_b},
\eqref{eq:W{-1,2}_bound_a_times_a_times_varphi}, \eqref{eq:W{-1,2}_bound_Phi_times_Phi_times_b}, and
\eqref{eq:W{-1,2}_bound_Phi_times_Phi_times_varphi} are analogous to that of estimate \eqref{eq:W{-1,2}_bound_a_times_Phi_times_b}. This completes Step~\ref{step:W-12_estimates_for_a_times_Phi_times_b_and_a_times_Phi_times_varphi_and_a_times_a_times_b_and_
a_times_a_times_varphi_and_Phi_times_Phi_times_b_and_Phi_times_Phi_times_varphi}.
\end{step}

\begin{step}[$W^{-1,2}$ estimates for $F_{A_1} \times b$ and $m\varphi$]
\label{step:W-12_estimates_for_FAinfty_times_b_and_m_varphi}
We claim that
\begin{subequations}
\begin{align}
\label{eq:W{-1,2}_bound_FAinfty_times_b}
\|F_{A_1} \times b\|_{W_{A_1}^{-1,2}(X)} 
&\leq
z\|F_{A_1}\|_{C(X)}\|b\|_{W^{1,2}_{A_1}(X)},
\\
\label{eq:W{-1,2}_bound_m_varphi}
\|m\varphi\|_{W_{A_1}^{-1,2}(X)} 
&\leq
\|m\|_{C(X)} \|\varphi\|_{W^{1,2}_{A_1}(X)},
\end{align}
\end{subequations}
where $z = z(g,G,p) \in [1,\infty)$.

We consider first the estimate \eqref{eq:W{-1,2}_bound_FAinfty_times_b}. Using duality, together with the Cauchy--Schwarz inequality, we estimate
\begin{align*}
\|F_{A_1}\times b\|_{W_{A_1}^{-1,2}(X)} 
&= 
\sup_{0<\|\beta\|_{W_{A_1}^{1,2}(X)}\leq 1} (F_{A_1}\times b,\beta)_{L^2(X)} 
\\
&\leq 
z\|F_{A_1}\|_{C(X)} \|b\|_{L^2(X)} \sup_{0<\|\beta\|_{W_{A_1}^{1,2}(X)}\leq 1} \|\beta\|_{L^2(X)}
\\
&\leq z\|F_{A_1}\|_{C(X)} \|b\|_{W_{A_1}^{1,2}(X)},
\end{align*}
as required. Estimate \eqref{eq:W{-1,2}_bound_m_varphi} is proved in the same way.
This completes Step~\ref{step:W-12_estimates_for_FAinfty_times_b_and_m_varphi}.
\end{step}

The estimates obtained in each of the preceding steps verify Claim~\ref{claim:Analytic_difference_extended_Hessian_operator_and_Laplacian_W12_to_W-12}.
\end{proof}

Next, we apply Proposition \ref{prop:fredbos} with $k=-1$ and $p=2$ to give

\begin{cor}[Fredholm and index zero properties of the extended Hessian operator for the boson coupled Yang--Mills energy function on a Coulomb-gauge slice]
\label{cor:extended_fredbos}
Assume the hypotheses of Proposition \ref{prop:fredbos}. If $(A_\infty,\Phi_\infty)$ is a $C^\infty$ pair on $(P,E)$, then the following operator is Fredholm with index zero,
\begin{multline*}
\sM'(A_\infty,\Phi_\infty): \Ker d_{A_\infty,\Phi_\infty}^* \cap W_{A_1}^{1,2}(X; \Lambda^1 \otimes \ad P \oplus E)
\\
\to \Ker d_{A_\infty,\Phi_\infty}^* \cap W_{A_1}^{-1,2}(X; \Lambda^1 \otimes \ad P \oplus E),
\end{multline*}
where is the Hessian operator $\sM'(A_\infty,\Phi_\infty)$ is defined in \eqref{eq:Relation_Hessian_and_Hessian_operator}, with explicit schematic expression in \eqref{eq:Hessian_boson_energy_function}.
\end{cor}

Recall from \eqref{eq:Gradient_map_boson_energy_function_slice} that
\begin{multline*}
\hat\sM \equiv \Pi_{A_\infty,\Phi_\infty}\sM: \Ker d_{A_\infty,\Phi_\infty}^* \cap W_{A_1}^{1,p}(X; \Lambda^1 \otimes \ad P \oplus E) \ni (A,\Phi)
\\
\mapsto \Pi_{A_\infty,\Phi_\infty}\sM(A,\Phi) \in \Ker d_{A_\infty,\Phi_\infty}^* \cap W_{A_1}^{-1,p}(X; \Lambda^1 \otimes \ad P \oplus E)
\end{multline*}
is the gradient at a $W^{1,p}$ pair $(A,\Phi)$ of the restriction,
\[
\sE: (A_\infty,\Phi_\infty)+\Ker d_{A_\infty,\Phi_\infty}^* \cap W_{A_1}^{1,p}(X; \Lambda^1 \otimes \ad P \oplus E) \to \RR,
\]
of the boson coupled Yang--Mills energy function \eqref{eq:Boson_Yang--Mills_energy_function} to the Coulomb-gauge slice through a $C^\infty$ pair $(A_\infty,\Phi_\infty)$ on $(P,E)$, where we recall that $\Pi_{A_\infty,\Phi_\infty}$ is the $L^2$-orthogonal projection \eqref{eq:L2-orthogonal_projection_onto_slice} onto the Coulomb-gauge slice. The Hessian operators are related by
\[
 \hat\sM'(A_\infty,\Phi_\infty) = \left(\Pi_{A_\infty,\Phi_\infty}\sM\right)'(A_\infty,\Phi_\infty) =  \Pi_{A_\infty,\Phi_\infty}\sM'(A_\infty,\Phi_\infty).
\]
Because the operator $\sM'(A_\infty,\Phi_\infty)$ preserves the Coulomb-gauge condition, then
\[
\Pi_{A_\infty,\Phi_\infty}\sM'(A_\infty,\Phi_\infty) = \sM'(A_\infty,\Phi_\infty),
\]
although for arbitrary $W^{1,p}$ pairs $(A,\Phi)$, we may have $\Pi_{A_\infty,\Phi_\infty}\sM'(A,\Phi) \neq \sM'(A,\Phi)$.

Recall from Lemma~\ref{lem:Analyticity_of_extension_to_boson_YM_augmented_Hessian_operator} that for any $W^{1,p}$ pair on $(P,E)$, the operator
\[
  \sM_1(A,\Phi) \in \sL(W^{1,2},W^{-1,2})
\]
was constructed to be the extension of the operator
\[
  \sM'(A,\Phi) \in \sL(W^{1,p},W^{-1,p}),
\]
where we as usual abbreviate $W^{\pm 1, p} = W_{A_1}^{\pm 1,p}(X; \Lambda^1 \otimes \ad P \oplus E)$ and similarly for $W^{\pm 1,2}$. Therefore,
\begin{multline}
\label{eq:Defn_hatM1}  
\hat\sM_1(A,\Phi) := \Pi_{A_\infty,\Phi_\infty}\sM_1(A,\Phi)
\\
\in \sL\left(\Ker d_{A_\infty,\Phi_\infty}^* \cap W^{1,2}, \Ker d_{A_\infty,\Phi_\infty}^* \cap W^{-1,2}\right)
\end{multline}
is the corresponding extension of 
\[
\hat\sM'(A,\Phi) \in \sL\left(\Ker d_{A_\infty,\Phi_\infty}^* \cap W^{1,p}, \Ker d_{A_\infty,\Phi_\infty}^* \cap W^{-1,p}\right).
\]
Note that
\begin{multline*}
  \Pi_{A_\infty,\Phi_\infty}\sM_1(A,\Phi) = \Pi_{A_\infty,\Phi_\infty}\left(\sM_1(A,\Phi) + d_{A_\infty,\Phi_\infty}d_{A_\infty,\Phi_\infty}^*\right)
  \\
  \text{on }
  \Ker d_{A_\infty,\Phi_\infty}^* \cap W^{-1,2}.
\end{multline*}
As a consequence of Lemma~\ref{lem:Analyticity_of_extension_to_boson_YM_augmented_Hessian_operator}, we have

\begin{cor}[Analyticity of the extended Hessian operator map for the boson coupled Yang--Mills energy function on a Coulomb-gauge slice]
\label{cor:Analyticity_extended_Hessian_boson_energy_function_slice}
Assume the hypotheses of Theorem \ref{mainthm:Lojasiewicz-Simon_gradient_inequality_boson_Yang--Mills_energy_function} and let $(A_\infty,\Phi_\infty)$ be a $C^\infty$ pair on $(P,E)$. Then the following map is well-defined and analytic,
\begin{multline*}
\Ker d_{A_\infty,\Phi_\infty}^* \cap W^{1, p} \ni (a,\phi)
\\
\mapsto \hat\sM_1(A_\infty+a,\Phi_\infty+\phi) \in \sL\left(\Ker d_{A_\infty,\Phi_\infty}^* \cap W^{1,2}, \Ker d_{A_\infty,\Phi_\infty}^* \cap W^{-1,2}\right).
\end{multline*}
\end{cor}

\begin{proof}
The conclusion follows from Lemma \ref{lem:Analyticity_of_extension_to_boson_YM_augmented_Hessian_operator} and the fact that the $L^2$-orthogonal projection operators $\Pi_{A_\infty,\Phi_\infty} \in \sL(W^{-1,2},\Ker d_{A_\infty,\Phi_\infty}^* \cap W^{-1,2})$ are bounded.
\end{proof}

As we noted prior to the statement of Lemma~\ref{lem:Analyticity_of_extension_to_boson_YM_augmented_Hessian_operator}, it would suffice for our application of Theorem \ref{mainthm:Lojasiewicz-Simon_gradient_inequality2} to prove Corollary \ref{maincor:Lojasiewicz-Simon_L2_gradient_inequality_boson_Yang-Mills_energy_function} for the map in Corollary \ref{cor:Analyticity_extended_Hessian_boson_energy_function_slice} to be continuous on a $W^{1,q}$ open neighborhood of the origin.

Finally, we shall need the following an extension of Lemma \ref{lem:gauged} that provides $W^{1,2}$ estimates for $u(a,\phi)$, thus stronger than the existing $W^{1,p'}$ estimates when $p=d/2$ and $p'=d/(d-2)<2$ for $d\geq 5$.

\begin{lem}[Estimate for the action of a $W^{2,q}$ gauge transformation]
\label{lem:L2_gauged}
Assume the hypotheses of Theorem \ref{mainthm:Lojasiewicz-Simon_gradient_inequality_boson_Yang--Mills_energy_function}. Then there is a constant $C = C(g,G,p) \in [1,\infty)$ with the following significance. If $u\in \Aut^{2,q}(P)$ and $(a,\phi) \in W^{1,2}_{A_1}(X;\Lambda^1\otimes\ad P)$, then
\begin{align*}
\|u(a,\phi)\|_{W_{A_1}^{1,2}(X)}
  &\leq C\left(1 + \|u\|_{W_{A_1}^{2,p}(X)}\right) \|(a,\phi)\|_{W_{A_1}^{1,2}(X)},
  \\
  \|(a,\phi)\|_{W_{A_1}^{1,2}(X)}
&\leq C\left(1 + \|u\|_{W_{A_1}^{2,p}(X)}\right) \|u(a,\phi)\|_{W_{A_1}^{1,2}(X)}.
\end{align*}
\end{lem}

\begin{proof}
We recall that $u(a,\phi) = (u(a),u(\phi)) = (u^{-1}au, u^{-1}\phi)$ from the proof of Lemma \ref{lem:gauged} and so
\[
\nabla_{A_1}(u(a)) = -u^{-1}(\nabla_{A_1}u)u^{-1} a u + u^{-1}(\nabla_{A_1}a)u + u^{-1}a (\nabla_{A_1}u).
\]
We make the

\begin{claim}
\label{claim:Continuous_Sobolev_multiplication_W1p_times_W12_to_L2}
Let $(X,g)$ be a closed, smooth Riemannian manifold of dimension $d \geq 2$. If $p \in [d/2,\infty)$, with $p \geq 2$ when $d=2$, then there is a continuous Sobolev multiplication map,
\begin{equation}
\label{eq:Continuous_Sobolev_multiplication_W1p_times_W12_to_L2}
W^{1,p}(X) \times  W^{1,2}(X) \to L^2(X).
\end{equation}
\end{claim}

Given Claim \ref{claim:Continuous_Sobolev_multiplication_W1p_times_W12_to_L2} and the preceding expression for $\nabla_{A_1}(u(a))$, we have
\begin{align*}
  {}&\|\nabla_{A_1}(u(a))\|_{L^2(X)}
  \\
  &\quad \leq \|u^{-1}(\nabla_{A_1}u)u^{-1} a u\|_{L^2(X)} + \|u^{-1}(\nabla_{A_1}a)u\|_{L^2(X)} + \|u^{-1}a (\nabla_{A_1}u)\|_{L^2(X)}
  \\
  &\quad \leq \|\nabla_{A_1}u\|_{L^s(X)} \|a\|_{L^r(X)} + \|\nabla_{A_1}a\|_{L^2(X)} + \|a\|_{L^r(X)} \|\nabla_{A_1}u\|_{L^s(X)}
    \\
&\quad \leq z\|\nabla_{A_1}u\|_{W_{A_1}^{1,p}(X)}\|a\|_{W_{A_1}^{1,2}(X)} + \|\nabla_{A_1}a\|_{L^2(X)},
\end{align*}
for $r, s \in (2,\infty)$ and continuous Sobolev embeddings as in the proof of Claim \ref{claim:Continuous_Sobolev_multiplication_W1p_times_W12_to_L2}, so that
\[
\|u(a)\|_{W_{A_1}^{1,2}(X)} \leq z\left(1 + \|\nabla_{A_1}u\|_{W_{A_1}^{1,p}(X)}\right)\|a\|_{W_{A_1}^{1,2}(X)},
\]
for $z = z(g,p) \in [1,\infty)$. We have the analogous estimate for $\|u(\phi)\|_{W_{A_1}^{1,2}(X)}$ and the combination yields the required estimate for $\|u(a,\phi)\|_{W_{A_1}^{1,2}(X)}$. 

\begin{proof}[Proof of Claim \ref{claim:Continuous_Sobolev_multiplication_W1p_times_W12_to_L2}]
The multiplication, $L^r(X)\times L^s(X) \to L^2(X)$, is continuous for $r,s \in (2,\infty)$ provided $1/r + 1/s \leq 1/2$. If these exponents $r,s$ also yield continuous Sobolev embeddings,
\[
W^{1,p}(X) \subset L^s(X) \quad\text{and}\quad W^{1,2}(X) \subset L^r(X),
\]
then we obtain the desired continuous Sobolev multiplication map \eqref{eq:Continuous_Sobolev_multiplication_W1p_times_W12_to_L2}. To confirm the existence of suitable exponents, $r,s$, we shall separately consider the cases $d\geq 3$ with $p \geq d$, and $d \geq 3$ with $d/2 \leq p <d$, and $d=2$ with $p \geq 2$.

\setcounter{case}{0}
\begin{case}[$d\geq 3$ and $p\geq d$]
\label{case:d_geq_3_p_geq_d}
We choose $r = 2^* = 2d/(d-2) \in (2, \infty)$ and $s=d \in (2,\infty)$.  By \cite[Theorem 4.12]{AdamsFournier}, the preceding Sobolev embeddings hold and additionally
\[
1/r + 1/s =  (d-2)/(2d) + 1/d = 1/2,
\]
so the proof of \eqref{eq:Continuous_Sobolev_multiplication_W1p_times_W12_to_L2} is completed in this case.
\end{case}

\begin{case}[$d \geq 3$ and $d/2 \leq p <d$]
\label{case:d_geq_3_p_less_than_d}
We choose $r = 2^* = 2d/(d-2) \in (2, \infty)$ and $s= p^* = (dp)/(d-p) \in (2,\infty)$.  By \cite[Theorem 4.12]{AdamsFournier}, the preceding Sobolev embeddings hold and additionally
\[
1/r + 1/s =  (d-2)/(2d) + (d-p)/(dp) = 1/2 - 1/d + 1/p- 1/d \leq 1/2,
\]
using $p \geq d/2$, which we assume by hypothesis in Corollary \ref{maincor:Lojasiewicz-Simon_L2_gradient_inequality_boson_Yang-Mills_energy_function}. Therefore, the proof of \eqref{eq:Continuous_Sobolev_multiplication_W1p_times_W12_to_L2} is complete in this case.
\end{case}

\begin{case}[$d=2$ and $p\geq 2$]
\label{case:d2_p_geq_2}
We choose any $r, s\in (2,\infty)$ sufficiently large so that $1/s + 1/r \leq 1/2$.  By \cite[Theorem 4.12]{AdamsFournier}, the preceding Sobolev embeddings hold and the proof of \eqref{eq:Continuous_Sobolev_multiplication_W1p_times_W12_to_L2} is complete in this case.
\end{case}

Combining these three cases completes the proof of Claim \ref{claim:Continuous_Sobolev_multiplication_W1p_times_W12_to_L2}.
\end{proof}

Symmetry yields the corresponding estimate for $\|(a,\phi)\|_{W_{A_1}^{1,2}(X)}$ and this completes the proof of Lemma \ref{lem:L2_gauged}.
\end{proof}

\section{Completion of the proofs of the the $W^{-1,2}$ gradient inequalities}
\label{sec:Completion_proof_Lojasiewicz-Simon_L2_gradient_inequality_boson_Yang--Mills_energy_function}
We can now proceed to the

\begin{proof}[Proof of Corollary \ref{maincor:Lojasiewicz-Simon_L2_gradient_inequality_boson_Yang-Mills_energy_function}]
We proceed by verifying the hypotheses of Theorem \ref{mainthm:Lojasiewicz-Simon_gradient_inequality2}. As in the proof of Theorem \ref{mainthm:Lojasiewicz-Simon_gradient_inequality_boson_Yang--Mills_energy_function} (which relied on Theorem \ref{mainthm:Lojasiewicz-Simon_gradient_inequality}), the proof of Corollary \ref{maincor:Lojasiewicz-Simon_L2_gradient_inequality_boson_Yang-Mills_energy_function} is divided into two cases. In the first case, we consider the simpler situation where the pair $(A,\Phi)$ is in Coulomb gauge relative to the critical point $(A_\infty, \Phi_\infty)$ and in the second case, we consider the general situation.

\setcounter{case}{0}
\begin{case}[$(A,\Phi)$ in Coulomb gauge relative to $(A_\infty, \Phi_\infty)$]
\label{case:APhi_Coulomb_gauge_wrt_APhi_infty_L2}
By hypothesis, the $W^{1,q}$ pair $(A_\infty, \Phi_\infty)$ is a critical point for the function $\sE$ in \eqref{eq:Boson_Yang--Mills_energy_function}. By the regularity Theorem~\ref{thm:Parker_1982_5-3}, there exists a $W^{2,q}$ gauge transformation $u_\infty$ such that $u_\infty(A_\infty,\Phi_\infty)$ is a $C^\infty$ pair; we take $u_\infty$ to be the identity if $(A_\infty,\Phi_\infty)$ is already $C^\infty$. In particular, $u_\infty(A_\infty, \Phi_\infty)$ is a $W^{2,q}$ pair and $u_\infty(A, \Phi)$ is in Coulomb gauge relative to $u_\infty(A_\infty, \Phi_\infty)$. As in the proof of Theorem \ref{mainthm:Lojasiewicz-Simon_gradient_inequality_boson_Yang--Mills_energy_function} in Section \ref{subsec:Completion_proof_Lojasiewicz-Simon_gradient_inequality_boson_Yang--Mills_energy_function}, we have the Banach spaces,
\begin{align*}
\sX &= \Ker\left(d_{u_\infty(A_\infty, \Phi_\infty)}^*: W_{A_1}^{1,p}(X; \Lambda^1 \otimes \ad P \oplus E)
\to L^p(X;\ad P) \right),
\\
\tilde\sX &= \Ker\left(d_{u_\infty(A_\infty, \Phi_\infty)}^*: W_{A_1}^{-1,p}(X; \Lambda^1 \otimes \ad P \oplus E)
\to W_{A_1}^{-2,p}(X;\ad P) \right),
\end{align*}
and in addition, in order to apply Theorem \ref{mainthm:Lojasiewicz-Simon_gradient_inequality2}, we also choose Hilbert spaces,
\begin{align*}
\sG &:= \Ker\left(d_{u_\infty(A_\infty, \Phi_\infty)}^*: W_{A_1}^{1,2}(X; \Lambda^1 \otimes \ad P \oplus E)
\to L^2(X;\ad P) \right),
\\
\tilde\sG &:= \Ker\left(d_{u_\infty(A_\infty, \Phi_\infty)}^*: W_{A_1}^{-1,2}(X; \Lambda^1 \otimes \ad P \oplus E)
\to W_{A_1}^{-2,2}(X;\ad P) \right).
\end{align*}
Hence, $\sX \subset \tilde\sX$ and $\sG \subset \tilde\sG$ are continuous embeddings of Banach spaces. Moreover,
\[
\sX^* = \Ker\left(d_{u_\infty(A_\infty, \Phi_\infty)}^*: L^{p'}(X;\ad P) \to W_{A_1}^{-1,p'}(X; \Lambda^1 \otimes \ad P \oplus E) \right),
\]
and $\tilde\sX \subset \sX^*$is a continuous embedding of Banach spaces, just as in Section \ref{subsec:Completion_proof_Lojasiewicz-Simon_gradient_inequality_boson_Yang--Mills_energy_function}. Since $p \geq 2$ by hypothesis of Corollary \ref{maincor:Lojasiewicz-Simon_L2_gradient_inequality_boson_Yang-Mills_energy_function}, then $\tilde\sX \subset \tilde \sG$ is also a continuous embedding of Banach spaces and the compositions
\[
\sX \subset \sG \subset \tilde\sG \quad \text{and}\quad \sX \subset \tilde\sX \subset \tilde\sG,
\]
induce the same embedding. By Proposition \ref{prop:analyticbos}, the function $\sE: x_\infty+\sX\to\RR$ is analytic, where as in Section \ref{subsec:Completion_proof_Lojasiewicz-Simon_gradient_inequality_boson_Yang--Mills_energy_function} we abbreviate $x_\infty = u_\infty(A_\infty,\Phi_\infty)$. By Corollary~\ref{cor:Analyticity_extended_Hessian_boson_energy_function_slice}, the Hessian operator map,
\[
x_\infty + \sX \ni x \mapsto \hat\sM'(x)\in \sL(\sX,\tilde\sX),
\]
extends to a continuous map,
\[
x_\infty+\sX \ni x \mapsto \hat\sM_1(x)\in \sL(\sG, \tilde\sG).
\]
By Proposition \ref{prop:fredbos}, the operator $\hat\sM'(x_\infty)$ is Fredholm with index zero and by Corollary~\ref{cor:extended_fredbos}, the extension $\hat\sM_1(x_\infty)$ is also Fredholm with index zero. Hence, all the hypotheses of Theorem \ref{mainthm:Lojasiewicz-Simon_gradient_inequality2} are fulfilled (see also Remark \ref{rmk:Lojasiewicz-Simon_gradient_inequality2_Hilbert}) and so there are constants $Z'\in (0,\infty)$ and $\sigma' \in (0,1]$ and $\theta \in [1/2, 1)$ (depending on $(A_1, \Phi_1)$, and  $u_\infty(A_\infty,\Phi_\infty)$, and $g$, $G$, $p$, $P$) such that if
\begin{equation}
\label{eq:Lojasiewicz-Simon_L2_gradient_inequality_boson_Yang--Mills_pair_neighborhood_gauged}
\|u_\infty(A,\Phi) -  u_\infty(A_\infty, \Phi_\infty)\|_{W_{A_1}^{1,p}(X)} < \sigma',
\end{equation}
then
\begin{equation}
\label{eq:L2_gauged_LS}
|\sE(u_\infty(A,\Phi)) - \sE(u_\infty(A_\infty, \Phi_\infty))|^\theta \leq Z'\|\hat\sM(u_\infty(A,\Phi))\|_{W^{-1,2}(X)}.
\end{equation}
As in the proof of Theorem \ref{mainthm:Lojasiewicz-Simon_gradient_inequality_boson_Yang--Mills_energy_function}, Lemma \ref{lem:gauged} gives
\[
\|u_\infty(A,\Phi) - u_\infty(A_\infty, \Phi_\infty)\|_{W_{A_1}^{1,p}(X)}
\leq
C_1 \|(A,\Phi) - (A_\infty, \Phi_\infty)\|_{W_{A_1}^{1,p}(X)}.
\]
for $C_1 = C_1(g,G,p,u_\infty) = C_1(A_\infty,\Phi_\infty,g,G,p) \in (0,1]$. 
Therefore, setting $\sigma := C_1^{-1}\sigma'$, we see that if $(A,\Phi)$ obeys the {\L}ojasiewicz--Simon neighborhood condition \eqref{eq:Lojasiewicz-Simon_gradient_inequality_boson_Yang--Mills_pair_neighborhood}, namely
\[
\|(A,\Phi) - (A_\infty, \Phi_\infty)\|_{W_{A_1}^{1,p}(X)} <\sigma,
\]
then \eqref{eq:Lojasiewicz-Simon_L2_gradient_inequality_boson_Yang--Mills_pair_neighborhood_gauged} holds and thus also \eqref{eq:L2_gauged_LS}. Moreover,
\begin{align*}
  \|\hat\sM(u_\infty(A,\Phi))\|_{W^{-1, 2}_{A_1}(X)} &\leq C_2 \|\hat\sM(A,\Phi)\|_{W^{-1, 2}_{A_1}(X)},
  \\
  \|\hat\sM(A,\Phi)\|_{W^{-1, 2}_{A_1}(X)} &\leq \|\sM(A,\Phi)\|_{W^{-1, 2}_{A_1}(X)},
\end{align*}
just as in the proof of the corresponding $W^{-1, p}$ estimate in the proof of Theorem \ref{mainthm:Lojasiewicz-Simon_gradient_inequality_boson_Yang--Mills_energy_function}, but now using Lemma~\ref{lem:L2_gauged} with constant $C_2 = C_2(g,G,u_\infty) = C_2(A_\infty,\Phi_\infty,g,G) \in (0,1]$.

Increasing $C_2(A_\infty, \Phi_\infty,g,G,p) \in (0,1]$ if necessary to give $C_2\geq C_1$ and substituting the preceding inequality into \eqref{eq:L2_gauged_LS} and following the same steps as in the proof of Theorem \ref{mainthm:Lojasiewicz-Simon_gradient_inequality_boson_Yang--Mills_energy_function} yields
\begin{equation*}
|\sE(A,\Phi) - \sE(A_\infty, \Phi_\infty)|^\theta\leq Z'C_2\|\sM(A,\Phi)\|_{W^{-1, 2}_{A_1}(X)}.
\end{equation*}
Hence, the {\L}ojasiewicz--Simon gradient inequality \eqref{eq:Lojasiewicz-Simon_gradient_inequality_boson_Yang-Mills_energy_function_Wminus12} holds for the pairs $(A, \Phi)$ and $(A_\infty, \Phi_\infty)$ with constants $(Z, \theta, \sigma)$, where $Z := C_2Z'$.
\end{case}

\begin{case}[$(A,\Phi)$ not in Coulomb gauge relative to $(A_\infty, \Phi_\infty)$]
\label{case:APhi_notin_Coulomb_gauge_wrt_APhi_infty_L2}
Let
\begin{align*}
  \zeta &= \zeta(A_1, A_\infty, \Phi_\infty, g, G,p,q)\in (0,1] \quad\text{and}
  \\
  N &= N(A_1,A_\infty, \Phi_\infty, g, G,p,q) \in [1,\infty)
\end{align*}
denote the constants in Theorem~\ref{mainthm:Feehan_proposition_3-4-4_Lp_pairs} and choose $\zeta_1 \in (0,\zeta]$ small enough that $2N\zeta_1 < \sigma_1$, where we now use $\sigma_1$ to denote the {\L}ojasiewicz--Simon constant from Case \ref{case:APhi_Coulomb_gauge_wrt_APhi_infty_L2}.  If $(A,\Phi)$ obeys
\[
\|(A,\Phi) - (A_\infty, \Phi_\infty)\|_{W_{A_1}^{1,p}(X)} < \zeta_1,
\]
then Theorem~\ref{mainthm:Feehan_proposition_3-4-4_Lp_pairs} provides $u \in \Aut^{2,q}(P)$, depending on the pair $(A, \Phi)$, such that
\begin{gather*}
d_{A_\infty, \Phi_\infty}^*(u(A, \Phi) - (A_\infty, \Phi_\infty)) = 0,
\\
\|u(A,\Phi) - (A_\infty, \Phi_\infty)\|_{W_{A_1}^{1,p}(X)} < 2N\zeta_1 < \sigma.
\end{gather*}
By applying Case \ref{case:APhi_Coulomb_gauge_wrt_APhi_infty_L2} to the pairs $u(A,\Phi)$ and $(A_\infty, \Phi_\infty)$, we obtain
\[
|\sE(u(A,\Phi)) - \sE(A_\infty, \Phi_\infty)|^\theta \leq C_2Z'\|\sM(u(A,\Phi))\|_{W^{-1, 2}_{A_1}(X)}.
\]
Estimating as in Case \ref{case:APhi_Coulomb_gauge_wrt_APhi_infty_L2}, with $u$ replacing $u_\infty$, we see that
\[
\|\sM(u(A,\Phi))\|_{W^{-1, 2}_{A_1}(X)} \leq C_3\|\sM(A,\Phi)\|_{W^{-1, 2}_{A_1}(X)},
\]
just as in the proof of Theorem \ref{mainthm:Lojasiewicz-Simon_gradient_inequality_boson_Yang--Mills_energy_function}, where $C_3 = C_3(A_1,A_\infty, \Phi_\infty,g,G,p,q) \in (0,1]$. By combining the preceding inequalities, we obtain
\begin{align*}
|\sE(A,\Phi) - \sE(A_\infty, \Phi_\infty)|^\theta
&= |\sE(u(A,\Phi)) - \sE(A_\infty, \Phi_\infty)|^\theta
\\
&\leq C_2C_3Z'\|\sM(A,\Phi)\|_{W^{-1, 2}_{A_1}(X)}.
\end{align*}
Hence, we obtain the {\L}ojasiewicz--Simon gradient inequality \eqref{eq:Lojasiewicz-Simon_gradient_inequality_boson_Yang--Mills_energy_function} with constants $(Z, \theta, \sigma)$, where we now choose $Z = C_2C_3Z'$ and $\sigma = \zeta_1$. This completes the proof of this case.
\end{case}

This completes the proof of Corollary \ref{maincor:Lojasiewicz-Simon_L2_gradient_inequality_boson_Yang-Mills_energy_function}.
\end{proof}

\begin{proof}[Proof of Corollary \ref{maincor:Lojasiewicz-Simon_L2_gradient_inequality_fermion_Yang-Mills_energy_function}]
The argument follows \mutatis that for Corollary \ref{maincor:Lojasiewicz-Simon_L2_gradient_inequality_boson_Yang-Mills_energy_function}. 
\end{proof}

\appendix
\numberwithin{equation}{chapter}
\numberwithin{thm}{chapter}

\chapter{Fredholm and index properties of elliptic operators on Sobolev spaces}
\label{chap:Fredholm_properties_elliptic_PDEs_Sobolev_spaces}
We note the following corollary of the standard Fredholm property \cite[Lemma 1.4.5]{Gilkey2}, \cite[Theorem 19.2.1]{Hormander_v3} of elliptic pseudo-differential operators on Sobolev spaces $W^{k+m,p}(M;V)$ when\footnote{Gilkey and H\"ormander allow $m, k\in\RR$, as we could also in Theorem \ref{thm:Gilkey_1-4-5_Sobolev}, but we shall omit this refinement.} $k \in \ZZ$ and $p=2$ and which may be compared with \cite[Theorem 5.7.2]{Volpert1}. We include a detailed proof of Theorem \ref{thm:Gilkey_1-4-5_Sobolev} because we are unaware of a published reference. The statement of Theorem \ref{thm:Gilkey_1-4-5_Sobolev} both generalizes \cite[Theorem 19.2.1]{Hormander_v3} by allowing $1<p<\infty$ rather than $p=2$ and specializes \cite[Theorem 19.2.1]{Hormander_v3} by restricting to partial differential operators (thus of integer order $m$) and restricting $k$ to be an integer. In our application, we shall only need the case $m=2$, but we provide the more general version for the benefit of other applications. We leave it to the reader to consider extensions to the general case of pseudo-differential operators with real order $m$ on $W^{k,p}$ spaces with real $k$.

\begin{thm}[Fredholm property for elliptic partial differential operators on Sobolev spaces]
\label{thm:Gilkey_1-4-5_Sobolev}
(See \cite[Lemma 41.1]{Feehan_yang_mills_gradient_flow_v4}.)
Let $V$ and $W$ be finite-rank, smooth vector bundles over a closed, smooth manifold, $M$, and $k\in\ZZ$ an integer, and $p \in (1,\infty)$. If $P:C^\infty(M;V)\to C^\infty(M;W)$ is an elliptic linear partial differential operator of integer order $m\geq 1$ with $C^\infty$ coefficients, then $P:W^{k+m,p}(M;V) \to W^{k,p}(M;W)$ is a Fredholm operator with index,
\begin{align*}
\Ind P &= \dim\Ker\left(P:C^\infty(M;V)\to C^\infty(M;W)\right)
\\
&\quad - \dim\Ker\left(P^*:C^\infty(M;W^*)\to C^\infty(M;V^*)\right),
\end{align*}
where $P^*$ is the formal ($L^2$) adjoint of $P$, and has range,
\[
\Ran\left(P:W^{k+m,p}(M;V) \to W^{k,p}(M;W)\right) = (K^*)^\perp \cap W^{k,p}(M;W),
\]
where $\perp$ denotes $L^2$-orthogonal complement and
\[
K = \Ker\left(P^*:C^\infty(M;W^*)\to C^\infty(M;V^*)\right) \cong K^* \subset C^\infty(M;W).
\]
If $V = W^*$ and $P-P^*:C^\infty(M;V)\to C^\infty(M;W)$ is a differential operator of order $m-1$, then $\Ind P=  0$.
\end{thm}

\begin{rmk}[On the index of elliptic operators with scalar principal symbol]
We recall from \cite[Proposition 2.4]{Treves1} that if $E$ is a smooth complex vector bundle over a closed, smooth manifold, $M$, and $P:C^\infty(M;E) \to C^\infty(M;E)$ is an elliptic differential operator whose principal symbol is a scalar multiple of the identity (as often occurs in Geometric Analysis), then $\Ind P = 0$. However, if $P:C^\infty(M;E) \to C^\infty(M;E)$ is more generally an elliptic pseudo-differential operator, then $P$ need not have index zero  \cite[Example 2.2]{Treves1} .
\end{rmk}

We refer to Gilkey \cite{Gilkey2} and H\"ormander \cite{Hormander_v3} for the definitions (over domains $\Omega \subseteqq\RR^d$) of a differential operator $P$ of integer order $m\geq 1$ \cite[Section 1.1]{Gilkey2}, a pseudo-differential operator $P$ of order $m \in \RR$ \cite[Section 1.2.1]{Gilkey2}, the vector spaces $\Psi^m$ of pseudo-differential operators of order $m\in\RR$ and $\Psi^{-\infty} = \cap_{m\in\RR}\Psi^m$ of infinitely smoothing pseudo-differential operators \cite[Section 1.2.1]{Gilkey2}, the $L^2$-adjoint $P^*$ and composition $PQ$ of pseudo-differential operators \cite[Section 1.2.2]{Gilkey2}, the extension of the definition of pseudo-differential operators on complex-valued functions over domains in $\RR^d$ to sections of finite-rank vector bundles over domains in $\RR^d$ \cite[Section 1.2.7]{Gilkey2}, elliptic pseudo-differential operators on sections of finite-rank vector bundles over domains in $\RR^d$ \cite[Section 1.3.1]{Gilkey2}, and elliptic pseudo-differential operators on sections of finite-rank vector bundles over closed manifolds \cite[Sections 1.3.3 and 1.3.5]{Gilkey2}. Similarly, for $s \in \RR$, we refer to Gilkey \cite{Gilkey2} and H\"ormander \cite{Hormander_v3} for the definitions of Sobolev spaces $W^{s,2}(\RR^d;\CC)$ \cite[Section 1.1.3]{Gilkey2}, $W^{s,2}(\Omega;\CC)$ \cite[Section 1.2.6]{Gilkey2}, and $W^{s,2}(M;\CC)$ and $W^{s,2}(M;V)$ \cite[Sections 1.3.4 and 1.3.5]{Gilkey2}.

When $k \in \NN$ is a non-negative integer and $1 \leq p \leq \infty$, we use the standard definitions (without appealing to symbols or the symbol calculus) of Sobolev spaces, $W^{k,p}(\Omega;\RR)$ and $W_0^{k,p}(\Omega;\RR)$, in Adams and Fournier \cite[Section 3.2]{AdamsFournier} for domains $\Omega \subseteqq\RR^d$ and Aubin \cite[Section 2.3]{Aubin_1998} and Gilkey \cite[Sections 1.3.4 and 1.3.5]{Gilkey2} for their extensions to functions on closed manifolds, $W^{k,p}(M;\RR)$, and sections of finite-rank vector bundles over closed manifolds, $W^{k,p}(M;V)$. For integers $k < 0$ and $1 < p < \infty$, we follow \cite[Sections 3.9 to 3.14]{AdamsFournier} and by analogy define
\[
W^{k,p}(M;V) := (W^{-k,p'}(M;V^*))^*,
\]
where the dual H\"older exponent, $1<p'<\infty$, is defined by $1/p+1/p'=1$ and $(W^{-k,p'}(M;V^*))^*$ is the continuous dual of the Banach space, $W^{-k,p'}(M;V^*)$.

For a differential operator, $P$, of integer order $m\geq 1$, any integer $k \in \ZZ$ and $v \in W^{k+m,p}(M;V)$, one may define $Pv \in (C^\infty(M;V))^*$ in the sense of distributions \cite[Section 1.62]{AdamsFournier} by
\[
(Pv)(w) := (v, P^*w)_{L^2(M)}, \quad\forall\, w \in C^\infty(M;W^*),
\]
where the $L^2$-adjoint, $P^*:C^\infty(M;W^*) \to C^\infty(M;V^*)$, is also a differential operator of order $m \geq 1$. For integer $k \geq 0$, then $Pv$ has its usual meaning and we have $Pv \in W^{k,p}(M;V)$, with
\[
P:W^{k+m,p}(M;V) \to W^{k,p}(M;W),
\]
defining a bounded operator by definition of the Sobolev spaces, $W^{k+m,p}(M;V)$. Similarly, for integers $k \leq -m$, we may define $Pv \in W^{k,p}(M;V)$ by
\[
(Pv)(w) := (v, P^*w)_{L^2(M)}, \quad\forall\, v \in W^{k+m,p}(M;V) \text{ and } w \in W^{-k,p'}(M;W^*),
\]
noting that $P^*:W^{-k,p'}(M;W^*) \to W^{-k-m,p'}(M;V^*)$ is a bounded operator and so $P^*w \in W^{-k-m,p'}(M;V^*)$ with its usual meaning.

Finally, for $m\geq 2$ and integers $k$ in the range $-m+1 \leq k \leq -1$, we choose a Riemannian metric, $g$, on $M$, a connection, $\nabla:C^\infty(M;V) \to C^\infty(M;T^*M\otimes V)$, and form the augmented connection Laplace operator,
\[
\nabla^*\nabla+1:C^\infty(M;V)\to C^\infty(M;V).
\]
The resulting operator, $\nabla^*\nabla+1:W^{l+2,p}(M;V)\to W^{l,p}(M;V)$, is invertible for integers $l\geq 0$ and $1<p<\infty$. By duality (as above) it extends to an invertible operator, $\nabla^*\nabla+1:W^{l+2,p}(M;V)\to W^{l,p}(M;V)$, for integers $l \leq -2$ and $1<p<\infty$ and because it has divergence form, it also defines an invertible operator, $\nabla^*\nabla+1:W^{1,p}(M;V)\to W^{-1,p}(M;V)$, for $1<p<\infty$. (Observe that the bilinear map,
\[
W^{1,p}(M;V)\times W^{1,p'}(M;V) \ni (u, v)
\mapsto
((\nabla^*\nabla+1)u, v)_{L^2(X)},
\]
is continuous since
\[
| (\nabla^*\nabla+1)u, v)_{L^2(X)} | \leq \|\nabla u\|_{L^p(X)}\|\nabla v\|_{L^{p'}(X)} + \|u\|_{L^p(X)}\|v\|_{L^{p'}(X)},
\]
We therefore obtain a continuous, linear map,
\[
\nabla^*\nabla+1: W^{1,p}(M;V)\to (W^{1,p'}(M;V^*))^* = W^{-1,p}(M;V),
\]
which is clearly injective. The map is surjective by \cite[Sections 3.5--3.14]{AdamsFournier} because every $\alpha \in (W^{1,p'}(M;V^*))^*$ is represented by
\[
\alpha(v) = (\nabla v,u_1)_{L^2(X)} +  (v,u_2)_{L^2(X)},
\]
for $u_i \in L^p(M;V)$ for $i=1,2$. If $p=2$, then we find that $u_2=u$ and $u_1=\nabla u$ for $u \in W^{1,2}(M;V)$ (compare \cite[Section 8.2]{GilbargTrudinger}) and thus also for $2\leq p <\infty$; duality yields the case $1<p\leq 2$. Hence, the map is an isomorphism by the Open Mapping Theorem.)

For $m=2n$ and $n \in \NN$, we may thus define $Pv \in W^{k,p}(M;V)$ by
\begin{multline*}
(Pv)(w) := ((\nabla^*\nabla+1)^{m/2}v, P^*w)_{L^2(M;W^*)},
\\
\forall\, v \in W^{k+m,p}(M;V) \text{ and } w \in W^{-k+m,p'}(M;V),
\end{multline*}
noting that $(\nabla^*\nabla+1)^{m/2}v \in W^{k,p}(M;V)$ and $P^*w \in W^{-k,p'}(M;V)$. For $m=2n+1$ and $n \in \NN$, one must first define the square root\footnote{This could be accomplished geometrically with the aid of a Dirac operator, $D:C^\infty(M;V)\to C^\infty(M;V)$, when $V$ and $W$ are spinor bundles and $M$ admits a \spinc structure \cite{LM}.}, $(\nabla^*\nabla+1)^{1/2}:W^{l+1,p}(M;V)\to W^{l,p}(M;V)$, for integers $l \in \ZZ$ and $1<p<\infty$; see Feehan \cite{Feehan_yang_mills_gradient_flow_v4} and references cited therein for a survey of approaches to the definition of the fractional powers, $(\nabla^*\nabla+1)^s$, for $s \in \RR$. Once the square root of the augmented connection Laplace operator is defined, then the definition of $Pv \in W^{k,p}(M;V)$ is identical to the case $m=2n$.

\begin{proof}[Proof of Theorem \ref{thm:Gilkey_1-4-5_Sobolev}]
There is a pseudo-differential operator
  \[
    S:C^\infty(M;W) \to C^\infty(M;V)
  \]
of order $-m$, according to \cite[Lemma 1.3.6]{Gilkey2}, so that
\[
SP - I \in \Psi^{-\infty}(M;V) \quad\text{and}\quad PS - I \in \Psi^{-\infty}(M;W),
\]
where $\Psi^{-\infty}(M;W)$ and $\Psi^{-\infty}(M;V)$ are the vector spaces of infinitely smoothing pseudo-differential operators \cite[Sections 1.2 and 1.3]{Gilkey2}.

The operator $P:W^{k+m,p}(M;V) \to W^{k,p}(M;W)$ is continuous since $P$ is an elliptic partial differential operator of order $m \geq 1$ and by definition of the Sobolev space $W^{k+m,p}(M;V)$. Combining the preceding observation with the \apriori elliptic estimate (see \cite[Theorem 14.60]{Feehan_yang_mills_gradient_flow_v4}, for example),
\[
\|v\|_{W^{k+m,p}(M)} \leq C\left(\|v\|_{W^{k,p}(M)} + \|Pv\|_{W^{k,p}(M)}\right),
\]
implies that the expression $\|v\|_{W^{k,p}(M)} + \|Pv\|_{W^{k,p}(M)}$ defines a norm for $v \in C^\infty(M;V)$ which is equivalent to the standard norm on $W^{k+m,p}(M;V)$. Furthermore, the following operators are continuous, 
\begin{align*}
SP - I:W^{k,p}(M;V) &\to W^{k+1,p}(M;V),
\\
PS - I:W^{k,p}(M;W) &\to W^{k+1,p}(M;W),
\end{align*}
because, denoting $H^l(M;V) = W^{l,2}(M;V)$ and $H^l(M;W) = W^{l,2}(M;W)$, the operators
\begin{align*}
SP - I: H^l(M;V) &\to H^{l+n}(M;V),
\\
PS - I: H^l(M;W) &\to H^{l+n}(M;W),
\end{align*}
are continuous by \cite[Lemma 1.3.5]{Gilkey2}, for any integers $l, n\in\ZZ$, and the standard Sobolev Embedding \cite[Theorem 4.12]{AdamsFournier} and duality \cite[Sections 3.5--3.14]{AdamsFournier}, which provide continuous embeddings, $W^{k,p}(M;V) \subset H^l(M;V)$ for sufficiently small $l \leq l_0(d,k,p)$ and $H^{l+n}(M;W) \subset W^{k+1,p}(M;W)$ for sufficiently large $n \geq n_0(d,k,p)$, where $d$ is the dimension of $M$. (The finite integers $l_0(d,k,p)$ and $n_0(d,k,p)$ can of course be determined explicitly from the full statement of the Sobolev Embedding \cite[Theorem 4.12]{AdamsFournier}, but their precise values are unimportant here.) Thus,
\begin{align*}
\|Sw\|_{W^{k+m,p}(M;V)} &\leq C\left(\|PSw\|_{W^{k,p}(M;W)} + \|w\|_{W^{k,p}(M;W)}\right)
\\
&\leq C\left(\|(PS-I)w\|_{W^{k,p}(M;W)} + \|w\|_{W^{k,p}(M;W)}\right)
\\
&\leq C\|w\|_{W^{k,p}(M;W)} \quad\text{(by continuity of $PS - I$ on $W^{k,p}(M;W)$),}
\end{align*}
and so the operator $S:W^{k,p}(M;V) \to W^{k+m,p}(M;W)$ is continuous.

The embeddings
\[
  W^{k+1,p}(M;V) \Subset W^{k,p}(M;V) \quad\text{and}\quad W^{k+1,p}(M;W) \Subset W^{k,p}(M;W)
\]
are compact by the Rellich--Kondrachov Theorem \cite[Theorem 6.3]{AdamsFournier} when $k \geq 0$ and when $k \leq -1$ using duality \cite[Theorem 6.4]{Brezis} and compactness of the embeddings
\begin{align*}
  W^{-k,p'}(M;V^*) &\Subset W^{-k-1,p'}(M;V^*) \quad\text{and}
  \\
  W^{-k,p'}(M;W^*) &\Subset W^{-k-1,p'}(M;W^*).
\end{align*}
Hence, the operators (now viewed as compositions with compact embeddings),
\begin{align*}
SP - I:W^{k,p}(M;V) &\to W^{k,p}(M;V),
\\
PS - I:W^{k,p}(M;W) &\to W^{k,p}(M;W),
\end{align*}
are compact by \cite[Proposition 6.3]{Brezis}. Thus, $P:W^{k+m,p}(M;V) \to W^{k,p}(M;W)$ is Fredholm by \cite[Section 1.4.2, Definition, p. 38]{Gilkey2}.

Because $P:W^{k+m,p}(M;V) \to W^{k,p}(M;W)$ is Fredholm, its index may be computed by \cite[Lemma 4.38]{Abramovich_Aliprantis_2002},
\begin{align*}
\Ind P &= \dim\Ker\left(P:W^{k+m,p}(M;V) \to W^{k,p}(M;W)\right)
\\
&\quad - \dim\Ker\left(P^*:W^{-k,p'}(M;W^*) \to W^{-k-m,p'}(M;V^*)\right).
\end{align*}
If $w \in (\Ker P^*)\cap W^{-k,p'}(M;W^*)$, then $w \in (\Ker P^*)\cap H^l(M;W^*)$ for all integers $l \leq l_0(d,k,p)$, where $l_0$ is given by the Sobolev Embedding \cite[Theorem 4.12]{AdamsFournier}, and consequently (since the Banach space dual $P^*$ is defined by the realization of the formal adjoint and $P^*$ is thus an elliptic partial differential operator of order $m$) we see that $w \in (\Ker P^*)\cap C^\infty(M;W^*)$ by elliptic regularity \cite[Lemma 1.3.2 and Section 1.3.5]{Gilkey2}. Of course, if $v \in (\Ker P)\cap W^{k+m,p}(M;V)$, then $v \in (\Ker P)\cap C^\infty(M;V)$ by the same argument. This yields the stated formula for the index of $P:W^{k+m,p}(M;V) \to W^{k,p}(M;W)$.

If $K \subset W^{-k,p'}(M;W^*)$ is any subspace, we recall that the \emph{annihilator} \cite[Section 4.6]{Rudin} of $K$ in $(W^{-k,p'}(M;W^*))^* = W^{k,p}(M;W)$ is
\[
K^\circ = \{w \in W^{k,p}(M;W): \langle \alpha, w\rangle = 0, \ \forall\, \alpha \in K \},
\]
and $\langle \cdot, \cdot\rangle:W^{-k,p'}(M;W^*)\times (W^{-k,p'}(M;W^*))^* \to \RR$ is the canonical pairing. We can therefore identify the range of $P:W^{k+m,p}(M;V) \to W^{k,p}(M;W)$ using
\begin{align*}
{}&\Ran\left(P:W^{k+m,p}(M;V) \to W^{k,p}(M;W)\right)
\\
&\quad = \overline{\Ran}\left(P:W^{k+m,p}(M;V) \to W^{k,p}(M;W)\right) \quad\text{(by closed range)}
\\
  &\quad = \Ker\left(P^*:W^{-k,p'}(M;W^*) \to W^{-k-m,p'}(M;V^*)\right)^\circ
  \\
  &\qquad\text{(by \cite[Corollary 2.18 (iv))]{Brezis})}.
\end{align*}
If $K := \Ker\left(P^*:W^{-k,p'}(M;W^*) \to W^{-k-m,p'}(M;V^*)\right)$, then $K \subset C^\infty(M;W^*)$, as we observed by elliptic regularity and
\begin{align*}
K^\circ &= \{w \in W^{k,p}(M;W): w(\alpha) = 0, \ \forall\, \alpha \in K \}
\\
        &= \{w \in W^{k,p}(M;W): \alpha(w) = 0, \ \forall\, \alpha \in K \}
  \\
  &\qquad\text{(by $(W^{k,p}(M;W))^{**} \cong W^{k,p}(M;W)$)}
\\
&=  \{w \in W^{k,p}(M;W): \iota(\kappa)(w) = 0, \ \forall\, \kappa \in K^* \}
\\
&=  \{w \in W^{k,p}(M;W): (w, \kappa)_{L^2(M;W)} = 0, \ \forall\, \kappa \in K^* \}
\\
&= (K^*)^\perp \cap W^{k,p}(M;W),
\end{align*}
where $K^* \subset C^\infty(M;W)$ and $\iota:K^*\cong K^{**}=K$ is the canonical isomorphism, $\kappa \mapsto \iota(\kappa) = (\cdot, \kappa)_{L^2(M;W)}$. This yields the claimed identification of the range of $P$.

Noting that $(W^{-k,p'}(M;V))^* = W^{k,p}(M;V^*)$ and  $(W^{-k-m,p'}(M;W))^* = W^{k+m,p}(M;W^*)$, we see that the operator,
\[
P^*:W^{k+m,p}(M;W^*) \to W^{k,p}(M;V^*),
\]
is Fredholm by \cite[Lemma 1.4.3]{Gilkey2}, since the same is true for the operator,
\[
P:W^{-k,p'}(M;V) \to W^{-k-m,p'}(M;W).
\]
If $V=W^*$, we may write $P^* = P + (P^*-P)$ and observe that $P^*-P:W^{k+m,p}(M;V) \to W^{k+1,p}(M;W)$ is bounded since $P^*-P$ has order $m-1$. The inclusion, $W^{k+1,p}(M;W)\Subset W^{k,p}(M;W)$, is compact by \cite[Theorem 6.3]{AdamsFournier}, so $P^*-P:W^{k+m,p}(M;V) \to W^{k,p}(M;W)$ is compact. Hence, $\Ind(P^*) = \Ind P$ by \cite[Corollary 19.1.8]{Hormander_v3}. But we also have $\Ind P^* = -\Ind P$ by \cite[Lemma 1.4.4 (a)]{Gilkey2} and thus $\Ind P = 0$.
\end{proof}

The following corollary is partially based on Abramovich and Aliprantis \cite[Theorem 4.46]{Abramovich_Aliprantis_2002}, Melrose \cite[Lecture 9, Proposition 18]{Melrose_Lectures_pseudodifferential_operators} and Treves \cite[Theorem 2.4]{Treves1}.

\begin{cor}[Green's operators for elliptic partial differential operators on Sobolev spaces]
\label{cor:Melrose_proposition_6-4}
Assume the hypotheses of Theorem \ref{thm:Gilkey_1-4-5_Sobolev} and let $G:W^{k,p}(M;W) \to W^{k+m,p}(M;V)$ be the \emph{Green's operator} for $P$ defined by
\[
Gw := \begin{cases} P^{-1}w, &w \in \Ran P \subset W^{k,p}(M;W), \\ 0, &w \in (\Ran P)^\perp. \end{cases}
\]
Then $G:C^\infty(M;W)\to C^\infty(M;V)$ is an elliptic pseudo-differential operator of order $-m$ with
\[
GP = \id - \Pi_1 \quad\text{and}\quad PG = \id - \Pi_2,
\]
where $\Pi_1$ is the $L^2$-orthogonal projection onto $\Ker P \subset W^{k+m,p}(M;V)$ and $\Pi_2$ is the $L^2$-orthogonal projection onto $(\Ran P)^\perp \subset W^{k,p}(M;W)$. Moreover, if $k\geq 0$, the following operator is bounded,
\[
G:W^{k,p}(M;W) \to W^{k+m,p}(M;V).
\]
\end{cor}

\begin{proof}
From Melrose \cite[Lecture 6, Proposition 11]{Melrose_Lectures_pseudodifferential_operators} or
\cite[Lecture 7, Theorem 3]{Melrose_Lectures_pseudodifferential_operators} or Treves \cite[Theorem 2.4]{Treves1}, there exists an elliptic pseudo-differential operator, $G:C^\infty(M;W)\to C^\infty(M;V)$, of integer order $-m$ such that $\id-GP$ and $\id-PG$ are infinitely smoothing operators. We then appeal to Melrose
\cite[Lecture 9, Proposition 18]{Melrose_Lectures_pseudodifferential_operators} to conclude that $\id-GP$ and $\id-PG$ can be identified as the stated $L^2$-orthogonal projections.

For $k \geq 0$ and any $w \in W^{k,p}(M;W)$, we have
\begin{align*}
\|Gw\|_{W^{k+m,p}(M)} &\leq C\left(\|PGw\|_{W^{k,p}(M)} + \|w\|_{W^{k,p}(M)} \right)
\quad\text{(by \cite[Theorem 14.60]{Feehan_yang_mills_gradient_flow_v4})}
\\
&\leq C\left(\|(\id - \Pi_2)w\|_{W^{k,p}(M)} + \|w\|_{W^{k,p}(M)} \right)
\\
&\leq C\|w\|_{W^{k,p}(M)},
\end{align*}
and so the conclusion on boundedness follows.
\end{proof}

\chapter[Equivalence of Sobolev norms]{Equivalence of Sobolev norms defined by Sobolev and smooth connections}
\label{chap:Equivalence_Sobolev_norms_for_Sobolev_and_smooth_connections}
Suppose that $(X,g)$ is a closed, smooth Riemannian manifold of dimension $d \geq 2$, and $G$ is a compact Lie group, and $P$ is a smooth principal $G$-bundle over $X$. In standard references for gauge theory \cite{DK, FU}, it is generally assumed in the construction of Sobolev completions of spaces such as $\Omega^l(X;\ad P)$ that one defines Sobolev norms using a covariant derivative $\nabla_A$ determined by a connection $A$ on $P$ that is smooth or of class $W^{k,p}$ for $p \geq 1$ and an integer $k \geq 1$ large enough that $kp>d$ or even $kp\gg d$. However, in this monograph, we often consider connections $A$ with more borderline regularity, for example of class $W^{1,q}$ for $q > d/2$, and in that situation, one must exercise care in the definition of Sobolev spaces using such connections. Lemmas \ref{lem:Second_order_Kato_inequality_and_Sobolev_norms} and \ref{lem:Equivalence_Sobolev_norms_for_Sobolev_and_smooth_connections} provide some guidance.

\begin{lem}[Second-order Kato inequality and second-order Sobolev norms]
\label{lem:Second_order_Kato_inequality_and_Sobolev_norms}
Let $(X,g)$ be a closed, smooth Riemannian manifold of dimension $d \geq 2$, and $G$ be a Lie group, and $P$ be a smooth principal $G$-bundle over $X$, and $V = P\times_\varrho\VV$ be a smooth Riemannian vector bundle over $X$ defined by a finite-dimensional, orthogonal representation, $\varrho: G \to \Aut_\RR(\VV)$. Then there exists a constant $C=C(g,q) \in [1,\infty)$ with the following significance. If $A$ is a $W^{1,q}$ connection on $P$ with $q>d/2$, then for all $v \in C^\infty(X;V)$,
\begin{equation}
\label{eq:Second_order_Kato_inequality_and_Sobolev_norms}
\|v\|_{C(X)} \leq C\|v\|_{W_A^{2,q}(X)}.
\end{equation}
\end{lem}

\begin{proof}
The first-order analogue of \eqref{eq:Second_order_Kato_inequality_and_Sobolev_norms}, namely,
\[
\|v\|_{C(X)} \leq \kappa_1 \|v\|_{W_A^{1,q}(X)},
\]
when $q > d$ and $\kappa_1 = \kappa_1(g) \in [1,\infty)$ is the norm of the Sobolev embedding $W^{1,q}(X)\subset C(X)$, is an immediate consequence of the pointwise first-order Kato inequality, $|\nabla|v|| \leq |\nabla_A v|$ from \cite[Inequality (6.20)]{FU}, in turn a consequence of the compatibility of the fiber metric on $V$ with $\nabla_A$.

We first note that, for $f \in C^\infty(X;\RR)$, the norm
\[
\|f\|_{W^{2,q}(X)} = \|\nabla^2 f\|_{L^q(X)} + \|\nabla f\|_{L^q(X)} + \|f\|_{L^q(X)}
\]
is equivalent (with respect to a constant depending at most on $(g, q)$) by virtue of \cite[Theorem 9.11]{GilbargTrudinger} to
\[
\|f\|_{W^{2,q}(X)} = \|\Delta f\|_{L^q(X)} + \|f\|_{L^q(X)},
\]
where $\Delta$ is the Laplace operator defined by the Riemannian metric $g$ on $X$. Now recall the pointwise identity \cite[Equation (6.18)]{FU},
\[
\Delta|v|^2 = 2\langle\nabla_A^*\nabla_A v, v\rangle - 2|\nabla_A v|^2.
\]
Hence, letting $\kappa_2 = \kappa_2(g) \in [1,\infty)$ denote the norm of the Sobolev embedding $W^{2,q}(X)\subset C(X)$,
\begin{align*}
\|v\|_{C(X)}^2 &= \||v|^2\|_{C(X)}
\\
&\leq \kappa_2 \||v|^2\|_{W^{2,q}(X)}
\\
&= \kappa_2\left( \|\Delta |v|^2\|_{L^q(X)} + \||v|^2\|_{L^q(X)} \right)
\\
&\leq \kappa_2\left( 2\|v\|_{C(X)}\|\nabla_A^*\nabla_A v\|_{L^q(X)}
+ 2\|\nabla_A v\|_{L^{2q}(X)}^2 + \|v\|_{C(X)}\|v\|_{L^q(X)} \right).
\end{align*}
Recall that $W^{1,q}(X) \subset L^{2q}(X)$, for $q < d$, if and only if $2q \leq q^* = dq/(d-q)$, that is, $2(d - q) \leq d$ or $q \geq d/2$; the embedding is immediate from \cite[Theorem 4.12]{AdamsFournier} when $q \geq d$. Thus, applying the first-order Kato Inequality and the preceding Sobolev embedding for functions,
\[
\| |v| \|_{L^{2q}(X)} \leq \kappa_1\left(\| \nabla|v| \|_{L^q(X)} + \|v\|_{L^q(X)}\right)
\leq \kappa_1\left(\|\nabla_Av\|_{L^q(X)} + \|v\|_{L^q(X)}\right),
\]
we obtain
\begin{align*}
\|v\|_{C(X)}^2
&\leq
\kappa_2\left( 2\|v\|_{C(X)}\|\nabla_A^*\nabla_A v\|_{L^q(X)}
+ 2\kappa_1^2\left(\|\nabla_A^2 v\|_{L^q(X)} + \|\nabla_Av\|_{L^q(X)}\right)^2 \right.
\\
&\quad + \left. \|v\|_{C(X)}\|v\|_{L^q(X)} \right).
\end{align*}
We now use Young's Inequality $2ab \leq \eps a^2 + \eps^{-1}b^2$ from \cite[Inequality (7.8)]{GilbargTrudinger} and rearrangement with a suitably small and universal $\eps$ to give
\begin{align*}
\|v\|_{C(X)}^2
&\leq
C^2\left(\|\nabla_A^*\nabla_A v\|_{L^q(X)}^2 + \|\nabla_A^2 v\|_{L^q(X)}^2 + \|\nabla_A v\|_{L^q(X)}^2 + \|v\|_{L^q(X)}^2 \right),
\end{align*}
where $C=C(g,q) \in [1,\infty)$. We simplify the right-hand side in the preceding inequality via
\[
\|\nabla_A^*\nabla_A v\|_{L^q(X)}
\leq
z\left(\|\nabla_A^2 v\|_{L^q(X)} + \|\nabla_A v\|_{L^q(X)} + \|v\|_{L^q(X)}\right),
\]
where $z$ is a constant depending at most on the Riemannian metric on $X$. The desired Sobolev inequality \eqref{eq:Second_order_Kato_inequality_and_Sobolev_norms} now follows by taking square roots.
\end{proof}

\begin{lem}[Equivalence of Sobolev norms defined by Sobolev and smooth connections]
\label{lem:Equivalence_Sobolev_norms_for_Sobolev_and_smooth_connections}
Let $(X,g)$ be a closed, smooth Riemannian manifold of dimension $d \geq 2$, and $G$ be a compact Lie group, and $P$ be a smooth principal $G$-bundle over $X$, and $V = P\times_\varrho\VV$ be a smooth Riemannian vector bundle over $X$ defined by a finite-dimensional, orthogonal representation, $\varrho: G \to \Aut_\RR(\VV)$, and $q > d/2$ and $p$ obey $d/2\leq p \leq q$. Let $A_1$ be a $C^\infty$ connection on $P$, and $A_0$ be a Sobolev connection on $P$, and $a_0 := A_0-A_1$.
\begin{enumerate}
\item
\label{item:Equivalence_W1p_A0_embed_Lr_and_W2pA1_embed_Lr}
There exists $C=C(g,p) \in [1,\infty)$ such that, if $a_0 \in W_{A_1}^{1,p}(X;\Lambda^1\otimes\ad P)$, then
\[
\|\xi\|_{L^r(X)} \leq C\|\xi\|_{W^{1,p}_{A_0}(X)},
\quad\text{for }
\begin{cases}
1\leq r \leq dp/(d-p) &\text{if } p < d,
\\
1\leq r < \infty &\text{if } p = d,
\\
r = \infty &\text{if } p > d,
\end{cases}
\]
for all $\xi \in C^\infty(X;V)$; moreover, there exists $C=C(A_1,g,p) \in [1,\infty)$ such that
\[
\|\xi\|_{L^r(X)} \leq C\|\xi\|_{W_{A_1}^{2,p}(X)},
\quad\text{for }
\begin{cases}
1\leq r \leq dp/(d-2p) &\text{if } p < d/2,
\\
1\leq r < \infty &\text{if } p = d/2,
\\
r = \infty &\text{if } p > d/2,
\end{cases}
\]
for all $\xi \in C^\infty(X;V)$.
\item
\label{item:Equivalence_W2qA0_embed_Linfty}
If $a_0\in W_{A_1}^{1,q}(X; \Lambda^1\otimes \ad P)$, then 
\[
\|\xi\|_{C(X)} \leq C\|\xi\|_{W^{2,q}_{A_0}(X)}, \quad\forall\, \xi \in C^\infty(X;V),
\]
for some $C=C(g,q) \in [1,\infty)$.
\item
\label{item:Equivalence_W1p_A0_and_W1pA1}
If $a_0\in W_{A_1}^{1,p}(X; \Lambda^1\otimes \ad P)$, then 
\[
C^{-1}\|\xi\|_{W_{A_1}^{1,p}(X)} \leq \|\xi\|_{W_{A_0}^{1,p}(X)} \leq C \|\xi\|_{W_{A_1}^{1,p}(X)}, \quad\forall\, \xi \in C^\infty(X;V),
\]
for some $C = C(g,G,p,\|a_0\|_{W_{A_1}^{1,p}(X)}) \in [1,\infty)$.
\item
\label{item:Equivalence_W2q_A0_and_W2qA1}
If $a_0\in W_{A_1}^{1,q}(X; \Lambda^1\otimes \ad P)$, then 
\[
C^{-1}\|\xi\|_{W_{A_1}^{2,q}(X)} \leq \|\xi\|_{W_{A_0}^{2,q}(X)} \leq C \|\xi\|_{W_{A_1}^{2,q}(X)}, \quad\forall\, \xi \in C^\infty(X;V),
\]
for some $C = C(g,G,p,q,\|a_0\|_{W^{1,q}_{A_1}(X)}) \in [1,\infty)$.
\item
\label{item:Equivalence_W2p_A1_embed_W2pA0}
If $a_0\in W_{A_1}^{1,q}(X; \Lambda^1\otimes \ad P)$, then 
\[
\|\xi\|_{W_{A_0}^{2,p}(X)} \leq C\|\xi\|_{W_{A_1}^{2,p}(X)}, \quad\forall\, \xi \in C^\infty(X;V),
\]
for some $C = C(g,G,p,q,\|a_0\|_{W^{1,q}_{A_1}(X)}) \in [1,\infty)$.
\item
\label{item:Equivalence_W2p_A0_and_W2pA1}
If $a_0\in W_{A_1}^{2,q}(X; \Lambda^1\otimes \ad P)$, then 
\[
C^{-1}\|\xi\|_{W_{A_1}^{2,p}(X)} \leq \|\xi\|_{W_{A_0}^{2,p}(X)} \leq C \|\xi\|_{W_{A_1}^{2,p}(X)}, \quad\forall\, \xi \in C^\infty(X;V),
\]
for some $C = C(g,G,p,q,\|a_0\|_{W^{2,q}_{A_1}(X)}) \in [1,\infty)$.
\end{enumerate}
\end{lem}

\begin{proof}
Item \eqref{item:Equivalence_W1p_A0_embed_Lr_and_W2pA1_embed_Lr} is a well-known consequence of the Sobolev Embedding \cite[Theorem 4.12]{AdamsFournier} for scalar functions and the Kato Inequality \cite[Equation (6.20)]{FU} in the case of the embedding $W^{1,p}(X) \subset L^r(X)$. Item \eqref{item:Equivalence_W2qA0_embed_Linfty} restates the conclusion of Lemma \ref{lem:Second_order_Kato_inequality_and_Sobolev_norms}.

For Item \eqref{item:Equivalence_W1p_A0_and_W1pA1}, we use $\nabla_{A_0}\xi = \nabla_{A_1}\xi + [a_0,\xi]$ and estimate
\begin{align*}
\|\nabla_{A_1}\xi\|_{L^p(X)}
&\leq
\|\nabla_{A_0}\xi\|_{L^p(X)} + \|[a_0,\xi]\|_{L^p(X)}
\\
&\leq  \|\xi\|_{W^{1,p}_{A_0}(X)} + z\|a_0\|_{L^{2p}(X)} \|\xi\|_{L^{2p}(X)}
\\
&\leq
C(1+ \|a_0\|_{W_{A_1}^{1,p}(X)}) \|\xi\|_{W^{1,p}_{A_0}(X)},
\end{align*}
where we used the continuous Sobolev embedding $W^{1,p}(X) \subset L^{2p}(X)$ for $p\geq d/2$ and Item \eqref{item:Equivalence_W1p_A0_embed_Lr_and_W2pA1_embed_Lr} to obtain the last inequality. Here, $z = z(g,G) \in [1,\infty)]$ and $C \in [1,\infty)$ has the stated dependencies. The analogous estimate with the roles of $A_0$ and $A_1$ reversed follows by a symmetric argument.

For Item \eqref{item:Equivalence_W2q_A0_and_W2qA1}, we first write
\begin{equation}
\label{eq:nabla_exp}
\nabla_{A_1}^2\xi = \nabla_{A_0}^2\xi + \nabla_{A_0}a_0\times \xi + a_0\times\nabla_{A_0}\xi + a_0\times a_0\times\xi.
\end{equation}
Taking $L^q$ norms of both sides of \eqref{eq:nabla_exp}, we see that
\begin{align*}
\|\nabla_{A_1}^2\xi\|_{L^q(X)}
&\leq
\|\nabla_{A_0}^2\xi\|_{L^q(X)} + \|\nabla_{A_0}a_0\times \xi\|_{L^q(X)} + \|a_0\times\nabla_{A_0}\xi\|_{L^q(X)}
\\
&\quad + \|a_0\times a_0\times\xi\|_{L^q(X)},
\end{align*}
and thus, for $z = z(g,G) \in [1,\infty)$,
\begin{multline}
\label{eq:qnorms}
\|\nabla_{A_1}^2\xi\|_{L^q(X)}
\leq
\|\nabla_{A_0}^2\xi\|_{L^q(X)} + z\|\nabla_{A_0}a_0\|_{L^{q}(X)}\|\xi\|_{C(X)}
\\
+ z\|a_0\|_{L^{2q}(X)} \|\nabla_{A_0}\xi\|_{L^{2q}(X)} + z\|a_0\|_{L^{2q}}^2 \|\xi\|_{C(X)}.
\end{multline}
By Item \eqref{item:Equivalence_W1p_A0_and_W1pA1}, we have
\[
\|\nabla_{A_0}a_0\|_{L^q(X)} \leq  \|a_0\|_{W^{1,q}_{A_0}(X)} \leq C\|a_0\|_{W^{1,q}_{A_1}(X)},
\]
and by Item \eqref{item:Equivalence_W1p_A0_embed_Lr_and_W2pA1_embed_Lr} and the fact that $W^{1,q}(X) \subset L^{2q}(X)$ for $q > d/2$, we obtain
\[
\|\nabla_{A_0}\xi\|_{L^{2q}(X)} \leq C\|\xi\|_{W^{2,q}_{A_0}(X)}.
\]
Similarly, Item \eqref{item:Equivalence_W2qA0_embed_Linfty} gives
\[
\|\xi\|_{C(X)} \leq C\|\xi\|_{W^{2,q}_{A_0}(X)}.
\]
By substituting the preceding inequalities into \eqref{eq:qnorms}, we find that
\[
\|\xi\|_{W_{A_1}^{2,q}(X)} \leq C \|\xi\|_{W_{A_0}^{2,q}(X)},
\]
where $C \in [1,\infty)$ has the stated dependencies. The analogous inequality with the roles of $A_0$ and $A_1$ reversed follows by a symmetric argument.

For Item \eqref{item:Equivalence_W2p_A1_embed_W2pA0}, define $r\in [p,\infty]$ by $1/p=1/q+1/r$, recall that $p=d/2<q$ or $d/2<p\leq q$, interchange the roles of $A_0$ and $A_1$ in \eqref{eq:nabla_exp}, and take $L^p$ norms to give
\begin{align*}
\|\nabla_{A_0}^2\xi\|_{L^p(X)}
&\leq
\|\nabla_{A_1}^2\xi\|_{L^p(X)} + \|\nabla_{A_1}a_0\times \xi\|_{L^p(X)} + \|a_0\times\nabla_{A_1}\xi\|_{L^p(X)}
\\
&\quad + \|a_0\times a_0\times \xi\|_{L^p(X)}
\\
&\leq
\|\nabla_{A_1}^2\xi\|_{L^p(X)}
+ z\|\nabla_{A_1}a_0\|_{L^q(X)}\|\xi\|_{L^r(X)}
+ z\|a_0\|_{L^{2p}(X)} \|\nabla_{A_1}\xi\|_{L^{2p}(X)}
\\
&\quad + z\|a_0\|_{L^{2q}}^2 \|\xi\|_{L^r(X)}
\\
&\leq
\|\nabla_{A_1}^2\xi\|_{L^p(X)}
+ C\|a_0\|_{W_{A_1}^{1,q}(X)}\|\xi\|_{W_{A_1}^{2,p}(X)}
+ C\|a_0\|_{W_{A_1}^{1,p}(X)} \|\nabla_{A_1}\xi\|_{W_{A_1}^{1,p}(X)}
\\
&\quad + C\|a_0\|_{W_{A_1}^{1,q}(X)}^2 \|\xi\|_{W_{A_1}^{2,p}(X)},
\end{align*}
where $z = z(g,G)\in [1,\infty)$ and, to obtain the last inequality, we use the continuous Sobolev embeddings $W^{1,p}(X)\subset L^{2p}(X)$ and $W^{1,q}(X)\subset L^{2q}(X)$ for $d/2\leq p \leq q$ and Item \eqref{item:Equivalence_W1p_A0_embed_Lr_and_W2pA1_embed_Lr} together with the continuous Sobolev embedding $W^{2,p}(X) \subset L^r(X)$, for $r\in [1,\infty)$ if $p=d/2$ and $r=\infty$ if $p>d/2$. Therefore, we obtain
\[
\|\xi\|_{W^{2,p}_{A_0}(X)} \leq C \|\xi\|_{W_{A_1}^{2,p}(X)},
\]
where $C \in [1,\infty)$ has the stated dependencies.

For Item \eqref{item:Equivalence_W2p_A0_and_W2pA1}, we take $L^p$ norms of \eqref{eq:nabla_exp} and use $\nabla_{A_0}a_0 = \nabla_{A_1}a_0 + [a_0,a_0]$ to give
\begin{align*}
\|\nabla_{A_1}^2\xi\|_{L^p(X)}
&\leq
\|\nabla_{A_0}^2\xi\|_{L^p(X)} + \|\nabla_{A_0}a_0\times\xi\|_{L^p(X)} + \|a_0\times \nabla_{A_0}\xi\|_{L^p(X)}
\\
&\quad + \|a_0\times a_0\times\xi\|_{L^p(X)}
\\
&\leq
\|\nabla_{A_0}^2\xi\|_{L^p(X)} + z\|\nabla_{A_1}a_0\|_{L^{2p}(X)} \|\xi\|_{L^{2p}(X)} + \|a_0\|_{C(X)} \|\nabla_{A_0}\xi\|_{L^p(X)}
\\
&\quad + 2z\|a_0\|_{C(X)}^2 \|\xi\|_{L^p(X)}.
\end{align*}
Applying the continuous Sobolev embeddings $W^{1,p}(X)\subset L^{2p}(X)$ and $W^{2,q}(X) \subset C(X)$ and Items \eqref{item:Equivalence_W1p_A0_embed_Lr_and_W2pA1_embed_Lr} and \eqref{item:Equivalence_W1p_A0_and_W1pA1}, we discover that
\begin{align*}
\|\nabla_{A_1}^2\xi\|_{L^p(X)}
&\leq
                                 \|\nabla_{A_0}^2\xi\|_{L^p(X)} + C\|\nabla_{A_1}a_0\|_{W_{A_1}^{1,p}(X)} \|\xi\|_{W_{A_0}^{1,p}(X)}
  \\
  &\quad + \|a_0\|_{W_{A_1}^{2,q}(X)} \|\nabla_{A_0}\xi\|_{L^p(X)} + 2z\|a_0\|_{W_{A_1}^{2,q}(X)}^2 \|\xi\|_{L^p(X)}
\\
&\leq C\|\xi\|_{W^{2,p}_{A_0}(X)},
\end{align*}
where $C \in [1,\infty)$ has the stated dependencies. The analogous inequality with the roles of $A_0$ and $A_1$ reversed follows by a symmetric argument.
\end{proof}

\chapter[Fredholm and index properties of Laplacians]{Fredholm and index properties of a Hodge Laplacian with Sobolev coefficients}
\label{chap:Fredholm_properties_Hodge_Laplacian_Sobolev_coefficients}
In this chapter we include proofs of results regarding the Fredholm properties of the Hodge Laplace operators encountered in Chapters \ref{chap:Coulomb_gauge_slice_quotient_space_connections_pairs} and \ref{chap:Lojasiewicz-Simon_gradient_inequality_coupled_Yang--Mills_energy_functions} that would be standard if the operator had smooth coefficients and acted on $L^2$ rather than $L^p$ Sobolev spaces as we allow here.

When $A$ is a smooth connection, the Hodge Laplace operator, $\Delta_A$, in \eqref{eq:Hodge_Laplace_operator} is an elliptic, second-order partial differential operator with smooth coefficients that is $L^2$-self-adjoint and so Theorem \ref{thm:Gilkey_1-4-5_Sobolev} immediately provides the

\begin{prop}[Fredholm and index zero properties of a Laplace operator with smooth coefficients]
\label{prop:Fredholmness_and_index_Laplace_operator_on_W2p_smooth_connection}
Let $(X,g)$ be a closed, smooth Riemannian manifold of dimension $d \geq 2$, and $G$ be a compact Lie group, and $P$ be a smooth principal $G$-bundle over $X$, and $l\geq 0$ an integer. If $A$ and $A_1$ are $C^\infty$ connections on $P$ and $1 < p < \infty$, then the operator,
\begin{equation}
\label{eq:Delta_A_smooth_W2p_to_Lp_Fredholm}
\Delta_A: W_{A_1}^{2,p}(X;\Lambda^l\otimes\ad P) \to L^p(X;\Lambda^l\otimes\ad P),
\end{equation}
is Fredholm with index zero and closed range $K^\perp \cap L^p(X;\Lambda^l\otimes\ad P)$, where $\perp$ denotes $L^2$-orthogonal complement and $K \subset W_{A_1}^{2,p}(X;\Lambda^l\otimes\ad P)$ is the kernel of $\Delta_A$ in \eqref{eq:Delta_A_smooth_W2p_to_Lp_Fredholm}.
\end{prop}

Proposition \ref{prop:Fredholmness_and_index_Laplace_operator_on_W2p_smooth_connection} and a compact operator perturbation argument provides following useful generalization from the case of a $C^\infty$ to a $W^{1,q}$ connection $A$.

\begin{cor}[Fredholm and index zero properties of a Laplace operator with Sobolev coefficients]
\label{cor:Fredholmness_and_index_Laplace_operator_on_W2p_Sobolev_connection}
Assume the hypotheses of Proposition \ref{prop:Fredholmness_and_index_Laplace_operator_on_W2p_smooth_connection}, but allow $A$ to be a $W^{1,q}$ connection with $d/2 < q < \infty$ and restrict $p \in (1,\infty)$ so that $d/2 \leq p \leq q$. Then the operator,
\begin{equation}
\label{eq:Delta_A_Sobolev_W2p_to_Lp_Fredholm}
\Delta_A: W_{A_1}^{2,p}(X;\Lambda^l\otimes\ad P) \to L^p(X;\Lambda^l\otimes\ad P),
\end{equation}
is Fredholm with index zero and closed range $K^\perp \cap L^p(X;\Lambda^l\otimes\ad P)$, where $\perp$ denotes $L^2$-orthogonal complement and $K \subset W_{A_1}^{2,p}(X;\Lambda^l\otimes\ad P)$ is the kernel of $\Delta_A$ in \eqref{eq:Delta_A_Sobolev_W2p_to_Lp_Fredholm}.
\end{cor}

\begin{proof}
By hypothesis, $A_1$ is a $C^\infty$ connection on $P$. We write $A = A_1+a$, for $a \in W_{A_1}^{1,q}(X;\Lambda^l\otimes\ad P)$ and proceed by modifying the proof of Proposition \ref{prop:W2p_apriori_estimate_Delta_A_Sobolev} to show that, for suitable $r \in (p,\infty]$ and $t \in [p,\infty)$, the operator
\begin{equation}
\label{eq:Delta_A_Sobolev_minus_Delta_A_smooth_W1u_to_Lp_bounded}
\Delta_A - \Delta_{A_1}: W_{A_1}^{1,t}(X;\Lambda^l\otimes\ad P)\cap L^r(X;\Lambda^l\otimes\ad P) \to L^p(X;\Lambda^l\otimes\ad P)
\end{equation}
is bounded and, because the Sobolev embedding $W^{2,p}(X) \Subset W^{1,t}(X)\cap L^r(X)$ will be compact by the Rellich--Kondrachov Theorem \cite[Theorem 6.3]{AdamsFournier}, then the following composition of that compact embedding and the preceding bounded operator,
\begin{equation}
\label{eq:Delta_A_Sobolev_minus_Delta_A_smooth_W2p_to_Lp_compact}
\Delta_A - \Delta_{A_1}: W_{A_1}^{2,p}(X;\Lambda^l\otimes\ad P) \to L^p(X;\Lambda^l\otimes\ad P)
\end{equation}
is compact by \cite[Proposition 6.3]{Brezis}.

By modifying the derivation of the estimate \eqref{eq:Delta_A_Sobolev_minus_Delta_As_smooth} in the proof of Proposition \ref{prop:W2p_apriori_estimate_Delta_A_Sobolev}, we claim that, for $r \in (p,\infty]$ defined by $1/p=1/q+1/r$,
\begin{multline}
\label{eq:Lp_estimate_DeltaASobolev_minus_DeltaA1smooth}
\|(\Delta_A - \Delta_{A_1})\xi\|_{L^p(X)}
\leq
z\left(\|a\|_{W_{A_1}^{1,q}(X)} + \|a\|_{W_{A_1}^{1,q}(X)}^2 \right)\|\xi\|_{L^r(X)}
\\
+ z\|a\|_{W_{A_1}^{1,q}(X)}\|\xi\|_{W_{A_1}^{1,t}(X)},
\end{multline}
where $z=z(g,G,l) \in [1,\infty)$ and
\begin{inparaenum}[\itshape 1\upshape)]
\item $q < d$ and $t\in(p,\infty)$ defined by $1/p=1/q^*+1/t$ with $q^*=dq/(d-q)$, or
\item $q = d$ and $t = p+\delta$ for a small enough $\delta\in(0,1]$, or
\item $q > d$ and $t = p$.
\end{inparaenum}

To see that \eqref{eq:Lp_estimate_DeltaASobolev_minus_DeltaA1smooth} holds, observe that \eqref{eq:Delta_A1+a_expansion_prelim} (with $A_s$ replaced by $A_1$) gives
\[
\Delta_A\xi = \Delta_{A_1}\xi + \nabla_{A_1}a \times \xi + a\times \nabla_{A_1}\xi + a\times a\times \xi.
\]
Proceeding as in the derivation of \eqref{eq:Delta_A_Sobolev_minus_Delta_As_smooth_prelim}, we see that the preceding identity yields
\begin{multline*}
\|(\Delta_A - \Delta_{A_1})\xi\|_{L^p(X)}
\leq
z\|\nabla_{A_1}a\|_{L^q(X)}\|\xi\|_{L^r(X)} + \|a \times \nabla_{A_1}\xi\|_{L^p(X)}
\\
+ z\||a|^2\|_{L^q(X)} \|\xi\|_{L^r(X)}.
\end{multline*}
Hence, using continuity of the Sobolev multiplication map, $L^{2q}(X)\times L^{2q}(X) \to L^q(X)$, and Sobolev embedding, $W^{1,q}(X) \subset L^{2q}(X)$, by \cite[Theorem 4.12, Part I, Cases A, B, and C]{AdamsFournier} (using $q\geq d/2$), and Kato Inequality \cite[Equation (6.20)]{FU}, we find that
\begin{multline*}
\|(\Delta_A - \Delta_{A_1})\xi\|_{L^p(X)}
\leq
z\left(\|a\|_{W_{A_1}^{1,q}(X)} + \|a\|_{W_{A_1}^{1,q}(X)}^2 \right)\|\xi\|_{L^r(X)}
\\
+ \|a \times \nabla_{A_1}\xi\|_{L^p(X)}.
\end{multline*}
Now, defining $v\in (p,\infty)$ by a choice of small $\delta\in (0,1]$ and $1/p=1/v+1/(p+\delta)$ and $w\in (p,\infty)$ by $1/p = 1/q^*+1/w$ and $q^*=dq/(d-q)$ (when $q<d$), we obtain that
\[
  \|a \times \nabla_{A_1}\xi\|_{L^p(X)}
\]
is bounded by
\[
  z\|a\|_{L^\infty(X)}\|\nabla_{A_1}\xi\|_{L^p(X)}
\leq z\|a\|_{W_{A_1}^{1,q}(X)}\|\nabla_{A_1}\xi\|_{L^p(X)} \quad\text{if }q>d,
\]
or
\[
  z\|a\|_{L^v(X)}\|\nabla_{A_1}\xi\|_{L^{p+\delta}(X)}
\leq z\|a\|_{W_{A_1}^{1,q}(X)}\|\nabla_{A_1}\xi\|_{L^{p+\delta}(X)} \quad\text{if }q=d,
\]
or
\[
  z\|a\|_{L^{q^*}(X)}\|\nabla_{A_1}\xi\|_{L^w(X)}
\leq z\|a\|_{W_{A_1}^{1,q}(X)}\|\nabla_{A_1}\xi\|_{L^w(X)} \quad\text{if }q<d.
\]
The final three inequalities follow from continuity of the Sobolev embeddings, $W^{1,q}(X) \subset C(X)$ when $q>d$, and $W^{1,q}(X) \subset L^v(X)$ when $q=d$, and $W^{1,q}(X) \subset L^{q^*}(X)$ when $q<d$, by \cite[Theorem 4.12, Part I, Cases A, B, and C]{AdamsFournier}, respectively. This proves Claim \eqref{eq:Lp_estimate_DeltaASobolev_minus_DeltaA1smooth}.

In particular, the operator \eqref{eq:Delta_A_Sobolev_minus_Delta_A_smooth_W1u_to_Lp_bounded} is bounded for each of the three above cases for $(q,t)$.

When $p<q$ and thus $r < \infty$, the embedding $W^{2,p}(X) \subset L^r(X)$ is compact since $p \geq d/2$ by hypothesis and so the embedding $W^{2,p}(X) \subset W^{2,d/2}(X)$ is continuous while the embedding $W^{2,p}(X) \subset L^r(X)$ is compact by \cite[Theorem 6.3, Part I]{AdamsFournier}. When $p=q$ and thus $r = \infty$, the embedding $W^{2,p}(X) \subset C(X)$ is compact since $q > d/2$ by hypothesis and thus $p > d/2$ so the embedding $W^{2,p}(X) \subset C(X)$ is compact by \cite[Theorem 6.3, Part II]{AdamsFournier}.

To see compactness of the embedding $W^{2,p}(X) \subset W^{1,t}(X)$ and hence of the embedding $W^{2,p}(X) \subset W^{1,t}(X)\cap L^r(X)$ and consequently the operator \eqref{eq:Delta_A_Sobolev_minus_Delta_A_smooth_W2p_to_Lp_compact}, we consider the three cases for $(q,t)$ separately.

\setcounter{case}{0}
\begin{case}[$q < d$ and $t\in(p,\infty)$ defined by $1/p=1/q^*+1/t$ with $q^*=dq/(d-q)$]
We have a compact Sobolev embedding, $W^{2,p}(X) \Subset W^{1,t}(X)$, by \cite[Theorem 6.3, Part I]{AdamsFournier} provided $t < p^* = dp/(d-p)$. To check that the strict inequality $t < p^*$ holds, observe that $p\leq q$ and so $1/p \geq 1/q$, which gives
\[
1/t = 1/p - 1/q^* = 1/p - 1/q + 1/d \geq 1/d.
\]
If $p>d/2$, then $p^* = dp/(d-p) > d$ and $1/p^* < 1/d$ and thus $1/t \geq 1/d > 1/p^*$ and $t < p^*$, as desired. If $p=d/2$, then $p < q$ by hypothesis and so $1/p < 1/q$, which now gives
\[
1/t > 1/d.
\]
But $p=d/2 \implies p^* = d$ and thus $1/t > 1/d = 1/p^*$ and $t < p^*$, again as desired.

Hence, the embedding $W^{2,p}(X) \subset W^{1,t}(X)$ is compact and so the operator \eqref{eq:Delta_A_Sobolev_minus_Delta_A_smooth_W2p_to_Lp_compact} is compact for this case.
\end{case}

\begin{case}[$q = d$ and $t=p+\delta$ for small enough $\delta>0$]
We have a compact Sobolev embedding, $W^{2,p}(X) \Subset W^{1,p+\delta}(X)$, by \cite[Theorem 6.3, Part I]{AdamsFournier}, where it suffices to choose $\delta\in (0,1]$ small enough that $p+\delta<p^*=dp/(d-p) \in [d,\infty)$ when $d/2\leq p<d$ and any $\delta\in(0,1]$ will do when $p=d$. Hence, the embedding $W^{2,p}(X) \subset W^{1,t}(X)$ is compact with $t=p+\delta$ and so the operator \eqref{eq:Delta_A_Sobolev_minus_Delta_A_smooth_W2p_to_Lp_compact} is compact for this case.
\end{case}

\begin{case}[$q > d$ and $t=p$]
We have a compact Sobolev embedding, $W^{2,p}(X) \Subset W^{1,p}(X)$, by \cite[Theorem 6.3, Parts I and II]{AdamsFournier}. Hence, the embedding $W^{2,p}(X) \subset W^{1,t}(X)$ is compact with $t=p$ and so the operator \eqref{eq:Delta_A_Sobolev_minus_Delta_A_smooth_W2p_to_Lp_compact} is compact for this case.
\end{case}

Proposition \ref{prop:Fredholmness_and_index_Laplace_operator_on_W2p_smooth_connection} implies that the operator
\[
\Delta_{A_1}: W_{A_1}^{2,p}(X;\Lambda^l\otimes\ad P) \to L^p(X;\Lambda^l\otimes\ad P)
\]
is Fredholm with index zero while the operator $\Delta_A - \Delta_{A_1}$ in \eqref{eq:Delta_A_Sobolev_minus_Delta_A_smooth_W2p_to_Lp_compact} is compact from each of the preceding cases, so the operator \eqref{eq:Delta_A_Sobolev_W2p_to_Lp_Fredholm} is Fredholm with index zero by \cite[Corollary 19.1.8]{Hormander_v3}.

The identification of the range of the operator \eqref{eq:Delta_A_Sobolev_W2p_to_Lp_Fredholm} follows \mutatis the proof of the corresponding fact in the statement of Theorem \ref{thm:Gilkey_1-4-5_Sobolev}. The only difference, after noting that $\Delta_A^* = \Delta_A$ and $p \leq q \implies q' \leq p'$ and thus $L^{p'}(X) \subset L^{q'}(X)$, is that we appeal to the following regularity result for distributional solutions, $b \in L^{p'}(X;\Lambda^l\otimes\ad P)$, to an elliptic linear partial differential equation, $\Delta_Ab = 0$, with Sobolev rather than $C^\infty$ coefficients,
\begin{multline*}
\Ker\left( \Delta_A: L^{p'}(X;\Lambda^l\otimes\ad P) \to W_{A_1}^{-2,p'}(X;\Lambda^l\otimes\ad P) \right)
\\
= \Ker\left( \Delta_A: W_{A_1}^{2,p}(X;\Lambda^l\otimes\ad P) \to L^p(X;\Lambda^l\otimes\ad P) \right).
\end{multline*}
Indeed, because $b \in L^{q'}(X;\Lambda^l\otimes\ad P)$ by the preceding remarks, Lemma \ref{lem:Regularity_distributional_solution_Hodge_Laplacian_Sobolev} implies that $b \in W_{A_1}^{2,q}(X;\Lambda^l\otimes\ad P)$ and, in particular, $b \in W_{A_1}^{2,p}(X;\Lambda^l\otimes\ad P)$ since $p \leq q$. This completes the proof of Corollary \ref{cor:Fredholmness_and_index_Laplace_operator_on_W2p_Sobolev_connection}.
\end{proof}

We shall also need to consider Fredholm properties of the perturbed Laplace operator,
\[
d_A^*d_{A+a}:C^\infty(X;\ad P) \to C^\infty(X;\ad P),
\]
when a $C^\infty$ connection, $A$, one-form, $a \in \Omega^l(X;\ad P)$, and Fr\'echet space, $\Omega^l(X;\ad P)$, are replaced by suitable Sobolev counterparts. As usual, we begin with the simpler case of smooth coefficients.

\begin{lem}[Fredholm and index zero properties of a perturbed Laplace operator with smooth coefficients]
\label{lem:Fredholm_index_zero_perturbed_Laplace_operator_smooth}
Let $(X,g)$ be a closed, smooth Riemannian manifold of dimension $d \geq 2$, and $G$ be a compact Lie group and $P$ be a smooth principal $G$-bundle over $X$. If $A$ and $A_1$ are $C^\infty$ connections on $P$ and $a \in C^\infty(X;\Lambda^1\otimes\ad P)$ and $1 < q < \infty$, then the operator,
\begin{equation}
\label{eq:Perturbation_Laplace_operator_W2p_to_Lp_smooth}
d_A^*d_{A+a}: W_{A_1}^{2,q}(X;\ad P) \to L^q(X;\ad P),
\end{equation}
is Fredholm with index zero.
\end{lem}

\begin{proof}
We observe that
\[
  d_A^*d_{A+a}:C^\infty(X;\ad P) \to C^\infty(X;\ad P)
\]
is an elliptic, linear, second-order partial differential operator such that
\[
d_A^*d_{A+a}-d_{A+a}^*d_A = d_A^*[a,\cdot\,] - [a,\cdot\,]^*d_A: C^\infty(X;\ad P) \to C^\infty(X;\ad P)
\]
is a first-order differential operator. The conclusions now follow from Theorem \ref{thm:Gilkey_1-4-5_Sobolev}.
\end{proof}

\begin{lem}[Fredholm and index zero properties of a perturbed Laplace operator with Sobolev coefficients]
\label{lem:Fredholm_index_zero_perturbed_Laplace_operator_Sobolev}
Let $(X,g)$ be a closed, smooth Riemannian manifold of dimension $d \geq 2$, and $G$ be a compact Lie group and $P$ be a smooth principal $G$-bundle over $X$. If $A_1$ is a $C^\infty$ connection on $P$ and $A$ is a $W^{1,q}$ connection on $P$ with $d/2 < q < \infty$ and $a \in W_{A_1}^{1,q}(X;\Lambda^1\otimes\ad P)$, then the operator,
\begin{equation}
\label{eq:Perturbation_Laplace_operator_W2p_to_Lp_Sobolev}
d_A^*d_{A+a}: W_{A_1}^{2,q}(X;\ad P) \to L^q(X;\ad P),
\end{equation}
is Fredholm with index zero.
\end{lem}

\begin{proof}
The argument is almost identical to the proof of Corollary \ref{cor:Fredholmness_and_index_Laplace_operator_on_W2p_Sobolev_connection}. Write $A = A_1+a_1$ and observe that
\[
d_A^*d_{A+a} = d_{A_1}^*d_{A_1} + d_{A_1}^*[a_1+a,\cdot\,] + [a_1,\cdot\,]^*d_{A_1} + [a_1,\cdot\,]^*[a_1+a,\cdot\,].
\]
By again retracing the steps
in the proof of Proposition \ref{prop:W2p_apriori_estimate_Delta_A_Sobolev}, we find that, for $r \in (p,\infty]$ defined by $1/p=1/q+1/r$ and $\xi \in W_{A_1}^{2,q}(X;\ad P)$,
\begin{equation}
\begin{aligned}
\label{eq:Lp_estimate_perturbed_Laplace_operator_Sobolev_minus_DeltaA1smooth}
{}&\|(d_A^*d_{A+a} - d_{A_1}^*d_{A_1})\xi\|_{L^p(X)}
\\
&\quad \leq
  z\left(\|a\|_{W_{A_1}^{1,q}(X)} + \|a_1\|_{W_{A_1}^{1,q}(X)} + \|a_1\|_{W_{A_1}^{1,q}(X)}^2 \right.
  \\
 &\qquad \left. + \|a_1\|_{W_{A_1}^{1,q}(X)} \|a\|_{W_{A_1}^{1,q}(X)} \right)\|\xi\|_{L^r(X)}
\\
&\qquad + z\left(\|a_1\|_{W_{A_1}^{1,q}(X)} + \|a\|_{W_{A_1}^{1,q}(X)}\right)\|\xi\|_{W_{A_1}^{1,t}(X)},
\end{aligned}
\end{equation}
with the values of $t$ specified in the proof of Corollary \ref{cor:Fredholmness_and_index_Laplace_operator_on_W2p_Sobolev_connection}. The remainder of the proof of Corollary \ref{cor:Fredholmness_and_index_Laplace_operator_on_W2p_Sobolev_connection} now applies to show that the operator \eqref{eq:Perturbation_Laplace_operator_W2p_to_Lp_Sobolev}is Fredholm with index zero.
\end{proof}

\chapter[Convergence of gradient flows]{Convergence of gradient flows under the validity of the {\L}ojasiewicz--Simon gradient inequality}
\label{chap:Convergence_gradient_flows_validity_Lojasiewicz-Simon_gradient_inequality}
As we noted in Section \ref{sec:Lojasiewicz-Simon_gradient_inequality_abstract_function}, the gradient inequality for an energy function, $\sE:\sX\supset \sU \to \RR$, with gradient map, $\sM:\sX\supset \sU \to \tilde\sX$, is most useful when it has the strong form \eqref{eq:Lojasiewicz-Simon_gradient_inequality_analytic_function_Hilbert_space} implied by Theorem \ref{mainthm:Lojasiewicz-Simon_gradient_inequality2}, namely,
\[
\|\sM(x)\|_{\sH} \geq Z|\sE(x) - \sE(x_\infty)|^\theta, \quad\forall\, x \in \sU \text{ with } \|x-x_\infty\|_\sX < \sigma,
\]
where $\sH$ is a Hilbert space, the Banach space $\tilde\sX$ is continuously embedded in $\sH$, and the Banach space $\sX$ is a dense subspace of $\sH$ with continuous embedding $\sX \subset \sH$ and thus $\sH^* \subset \sX^*$ is also a continuous embedding.

In this Appendix, we briefly explain why Theorem \ref{mainthm:Lojasiewicz-Simon_gradient_inequality2} is so useful in applications to questions of global existence and convergence of a \emph{strong solution}, that is, $u \in C([0,T); \sX)$ with time derivative $\dot{u} \in C((0,T); \sH)$ (for $T\in (0,\infty]$), to the Cauchy problem for the gradient system
\begin{equation}
\label{eq:Huang_3-3a_Hilbert_space}
\dot u(t) = -\sM(u(t)) \quad\text{in } \sH, \quad t \in (0,T), \quad u(0) = u_0.
\end{equation}
The importance of a geometric version of Theorem \ref{mainthm:Lojasiewicz-Simon_gradient_inequality2} to a more specific setting in geometric analysis was famously pioneered by Simon in \cite{Simon_1983}, generalizing a result of {\L}ojasiewicz for gradient flows in Euclidean spaces \cite{Lojasiewicz_1984}.

A \emph{weak solution} to the gradient system for $\sE$ has the form $u \in C([0,T); \sX)$ with time derivative $\dot{u} \in C((0,T); \sX^*)$, obeying
\begin{equation}
\label{eq:Huang_3-3a_Banach_dual_space}
\dot u(t) = -\sE'(u(t)) \quad\text{in } \sX^*, \quad t \in (0,T), \quad u(0) = u_0.
\end{equation}
To illustrate the application of Theorem \ref{mainthm:Lojasiewicz-Simon_gradient_inequality2}, we include from \cite{Feehan_yang_mills_gradient_flow_v4} a proof of a simplified version of our \cite[Proposition 24.12]{Feehan_yang_mills_gradient_flow_v4} that yields convergence, $u(t) \to u_\infty$ in $\sH$ as $t \to \infty$ for a global strong solution, $u \in C([0,\infty);\sX) \cap C^1([0,\infty);\sH)$ to \eqref{eq:Huang_3-3a_Hilbert_space}, when $\sM$ and $\sE$ obey the version of {\L}ojasiewicz--Simon gradient inequality in \eqref{eq:Lojasiewicz-Simon_gradient_inequality_analytic_function_Hilbert_space}.

The statement and proof of the forthcoming Proposition \ref{prop:Huang_3-3-2_Hilbert_space} are closely modeled on Huang's \cite[Proposition 3.3.2]{Huang_2006}, but for the gradient system \eqref{eq:Huang_3-3a_Hilbert_space} in a Hilbert space. By contrast, Huang's version allows apparently more general weak gradient-like differential inequalities in Banach spaces, namely \cite[Equation (3.10a) or Equation (3.10$'$)]{Huang_2006}, with auxiliary conditions such as those in his \cite[Equation (3.10b)]{Huang_2006} or \cite[Equation (3.10$'$)]{Huang_2006}. However, examples satisfying Huang's gradient-like differential inequalities and auxiliary conditions appear to us to be difficult to find except when they reduce to a pure gradient system \eqref{eq:Huang_3-3a_Hilbert_space} in a Hilbert space or Simon's gradient-like system \cite[Equation (3.1)]{Simon_1983} in a Hilbert space,
\[
\dot u(t) = -\sM(u(t)) + \sR(t) \quad\text{in } \sH, \quad t \in (0,T), \quad u(0) = u_0,
\]
where $\sR \in C((0,\infty);\sH)$ obeys a decay condition (as $t\to\infty$) implying Huang's \cite[Equation (3.10b)]{Huang_2006} or Simon's hypothesis in \cite[Equation (3.1)]{Simon_1983},
\[
\|\sR(t)\|_\sH \leq \alpha\|\dot{u}(t)\|_\sH,
\]
where $\alpha \in (0,1)$ is a constant.

\begin{prop}[Convergence of gradient flow under the validity of the {\L}ojasiewicz--Simon gradient inequality]
\label{prop:Huang_3-3-2_Hilbert_space}
Let $\sU$ be an open subset of a Banach space, $\sX$, that is continuously embedded and dense in a Hilbert space, $\sH$. Let $\sE:\sU\subset \sX \to \RR$ be a $C^1$ function on an open subset, $\sU \subset \sX$, with gradient map $\sM:\sU\subset \sX \to \sH$, and $x_\infty\in\sU$ be a critical point of $\sE$, that is, $\sE'(x_\infty) = 0$. Let $u \in C([0,\infty);\sU) \cap C^1([0,\infty);\sH)$ be a strong solution to \eqref{eq:Huang_3-3a_Hilbert_space} such that
\begin{equation}
\label{eq:Huang_3-14}
\inf\{|\sE(u(t))| : t \geq 0\} > -\infty.
\end{equation}
If $\sE$ and $\sM$ satisfy a {\L}ojasiewicz--Simon gradient inequality \eqref{eq:Lojasiewicz-Simon_gradient_inequality_analytic_function_Hilbert_space} in the orbit $O(u) = \{u(t): t \geq 0\}$, that is,
\begin{equation}
\label{eq:Huang_3-15a_Hilbert_space}
\|\sM(u(t))\|_\sH \geq Z|\sE(u(t)) - \sE(x_\infty)|^\theta, \quad \forall\, t \geq 0,
\end{equation}
for constants $Z \in (0,\infty)$ and $\theta \in [1/2, 1)$, then
\begin{equation}
\label{eq:Huang_3-15b_Hilbert_space}
\int_0^\infty \|\dot u(t)\|_\sH\,dt
\leq
\int_{\sE_\infty}^{\sE(u(0))} \frac{1}{Z|s - \sE(x_\infty)|^\theta} \,ds < \infty,
\end{equation}
where $\sE_\infty := \lim_{t\to\infty}\sE(u(t)) \in \RR$, and thus
\[
u(t) \to u_\infty \quad\text{in } \sH, \quad\text{as } t \to \infty,
\]
for some $u_\infty \in \sH$.
\end{prop}

\begin{proof}
The function $[0, \infty) \ni t \mapsto \sE(u(t)) \in \RR$ is $C^1$ by direct calculation and obeys (see \cite[Proposition 3.1.2]{Huang_2006})
\begin{align*}
-\frac{d}{dt}\sE(u(t)) &= -\sE'(u(t))(\dot u(t)) \quad\text{(Chain Rule)}
\\
&= -\langle \dot u(t),\sM(u(t)) \rangle_{\sX\times \sX^*} \quad\hbox{(by Definition \ref{defn:Huang_2-1-1})}
\\
&= -(\dot u(t),\sM(u(t)))_\sH  \quad\hbox{(by $\sH\subset \sX^*$ via $\sH \ni h \mapsto (\cdot, h)_\sH \in \sX^*$)}
\\
                       &= \|\sM(u(t))\|_\sH^2 = \|\sM(u(t))\|_\sH\|\dot u(t)\|_\sH \geq 0,
  \\
  &\qquad \forall\, t \in (0, \infty) \quad\hbox{(by \eqref{eq:Huang_3-3a_Hilbert_space})}.
\end{align*}
Hence, $\sE(u(t))$ is a nonincreasing and uniformly bounded function of $t \in [0,\infty)$ by \eqref{eq:Huang_3-14}, so $\sE_\infty = \lim_{t\to\infty}\sE(u(t))$ exists, as asserted by the proposition. Set $H(t) := \sE(u(t))$, for all $t \in [0, \infty)$, and observe that $H(t)$ is monotone and absolutely continuous on $[0, \infty)$ and obeys, by the preceding equality,
\begin{equation}
\label{eq:Huang_3-16a}
-\frac{d}{dt}H(t) = \|\sM(u(t))\|_\sH\|\dot u(t)\|_\sH, \quad\forall\, t \in [0, \infty).
\end{equation}
Let $\phi:\RR\to\RR$ be the function defined by $\phi(s) = Z|s - \sE(x_\infty)|^\theta$, for all $s \in \RR$, and let $\Phi:\RR \to \RR$ be the absolutely continuous function given by
$$
\Phi(x) := \int_{\sE_\infty}^x \frac{1}{\phi(s)} \,ds = \int_{\sE_\infty}^x \frac{1}{Z|s - \sE(x_\infty)|^\theta} \,ds, \quad \forall\, x \in \RR,
$$
where $\lim_{t\to\infty}H(t) = \sE_\infty$. The function $\Phi$ is differentiable a.e. on $\RR$ with $\Phi'(x) = 1/\phi(x)$ for a.e. $x \in \RR$. According to \cite[Lemma 3.2.1]{Huang_2006}, the composition $\Phi\circ H$ is absolutely continuous on $[0, \infty)$ and there holds
\begin{equation}
\label{eq:Huang_3-16b}
\frac{d}{dt}\Phi(H(t)) = \frac{H'(t)}{\phi(H(t))}, \quad \forall\, t \in \Lambda,
\end{equation}
where $\Lambda \subset [0, \infty)$ is such that the complement, $[0, \infty) \less \Lambda$, has zero Lebesgue measure.

For any $t \in \Lambda$, we have two possibilities: either
\begin{inparaenum}[i\upshape)]
\item \label{item:Norm_gradient_at_u(t)_is_zero}
$\|\sM(u(t))\|_\sH = 0$, or
\item \label{item:Norm_gradient_at_u(t)_is_positive}
$\|\sM(u(t))\|_\sH > 0$.
\end{inparaenum}
For Case \eqref{item:Norm_gradient_at_u(t)_is_positive}, we observe that the {\L}ojasiewicz--Simon gradient inequality \eqref{eq:Huang_3-15a_Hilbert_space} takes the shape,
\begin{equation}
\label{eq:Huang_3-16c}
\phi(H(t)) = Z|\sE(u(t)) - \sE(x_\infty)|^\theta \leq \|\sM(u(t))\|_\sH ,
\end{equation}
and so
\begin{align*}
-\frac{d}{dt}\Phi(H(t)) &= -\frac{H'(t)}{\phi(H(t))} \quad\text{(by \eqref{eq:Huang_3-16b})}
\\
&= \frac{\|\sM(u(t))\|_\sH\|\dot u(t)\|_\sH}{\phi(H(t))} \quad\text{(by  \eqref{eq:Huang_3-16a})}
\\
&\geq \frac{\|\sM(u(t))\|_\sH\|\dot u(t)\|_\sH}{\|\sM(u(t))\|_\sH}  \quad\text{(by \eqref{eq:Huang_3-16c})}
\\
&= \|\dot u(t)\|_\sH,
\end{align*}
that is,
\begin{equation}
\label{eq:Huang_3-16d}
-\frac{d}{dt}\Phi(H(t))  \geq \|\dot u(t)\|_\sH.
\end{equation}
Therefore, by the non-negativity of the function $-d\Phi(H(t))/dt$, combined with the fact that $[0, \infty) \less \Lambda$ has Lebesgue measure zero, we obtain the estimate,
\begin{equation}
\label{eq:Huang_3-18a_Hprelim}
-\frac{d}{dt}\Phi(H(t))  \geq \|\dot u(t)\|_\sH, \quad \hbox{a.e. } t \in [0, \infty),
\end{equation}
for both Cases \eqref{item:Norm_gradient_at_u(t)_is_zero} and\eqref{item:Norm_gradient_at_u(t)_is_positive}. Integration and the fact that $\lim_{t\to\infty}H(t) = \sE_\infty$ yields
$$
\int_0^\infty \|\dot u(t)\|_\sH\,dt \leq \Phi(H(0)) - \lim_{t\to\infty}\Phi(H(t)) = \Phi(H(0)) - \Phi(\sE_\infty) = \Phi(H(0)).
$$
By the definitions of $\Phi(x)$ and $H(t)$, this is \eqref{eq:Huang_3-15b_Hilbert_space}, since
$$
\Phi(H(0))
=
\int_{\sE_\infty}^{H(0)} \frac{1}{\phi(s)} \,ds
=
\int_{\sE_\infty}^{\sE(u(0))} \frac{1}{Z|s - \sE(x_\infty)|^\theta} \,ds.
$$
The final convergence assertion follows from the fact that
\[
u(t_n) - u(t_m) = \int_{t_m}^{t_n} \dot{u}(t) \,dt, \quad\forall\, t_m, t_n \in [0,\infty),
\]
and thus, for any unbounded sequence of times, $\{t_n\}_{n=1}^\infty \subset [0,\infty)$, the sequence of points, $\{u(t_n)\}_{n=1}^\infty \subset \sH$, is Cauchy in $\sH$ and thus converges to a limit, $u_\infty \in \sH$, as $t\to\infty$, independent of the choice, $\{t_n\}_{n=1}^\infty$. This completes the proof of Proposition \ref{prop:Huang_3-3-2_Hilbert_space}.
\end{proof}

The convergence of $u(t)$ in $\sH$ provided by Proposition \ref{prop:Huang_3-3-2_Hilbert_space} can be improved with the assumption of the following

\begin{hyp}[Regularity and \apriori interior estimate for a trajectory]
\label{hyp:Abstract_apriori_interior_estimate_trajectory}
(See \cite[Hypothesis 24.10]{Feehan_yang_mills_gradient_flow_v4}.)
Let $C_1$ and $\rho$ be positive constants and let $T \in (0, \infty]$. Given $u \in C([0,\infty);\sX) \cap C^1([0,\infty);\sH)$, we say that $\dot u:[0,T) \to \sH$ obeys an \apriori \emph{interior estimate on $(0, T]$} if, for every $S \geq 0$ and $\delta > 0$ obeying $S+\delta \leq T$, the map $\dot u:[S+\delta,T) \to \sX$ is Bochner integrable and there holds
\begin{equation}
\label{eq:Abstract_apriori_interior_estimate_trajectory}
\int_{S+\delta}^T \|\dot u(t)\|_\sX\,dt \leq C_1(1+\delta^{-\rho})\int_S^T \|\dot u(t)\|_\sH\,dt.
\end{equation}
\end{hyp}

Given Hypothesis \ref{hyp:Abstract_apriori_interior_estimate_trajectory}, the bound \eqref{eq:Huang_3-15b_Hilbert_space} in Proposition \ref{prop:Huang_3-3-2_Hilbert_space} improves to
\begin{equation}
\label{eq:Huang_3-15b_Banach_space}
\int_\delta^\infty \|\dot u\|_\sX\,dt
\leq
C_1(1+\delta^{-\rho}) \int_{\sE_\infty}^{\sE(u(0))} \frac{1}{c|s - \sE(x_\infty)|^\theta} \,ds < \infty,
\end{equation}
for all $\delta > 0$, and so the convergence in $\sH$ improves to $u(t) \to u_\infty$ in $\sX$ as $t \to \infty$.

The application of Proposition \ref{prop:Huang_3-3-2_Hilbert_space} and similar results to proofs of global existence, convergence, convergence rates, and stability of solutions to \eqref{eq:Huang_3-3a_Hilbert_space} are described at length in \cite[Section 2.1]{Feehan_yang_mills_gradient_flow_v4}.

\chapter[Huang's {\L}ojasiewcz--Simon gradient inequality]{Huang's {\L}ojasiewcz--Simon gradient inequality for analytic functions on Banach spaces}
\label{chap:Huang_2-4-2}
For the convenience of the reader, we quote Huang's \cite[Theorem 2.4.2 (i)]{Huang_2006} for the {\L}ojasiewcz--Simon gradient inequality for analytic functions on Banach spaces. We first recall Huang's hypotheses for \cite[Theorem 2.4.2 (i)]{Huang_2006}.

\begin{hyp}[Hypotheses for the abstract {\L}ojasiewicz-Simon gradient inequality with Hilbert space gradient norm]
\label{hyp:Huang_2-4_H1_H2_H3}
(See \cite[pp. 34--35]{Huang_2006}.) Assume the following conditions.
\begin{enumerate}
\item Let $\sH$ be a Hilbert space, $\sA:\sD(\sA)\subset \sH \to \sH$ a linear, positive definite, self-adjoint operator, and $\sH_\sA := (\sD(\sA), (\cdot,\cdot)_\sA)$ be the Hilbert space with inner product,
\[
(u,v)_\sA := (\sA u, \sA v)_\sH, \quad\forall\, u, v \in \sD(\sA),
\]
where $(\cdot, \cdot)_\sH$ is the inner product on $\sH$.

\item Let $\sX \subset \tilde\sX$ be Banach spaces such the following embeddings are continuous,
\[
\sX \subset \sH_\sA, \quad \tilde \sX \subset \sH.
\]
\item Let $\sE:\sX\to\RR$ be a function with $C^1$ gradient map, $\sM:\sU\subset \sX\to \tilde\sX$, so that
    \[
    \sE'(x)v = (v, \sM(x))_\sH, \quad\forall\, x\in \sU \text{ and } v \in \sX,
    \]
    where $\sU\subset \sX$ is an open subset, and having the following properties:
\begin{enumerate}
\item $\sM$ is a Fredholm map of index zero, that is, for each $x \in \sU$,
\[
\sM'(x): \sX \to \tilde \sX,
\]
is a Fredholm operator of index zero.

\item For each $x \in \sU$, the bounded, linear operator,
\[
\sM'(x): \sX \to \tilde \sX,
\]
has an extension
\[
\sM_1(x): \sH_A \to \sH
\]
which is symmetric\footnote{See \cite[Section 7.4]{Brezis}, \cite[Section 5.3.3]{Kato}, or \cite[Section 7.3]{Yosida}.} and also a Fredholm operator of index zero and such that the map
\[
\sU \ni x \mapsto \sM_1(x) \in \sL(\sH_\sA, \sH) \quad\hbox{is continuous},
\]
or, equivalently, the map $\sU \ni x \mapsto \sM_1(x)\sA^{-1} \in \sL(\sH)$ is continuous.
\end{enumerate}
\end{enumerate}
\end{hyp}

\begin{thm}[{\L}ojasiewcz--Simon gradient inequality for analytic functions on Banach spaces and Hilbert space gradient]
\cite[Theorem 2.4.2 (i)]{Huang_2006}
\label{thm:Huang_2-4-2}
Assume Hypothesis \ref{hyp:Huang_2-4_H1_H2_H3} on $\sE$, $\sM$, $\sM_1$, $\sH$, $\sU$, $\sX$, and $\tilde \sX$, and that $\sM:\sU \subset \sX \to \tilde\sX$ is real analytic. If $x_\infty \in \sU$ is a critical point of $\sE$, that is, $\sM(x_\infty) = 0$, then there are positive constants, $c$, $\sigma$, and $\theta \in [1/2, 1)$ such that
\begin{equation}
\label{eq:Simon_2-2}
\|\sM(x)\|_\sH \geq c|\sE(x) - \sE(x_\infty)|^\theta, \quad \forall\, x \in \sU \hbox{ such that } \|x-x_\infty\|_\sX < \sigma.
\end{equation}
\end{thm}

\chapter{Quantitative implicit and inverse function theorems}
\label{chap:Quantitative_implicit_function_theorem}
Statements of the Implicit Function Theorem are often most useful in applications when equipped with explicit estimates for the radii of the balls containing the domain and range of the implied function that the statement produces, along with an explicit estimate for the Lipschitz constant of that function. We shall state and prove a version of such a result in more generality than we need in our current application, since precise statements are not easy to find in the literature.

\begin{thm}[Quantitative implicit function theorem for maps of Banach spaces]
\label{thm:Quantitative_implicit_function_theorem}
Let $\KK = \RR$ or $\CC$, and $k \geq 1$ be an integer or $\infty$, and $\sX$, $\sY$, and $\sZ$ be Banach spaces over $\KK$, and $\sU \subset \sX$ and $\sV \subset \sY$ be open neighborhoods of points $x_0 \in \sX$ and $y_0 \in \sY$, and $f:\sU\times\sV\to \sZ$ be a $C^k$ (respectively, analytic) map such that $f(x_0,y_0) = 0$ and the partial derivative of $f$ at $(x_0,y_0)$ with respect to the second variable,
\begin{equation}
\label{eq:Partial_derivative_f_isomorphism}
  D_2f(x_0,y_0) \in \sL(\sY,\sZ),
\end{equation}
is an isomorphism of Banach spaces. Let $\zeta \in (0,1]$ be small enough that $B_\zeta(x_0) \subset \sU$ and $B_\zeta(y_0) \subset \sV$ and assume that
\begin{subequations}
  \label{eq:D_1_and_D_2f_conditions_implicit_function_theorem}  
\begin{gather}
  \label{eq:D_2f_x0y_0_inverse}
  M := \|(D_2f(x_0,y_0))^{-1}\|_{\sL(\sZ,\sY)},
  \\
  \label{eq:Uniform_Lipschitz_bound_D_2fxy}
  \sup_{(x,y)\in B_\zeta(x_0) \times B_\zeta(y_0)}\|D_2f(x,y) - D_2f(x_0,y_0)\|_{\sL(\sX\times\sY,\sZ)} \leq \frac{1}{2M},
  \\
  \label{eq:D1f_uniform_bound}
  \beta := \sup_{(x,y)\in B_\zeta(x_0) \times B_\zeta(y_0)}\|D_1f(x,y)\|_{\sL(\sX,\sZ)} < \infty.
\end{gather}
\end{subequations}
Then there are a constant $\delta \in (0,\min\{\zeta,\zeta/(2\beta M)\}]$ and a unique $C^k$ (respectively, analytic) map $g:\sX \supset B_\delta(x_0) \to B_\zeta(y_0) \subset \sY$ such that $y_0=g(x_0)$ and
\begin{subequations}
\begin{gather}
\label{eq:f_x_gx_equation}
  f(x,g(x)) = 0, \quad\forall\, x \in B_\delta(x_0),
  \\
  \label{eq:Derivative_g}
  Dg(x) = -(D_2f(x,g(x)))^{-1}D_1f(x,g(x)) \in \sL(\sX,\sY), \quad\forall\, x \in B_\delta(x_0),
    \\
  \label{eq:Lipschitz_constant_implied_function_g}
  \|g(x_1)-g(x_2)\|_\sY \leq 2\beta M\|x_1 - x_2\|_\sX, \quad\forall\, x_1, x_2 \in B_\delta(x_0).
\end{gather}
\end{subequations}
\end{thm}

\begin{rmk}[Inverse and implicit function theorems for smooth or analytic maps of Banach spaces]
\label{rmk:Smooth_and_analytic_inverse_and_implicit_function_theorems}  
Statements and proofs of the Inverse Function Theorem for $C^k$ maps of Banach spaces are provided by Abraham, Marsden, and Ratiu \cite[Theorem 2.5.2]{AMR}, Apostol \cite[Theorem 13.6]{Apostol_math_analysis} (Euclidean space only), Deimling \cite[Theorem 4.15.2]{Deimling_1985}, Lang \cite[Theorem XIV.1.2]{Lang_analysis}, and Zeidler \cite[Theorem 4.F]{Zeidler_nfaa_v1}; statements and proofs of the Inverse Function Theorem for \emph{analytic} maps of Banach spaces are provided by Berger \cite[Corollary 3.3.2]{Berger_1977} (complex), Deimling \cite[Theorem 4.15.3]{Deimling_1985} (real or complex), and Zeidler \cite[Corollary 4.37]{Zeidler_nfaa_v1} (real or complex). The corresponding $C^k$ or Analytic Implicit Function Theorems are proved in the standard way as corollaries, for example \cite[Theorem 2.5.7]{AMR} and \cite[Theorem 4.H]{Zeidler_nfaa_v1}, although the reverse strategy is used by Jost \cite[Theorem 10.1]{Jost_postmodern_analysis} and others.

    A quantitative version of the Inverse Function Theorem is given by Abraham, Marsden, and Ratiu as \cite[Proposition 2.5.6]{AMR}, although in the latter version the conclusion \eqref{eq:Lipschitz_constant_implied_function_g} on the Lipschitz constant for the inverse map, $g$, is omitted. See also Fr{\o}yshov \cite[Proposition B.0.2]{Froyshov_2008}, Kronheimer and Mrowka \cite[Proposition 18.3.6]{KMBook}, McDuff and Salamon \cite[Proposition A.3.4]{McDuffSalamon1}, Mrowka and Rollin \cite[Proposition 2.3.5]{Mrowka_Rollin_2006}, and Salamon \cite[Theorem B.1]{SalamonSWBook} for closely-related versions of the Quantitative Inverse Function Theorem in the gauge-theory literature, though none quite as general as our Theorem \ref{thm:Quantitative_inverse_function_theorem}.
  \end{rmk}

\begin{rmk}[Alternative hypothesis for implicit function theorems for $C^2$ maps of Banach spaces]
\label{rmk:C^2_implicit_function_theorems}
When $f$ is $C^2$, one can replace the hypothesis \eqref{eq:Uniform_Lipschitz_bound_D_2fxy} in Theorem \ref{thm:Quantitative_implicit_function_theorem} by the simpler condition,
\begin{equation}
    \label{eq:Uniform_bound_D^2fxy}
    K := \sup_{(x,y)\in B_\zeta(x_0) \times B_\zeta(y_0)}\|D^2f(x,y)\|_{\sL(\sX\times\sY,\sZ)} < \infty.
\end{equation}
The Mean Value Theorem for $C^1$ maps of Banach spaces (see \eqref{eq:MVT_h}) gives
\begin{align*}
  D_2f(x,y) - D_2f(x_0,y_0)  &= \int_0^1 D_{21}f(tx + (1-t)x_0, ty + (1-t)y_0)(x-x_0)\,dt
  \\
                           &\quad + \int_0^1 D_2^2h(tx + (1-t)x_0, ty + (1-t)y_0)(y-y_0)\,dt,
  \\
  &\qquad\forall\, x \in \sU \text{ and } y \in \sV,
\end{align*}
assuming that $\sU$ and $\sV$ are convex without loss of generality. Therefore,
\[
  \|D_2f(x,y) - D_2(x_0,y_0)\|_{\sL(\sX\times\sY,\sZ)}
  \leq
  K(\|x-x_0\|_\sX + \|y-y_0\|_\sY)
\]
and thus it suffices to choose $\eta \in (0,\zeta]$ small enough that $\eta \leq 1/(4KM)$ in order to achieve \eqref{eq:Uniform_Lipschitz_bound_D_2fxy}, with $\zeta$ replaced by $\eta$ there (so $(x,y) \in B_\eta(x_0)\times B_\eta(y_0)$) and in the conclusions of Theorem \ref{thm:Quantitative_implicit_function_theorem}.    
\end{rmk}  

\begin{proof}[Proof of Theorem \ref{thm:Quantitative_implicit_function_theorem}]
When $\KK=\RR$ and in the absence of the supplementary hypothesis \eqref{eq:Uniform_Lipschitz_bound_D_2fxy}, the existence of a unique $C^k$ map $g$ obeying \eqref{eq:f_x_gx_equation} is provided by standard statements of the Implicit Function Theorem (see Abraham, Marsden, and Ratiu \cite[Theorem 2.5.7]{AMR}, Lang \cite[Theorem XIV.2.1]{Lang_analysis}, and, in the case of Euclidean spaces, Apostol \cite[Theorem 13.7]{Apostol_math_analysis}). For the cases $\KK=\CC$ or $f$ real or complex analytic, see the references cited in Remark \ref{rmk:Smooth_and_analytic_inverse_and_implicit_function_theorems} for proofs of the Inverse Function Theorem and recall that the Implicit Function Theorem can be obtained in a straightforward way as a corollary of the Inverse Function Theorem.

Hence, we need only focus on the quantitative conclusions of Theorem \ref{thm:Quantitative_implicit_function_theorem} and, for this purpose, we shall adapt Lang's proof of the Inverse Function Theorem (see Lang \cite[Theorem XIV.1.2]{Lang_analysis}) and Jost's proof of the Implicit Function Theorem \cite[Theorem  10.1]{Jost_postmodern_analysis}. Recall that given $x$ in an open neighborhood of $x_0$, the existence of $y$ in an open neighborhood of $y_0$ solving $f(x,y) = 0$ in \eqref{eq:f_x_gx_equation} is equivalent to the existence of $y$ solving $h(x,y) = y$, where
\begin{equation}
  \label{eq:Defn_h}
  h(x,y) := y - (D_2f(x_0,y_0))^{-1}f(x,y).
\end{equation}
Observe that the smooth map $h:\sU\times\sV \to \sY$ has partial derivatives,
\begin{align*}
  D_1h(x,y) &= -(D_2f(x_0,y_0))^{-1}D_1f(x,y), 
  \\
  D_2h(x,y) &= \id_\sY - (D_2f(x_0,y_0))^{-1}D_2f(x,y), \quad\forall\, (x,y) \in \sU\times\sV,
\end{align*}
and $D_2h(x_0,y_0) = 0$. Moreover,
\begin{align*}
  \|D_2h(x,y)\|_{\sL(\sY)} &= \left\|(D_2f(x_0,y_0))^{-1}\left(D_2f(x_0,y_0) - D_2f(x,y)\right)\right\|_{\sL(\sY)}
  \\
  &\leq \|(D_2f(x_0,y_0))^{-1}\|_{\sL(\sZ,\sY)} \|D_2f(x_0,y_0) - D_2f(x,y)\|_{\sL(\sY,\sZ)}.
\end{align*}
By hypothesis \eqref{eq:Uniform_Lipschitz_bound_D_2fxy} and the preceding bound for $\|D_2h(x,y)\|_{\sL(\sY)}$ and by hypothesis \eqref{eq:D1f_uniform_bound} and the preceding expression for $D_1h(x,y)$, we therefore have
\begin{subequations}
  \label{eq:D_ihxy_bounds}
  \begin{align}
  \label{eq:D_1hxy_leq_Mbeta}
  \|D_1h(x,y)\|_{\sL(\sY)} &\leq \beta M, 
    \\
    \label{eq:D_2hxy_leq_1over2}
  \|D_2h(x,y)\|_{\sL(\sY)} &\leq \frac{1}{2}, \quad\forall\, (x,y) \times B_\zeta(x_0)\times B_\zeta(y_0).
\end{align}
\end{subequations}
The Mean Value Theorem for $C^1$ maps of Banach spaces \cite[Theorem XIII.4.2]{Lang_analysis} gives
\begin{align}
  \label{eq:MVT_h}
  {}&h(x_1,y_1) - h(x_2,y_2)
  \\
  &\quad = \int_0^1 D_1h(tx_1 + (1-t)x_2, ty_1 + (1-t)y_2)(x_1-x_2)\,dt \nonumber
  \\
                           &\qquad + \int_0^1 D_2h(tx_1 + (1-t)x_2, ty_1 + (1-t)y_2)(y_1-y_2)\,dt, \nonumber
  \\
  &\qquad\qquad\forall\, x_1,x_2 \in \sU \text{ and } y_1,y_2 \in \sV, \nonumber
\end{align}
assuming that $\sU$ and $\sV$ are convex without loss of generality. Hence, by \eqref{eq:D_2hxy_leq_1over2} and \eqref{eq:MVT_h}
\begin{equation}
  \label{eq:Jost_10-9}
  \|h(x,y_1)-h(x,y_2)\|_\sY \leq \frac{1}{2}\|y_1-y_2\|_\sY, \quad\forall\, x \in \bar B_\zeta(x_0) \text{ and } y_1,y_2 \in \bar B_\zeta(y_0).
\end{equation}
Moreover, 
\begin{align*}
  \|h(x,y)-y_0\|_\sY &= \|h(x,y)-h(x_0,y_0)\|_\sY \quad\text{(by $h(x_0,y_0) = y_0$)}
  \\
  &\leq \|h(x,y)-h(x_0,y)\|_\sY + \|h(x_0,y)-h(x_0,y_0)\|_\sY
  \\
  &\leq  \beta M\|x-x_0\|_\sX +  \frac{1}{2}\|y-y_0\|_\sY \quad\text{(by \eqref{eq:D_1hxy_leq_Mbeta}, \eqref{eq:MVT_h}, and \eqref{eq:Jost_10-9}).}
\end{align*}
Hence, for $\delta \in (0,\zeta]$ obeying $\delta \leq \zeta/(2\beta M)$, we obtain
\begin{equation}
   \label{eq:Jost_10-10b}
  \|h(x,y)-y_0\|_\sY \leq \zeta, \quad\forall\, (x,y) \in B_\delta(x_0)\times B_\zeta(y_0).
\end{equation}
Thus, for each $x \in \bar B_\delta(x_0)$, the map $h(x,\cdot): \bar B_\zeta(y_0) \to \bar B_\zeta(y_0)$ is a contraction. We can now apply the Banach Contraction Mapping Lemma (for example, see \cite[Lemma 10.2]{Jost_postmodern_analysis} or \cite[Lemma XIV.1.1]{Lang_analysis}) to find a unique $y = g(x) \in \bar B_\zeta(y_0)$ for every $x \in \bar B_\delta(x_0)$ such that $h(x,y)=y$, that is, $f(x,y) = 0$, and $y_0 = g(x_0)$.

To prove \eqref{eq:Lipschitz_constant_implied_function_g}, we shall apply the Mean Value Theorem,
\begin{equation}
  \label{eq:MVT_g}
  g(x_1) - g(x_2) = \int_0^1 Dg(tx_1 + (1-t)x_2)(x_1-x_2)\,dt, \quad\forall\, x \in B_\delta(x_0).
\end{equation}
Computing the derivative of the equation \eqref{eq:f_x_gx_equation} with respect to $x$ yields
\[
  D_1f(x,g(x)) + D_2f(x,g(x))Dg(x) = 0 \in \sL(\sX,\sZ), \quad\forall\, x \in B_\delta(x_0),
\]
so that \eqref{eq:Derivative_g} holds provided $D_2f(x,g(x))$ is invertible. To verify invertibility, we define an operator $T_x \in \sL(\sY)$ by
\[
  (D_2f(x_0,g(x_0)))^{-1}D_2f(x,g(x)) =: \id_\sY - T_x \in \sL(\sY), \quad\forall\, x \in B_\delta(x_0).
\]
Observe that
\[
  T_x = \id_\sY - (D_2f(x_0,g(x_0)))^{-1}D_2f(x,g(x)) = D_2h(x,g(x)) \quad\text{(by \eqref{eq:Defn_h})}
\]
and so $\|T_x\|_{\sL(\sY)} \leq 1/2$ for all $x \in B_\delta(x_0)$ by \eqref{eq:D_2hxy_leq_1over2}. Since $\|T_x\|_{\sL(\sY)} < 1$, for all $x \in B_\delta(x_0)$, the operator $\id_\sY - T_x$ is invertible, for all $x \in B_\delta(x_0)$. Hence, $D_2f(x,g(x)) \in \sL(\sY,\sZ)$ is invertible for all $x \in B_\delta(x_0)$, as we had claimed in \eqref{eq:Derivative_g}, with
\begin{equation}
  \label{eq:D_2fxgx_inverse}
  (D_2f(x,g(x)))^{-1} = (\id_\sY - T_x)^{-1}(D_2f(x_0,g(x_0)))^{-1}, \quad\forall\, x \in B_\delta(x_0).
\end{equation}
We can estimate the operator norm of $(\id_\sY - T_x)^{-1}$ (and hence that of $(D_2f(x,g(x)))^{-1}$) using the Neumann series,
\[
  (\id_\sY - T_x)^{-1} = \sum_{k=0}^\infty T_x^k,
\]
to give

\[
  \|(\id_\sY - T_x)^{-1}\|_{\sL(\sY)} \leq \sum_{k=0}^\infty \|T_x^k\|_{\sL(\sY)}  \leq \sum_{k=0}^\infty \|T_x\|_{\sL(\sY)}^k = \frac{1}{1 - \|T_x\|_{\sL(\sY)}},
\]
and consequently,
\begin{equation}
  \label{eq:Bound_operator_norm_id-Tx_inverse}
  \|(\id_\sY - T_x)^{-1}\|_{\sL(\sY)} \leq 2, \quad\forall\, x \in B_\delta(x_0).
\end{equation}
We can therefore estimate the operator norm of $(D_2f(x,g(x)))^{-1}$ by
\begin{align*}
  {}&\|(D_2f(x,g(x)))^{-1}\|_{\sL(\sZ,\sY)}
  \\
  &\quad = \|(\id_\sY - T_x)^{-1}(D_2f(x_0,g(x_0)))^{-1}\|_{\sL(\sZ,\sY)}
                                           \quad \text{(by \eqref{eq:D_2fxgx_inverse})}
  \\
                                         &\quad \leq \|(\id_\sY - T_x)^{-1}\|_{\sL(\sY)}\|(D_2f(x_0,g(x_0)))^{-1}\|_{\sL(\sZ,\sY)}
  \\
                                         &\quad \leq 2M \quad \text{(by \eqref{eq:D_2f_x0y_0_inverse} and $y_0=g(x_0)$ and \eqref{eq:Bound_operator_norm_id-Tx_inverse})}, \quad\forall\, x \in B_\delta(x_0).
\end{align*}
By hypothesis \eqref{eq:D1f_uniform_bound} and the fact that $g(B_\delta(x_0)) \subset B_\zeta(x_0)$ by construction, we have
\[
  \sup_{x \in B_\delta(x_0)}\|D_1f(x,g(x))\|_{\sL(\sX\times\sY,\sZ)} \leq \beta
\]
and so the expression \eqref{eq:Derivative_g} for $Dg(x)$ and the preceding two estimates give
\begin{align*}
  \|Dg(x)\|_{\sL(\sX,\sY)} &= \|(D_2f(x,g(x)))^{-1}D_1f(x,g(x))\|_{\sL(\sX,\sY)}
  \\
                           &\leq \|(D_2f(x,g(x)))^{-1}\|_{\sL(\sZ,\sY)}\|D_1f(x,g(x))\|_{\sL(\sX,\sZ)}
  \\
                           &\leq 2\beta M,   \quad\forall\, x \in B_\delta(x_0).
\end{align*}                             
The Lipschitz bound \eqref{eq:Lipschitz_constant_implied_function_g} now follows from the preceding estimate and \eqref{eq:MVT_g}. This completes the proof of Theorem \ref{thm:Quantitative_implicit_function_theorem}.
\end{proof}

The quantitative inverse function theorem below (see \cite[Theorem 3.2]{FeehanSlice}) follows in a standard way from the preceding quantitative implicit function theorem.

\begin{thm}[Quantitative inverse function theorem for maps of Banach spaces]
\label{thm:Quantitative_inverse_function_theorem}
Let $\KK = \RR$ or $\CC$, and $k \geq 1$ be an integer or $\infty$, and $\sX$ and $\sY$ be Banach spaces over $\KK$, and $\sU \subset \sX$ be an open neighborhood of a point $x_0 \in \sX$, and $f:\sU\to \sY$ be a $C^k$ (respectively, analytic) map such that $f(x_0) = y_0$ and the derivative of $f$ at $x_0$,
\begin{equation}
\label{eq:Derivative_f_isomorphism}
  Df(x_0) \in \sL(\sX,\sY),
\end{equation}
is an isomorphism of Banach spaces.  Let $\zeta \in (0,1]$ be small enough that $B_\zeta(x_0) \subset \sU$ and assume that
\begin{subequations}
  \label{eq:Df_conditions_inverse_function_theorem}  
  \begin{gather}
  \label{eq:Df_x0_inverse}
  M := \|(Df(x_0))^{-1}\|_{\sL(\sY,\sX)},
  \\
  \label{eq:Uniform_Lipschitz_bound_Dfx}
  \sup_{x \in B_\zeta(x_0)}\|Df(x) - Df(x_0)\|_{\sL(\sX,\sY)} \leq \frac{1}{2M}.
\end{gather}
\end{subequations}
Then there are a constant $\delta \in (0,\min\{\zeta,\zeta/(2M)\}]$ and a unique $C^k$ (respectively, analytic) map $g:\sY \supset B_\delta(y_0) \to B_\zeta(x_0) \subset \sX$ such that $x_0=g(y_0)$ and
\begin{subequations}
\begin{align}
\label{eq:fgy_equation}
  f(g(y)) &= y, \quad\forall\, y \in B_\delta(y_0),
  \\
  \label{eq:Derivative_finverse}
  Dg(x) &= (Df(x))^{-1} \in \sL(\sY,\sX), \quad\forall\, x \in B_\delta(x_0),
  \\
  \label{eq:Lipschitz_constant_inverse_function_g}
  \|g(y_1)-g(y_2)\|_\sX &\leq 2M\|y_1 - y_2\|_\sY, \quad\forall\, y_1, y_2 \in B_\delta(y_0).
\end{align}
\end{subequations}
\end{thm}

\begin{rmk}[Alternative hypothesis for inverse function theorems for $C^2$ maps of Banach spaces]
\label{rmk:C^2_inverse_function_theorems}
When $f$ is $C^2$, one can replace the hypothesis \eqref{eq:Uniform_Lipschitz_bound_Dfx} in Theorem \ref{thm:Quantitative_inverse_function_theorem} by the simpler condition,
\begin{equation}
    \label{eq:Uniform_bound_D^2fx}
    K := \sup_{x\in B_\zeta(x_0)}\|D^2f(x)\|_{\sL(\sX,\sY)} < \infty.
\end{equation}
Just as in Remark \ref{rmk:C^2_implicit_function_theorems}, the Mean Value Theorem for $C^1$ maps of Banach spaces ensures that we can then choose $\eta \in (0,\zeta]$ small enough that $\eta \leq 1/(2KM)$ in order to achieve \eqref{eq:Uniform_Lipschitz_bound_Dfx}, with $\zeta$ replaced by $\eta$ there (so $x \in B_\eta(x_0)$) and in the conclusions of Theorem \ref{thm:Quantitative_inverse_function_theorem}. The hypothesis \eqref{eq:Uniform_bound_D^2fx} is the one used in \cite[Proposition 2.5.6]{AMR} and \cite[Proposition B.0.2]{Froyshov_2008}, whereas the remaining examples of quantitative inverse function theorems cited in Remark \ref{rmk:Smooth_and_analytic_inverse_and_implicit_function_theorems} use the hypothesis \eqref{eq:Uniform_Lipschitz_bound_Dfx}.   
\end{rmk}  

\begin{proof}[Proof of Theorem \ref{thm:Quantitative_inverse_function_theorem}]
  Set $\sZ := \sY$ and $F(x,y) := f(x) - y$ and observe that $D_1F(x,y) = Df(x)$, so \eqref{eq:Partial_derivative_f_isomorphism} is obeyed (interchanging the roles of $x$ and $y$). Moreover, $(D_1F(x_0,y_0))^{-1} = (Df(x_0))^{-1}$, so $\|(D_1F(x_0,y_0))^{-1}\|_{\sL(\sY,\sX)} = M$ in \eqref{eq:D_2f_x0y_0_inverse} by \eqref{eq:Df_x0_inverse}. Also, \eqref{eq:Uniform_Lipschitz_bound_Dfx} yields \eqref{eq:Uniform_Lipschitz_bound_D_2fxy} for all $y\in\sY$.  Finally, $D_2F(x,y) = -\id_\sY$ for all $x \in \sU$ and $y\in\sY$, so $\beta = 1$ in \eqref{eq:D1f_uniform_bound}. Hence, Theorem \ref{thm:Quantitative_implicit_function_theorem} applied to $F(x,y)$ yields $\delta \in (0,\min\{\zeta, \zeta/(2M)\}]$ and a $C^k$ (respectively, analytic) map $g:\sY\supset B_\delta(y_0) \to B_\zeta(x_0) \subset \sX$ such that $F(g(y),y) = 0$, that is, $f(g(y)) = y$ for all $y \in B_\delta(y_0)$. Lastly, \eqref{eq:Derivative_g} yields \eqref{eq:Derivative_finverse} and \eqref{eq:Lipschitz_constant_implied_function_g} yields \eqref{eq:Lipschitz_constant_inverse_function_g}.
\end{proof}

\backmatter

%
%

\bibliography{master,mfpde}

\def\cprime{$'$} \def\cprime{$'$}
  \def\ocirc#1{\ifmmode\setbox0=\hbox{$#1$}\dimen0=\ht0 \advance\dimen0
  by1pt\rlap{\hbox to\wd0{\hss\raise\dimen0
  \hbox{\hskip.2em$\scriptscriptstyle\circ$}\hss}}#1\else {\accent"17 #1}\fi}
  \def\cprime{$'$} \def\cprime{$'$} \def\cprime{$'$} \def\cprime{$'$}
  \def\polhk#1{\setbox0=\hbox{#1}{\ooalign{\hidewidth
  \lower1.5ex\hbox{`}\hidewidth\crcr\unhbox0}}} \def\cprime{$'$}
  \def\cprime{$'$} \def\cprime{$'$}
  \def\lfhook#1{\setbox0=\hbox{#1}{\ooalign{\hidewidth
  \lower1.5ex\hbox{'}\hidewidth\crcr\unhbox0}}} \def\cprime{$'$}
  \def\cprime{$'$} \def\cprime{$'$} \def\cprime{$'$} \def\cprime{$'$}
\providecommand{\bysame}{\leavevmode\hbox to3em{\hrulefill}\thinspace}
\providecommand{\MR}{\relax\ifhmode\unskip\space\fi MR }
\providecommand{\MRhref}[2]{%
  \href{http://www.ams.org/mathscinet-getitem?mr=#1}{#2}
}
\providecommand{\href}[2]{#2}
\begin{thebibliography}{100}

\bibitem{AMR}
R.~Abraham, J.~E. Marsden, and T.~Ratiu, \emph{Manifolds, tensor analysis, and
  applications}, second ed., Springer, New York, 1988. \MR{960687 (89f:58001)}

\bibitem{Abramovich_Aliprantis_2002}
Y.~A. Abramovich and C.~D. Aliprantis, \emph{An invitation to operator theory},
  Graduate Studies in Mathematics, vol.~50, American Mathematical Society,
  Providence, RI, 2002. \MR{1921782 (2003h:47072)}

\bibitem{Ache_2011arxiv}
A.~G. Ache, \emph{On the uniqueness of asymptotic limits of the {R}icci flow},
  arXiv:1211.3387.

\bibitem{AdamsFournier}
R.~A. Adams and J.~J.~F. Fournier, \emph{Sobolev spaces}, second ed.,
  Elsevier/Academic Press, Amsterdam, 2003. \MR{2424078 (2009e:46025)}

\bibitem{Apostol_math_analysis}
T.~M. Apostol, \emph{Mathematical analysis}, second ed., Addison-Wesley
  Publishing Co., Reading, Mass.-London-Don Mills, Ont., 1974. \MR{0344384}

\bibitem{Aubin_1998}
T.~Aubin, \emph{Some nonlinear problems in {R}iemannian geometry}, Springer,
  Berlin, 1998. \MR{1636569 (99i:58001)}

\bibitem{Berger_1977}
M.~Berger, \emph{Nonlinearity and functional analysis}, Academic Press, New
  York, 1977. \MR{0488101 (58 \#7671)}

\bibitem{Bourguignon_1981}
J.-P. Bourguignon, \emph{Formules de {W}eitzenb\"ock en dimension {$4$}},
  Riemannian geometry in dimension 4 ({P}aris, 1978/1979), Textes Math.,
  vol.~3, CEDIC, Paris, 1981, pp.~308--333. \MR{769143}

\bibitem{Bourguignon_1990}
J.-P. Bourguignon, \emph{The ``magic'' of {W}eitzenb\"ock formulas},
  Variational methods, Proceedings of the Conference on Variational Problems,
  Paris, 1988 (H.~Berestycki, J-M. Coron, and I.~Ekeland, eds.), Progress in
  nonlinear differential equations and their applications, vol.~4,
  Birkh\"auser, Boston, MA, 1990, pp.~251--271.

\bibitem{Bradlow_1990}
S.~B. Bradlow, \emph{Vortices in holomorphic line bundles over closed
  {K}\"ahler manifolds}, Comm. Math. Phys. \textbf{135} (1990), no.~1, 1--17.
  \MR{1086749 (92f:32053)}

\bibitem{Bradlow_1991}
S.~B. Bradlow, \emph{Special metrics and stability for holomorphic bundles with
  global sections}, J. Differential Geom. \textbf{33} (1991), 169--213.
  \MR{1085139 (91m:32031)}

\bibitem{BradlowGP}
S.~B. Bradlow and O.~Garc{\'{\i}}a-Prada, \emph{Non-abelian monopoles and
  vortices}, Geometry and physics ({A}arhus, 1995), Lecture Notes in Pure and
  Appl. Math., vol. 184, Dekker, New York, 1997, arXiv:alg-geom/9602010,
  pp.~567--589. \MR{1423193 (97k:53032)}

\bibitem{Brendle_2005}
S.~Brendle, \emph{Convergence of the {Y}amabe flow for arbitrary initial
  energy}, J. Differential Geom. \textbf{69} (2005), 217--278. \MR{2168505
  (2006e:53119)}

\bibitem{Brezis}
H.~Br{\'e}zis, \emph{Functional analysis, {S}obolev spaces and partial
  differential equations}, Universitext, Springer, New York, 2011. \MR{2759829
  (2012a:35002)}

\bibitem{BrockertomDieck}
T.~Br{\"o}cker and T.~tom Dieck, \emph{Representations of compact {L}ie
  groups}, Graduate Texts in Mathematics, vol.~98, Springer, New York, 1995.
  \MR{1410059 (97i:22005)}

\bibitem{Carlotto_Chodosh_Rubinstein_2015}
A.~Carlotto, O.~Chodosh, and Y.~A. Rubinstein, \emph{Slowly converging {Y}amabe
  flows}, Geom. Topol. \textbf{19} (2015), no.~3, 1523--1568, arXiv:1401.3738.
  \MR{3352243}

\bibitem{Chavel}
I.~Chavel, \emph{Eigenvalues in {R}iemannian geometry}, Pure and Applied
  Mathematics, vol. 115, Academic Press, Inc., Orlando, FL, 1984, Including a
  chapter by Burton Randol, With an appendix by Jozef Dodziuk. \MR{768584}

\bibitem{Chill_2003}
R.~Chill, \emph{On the {{\L}}ojasiewicz--{S}imon gradient inequality}, J.
  Funct. Anal. \textbf{201} (2003), 572--601. \MR{1986700 (2005c:26019)}

\bibitem{Chill_2006}
R.~Chill, \emph{The {{\L}}ojasiewicz--{S}imon gradient inequality in {H}ilbert
  spaces}, Proceedings of the 5th European-Maghrebian Workshop on Semigroup
  Theory, Evolution Equations, and Applications (M.~A. Jendoubi, ed.), 2006,
  pp.~25--36.

\bibitem{Chill_Fiorenza_2006}
R.~Chill and A.~Fiorenza, \emph{Convergence and decay rate to equilibrium of
  bounded solutions of quasilinear parabolic equations}, J. Differential
  Equations \textbf{228} (2006), 611--632. \MR{2289546 (2007k:35226)}

\bibitem{Chill_Haraux_Jendoubi_2009}
R.~Chill, A.~Haraux, and M.~A. Jendoubi, \emph{Applications of the
  {{\L}}ojasiewicz--{S}imon gradient inequality to gradient-like evolution
  equations}, Anal. Appl. (Singap.) \textbf{7} (2009), 351--372. \MR{2572850
  (2011a:35557)}

\bibitem{Chill_Jendoubi_2003}
R.~Chill and M.~A. Jendoubi, \emph{Convergence to steady states in
  asymptotically autonomous semilinear evolution equations}, Nonlinear Anal.
  \textbf{53} (2003), 1017--1039. \MR{1978032 (2004d:34103)}

\bibitem{Chill_Jendoubi_2007}
R.~Chill and M.~A. Jendoubi, \emph{Convergence to steady states of solutions of
  non-autonomous heat equations in {$\Bbb R^N$}}, J. Dynam. Differential
  Equations \textbf{19} (2007), 777--788. \MR{2350247 (2009h:35208)}

\bibitem{Clarke_1976}
Frank~H. Clarke, \emph{On the inverse function theorem}, Pacific J. Math.
  \textbf{64} (1976), no.~1, 97--102. \MR{0425047}

\bibitem{Colding_Minicozzi_2014sdg}
T.~H. Colding and W.~P. Minicozzi, II, \emph{{\L}ojasiewicz inequalities and
  applications}, Surveys in Differential Geometry \textbf{XIX} (2014), 63--82,
  arXiv:1402.5087.

\bibitem{Deimling_1985}
K.~Deimling, \emph{Nonlinear functional analysis}, Springer--Verlag, Berlin,
  1985. \MR{787404 (86j:47001)}

\bibitem{DK}
S.~K. Donaldson and P.~B. Kronheimer, \emph{The geometry of four-manifolds},
  Oxford University Press, New York, 1990.

\bibitem{Donaldson_Segal_2011}
S.~K. Donaldson and E.~Segal, \emph{Gauge theory in higher dimensions, {II}},
  Surveys in differential geometry. {V}olume {XVI}. {G}eometry of special
  holonomy and related topics, Surv. Differ. Geom., vol.~16, Int. Press,
  Somerville, MA, 2011, pp.~1--41. \MR{2893675}

\bibitem{Doria_2005}
C.~M. Doria, \emph{Boundary value problems for the second-order
  {S}eiberg-{W}itten equations}, Bound. Value Probl. (2005), no.~1, 73--91.
  \MR{2148375}

\bibitem{Doria_2006}
C.~M. Doria, \emph{Variational principle for the {S}eiberg-{W}itten equations},
  Contributions to nonlinear analysis, Progr. Nonlinear Differential Equations
  Appl., vol.~66, Birkh\"auser, Basel, 2006, pp.~247--261. \MR{2187807}

\bibitem{Evans2}
L.~C. Evans, \emph{Partial differential equations}, second ed., Graduate
  Studies in Mathematics, vol.~19, American Mathematical Society, Providence,
  RI, 2010. \MR{2597943 (2011c:35002)}

\bibitem{Feehan_yang_mills_gradient_flow_v4}
P.~M.~N. Feehan, \emph{Global existence and convergence of solutions to
  gradient systems and applications to {Y}ang--{M}ills gradient flow},
  submitted to a refereed monograph series on September 4, 2014,
  arXiv:1409.1525v4, xx+475 pages.

\bibitem{Feehan_lojasiewicz_inequality_ground_state}
P.~M.~N. Feehan, \emph{Optimal {{\L}}ojasiewicz--{S}imon inequalities and
  {M}orse--{B}ott {Y}ang--{M}ills energy functions}, submitted to a refereed
  journal on June 28, 2017, arXiv:1706.09349.

\bibitem{FeehanSlice}
P.~M.~N. Feehan, \emph{Critical-exponent {S}obolev norms and the slice theorem
  for the quotient space of connections}, Pacific J. Math. \textbf{200} (2001),
  no.~1, 71--118, arXiv:dg-ga/9711004. \MR{1863408}

\bibitem{Feehan_yangmillsenergygapflat}
P.~M.~N. Feehan, \emph{Energy gap for {Y}ang--{M}ills connections, {II}:
  {A}rbitrary closed {R}iemannian manifolds}, Adv. Math. \textbf{312} (2017),
  547--587, arXiv:1502.00668. \MR{3635819}

\bibitem{FL1}
P.~M.~N. Feehan and T.~G. Leness, \emph{{$\rm PU(2)$} monopoles. {I}.
  {R}egularity, {U}hlenbeck compactness, and transversality}, J. Differential
  Geom. \textbf{49} (1998), 265--410. \MR{1664908 (2000e:57052)}

\bibitem{FL2a}
P.~M.~N. Feehan and T.~G. Leness, \emph{{$\rm PU(2)$} monopoles and links of
  top-level {S}eiberg-{W}itten moduli spaces}, J. Reine Angew. Math.
  \textbf{538} (2001), 57--133, arXiv:math/0007190. \MR{1855754}

\bibitem{Feehan_Maridakis_Lojasiewicz-Simon_harmonic_maps_v6}
P.~M.~N. Feehan and M.~Maridakis, \emph{{{\L}}ojasiewicz--{S}imon gradient
  inequalities for analytic and {M}orse--{B}ott functions on {B}anach spaces
  and applications to harmonic maps}, submitted to a refereed journal on
  October 13, 2015, arXiv:1510.03817v6.

\bibitem{Feireisl_Laurencot_Petzeltova_2007}
E.~Feireisl, P.~Lauren{\c{c}}ot, and H.~Petzeltov{\'a}, \emph{On convergence to
  equilibria for the {K}eller-{S}egel chemotaxis model}, J. Differential
  Equations \textbf{236} (2007), 551--569. \MR{2322024 (2008c:35121)}

\bibitem{Feireisl_Simondon_2000}
E.~Feireisl and F.~Simondon, \emph{Convergence for semilinear degenerate
  parabolic equations in several space dimensions}, J. Dynam. Differential
  Equations \textbf{12} (2000), 647--673. \MR{1800136 (2002g:35116)}

\bibitem{Feireisl_Takac_2001}
E.~Feireisl and P.~Tak{\'a}{\v{c}}, \emph{Long-time stabilization of solutions
  to the {G}inzburg-{L}andau equations of superconductivity}, Monatsh. Math.
  \textbf{133} (2001), no.~3, 197--221. \MR{1861137 (2003a:35022)}

\bibitem{Folland}
G.~B. Folland, \emph{Introduction to partial differential equations}, second
  ed., Princeton University Press, Princeton, NJ, 1995. \MR{1357411
  (96h:35001)}

\bibitem{FU}
D.~S. Freed and K.~K. Uhlenbeck, \emph{Instantons and four-manifolds}, second
  ed., Mathematical Sciences Research Institute Publications, vol.~1, Springer,
  New York, 1991. \MR{1081321 (91i:57019)}

\bibitem{FrM}
R.~Friedman and John~W. Morgan, \emph{Smooth four-manifolds and complex
  surfaces}, Ergebnisse der Mathematik und ihrer Grenzgebiete (3) [Results in
  Mathematics and Related Areas (3)], vol.~27, Springer--Verlag, Berlin, 1994.
  \MR{1288304}

\bibitem{Frigeri_Grasselli_Krejcic_2013}
S.~Frigeri, M.~Grasselli, and P.~Krej{\v{c}}{\'{\i}}, \emph{Strong solutions
  for two-dimensional nonlocal {C}ahn-{H}illiard-{N}avier--{S}tokes systems},
  J. Differential Equations \textbf{255} (2013), no.~9, 2587--2614.
  \MR{3090070}

\bibitem{Froyshov_2008}
K.~A. Fr{\o}yshov, \emph{Compactness and gluing theory for monopoles}, Geometry
  \& Topology Monographs, vol.~15, Geometry \& Topology Publications, Coventry,
  2008, available at \url{msp.warwick.ac.uk/gtm/2008/15/}. \MR{2465077
  (2010a:57050)}

\bibitem{GilbargTrudinger}
D.~Gilbarg and N.~S. Trudinger, \emph{Elliptic partial differential equations
  of second order}, second ed., Grundlehren der Mathematischen Wissenschaften
  [Fundamental Principles of Mathematical Sciences], vol. 224, Springer-Verlag,
  Berlin, 1983. \MR{737190}

\bibitem{Gilkey2}
P.~B. Gilkey, \emph{Invariance theory, the heat equation, and the
  {A}tiyah--{S}inger index theorem}, second ed., Studies in Advanced
  Mathematics, CRC Press, Boca Raton, FL, 1995. \MR{1396308 (98b:58156)}

\bibitem{Grasselli_Wu_2013}
M.~Grasselli and H.~Wu, \emph{Long-time behavior for a hydrodynamic model on
  nematic liquid crystal flows with asymptotic stabilizing boundary condition
  and external force}, SIAM J. Math. Anal. \textbf{45} (2013), no.~3,
  965--1002. \MR{3048212}

\bibitem{Grasselli_Wu_Zheng_2009}
M.~Grasselli, H.~Wu, and S.~Zheng, \emph{Convergence to equilibrium for
  parabolic-hyperbolic time-dependent {G}inzburg-{L}andau-{M}axwell equations},
  SIAM J. Math. Anal. \textbf{40} (2008/09), no.~5, 2007--2033.

\bibitem{GroisserParkerGeometryDefinite}
D.~Groisser and Thomas~H. Parker, \emph{The geometry of the {Y}ang--{M}ills
  moduli space for definite manifolds}, J. Differential Geom. \textbf{29}
  (1989), 499--544. \MR{992329 (90f:58021)}

\bibitem{Guentner_1993}
E.~Guentner, \emph{{$K$}-homology and the index theorem}, Index theory and
  operator algebras ({B}oulder, {CO}, 1991), Contemp. Math., vol. 148, Amer.
  Math. Soc., Providence, RI, 1993, pp.~47--66. \MR{1228499 (94h:19006)}

\bibitem{Haraux_2012}
A.~Haraux, \emph{Some applications of the {{\L}}ojasiewicz gradient
  inequality}, Commun. Pure Appl. Anal. \textbf{11} (2012), 2417--2427.
  \MR{2912754}

\bibitem{Haraux_Jendoubi_1998}
A.~Haraux and M.~A. Jendoubi, \emph{Convergence of solutions of second-order
  gradient-like systems with analytic nonlinearities}, J. Differential
  Equations \textbf{144} (1998), 313--320. \MR{1616968 (99a:35182)}

\bibitem{Haraux_Jendoubi_2007}
A.~Haraux and M.~A. Jendoubi, \emph{On the convergence of global and bounded
  solutions of some evolution equations}, J. Evol. Equ. \textbf{7} (2007),
  449--470. \MR{2328934 (2008k:35480)}

\bibitem{Haraux_Jendoubi_2011}
A.~Haraux and M.~A. Jendoubi, \emph{The {{\L}}ojasiewicz gradient inequality in
  the infinite-dimensional {H}ilbert space framework}, J. Funct. Anal.
  \textbf{260} (2011), 2826--2842. \MR{2772353 (2012c:47168)}

\bibitem{Haraux_Jendoubi_Kavian_2003}
A.~Haraux, M.~A. Jendoubi, and O.~Kavian, \emph{Rate of decay to equilibrium in
  some semilinear parabolic equations}, J. Evol. Equ. \textbf{3} (2003),
  463--484. \MR{2019030 (2004k:35187)}

\bibitem{Haslhofer_2012cvpde}
R.~Haslhofer, \emph{Perelman's lambda-functional and the stability of
  {R}icci-flat metrics}, Calc. Var. Partial Differential Equations \textbf{45}
  (2012), 481--504. \MR{2984143}

\bibitem{Haslhofer_Muller_2014}
R.~Haslhofer and R.~M{\"u}ller, \emph{Dynamical stability and instability of
  {R}icci-flat metrics}, Math. Ann. \textbf{360} (2014), no.~1-2, 547--553,
  arXiv:1301.3219. \MR{3263173}

\bibitem{Haydys_2017arxiv}
Andriy Haydys, \emph{${G}_2$ instantons and the {S}eiberg--{W}itten monopoles},
  arXiv:1607.01763.

\bibitem{Haydys_2019}
Andriy Haydys, \emph{The infinitesimal multiplicities and orientations of the
  blow-up set of the {S}eiberg--{W}itten equation with multiple spinors}, Adv.
  Math. \textbf{343} (2019), 193--218, arXiv:1607.01763. \MR{3880858}

\bibitem{Haydys_Walpuski_2015}
Andriy Haydys and Thomas Walpuski, \emph{A compactness theorem for the
  {S}eiberg-{W}itten equation with multiple spinors in dimension three}, Geom.
  Funct. Anal. \textbf{25} (2015), no.~6, 1799--1821. \MR{3432158}

\bibitem{Hilgert_Neeb_structure_geometry_lie_groups}
J.~Hilgert and K.-H. Neeb, \emph{Structure and geometry of {L}ie groups},
  Springer Monographs in Mathematics, Springer, New York, 2012. \MR{3025417}

\bibitem{Hitchin_1987}
N.~J. Hitchin, \emph{The self-duality equations on a {R}iemann surface}, Proc.
  London Math. Soc. (3) \textbf{55} (1987), no.~1, 59--126. \MR{887284
  (89a:32021)}

\bibitem{Hong_2001}
M-C. Hong, \emph{Heat flow for the {Y}ang--{M}ills-{H}iggs field and the
  {H}ermitian {Y}ang--{M}ills-{H}iggs metric}, Ann. Global Anal. Geom.
  \textbf{20} (2001), 23--46. \MR{1846895 (2002h:53040)}

\bibitem{Hong_Schabrun_2010}
M-C. Hong and L.~Schabrun, \emph{Global existence for the {S}eiberg--{W}itten
  flow}, Comm. Anal. Geom. \textbf{18} (2010), 433--473. \MR{2747435
  (2012b:53139)}

\bibitem{Hormander_v3}
L.~H{\"o}rmander, \emph{The analysis of linear partial differential operators,
  {III}. {P}seudo-differential operators}, Springer, Berlin, 2007. \MR{2304165
  (2007k:35006)}

\bibitem{Huang_2006}
S.-Z. Huang, \emph{Gradient inequalities}, Mathematical Surveys and Monographs,
  vol. 126, American Mathematical Society, Providence, RI, 2006. \MR{2226672
  (2007b:35035)}

\bibitem{Huang_Takac_2001}
S.-Z. Huang and P.~Tak{\'a}{\v{c}}, \emph{Convergence in gradient-like systems
  which are asymptotically autonomous and analytic}, Nonlinear Anal.
  \textbf{46} (2001), 675--698. \MR{1857152 (2002f:35125)}

\bibitem{IrwinThesis}
C.~A. Irwin, \emph{Bubbling in the harmonic map heat flow}, {Ph.D.} thesis,
  Stanford University, Palo Alto, CA, 1998. \MR{2698290}

\bibitem{Jendoubi_1998jfa}
M.~A. Jendoubi, \emph{A simple unified approach to some convergence theorems of
  {L}. {S}imon}, J. Funct. Anal. \textbf{153} (1998), 187--202. \MR{1609269
  (99c:35101)}

\bibitem{Jost_Peng_Wang_1996}
J.~Jost, X.~Peng, and G.~Wang, \emph{Variational aspects of the
  {S}eiberg--{W}itten functional}, Calc. Var. Partial Differential Equations
  \textbf{4} (1996), 205--218. \MR{1386734 (97d:58055)}

\bibitem{Jost_postmodern_analysis}
J\"{u}rgen Jost, \emph{Postmodern analysis}, third ed., Universitext,
  Springer-Verlag, Berlin, 2005. \MR{2166001}

\bibitem{Joyce_compact_manifolds_special_holonomy}
D.~Joyce, \emph{Compact manifolds with special holonomy}, Oxford Mathematical
  Monographs, Oxford University Press, Oxford, 2000. \MR{1787733}

\bibitem{KadisonRingrose1}
R.~V. Kadison and J.~R. Ringrose, \emph{Fundamentals of the theory of operator
  algebras. {V}ol. {I}}, Pure and Applied Mathematics, vol. 100, Academic
  Press, Inc. [Harcourt Brace Jovanovich, Publishers], New York, 1983,
  Elementary theory. \MR{719020}

\bibitem{Kato}
T.~Kato, \emph{Perturbation theory for linear operators}, second ed., Springer,
  New York, 1984.

\bibitem{Kovalev_2003}
A.~Kovalev, \emph{Twisted connected sums and special {R}iemannian holonomy}, J.
  Reine Angew. Math. \textbf{565} (2003), 125--160. \MR{2024648}

\bibitem{Kroncke_2013arxiv}
K.~Kr{\"o}ncke, \emph{Stability of {E}instein metrics under {R}icci flow},
  Commun. Anal. Geom., to appear, arXiv:1312.2224.

\bibitem{Kroncke_2015cvpde}
K.~Kr{\"o}ncke, \emph{Stability and instability of {R}icci solitons}, Calc.
  Var. Partial Differential Equations \textbf{53} (2015), no.~1-2, 265--287,
  arXiv:1403.3721. \MR{3336320}

\bibitem{KMBook}
P.~B. Kronheimer and Tomasz~S. Mrowka, \emph{Monopoles and three-manifolds},
  Cambridge University Press, Cambridge, 2007. \MR{2388043 (2009f:57049)}

\bibitem{KwonThesis}
H.~Kwon, \emph{Asymptotic convergence of harmonic map heat flow}, {Ph.D}.
  thesis, Stanford University, Palo Alto, CA, 2002. \MR{2703296}

\bibitem{Lang_analysis}
S.~Lang, \emph{Real and functional analysis}, third ed., Graduate Texts in
  Mathematics, vol. 142, Springer--Verlag, New York, 1993. \MR{1216137}

\bibitem{Lawson}
H.~B. Lawson, Jr., \emph{The theory of gauge fields in four dimensions}, CBMS
  Regional Conference Series in Mathematics, vol.~58, Published for the
  Conference Board of the Mathematical Sciences, Washington, DC; by the
  American Mathematical Society, Providence, RI, 1985. \MR{799712}

\bibitem{LM}
H.~B. Lawson, Jr. and M.-L. Michelsohn, \emph{Spin geometry}, Princeton
  Mathematical Series, vol.~38, Princeton University Press, Princeton, NJ,
  1989. \MR{1031992 (91g:53001)}

\bibitem{Li_Zhang_2011}
J.~Li and X.~Zhang, \emph{The gradient flow of {H}iggs pairs}, J. Eur. Math.
  Soc. (JEMS) \textbf{13} (2011), 1373--1422. \MR{2825168 (2012m:53043)}

\bibitem{Liu_Yang_2010}
Q.~Liu and Y.~Yang, \emph{Rigidity of the harmonic map heat flow from the
  sphere to compact {K}\"ahler manifolds}, Ark. Mat. \textbf{48} (2010),
  121--130. \MR{2594589 (2011a:53066)}

\bibitem{Lojasiewicz_1965}
S.~{\L}ojasiewicz, \emph{Ensembles semi-analytiques},  (1965), Publ. Inst.
  Hautes Etudes Sci., Bures-sur-Yvette. LaTeX version by M. Coste, August 29,
  2006 based on mimeographed course notes by S. {\L}ojasiewicz, available at
  \url{perso.univ-rennes1.fr/michel.coste/Lojasiewicz.pdf}.

\bibitem{Lojasiewicz_1984}
S.~{\L}ojasiewicz, \emph{Sur les trajectoires du gradient d'une fonction
  analytique}, Geometry seminars, 1982--1983 ({B}ologna, 1982/1983), Univ.
  Stud. Bologna, Bologna, 1984, pp.~115--117. \MR{771152 (86m:58023)}

\bibitem{Maly_Ziemer_1997}
J.~Mal{\'y} and W.~P. Ziemer, \emph{Fine regularity of solutions of elliptic
  partial differential equations}, Mathematical Surveys and Monographs,
  vol.~51, American Mathematical Society, Providence, RI, 1997. \MR{1461542
  (98h:35080)}

\bibitem{McDuffSalamon1}
Dusa McDuff and Dietmar Salamon, \emph{{$J$}-holomorphic curves and symplectic
  topology}, American Mathematical Society Colloquium Publications, vol.~52,
  American Mathematical Society, Providence, RI, 2004. \MR{2045629}

\bibitem{Melrose_Lectures_pseudodifferential_operators}
R.~B. Melrose, \emph{Lectures on pseudodifferential operators}, Massachusetts
  Institute of Technology, November 2006,
  \url{www-math.mit.edu/~rbm/Lecture_notes.html}.

\bibitem{MMR}
John~W. Morgan, Tomasz~S. Mrowka, and Daniel Ruberman, \emph{The {$L^2$}-moduli
  space and a vanishing theorem for {D}onaldson polynomial invariants},
  Monographs in Geometry and Topology, vol.~2, International Press, Cambridge,
  MA, 1994. \MR{1287851 (95h:57039)}

\bibitem{Mrowka_Rollin_2006}
Tomasz~S. Mrowka and Yann Rollin, \emph{Legendrian knots and monopoles},
  Algebr. Geom. Topol. \textbf{6} (2006), 1--69. \MR{2199446}

\bibitem{NicolaescuSWNotes}
L.~I. Nicolaescu, \emph{Notes on {S}eiberg--{W}itten theory}, Graduate Studies
  in Mathematics, vol.~28, American Mathematical Society, Providence, RI, 2000.
  \MR{1787219 (2001k:57037)}

\bibitem{OTVortex}
C.~Okonek and A.~Teleman, \emph{The coupled {S}eiberg-{W}itten equations,
  vortices, and moduli spaces of stable pairs}, Internat. J. Math. \textbf{6}
  (1995), no.~6, 893--910, arXiv:alg-geom/9505012. \MR{1354000}

\bibitem{Papi_2005}
Marco Papi, \emph{On the domain of the implicit function and applications}, J.
  Inequal. Appl. (2005), no.~3, 221--234. \MR{2206097}

\bibitem{ParkerGauge}
Thomas~H. Parker, \emph{Gauge theories on four-dimensional {R}iemannian
  manifolds}, Comm. Math. Phys. \textbf{85} (1982), 563--602. \MR{677998
  (84b:58036)}

\bibitem{PTLocal}
Victor~Ya. Pidstrigach and Andre\u{i}~Nikolaevic Tyurin, \emph{Localisation of
  {D}onaldson invariants along the {S}eiberg--{W}itten classes},
  arXiv:dg-ga/9507004.

\bibitem{Reed_Simon_v1}
M.~Reed and B.~Simon, \emph{Methods of modern mathematical physics. {I}},
  second ed., Academic Press, New York, 1980, Functional analysis. \MR{751959
  (85e:46002)}

\bibitem{Rade_1992}
J.~R\r{a}de, \emph{On the {Y}ang--{M}ills heat equation in two and three
  dimensions}, J. Reine Angew. Math. \textbf{431} (1992), 123--163. \MR{1179335
  (94a:58041)}

\bibitem{Rudin}
W.~Rudin, \emph{Functional analysis}, second ed., International Series in Pure
  and Applied Mathematics, McGraw-Hill, Inc., New York, 1991. \MR{1157815}

\bibitem{Rybka_Hoffmann_1998}
P.~Rybka and K.-H. Hoffmann, \emph{Convergence of solutions to the equation of
  quasi-static approximation of viscoelasticity with capillarity}, J. Math.
  Anal. Appl. \textbf{226} (1998), 61--81. \MR{1646449 (99h:35146)}

\bibitem{Rybka_Hoffmann_1999}
P.~Rybka and K.-H. Hoffmann, \emph{Convergence of solutions to
  {C}ahn-{H}illiard equation}, Comm. Partial Differential Equations \textbf{24}
  (1999), 1055--1077. \MR{1680877 (2001a:35028)}

\bibitem{SalamonSWBook}
D.~A. Salamon, \emph{Spin geometry and {S}eiberg--{W}itten invariants},
  unpublished book, available at \url{math.ethz.ch/~salamon/publications.html}.

\bibitem{Simon_1983}
L.~Simon, \emph{Asymptotics for a class of nonlinear evolution equations, with
  applications to geometric problems}, Ann. of Math. (2) \textbf{118} (1983),
  525--571. \MR{727703 (85b:58121)}

\bibitem{Simon_1985}
L.~Simon, \emph{Isolated singularities of extrema of geometric variational
  problems}, Lecture Notes in Math., vol. 1161, Springer, Berlin, 1985.
  \MR{821971 (87d:58045)}

\bibitem{Simpson_1988}
C.~T. Simpson, \emph{Constructing variations of {H}odge structure using
  {Y}ang--{M}ills theory and applications to uniformization}, J. Amer. Math.
  Soc. \textbf{1} (1988), no.~4, 867--918. \MR{944577 (90e:58026)}

\bibitem{Stakgold_Holst}
I.~Stakgold and M.~Holst, \emph{Green's functions and boundary value problems},
  third ed., Wiley, Hoboken, NJ, 2011.

\bibitem{Takac_2000}
P.~Tak{\'a}{\v{c}}, \emph{Stabilization of positive solutions for analytic
  gradient-like systems}, Discrete Contin. Dynam. Systems \textbf{6} (2000),
  947--973. \MR{1788263 (2001i:35162)}

\bibitem{TauFrame}
C.~H. Taubes, \emph{A framework for {M}orse theory for the {Y}ang--{M}ills
  functional}, Invent. Math. \textbf{94} (1988), 327--402. \MR{958836
  (90a:58035)}

\bibitem{ToppingThesis}
P.~Topping, \emph{The harmonic map heat flow from surfaces}, {Ph.D.} thesis,
  University of Warwick, United Kingdom, April 1996.

\bibitem{Topping_1997}
P.~Topping, \emph{Rigidity in the harmonic map heat flow}, J. Differential
  Geom. \textbf{45} (1997), 593--610. \MR{1472890 (99d:58050)}

\bibitem{Treves1}
F.~Tr{\`e}ves, \emph{Introduction to pseudodifferential and {F}ourier integral
  operators. {V}ol. 1}, Plenum Press, New York, 1980.

\bibitem{UhlLp}
K.~K. Uhlenbeck, \emph{Connections with {$L^{p}$} bounds on curvature}, Comm.
  Math. Phys. \textbf{83} (1982), 31--42. \MR{648356 (83e:53035)}

\bibitem{UhlRem}
K.~K. Uhlenbeck, \emph{Removable singularities in {Y}ang--{M}ills fields},
  Comm. Math. Phys. \textbf{83} (1982), 11--29. \MR{648355 (83e:53034)}

\bibitem{Varadarajan}
V.~S. Varadarajan, \emph{Lie groups, {L}ie algebras, and their
  representations}, Graduate Texts in Mathematics, vol. 102, Springer--Verlag,
  New York, 1984, Reprint of the 1974 edition. \MR{746308 (85e:22001)}

\bibitem{Volpert1}
V.~Volpert, \emph{Elliptic partial differential equations. {V}olume 1:
  {F}redholm theory of elliptic problems in unbounded domains}, Monographs in
  Mathematics, vol. 101, Birkh\"auser/Springer Basel AG, Basel, 2011.
  \MR{2778694 (2012g:35003)}

\bibitem{WalpuskiThesis}
Thomas Walpuski, \emph{Gauge theory on ${G}_2$-manifolds}, {Ph.D}. thesis,
  Imperial College London, United Kingdom, 2013,
  \url{http://hdl.handle.net/10044/1/14365}.

\bibitem{Warner}
F.~W. Warner, \emph{Foundations of differentiable manifolds and {L}ie groups},
  Graduate Texts in Mathematics, vol.~94, Springer, New York, 1983. \MR{722297
  (84k:58001)}

\bibitem{Wehrheim_2004}
K.~Wehrheim, \emph{Uhlenbeck compactness}, EMS Series of Lectures in
  Mathematics, European Mathematical Society (EMS), Z\"urich, 2004. \MR{2030823
  (2004m:53045)}

\bibitem{Wilkin_2008}
G.~Wilkin, \emph{Morse theory for the space of {H}iggs bundles}, Comm. Anal.
  Geom. \textbf{16} (2008), 283--332. \MR{2425469 (2010f:53115)}

\bibitem{Wu_Xu_2013}
H.~Wu and X.~Xu, \emph{Strong solutions, global regularity, and stability of a
  hydrodynamic system modeling vesicle and fluid interactions}, SIAM J. Math.
  Anal. \textbf{45} (2013), no.~1, 181--214. \MR{3032974}

\bibitem{Wu_1988}
H.~H. Wu, \emph{The {B}ochner technique in differential geometry}, Math. Rep.
  \textbf{3} (1988), no.~2, i--xii and 289--538. \MR{1079031 (91h:58031)}

\bibitem{Yang_2003aim}
B.~Yang, \emph{The uniqueness of tangent cones for {Y}ang--{M}ills connections
  with isolated singularities}, Adv. Math. \textbf{180} (2003), 648--691.
  \MR{2020554 (2004m:58026)}

\bibitem{Yosida}
K.~Yosida, \emph{Functional analysis}, sixth ed., Springer, New York, 1980.

\bibitem{Zeidler_nfaa_v1}
E.~Zeidler, \emph{Nonlinear functional analysis and its applications, {I}.
  {F}ixed-point theorems}, Springer, New York, 1986. \MR{816732 (87f:47083)}

\bibitem{Zeidler_nfaa_v2a}
E.~Zeidler, \emph{Nonlinear functional analysis and its applications.
  {II}/{A}}, Springer--Verlag, New York, 1990, Linear monotone operators.
  \MR{1033497 (91b:47001)}

\end{thebibliography}

\bibliographystyle{amsplain-nodash}

\end{document}